%% file: 2adicDihedral.tex
\documentclass{amsart}

\usepackage{amsmath,amssymb,amsthm,comment}
\usepackage[lite,shortalphabetic,nobysame]{amsrefs}
\usepackage[all]{xy}
\usepackage[usenames]{color}
\usepackage{mathrsfs}
\usepackage[margin=1.0in]{geometry}
\usepackage{hyperref}

\newtheorem{lem}[subsubsection]{Lemma}
\newtheorem{thm}[subsubsection]{Theorem}
\newtheorem{cor}[subsubsection]{Corollary}
\newtheorem{prop}[subsubsection]{Proposition}

\newtheorem*{TheMainThm}{Theorem}
\newtheorem*{EllCurveCor}{Corollary}

\newcommand{\Z}{\mathbb{Z}}
\newcommand{\Q}{\mathbb{Q}}
\newcommand{\R}{\mathbb{R}}
\newcommand{\C}{\mathbb{C}}

\newcommand{\F}{\mathbb{F}}
\newcommand{\A}{\mathbb{A}}
\newcommand{\bbP}{\mathbb{P}}

\newcommand{\Qbar}{\overline{\mathbb{Q}}}

\newcommand{\GL}{\mathrm{GL}}
\newcommand{\SL}{\mathrm{SL}}
\newcommand{\PGL}{\mathrm{PGL}}
\newcommand{\Hom}{\mathrm{Hom}}
\newcommand{\End}{\mathrm{End}}
\newcommand{\Gal}{\mathrm{Gal}}
\newcommand{\Spec}{\mathrm{Spec}\,}
\newcommand{\Spf}{\mathrm{Spf}}

\newcommand{\Ind}{\mathrm{Ind}}
\newcommand{\tr}{\mathrm{tr}\,}
\newcommand{\coker}{\mathrm{coker}}
\newcommand{\im}{\mathrm{im}}
\newcommand{\Sym}{\mathrm{Sym}}
\newcommand{\Frob}{\mathrm{Frob}}
\newcommand{\Nm}{\mathrm{Nm}}

\newcommand{\CNL}{\mathrm{CNL}}
\newcommand{\CNLO}{\mathrm{CNL}_{\mathcal{O}}}

\newcommand{\inv}{\mathrm{inv}}
\newcommand{\calO}{\mathcal{O}}
\newcommand{\frakp}{\mathfrak{p}}
\newcommand{\frakm}{\mathfrak{m}}
\newcommand{\val}{\mathrm{val}}
\newcommand{\Aut}{\mathrm{Aut}}

\newcommand{\rhobar}{\overline{\rho}}
\newcommand{\chibar}{\overline{\chi}}
\newcommand{\Ad}{\mathrm{Ad}}
\newcommand{\depth}{\mathrm{depth}}
\newcommand{\rk}{\mathrm{rk}}
\newcommand{\rank}{\mathrm{rank}}
\newcommand{\Nil}{\mathrm{Nil}}

\title[Modularity of nearly ordinary $2$-adic residually dihedral Galois representations]{Modularity of nearly ordinary $2$-adic residually dihedral Galois representations}
\author{Patrick B. Allen}
\date{}
\address{Department of Mathematics, Northwestern University,
2033 Sheridan Road, Evanston, IL 60208-2730, USA}
\email{pballen@math.northwestern.edu}

\begin{document}

\maketitle

\begin{abstract}

We prove modularity of some two dimensional, $2$-adic Galois representations over a totally real field that are nearly ordinary at all places above $2$ and that are residually dihedral. We do this by employing the strategy of Skinner and Wiles, using Hida families, together with the $2$-adic patching method of Khare and Wintenberger. As an application we deduce modularity of some elliptic curves over totally real fields that have good ordinary or multiplicative reduction at places above $2$.

\end{abstract}

\setcounter{tocdepth}{2}

\tableofcontents

\input{Intro/IntroIntro}

\section*{Acknowledgements}

I would like to thank my PhD advisor Chandrashekhar Khare for suggesting to me this problem and for his patience and guidance throughout its resolution. I would like to thank the entire number theory group at UCLA for providing a stimulating place to learn and work. I would like to thank David Geraghty, Gebhard B\"{o}ckle, Mark Kisin and Andrew Snowden for helpful conversations and correspondence. I would also like to thank Frank Calegari, Matthew Emerton, and Toby Gee for some helpful comments and corrections on a previous draft of this paper. I would like to thank Anush Tserunyan for constant moral support and encouragement. Lastly, I would like to thank the referee for helpful comments and corrections.

\input{DefTheory/DefTheoryIntro}

\input{HidaFam/HidaFamIntro}

\input{AuxPrimes/AuxPrimesIntro}

\input{Patching/PatchingIntro}

\input{MainThms/MainThmsIntro}

\bibliographystyle{uclathes}

\end{document}

%% file: Intro/IntroIntro.tex
\section*{Introduction}\label{Intro}

The Fontaine-Mazur-Langlands conjecture (in a special case) predicts that totally odd, geometric, absolutely irreducible $p$-adic representations of $\mathrm{Gal}(\overline{\mathbb{Q}}/F)$, for $F\subset \Qbar$ a totally real number field, arise from Hilbert modular forms. When $F=\Q$ and $p$ is odd, this conjecture has been resolved in almost all cases. When $[F:\Q] > 1$ and $p$ is odd, much is still known (especially if $p$ is split in $F$). Less is known when $p=2$ because the Taylor-Wiles method encounters technical difficulties. The first $2$-adic modularity lifting theorem (known to the author) was proved by Dickinson \cite{Dickinson2adic} for $F=\Q$. Recently, Khare and Wintenberger \cite{KW2} and Kisin \cite{Kisin2adic} developed an extension of the Taylor-Wiles method to prove modularity of a wide class of $2$-adic representations, and this was essential in their proof of Serre's conjecture. It's interesting to note that, since proofs of the known cases of the Fontaine-Mazur-Langlands conjecture for $\GL_2$ use Serre's conjecture as an important ingredient, $2$-adic modularity lifting theorems have had applications in proving modularity of $p$-adic representations even when $p$ is odd.

Due to their technical nature, the current $2$-adic modularity lifting theorems require stronger assumptions than their $p>2$ counterparts. One such assumption is that the residual representation has non-solvable image (the main theorem of \cite{Dickinson2adic} has a gap in the residually solvable case; namely in the proof of Lemma 40 on page 369 of \textit{loc. cit.} it is incorrectly assumed that certain matrices have distinct eigenvalues). A $2$-adic modularity lifting theorem in the residually solvable case is desirable for a number of reasons, one of which is that the $2$-adic representations arising from elliptic curves are always residually solvable. The main result of this paper is such a theorem.

This is done by employing a strategy of Skinner and Wiles. They showed, in the $p>2$ case, that certain representations of $\mathrm{Gal}(\Qbar/F)$ for (most) totally real fields obeying an ordinarity hypothesis are modular, assuming only that the residual representation is absolutely irreducible, cf. \cite{SWirreducible}.  Usually, one assumes further that the residual representation is absolutely irreducible when restricted to the finite index subgroup $\Gal(\Qbar/F(\zeta_p))$. Their strategy is to use Hida families in order to move to a new ``residual" representation where one can assume stronger conditions. We carry out the Skinner and Wiles method in the $2$-adic case and prove modularity of representations that are ordinary at places above $2$ and whose reductions are absolutely irreducible with solvable image. It is worth mentioning that this strategy is based in turn on a strategy of theirs for proving modularity in the more difficult case of residually reducible representations, albeit for less general totally real fields, c.f. \cite{SWreducible}.

The main result of this paper is the following theorem. In assumption (2) below, the isomorphism of local class field theory is normalized so that uniformizers correspond to arithmetic Frobenii. Also, the extension $L/F$ in assumption (5) is unique since any absolutely irreducible, $2$-dimensional, mod $2$ representation with solvable image is dihedral, by the classification of subgroups of $\mathrm{PGL}_2(\overline{\F}_2)$, and the image of a mod $2$ dihedral representation has order not divisible by $4$.

\input{Intro/IntroThms}

\subsection*{Strategy} I will now elaborate on the strategy of the proof.

\input{Intro/Strategy}

\subsection*{Outline}\label{outlinesec}

The paper is organized as follows.

\input{Intro/Outline}

\subsection*{Notation and conventions}\label{conventionsec}

\input{Intro/Conventions}

%% file: Intro/IntroThms.tex
\begin{TheMainThm}

Let $F$ be a totally real subfield of $\overline{\Q}$.  Let $J_F$ denote set of embeddings $F \hookrightarrow \overline{\Q}$. Fix embeddings $\overline{\Q}\hookrightarrow \Qbar_2$ and $\overline{\Q}\hookrightarrow \C$. Via these embeddings we view $J_F$ as the set of embeddings $\{F \hookrightarrow \R\}$ as well as the set of embeddings $\{F \hookrightarrow \Qbar_2\}$. For any $v|2$ in $F$, let $J_{F_v}\subseteq J_F$ denote the subset of $\tau$ that give rise to $v$. We identify $J_{F_v}$ with the set of embedding $F_v \hookrightarrow \Qbar_2$.

Let
	\[ \rho: G_F \longrightarrow \GL_2(\Qbar_2)
	\]
be a continuous representation unramified outside finitely many primes. Assume there is some $(\mathbf{k},\mathbf{w}) \in J_F^2$, such that $k_\tau \ge 2$ for each $\tau\in J_F$ and $ w= k_\tau + 2w_\tau$ is independent of $\tau$, and such that:
\begin{enumerate}
	\item $\det\rho = \phi\epsilon_2^{w-1}$, with $\phi$ a finite order character and $\epsilon_2$ the $2$-adic cyclotomic character;
	\item for each $v|2$, $\rho|_{G_v} \cong \left( \begin{array}{cc} \ast & \ast \\ 
			& \chi_v \end{array} \right)$ and $\chi_v(y) = \prod_{\tau \in J_{F_v}} y^{-w_\tau}$ on some open subgroup of $\mathcal{O}_{F_v}^\times$, viewing $\chi_v$ as a character of $F_v^\times$ via class field theory;
	\item for each choice of complex conjugation $c$, $\det\rho(c) = -1$.
\end{enumerate}
	Let $\overline{\rho}: G_F \rightarrow \GL_2(\overline{\F_2})$ denote the residual representation associated to $\rho$. We also assume:
\begin{enumerate}
	\item[(4)] $\overline{\rho}$ is absolutely irreducible with solvable image;
	\item[(5)] letting $L/F$ denote the unique quadratic extension such that $\overline{\rho}|_{G_L}$ is abelian, if $L/F$ is CM, then there is some $v|2$ in $F$ that does not split in $L$.
\end{enumerate}

	Under these assumptions $\rho$ is modular, i.e. there is a $2$-nearly ordinary, regular algebraic, cuspidal automorphic representation $\pi$ of $\GL_2(\A_F)$ such that $\rho \cong \rho_\pi$.
\end{TheMainThm}

We will actually prove this with assumption (4) replaced by the assumption that $\rhobar$ is absolutely irreducible and admits a $2$-nearly ordinary modular lift, c.f. \ref{finalthm}. A theorem of Wiles allows us to produce ordinary lifts in the residually dihedral case, cf. \ref{OrdinaryLift}, and the main theorem follows.

First, some comments on the assumptions in the theorem. Condition (2) is the \textit{near-ordinarity} condition, and it is essential to the method as we use Hida families. We do not assume distinguishedness, as in \cite{SWirreducible}, because of the advances in modularity lifting theorems since their paper; namely the use of framed deformation rings. The distinguishedness of certain deformations of $\overline{\rho}$ will, however, play an imoportant role in the argument. Namely, to show that certain points of the local deformation ring are normal. This will be elaborated on below. The condition (4), together with allowing general totally real fields $F$, is the main improvement in this paper. As mentioned above, both the results of \cite{KW2} and \cite{Kisin2adic} exclude the residually solvable case, and they also assume $F$ is unramified above $2$ in the potentially ordinary case. Condition (5) is technical and is related to the fact that when the extension $L/F$ is CM and every $v|2$ in $F$ splits in $L$, Hida's universal nearly ordinary Hecke algebra has CM components. I will elaborate on this below.

The main theorem has the following corollary.

\begin{EllCurveCor}

Let $F$ be a totally real field and let $E$ be an elliptic curve over $F$ with $j$-invariant $j_E$. Let $\Delta$ be the discriminant of some Weierstrass equation defining $E$. Assume:
\begin{enumerate} 
	\item for every $v|2$ in $F$, the valuation of $j_E$ at $v$ is $\le 0$;
	\item $E$ has no $2$-torsion defined over $F$ and $\Delta$ is not a square in $F$;
	\item if $\Delta$ is totally negative, then there is some $v|2$ in $F$ such that $\Delta$ is not a square in $F_v$.
\end{enumerate}
	Then $E$ is modular.
\end{EllCurveCor}

This essentially follows immediately from the main theorem; the details are given in \ref{MainCorProof}. Some examples of elliptic curves satisfying the assumption are given in \ref{TheExamples}. Note that the discriminants of two different Weierstrass equations differ by a square in $F$; so, the assumptions of the corollary are insensitive to the choice of Weierstrass equation.

%% file: Intro/Strategy.tex
We consider a certain large deformation ring where at places above $2$ we enforce no $p$-adic Hodge theory conditions but demand that the representations be reducible. Using standard arguments, we show that there is a surjective morphism from our large deformation ring $R$ (tensored with a certain Iwasawa algebra) to a $2$-nearly ordinary Hecke algebra $\mathbf{T}$, and we call a prime ideal of $R$ \textit{pro-modular} if it is the pullback of one from $\mathbf{T}$. If we prove that all prime ideals of $R$ are pro-modular, then Hida's classicality theorem implies that any representation satisfying (1), (2), and (3) in the statement of the theorem is a classical point of this large Hecke algebra.

In order to show that all prime ideals of $R$ are pro-modular, following Skinner and Wiles, we prove two things:
\begin{enumerate}
	\item there exist certain dimension $1$ characteristic $2$ pro-modular prime ideals $\mathfrak{p}$ of the Hecke algebra, that we will call ``nice primes" below, such that we can adapt the $2$-adic Taylor-Wiles method to prove the localizations and completions of $R$ and $\mathbf{T}$ at such primes are isomorphic;
	\item that every minimal prime of $R$ is contained in a nice prime.
\end{enumerate}

To prove step (1), there are two main difficulties in combining the methods of Skinner and Wiles with those of Khare and Wintenberger: namely, proving the existence of Taylor-Wiles primes and performing the patching. The reason why the $2$-adic Taylor-Wiles method fails in the residually dihedral case is because the image is too small. We thus need to ensure that our nice primes do not generate dihedral deformations; as in \cite{SWirreducible}, we require that they have dimension $1$, characteristic $2$, and when the universal deformation is specialized at $\mathfrak{p}$, the resulting representation is non-dihedral. The calculations of \cite{SWreducible} and \cite{SWirreducible} bounding the size of a certain cohomology group do not work in the $p=2$ case because they rely heavily on complex conjugation acting semisimply. To overcome this we further require that the image of the representation associated to our nice prime contains a non-trivial unipotent element. This allows us to use of a result of Pink \cite{PinkCompactSubs} to show that the image of the Galois representation attached to one of these prime ideals is open in $\SL_2(K_0)$, where $K_0$ is a finite index subfield of the residue field of $\mathfrak{p}$. This fact allows us to explicitly compute the cohomology groups in question, and to show the existence of the Taylor-Wiles primes. As in the work of Skinner and Wiles, there is a technicality that must be dealt with all the while; namely, we cannot adapt the Taylor-Wiles patching argument directly to the localized deformation rings and Hecke algebras. This is because at the heart of the patching argument is a simple diagonalization in the spirit of Cantor, based on the fact that we have infinitely many objects, each being finite. This finiteness is lost if one tries to perform patching after localizing, so we must perform the patching integrally, and then localize after taking limiting objects. Because things are done integrally, we must take care to control torsion submodules of the cohomology groups we compute. In particular, we will have to ensure that the torsion subgroups of certain cohomology groups and Hecke modules do not depend on the choice of Taylor-Wiles primes.

After carrying out the patching there is another complication that arises not present in the work of Skinner and Wiles. In the Taylor-Wiles argument as improved by Kisin, the point of the patching is to create a limiting object that, for dimension reasons, is isomorphic to a certain power series ring over a local deformation ring. One then uses a description of the minimal primes of the local deformation ring in order to show that the Hecke module in question is supported on the component of the universal deformation ring containing the point corresponding to the representation we wish to prove is modular. In our situation, the dimension argument can only be applied after localizing and completing the appropriate rings at our fixed nice prime. The completion may increase the number of minimal primes of our local deformation ring, and if we cannot describe them anymore, then we cannot hope to say our Hecke module is supported on the ones we want. In order to avoid this, we require further that the nice primes lie over the distinguished locus of the local deformation ring at all places above $2$. This allows us to show that the pullback of nice primes to the local deformation ring lie in the normal locus. Then the completion of the localized local deformation ring will again be normal, hence a domain.

In order to show that the local deformation ring is normal at nice primes, it is not enough to analyze its characteristic zero points. We must know something about its mod $2$ points as well. In the case of the local deformation ring at places above $2$, we use the Deumu\u{s}kin relation to explicitly compute the deformation ring on the locus of distinguished deformations. For the local deformation rings at primes not above $2$, we use recent ideas of Snowden that allow us to show the deformation ring mod $2$ may be thought of as representing a certain moduli space of unipotent matrices, whose structure can be described in detail.

The step (2) is carried out as in the work of Skinner and Wiles, except that we use a connectedness theorem of Gothendieck instead of the theorem of Mme. Raynaud. Grothendieck's theorem is more readily applicable to our situation, as we do not show that our local deformation ring satisfies Serre's property $(S_2)$, which is necessary (at least on a certain open locus) to apply Raynaud's theorem. We first show that a nice prime exists. Using step (1) we know that any irreducible component containing it is pro-modular. Let $C$ be one such component and let $C'$ be any other component. The connectedness theorem of Grothendieck implies that there is a chain of irreducible components $C = C_0,\ldots,C_m = C'$ such that the dimension of $C_i\cap C_{i+1}$ is large. Using this, we show that if $C_i$ is pro-modular, then $C_i\cap C_{i+1}$ contains a nice prime; applying step (1) again we deduce that $C_{i+1}$ is pro-modular. Continuing in this way, we deduce the pro-modularity of $C'$. 

In order to show that $C_i\cap C_{i+1}$ having large dimension implies that it contains a nice prime, we in particular need that $C_i \cap C_{i+1}$ is not contained in the dihedral locus. This is why we need the extra condition in the CM case. If $L/F$ is CM and every place above $2$ in $F$ splits in the quadratic extension $L/F$, the dihedral locus will form an irreducible component of the Hecke algebra, hence conjecturally also of the deformation ring. In this case, even if we start out with non-dihedral components $C$ and $C'$, the author does not know how to guarantee that the $C_i$ appearing above after applying Grothendieck's theorem are non-dihedral. If our extra assumption is satisfied, i.e. that there is some $v|2$ that does not split in $L$, then having large dimension implies that our representation is distinguished at this $v$, which then implies it is not dihedral. In the case that $L/F$ is not CM, we use known cases of Leopoldt's conjecture to base change to a situation where the dihedral locus has sufficiently large codimension so that it doesn't contain any of the intersections $C_i\cap C_{i+1}$.

%% file: Intro/Outline.tex
In \S \ref{DefTheory}, we start by recalling some commutative algebra facts about complete Noetherian local rings. In particular, we recall facts about their completed tensor product. We also record Grothendieck's connectedness theorem. We then recall the notion of group actions and group chunk actions from \cite{KW2} that will be necessary for the patching argument. After this, we turn to deformation theory, first recalling some general facts. We then study the nearly-ordinary deformation rings at places above $p$, the main points being to show that it is a domain of the correct dimension, that over a certain locus its reduction mod $p$ is still a domain, and that a certain locus of its characteristic zero points are smooth. We then recall some facts about the local deformation rings away from $p$ and use ideas of Snowden to show that they are domains mod $p$. After that we turn to global deformations. The main points of this subsection are to recall that our global deformation ring has an appropriate presentation over the local one, in order to later apply Grothendieck's connectedness theorem, and to show that a certain group action on the universal deformation ring is free. Lastly, we prove some small lemmas regarding deformations of dihedral representations, which will be useful in proving certain deformations are non-dihedral as well as to determine some properties of non-dihedral deformations.

In \S \ref{ModForms}, we recall Hida's theory of nearly ordinary automorphic forms in the totally definite quaternionic case. In the first subsection, we state definitions, analyse certain neatness properties of the open subgroups that will comprise our level, and state the relation to cuspidal automorphic representations of $\GL_2$. In the next subsection, we define the nearly ordinary Hecke algebra, and in the following subsection construct the universal nearly ordinary Hecke algebra. If the following subsection, we recall the Galois representations associated to eigenforms, and show how they give a Galois representation into the universal Hecke algebra such that the induced map from the universal deformation ring is surjective and factors through the quotient defined in \ref{globalquotient}. In the final subsection, we augment the level with Taylor-Wiles primes. The Taylor-Wiles primes we use to augment the level may have the property that the corresponding Frobenii do not have distinct eigenvalues under the residual representation. Because of this we cannot prove the standard control theorem, as there may be lifts of $\rhobar$ that are Steinberg at these primes. We show however, that the obstruction to the usual control theorem is annihilated by an element that becomes invertible after localizing at one of our nice prime ideals.

In \S \ref{AuxPrimesSec}, we show the existence of the Taylor-Wiles prime associated to a representation into $\GL_2(A)$, with $A$ the ring of integers in a characteristic $2$ local field $K$, satisfying some technical hypotheses. The main input is Pink's result \cite{PinkCompactSubs} that allows us to conclude that the intersection of the image with $\SL_2(A)$ is conjugate to an open subgroup in $\SL_2(A_0)$, where $A_0$ is the ring of integers of a characteristic $2$ local subfield $K_0 \subset K$. This allows us to compute explicitly with cocycles. As mentioned above, we must compute these cohomology groups integrally, and keep track (or at least bound) the torsion subgroups.

In \S \ref{LocRequalsT}, we prove the $R^\mathrm{red} = \mathbf{T}$ theorem. This section together with \S \ref{AuxPrimesSec} comprise the technical heart of the paper. The patching argument is a synthesis of the proof of \cite{KW2}*{Proposition 9.3} and \cite{SWreducible}*{\S 5}. The idea is to mimic the proof of \cite{KW2}*{Proposition 9.3}, but to define the maps from a power series over the local deformation ring $R_\mathrm{loc}[[x_1,\ldots,x_k]]$ to our augmented level global deformation rings $R_n$, in such a way that the $x_1,\ldots,x_k$ are mapped to the pullback to $R_n$ of our fixed nice prime ideal, instead of the maximal ideal. In this way, we get a surjection after localizing and completing at the nice prime. When defining these maps, we need to ensure that certain cokernels are finite of bounded size, so that when we take a projective limit the ranks of the resulting limiting modules do not grow. It is due to this reason that we needed to ensure that the torsion subgroups of the cohomology groups computed in \S \ref{AuxPrimesSec} did not depend on the Taylor-Wiles primes. After proving the localized $R^\mathrm{red} = \mathbf{T}$ theorem, we apply the connectivity argument to conclude that $R^\mathrm{red} = \mathbf{T}$.

The last section, \S \ref{MainThms}, proves the main theorem. We first recall some congruences proved in \cites{KisinFinFlat,KW2} necessary to show the existence of appropriate automorphic lifts after base change. We also prove a small lemma that shows the existence of ordinary lifts in the residually dihedral case, using a result of Wiles \cite{WilesOrdinary} that allows one to insert an ordinary Hilbert modular form of parallel weight $1$ into a $p$-adic family. We then prove the main theorem, by applying base change, combining the aforementioned congruences together with known cases of Leopoldt's conjecture so that we satisfy the assumptions of the $R^\mathrm{red} = \mathbf{T}$ theorem of \S \ref{LocRequalsT}.

%% file: Intro/Conventions.tex
We state some notation and conventions used in this paper. We will denote by $p$ a rational prime throughout. In the later sections we will usually take $p=2$. Let $\Qbar$ be the algebraic closure of $\Q$ in $\C$. For any subfield $L\subseteq\Qbar$ we set $G_L = \Gal(\Qbar/L)$. Throughout $F$ will denote a totally real number field inside $\Qbar$. Given a finite set $S$ of places of $F$ we denote by $G_{F,S} = \Gal(F_S/F)$, where $F_S$ is the maximal Galois extension of $F$ contained in $\Qbar$, unramified outside $S$. We let $J_F$ denote the set of embeddings $F\hookrightarrow \Qbar$.

For each rational prime $\ell$, fix an algebraic closure $\Qbar_\ell$ of $\Q_\ell$ and embeddings $\Qbar \hookrightarrow \Qbar_\ell$. We can then view $J_F$ as the set of embedding of $F$ into any of the $\Qbar_\ell$ or $\C$. For any place of $F$ we denote by $F_v$ the completion of $F$ at $v$ inside $\Qbar_\ell$, or $\C$. In the case that $v$ is non-archimedean, we write $G_v = \Gal(\Qbar_\ell/F_v)$ and let $I_v$ denote the inertia subgroup. For $v$ archimedean, we let $G_v = \Gal(\C/F_v)$. In either case we identify $G_v$ with a decomposition group of $G_F$ at $v$ via the embedding $\Qbar \hookrightarrow \Qbar_\ell$, or $\Qbar \hookrightarrow \C$.

As usual $\A_F$ will denote the ring of adeles of $F$ and $\A_F^\infty$ will denote the subring of finite adeles. We normalize the isomorphism of local class field theory so that uniformizers correspond to arithmetic Frobenii, and normalize global class field theory compatibly.

We will denote by $\epsilon_p$ the $p$-adic cyclotomic character and $\overline{\epsilon_p}$ is reduction mod $p$. We use homological conventions for our Galois representations; for example, the Galois representation attached to an elliptic curve is the one coming from its Tate module, not cohomology. With this convention, a representation is ordinary at $v$ if the local representation is reducible with an unramified quotient.

We will let $E$ denote a finite extension of $\Q_p$ contained in $\Qbar_p$. We will let $\mathcal{O}$ denote its ring of integers and $\F$ its residue field. We will occasionally enlarge $E$ if necessary. We let $\CNLO$ denote the category of complete, Noetherian, local, $\mathcal{O}$-algebras $A$ such that the structure morphism $\mathcal{O}\rightarrow A$ induces an isomorphism of residue fields. The morphisms in $\CNLO$ are local $\calO$-algebra morphisms. Given an object $B$ in $\CNLO$ we let $\CNL_B$ denote the subcategory whose objects are $B$-algebra and whose morphisms are $B$-algebra morphisms. We let $\CNL_B^\mathrm{op}$ be the opposite category of $\CNL_B$ and we identify it with the category of representable set-valued functors on $\CNL_B$ via the Yoneda embedding. Given a $\CNL_B$-algebra $A$, we denote by $\Spf A$ the correpsonding object of $\CNL_B^\mathrm{op}$, which should cause no confusion as the natural embedding of $\CNL_B^\mathrm{op}$ into the category of formal $B$-schemes is fully faithful. Given any local ring $A$, we denote by $\mathfrak{m}_A$ its maximal ideal.

Given a separable field extension $L/K$, we let $\mathrm{Nm}_{L/K}$ denote the norm from $L/K$. We refer the reader to the beginning of each section for more notation and conventions that will be employed in that particular section.

%% file: DefTheory/DefTheoryIntro.tex
\section{Deformation Theory}\label{DefTheory}

In this section we develop the necessary deformation theory. In \S \ref{ComAlg}, we recall some commutative algebra facts for complete Noetherian local rings with finite residue field. Of particular importance to our applications later will be recalling some facts regarding the completed tensor product of such rings, cf. \ref{TensorCNLO}, and the connedness theorem of Grothendieck, cf. \ref{GrothConnectivity}.

The following subsection recalls some facts from \cite{KW2} regarding group actions in the category $\CNLO$. In particular, the notion of group chunks and the building of free actions using them is essential to the $2$-adic Taylor-Wiles method developed in \cite{KW2}, and we will use these results in \S \ref{LocRequalsT}. The following subsection recalls some general definitions and results about deformation theory.

The subsequent subsection deals with determining the local deformation rings at the prime $p$. We follow the construction of Geraghty \cite{GeraghtyOrdinary}. However, as we deal only with dimension $2$, we can be more explicit and use the Demu\u{s}kin relation to prove that a certain open subset of the special fibre is integral. This will be important for showing that a certain localization of a completed tensor product of local deformation rings is normal in \S \ref{LocRequalsT},  c.f. \ref{localdefringnormal}. Also important to establish this normality is to prove a certain open subset of the generic fibre is smooth, and we prove this following \cite{GeraghtyOrdinary}.

The next subsection deals with determining the local deformation rings at the places not equal to $p$. Most of this is simply recalling results proved in \cites{KisinFinFlat,Kisin2adic,KW2}. However we will also need information about the special fibre of such rings. This is easy in the archimedean case, and to do this in the case $l\ne p$ we use ideas of Snowden \cite{SnowdenDefRing}. This is important for showing the normality of the ring in \ref{localdefringnormal}, mentioned above.

In the next subsection we deal with deformations of global Galois groups. We recall some definitions and properties and show that a certain quotient of the universal deformation ring (tensored with an Iwasawa algebra) can be presented as a complete Noetherian domain modulo ``few" relations, which will be necessary to apply the connectivity result \ref{Connectivity} later. We also recall some facts regarding twists of the universal deformation by characters, as in \cite{KW2}, which is essential to the $2$-adic method.

In the last subsection, we prove some easy facts regarding deformations of dihedral representations, that will be useful for determining the images of non-dihedral deformations to characteristic $p$ local fields as well as to establish criteria for a deformation to be non-dihedral.

 We now set up some more notation. Fix a choice of uniformizer $\varpi_E$ for $E$. Given a $\CNLO$-algebra $B$, we denote by $\mathrm{Ar}_B$ the full subcategory of $\CNL_B$ consisting of Artinian objects. Given a finite extension $E'/E$, we also let $\mathrm{Ar}_{E'}$ denote the category of Artinian local rings with residue field $E'$ with topology given by its structure as a finite dimensional $E'$-vector space. Note that such rings are canonically $E'$-algebras. The morphisms in $\mathrm{Ar}_{E'}$ are local $E'$-algebra morphisms.
  
 \subsection{Some commutative algebra}\label{ComAlg}
 
\input{DefTheory/ComAlg}

 \subsection{Group actions on $\CNLO$}\label{groupactions}
 
\input{DefTheory/GroupActions}

 \subsection{Some general deformation theory}\label{GenDefTheory}

\input{DefTheory/GenDefTheory}

\subsection{Local deformation rings at $p$}\label{LocalDefsp}

\input{DefTheory/LocalDefsp}

\subsection{Local deformation rings away from $p$}\label{LocalDefslNotp}

\input{DefTheory/LocalDefsNotp}

\subsection{Global deformations}\label{GlobalDefs}

\input{DefTheory/GlobalDefs}

\subsection{Deformations of dihedral representations}\label{DihedralDefs}

\input{DefTheory/Dihedral}

%% file: DefTheory/ComAlg.tex
 We recall some facts about objects in $\CNLO$ and their completed tensor products that will be useful to us later. 
 
 \subsubsection{}\label{Jacobsonness}
 
 Recall a scheme $X$ is called \textit{Jacobson} if for any closed subset $Z \subseteq X$, the set of closed points in $Z$ is dense. We say a ring $R$ is Jacobson if $\Spec R$ is Jacobson. If $R$ is a $\CNLO$-algebra, \cite{EGA4.3}*{ Corollary 10.5.8} shows that $\Spec R \smallsetminus \{\mathfrak{m}_R\}$ is Jacobson, and if $p$ is not nilpotent in $R$ then $R[1/p]$ is Jacobson. Parts (1) and (2) of the following proposition is \cite{KW2}*{Proposition 2.2 (i)} and part (3) is \cite{KW2}*{Corollary 2.3}.

\begin{prop}\label{OPointsOfR}

Let $R$ be an $\mathcal{O}$-flat $\CNLO$-algebra.

\begin{enumerate}
	\item There is a finite extension $E'/E$ with ring of integers $\mathcal{O}'$ and a local $\mathcal{O}$-algebra morphism $R\rightarrow \mathcal{O}'$.
	\item Ever maximal ideal of $R[1/p]$ is the image of the generic point of $\Spec\mathcal{O}' \rightarrow \Spec R$, with $R\rightarrow \mathcal{O}'$ as in \textnormal{(1)}.
	\item Let $I$ be an ideal of $R$ and let $X(I)$ denote the set of all morphisms $R \rightarrow \mathcal{O}'$ as in \textnormal{(1)} whose kernel contains $I$. If $R/I$ is $\mathcal{O}$-flat and reduced, then $I = \cap_{x\in X(I)} \ker(x)$
\end{enumerate}

\end{prop}

\subsubsection{}\label{CompTensorProd}

Let $R_1$ and $R_2$ be $\CNLO$-algebras. Then their completed tensor product over $\mathcal{O}$ is again a $\CNLO$-algebra. To see this, let $\mathfrak{m}$ denote the kernel of the natural map $R_1 \hat{\otimes}_\mathcal{O} R_2 \rightarrow \F \otimes_\mathcal{O} \F = \F$. Since $R_1 \hat{\otimes}_\mathcal{O} R_2$ is complete for the $\mathfrak{m}$-adic topology and we have an injection $\F^\times \hookrightarrow (R_1 \hat{\otimes}_\mathcal{O} R_2)^\times$, any element which does not belong to $\mathfrak{m}$ is invertible, and so $R_1 \hat{\otimes}_\mathcal{O} R_2$ is local. We can write $R_i$ as a quotient of a power series ring $\mathcal{O}[[x_1,\ldots,x_{d_i}]]$, and so $R_1 \hat{\otimes}_\mathcal{O} R_2$ can be written as a quotient of the power series ring $\mathcal{O}[[x_1,\ldots,x_{d_1+d_2}]]$, hence is Noetherian.

If $R_1$ and $R_2$ are flat $\CNLO$-algebras then $R_i \rightarrow R_1\hat{\otimes}_\mathcal{O} R_2$ is flat, cf. Lemma 19.7.1.2 of section 0 of \cite{EGA4.1}. In particular $R_1\hat{\otimes}_\mathcal{O} R_2$ is $\mathcal{O}$-flat.

\begin{prop}\label{TensorCNLO}

Let $R_1$ and $R_2$ be $\CNLO$-algebras and let $R = R_1 \hat{\otimes}_\mathcal{O} R_2$.

\begin{enumerate}
	\item Let $E'/E$ be a finite extension and for each $i=1,2$ let $x_i : R_i[1/p] \rightarrow E'$ be an $E'$-point that is formally smooth over $E$. The $E'$-point $(x_1,x_2): R[1/p] \rightarrow E'$ is formally smooth over $E$.
	\item If for each $i=1,2$, $R_i$ is $\mathcal{O}$-flat and $R_i[1/p]$ is geometrically integral, then so is $R[1/p]$. In particular, $R$ is a domain.
	\item Assume that each $R_i$ is $\mathcal{O}$-flat and that for any minimal primes $\mathfrak{q}_i$ of $R_i$, $R_i/\mathfrak{q}_i[1/p]$ is geometrically integral. Then any minimal prime of $R$ is of the form $\mathfrak{q}_1\hat{\otimes} R_1 + R_1 \hat{\otimes}\mathfrak{q}_2$, with $\mathfrak{q}_i$ a minimal prime of $R_i$.
	\item Assume each $R_i$ is an $\F$-algebra and let $\mathrm{Nil}(R_i)$ denote the nilradical of $R_i$. The nilradical of $R$ is $\Nil(R_1)\hat{\otimes}_\F R_2 + R_1\hat{\otimes}_\F \mathrm{Nil}(R_2)$.
\end{enumerate}

\end{prop}

\begin{proof}

Parts (1) and (2) are \cite{KisinFinFlat}*{Lemma 3.4.12}; see also \cite{KW2}*{Proposition 2.3}.

Let $\mathfrak{q}$ be a minimal prime of $R$ and let $\mathfrak{q}_i$ be its pullback to $R_i$. We have $\mathfrak{q}_1\hat{\otimes}_\mathcal{O} R_2 + R_1 \hat{\otimes}_\mathcal{O} \mathfrak{q}_2 \subset \mathfrak{q}$. Note that $(R_1/\mathfrak{q}_1)\hat{\otimes}_\mathcal{O} (R_2/\mathfrak{q}_2)$ is a domain by (2), and has the same dimension as $R$. The natural surjection $R \rightarrow (R_1/\mathfrak{q}_1)\hat{\otimes}_\mathcal{O} (R_2/\mathfrak{q}_2)$ then must have kernel $\mathfrak{q}$.

We now prove (4). Let $\Nil(R)$ denote the nilradical of $R$. Since each $\mathrm{Nil}(R_i)$ is finite length, $\mathrm{Nil}(R_1)\hat{\otimes}_\F R_2 = \mathrm{Nil}(R_1)\otimes_\F R_2$ and $R_1\hat{\otimes}_\F\mathrm{Nil}(R_2) = R_1\otimes_\F \mathrm{Nil}(R_1)$. By considering simple tensors, we see that $\mathrm{Nil}(R_1) \otimes_\F R_2 + R_1\otimes_\F \mathrm{Nil}(R_2)\subseteq \mathrm{Nil}(R)$. Consider the surjection
	\[ R \longrightarrow R/(R_1\otimes_\F\mathrm{Nil}(R_2) + R_1\otimes_\F \mathrm{Nil}(R_1))
	\cong (R_1/\mathrm{Nil}(R_1))\hat{\otimes}_\F (R_2/\mathrm{Nil}(R_2)).
	\]

A theorem of Chevalley, cf. \cite{EGA4.2}*{Corollary 7.5.7}, shows that $(R_1/\mathrm{Nil}(R_1))\hat{\otimes}_\F (R_2/\mathrm{Nil}(R_2))$ is reduced, and so $\mathrm{Nil}(R)\subseteq \mathrm{Nil}(R_1) \otimes_\F R_2 + R_1\otimes_\F \mathrm{Nil}(R_2)$.
	\end{proof}

\subsubsection{}\label{ConnectivitySubSec}

We record a connectedness theorem of Grothendieck. This will be used in \ref{ProMod} in the same way as \cite{SWreducible}*{Corollary A.2} is used in the proof of \cite{SWirreducible}*{Proposition 4.2}. We use Grothendieck's theorem because it is more readily implies connectivity results for the spectra of quotients of complete local Noetherian domains than does the theorem of Raynaud used in \cite{SWreducible}*{Corollary A.2}.

We say that a Noetherian scheme $X$ is $k$-\textit{connected} if $\dim X > k$ and for any closed subset $Z\subseteq X$ with $\dim Z < k$, the topological space $X\smallsetminus Z$ is connected. Clearly, if $X$ is $k$-connected, it is $k'$-connected for any $k'\le k$. Note that a Noetherian scheme $X$ is $k$-connected if and only if 
\begin{itemize}
	\item[-] every irreducible component of $X$ has dimension $> k$; and
	\item[-] for any two irreducible components $C$ and $C'$ of $X$, there is a sequence of irreducible components
	\[ C = C_0, C_1,\ldots, C_n = C'\]
such that $\dim (C_i\cap C_{i+1}) \ge k$ for each $0\le i < n$.
\end{itemize}
We will use this formulation later. The following is \cite{SGA2}*{Expos\'{e} XIII, Theorem 2.1}.

\begin{prop}\label{GrothConnectivity}

Let $A$ be a complete Noetherian local ring and let $f_1,\ldots,f_r \in \mathfrak{m}_A$. If $\Spec A$ is $k$-connected and $r\le k$, then $\Spec A/(f_1,\ldots,f_r)$ is $k-r$-connected.

\end{prop}

For a detailed proof of this proposition see \cite{FlennerOCarrollVogel}*{\S 3.1}; in particular, see the proof of Theorem 3.1.7 found there. Upon noting that any Noetherian domain $A$ is $(\dim A - 1)$-connected, we have the following corollary.

\begin{cor}\label{Connectivity}

Let $A$ be a complete Noetherian local domain and let $f_1,\ldots,f_r \in \mathfrak{m}_A$. If $r \le \dim A - 1$, then $\Spec A/(f_1,\ldots,f_r)$ is $(\dim A - r -1)$-connected.

\end{cor}

%% file: DefTheory/GroupActions.tex
We quote some results and definitions regarding group action on the category $\CNLO$ from \cite{KW2}, which will be necessary for the patching argument in \S \ref{LocRequalsT}. This material is taken directly from \cite{KW2}*{ \S 2.4-2.6}, and we refer the reader there for proofs.

\subsubsection{}\label{CNLOgroups}

Let $G$ be a group object in $\CNLO^\mathrm{op}$. We call $G$ a $\CNLO$-\textit{group}. Leting $A(G)$ denote the $\CNLO$-algebra representing $G$, we note that the group structure on $G$ is defined in the same way as the Hopf algebra structure on an affine algebraic group except that our comultiplication takes values in the completed tensor product $A(G) \rightarrow A(G)\hat{\otimes}_\mathcal{O} A(G)$.

Let $X$ be an element of $\CNLO^\mathrm{op}$, and let $A(X)$ be its affine $\CNLO$-algebra. We similarly define an action of $G$ on $X$, i.e. a $\CNLO$-morphism $\gamma : A(X) \rightarrow A(G)\hat{\otimes}_\mathcal{O} A(X)$ making $G(A) \times X(A) \rightarrow X(A)$ a group action that is functorial in $A$. An action is \textit{free} if the map $G\times X \rightarrow X\times X$ given by $(g,x) = (gx,x)$ on points is a closed immersion. This is equivalent to requiring that for any $\CNLO$-algebra $A$, $G(A)$ acts freely on $X(A)$. We let $A(X)_0$ denote the subalgebra of $A(X)$ consisting of elements $a$ such that $\gamma(a) = 1 \otimes a$. The elements $a\in A(X)_0$ are the functions on $X$ that are constant on orbits, i.e. $a(gx) = a(x)$ for any $\CNLO$-algebra $A$ and $x\in X(A)$, $g\in G(A)$.

\subsubsection{}\label{diaggroups}

Let $H$ be a finitely generated abelian group whose torsion is a power of $p$. We call the \textit{diagonalizable} $\CNLO$-\textit{group} associated to $H$, denoted $H^\ast$, the $\CNLO$-group defined by completing the diagonalizable group associated to $H$ as in \cite{SGA2}*{Expos\'{e} VIII} at the identity element of the special fibre. Concretely, we have $\Z^\ast = \mathbb{G}_m^\wedge$, the formal torus on $\CNLO$, and $(\Z/p^r\Z)^\ast = \mu_{p^r}$. All other $H^\ast$ are products of these two examples. Note that a surjection of abelian groups $H \rightarrow H'$ induces a closed immersion $(H')^\ast \rightarrow H^\ast$. The following is a combination of \cite{KW2}*{Proposition 2.5} and \cite{KW2}*{Proposition 2.6}.

\begin{prop}\label{diaggroupaction}

Let $X$ be an object in $\CNLO^\mathrm{op}$ and let $H$ be a finitely generated group with a free action $H^\ast \times X \rightarrow X$.

\begin{enumerate}
	\item A quotient $H^\ast\backslash X$ exists in $\CNLO^\mathrm{op}$, and if $H$ is torsion free, then $X\rightarrow H^\ast\backslash X$ is formally smooth of relative dimension equal to the rank of $H$.
	\item The morphism $X \rightarrow H^\ast \backslash X$ makes $X$ a torsor over $H^\ast\backslash X$ for $H^\ast \times (H^\ast\backslash X)$, i.e. the map
		\[ (H^\ast \times (H^\ast\backslash X)) \times_{H^\ast\backslash X} X \longrightarrow X 
			 \times_{H^\ast\backslash X} X
			 \]
given on points by $(g,x) \mapsto (gx,x)$, is an isomorphism.
	\item If $H \rightarrow H'$ is a surjective morphism of abelian groups, then $(H')^\ast\backslash X$ has a natural free action of $H^\ast / (H')^\ast$ and $H^\ast \backslash X$ is naturally isomorphic to the quotient of $(H')^\ast\backslash X$ by the action of $H^\ast / (H')^\ast$.
\end{enumerate}

\end{prop}

\subsubsection{}\label{groupchunks}

Let $m\ge 1$ be an integer. Denote by $\CNLO^{[m]}$ the full subcategory of $\CNLO$ consisting of objects $A$ such that $\mathfrak{m}_A^m = 0$. For a $\CNLO$-algebra $A$, we denote by $A^{[m]}$ the $\CNLO^{[m]}$-algebra $A/\mathfrak{m}_A^m$. Note that $A\mapsto A^{[m]}$ defines a functor from $\CNLO$ to $\CNLO^{[m]}$, and we call $A^{[m]}$ the \textit{truncation to level} $m$ of $A$. For $A$ a $\CNLO$-algebra, the restriction of $\Spf A$ to $\CNLO^{[m]}$ is isomorphic to $\Spf A^{[m]}$. If $X$ is an object in $\CNLO^\mathrm{op}$, $X = \Spf A(X)$, we let $X^{[m]} = \Spf A(X)^{[m]}$ on $\CNLO^{[m]}$, and call $X^{[m]}$ the \textit{truncation to level} $m$ of $X$. If $X = X^{[m]}$, then any $\CNLO^\mathrm{op}$-map $X \rightarrow Y$ factors through $Y^{[m]}$. A map $X \rightarrow Y$ is a closed immersion if and only if $X^{[m]} \rightarrow Y^{[m]}$ is a closed immersion for each $m$. Note that if $A_1$ and $A_2$ are $\CNLO^{[m]}$-algebras, then $A_1\hat{\otimes}_\mathcal{O}A_2$ may not be, and the restriction of $\Spf A_1 \times \Spf A_2$ to $\CNLO^{[m]}$ is represented by $(A_1\hat{\otimes}_\mathcal{O}A_2)^{[m]}$.

We define a \textit{group chunk of level} $m$ to be $G = \Spf A(G)$, where $A(G)$ is a $\CNLO^{[m]}$-algebra equiped with $\CNLO^{[m]}$-morphisms $A(G) \rightarrow (A(G)\hat{\otimes}_\mathcal{O} A(G))^{[m]}$, $A(G) \rightarrow A(G)$, and $A(G) \rightarrow \mathcal{O}^{[m]}$ satisfying the usual diagrams defining a Hopf algebra. A group chunk of level $m$ defines group functor on $\CNLO^{[m]}$. Note that if we are given a $\CNLO$-group $G$, $G^{[m]}$ is a group chunk of level $m$, for all $m\ge 1$.

Let $X = \Spf A(X)$ with $A(X)$ a $\CNLO^{[m]}$-algebra, and let $G$ be a group chunk of level $m$. We define a \textit{group action chunk of level} $m$, to be a morphism $(G\times X)^{[m]} \rightarrow X$ defining a functorial group action on $\CNLO^{[m]}$. Note that the map $(G\times X)^{[m]} \rightarrow X \times X$ given on points by $(g,x) \mapsto (gx,x)$ factors through $(X\times X)^{[m]}$. We call the group chunk action \textit{free} if $(G \times X)^{[m]} \rightarrow (X \times X)^{[m]}$ is a closed immersion. Given a $\CNLO$-action $G\times X\rightarrow X$, $(G^{[m]} \times X^{[m]})^{[m]} \rightarrow X^{[m]}$ is a group action chunk of level $m$. We record \cite{KW2}*{Proposition 2.7}.

\begin{prop}\label{goupchunkaction}

Let $G$ be a $\CNLO$-group. Suppose for each $m\ge 1$, we have $\CNLO^{[m]}$-algebras $A_m$ and $\CNLO$-morphisms $A_{m+1} \rightarrow A_m$, such that $A_\infty = \varprojlim A_m$ is in $\CNLO$ (i.e. is Noetherian). Assume that for each $m$, we have a group action chunk of $G^{[m]}$ on $\Spf A_m$. 
\begin{enumerate}\item Then there is a unique $\CNLO$-group action $G \times \Spf A_\infty \rightarrow \Spf A_\infty$ such that for each $m$, the group action chunk of $G^{[m]}$ on $\Spf A_\infty^{[m]}$ is compatible with the group action chunk of $G^{[m]}$ on $A_m$ via the closed immersion $\Spf A_m \rightarrow \Spf A_\infty^{[m]}$.
\item If the group action chunks of $G^{[m]}$ on $\Spf A_m$ are free, then so is the action of $G$ on $\Spf A_\infty$.
\end{enumerate}

\end{prop}

%% file: DefTheory/GenDefTheory.tex
We first introduce some notation and state some useful facts. Our references for this subsection are \cite{MazurDefGalRep} and \cite{KW2}*{\S 2}.

\subsubsection{}\label{DefDefs}

Let $G$ be a profinite group and let
	\[ \overline{\rho} : G \longrightarrow \GL_n(\F) \]
be a continuous homomorphism. Denote by $V_\F$ the representation space of $\rhobar$. Given $A \in \mathrm{Ob}(\CNL_\calO)$, a \textit{lift} of $\rhobar$ to $A$ is a continuous homomorphism $\rho : G \rightarrow \GL_n(A)$ whose reduction modulo $\mathfrak{m}_A$ is equal to $\rhobar$. A \textit{deformation} of $V_\F$ to $A$ is a pair $(V_A,\phi_A)$, where $V_A$ is a free rank two $A$-module with continuous $G$-action, and $\phi$ is an isomorphism $V_A \otimes_A \F \stackrel{\sim}{\longrightarrow} V_\F$. We will usually drop $\phi_A$ from the notation.  A deformation is an equivalence class of lifts, two lifts begin equivalent if they are conjugate by an element of $\ker(\GL_n(A)\rightarrow \GL_n(\F))$. Also, a deformation $V_A$ of $V_\F$ to $A$ together with a choice of basis for $V_A$ lifting our fixed basis of $V_\F$ determines a lift of $\rhobar$. For this reason we will also call lifts \textit{framed deformations}.

We define set valued functors $\mathcal{D}$ and $\mathcal{D}^\square$ on $\CNL_\calO$ by letting $\mathcal{D}(A)$, respectively $\mathcal{D}^\square(A)$, denote the set of deformations of $V_\F$ to $A$, respectively the set of lifts of $\rhobar$ to $A$. Sending a lift to its equivalence class of deformations gives a natural morphism $\mathcal{D}^\square \rightarrow \mathcal{D}$. If $G$ satisfies the $p$-finiteness condition, i.e. if $\Hom(G',\Z/p\Z)$ is finite for all finite index subgroups $G'$ of $G$, then $\mathcal{D}^\square$ is representable. We will give a proof of this fact bellow in \ref{RepOnTop}. If further $\End_{\F[G]}(V_\F) = \F$, then $\mathcal{D}$ is also representable. This can be checked using Schlessinger's criteria, cf. \cite{MazurDefGalRep}, or by taking the quotient of $\mathcal{D}^\square$ by the free action of $\mathrm{PGL}_2^\wedge$, the completion of $\mathrm{PGL}_{2/\mathcal{O}}$ at the identity section of the special fibre. If $R$ and $R^\square$ denote the objects representing $\mathcal{D}$ and $\mathcal{D}^\square$, respectively, then the natural morphism $R\rightarrow R^\square$ is formally smooth of relative dimension $n^2-1$, cf. \cite{KW2}*{\S 2.2}.

Let $\psi$ denote an $\mathcal{O}$-valued character of $G$ whose reduction is equal to $\det \rhobar$. In the case that $G$ is a local or global Galois group, or a quotient of one of these groups, we define a subfunctor $\mathcal{D}^\psi$ of $\mathcal{D}$ by letting $\mathcal{D}^\psi(A)$ be the subset of $\mathcal{D}(A)$ consisting of deformations $V_A$ such that $\det V_A = \psi\epsilon_p$. We define $\mathcal{D}^{\square,\psi}$ similarly. Let $R^\square$ and $\rho^{\mathrm{univ}}$ denote the object representing $\mathcal{D}^{\square}$ and its universal lift. If $I$ is the ideal generated by the elements $\det\rho^\mathrm{univ}(\sigma) - \psi\epsilon_p(\sigma)$, it is easy to see that $R^\square/I$ represents $\mathcal{D}^{\square,\psi}$. Similarly if $\mathcal{D}$ is representable, then so is $\mathcal{D}^\psi$.

\subsubsection{}\label{TOPCat}

We extend the functor $\mathcal{D}^\square$ to a larger category. Let $\mathrm{Top}_\mathcal{O}$ be the category whose objects are pairs $(A,I)$, where $A$ is a topological $\mathcal{O}$-algebra and $I$ is an ideal of $A$ defining the topology of $A$ such that $I$ contains the image of $\varpi_E$ under the structure map $\mathcal{O}\rightarrow A$ and such that $A$ is $I$-adically complete. The morphisms $(A,I) \rightarrow (A',I')$ in $\mathrm{Top}_\mathcal{O}$ are $\mathcal{O}$-algebra morphisms $\varphi : A \rightarrow A'$ such that $\varphi(I) \subseteq I'$. Note that the map $A \mapsto (A,\mathfrak{m}_A)$ embeds $\CNLO$ as a full subcategory of $\mathrm{Top}_\mathcal{O}$, and that for any $(A,I)$ in $\mathrm{Top}_\mathcal{O}$ the map $\mathcal{O} \rightarrow A$ induces an injection $\F \rightarrow A/I$.

Let $G$ and $\overline{\rho}$ be as in \ref{DefDefs}, and assume that $G$ satisfies the $p$-finiteness condition. We define a functor $\mathcal{D}^\square$ on $\mathrm{Top}_\mathcal{O}$ by letting $\mathcal{D}^\square (A,I)$ denote the set of continuous homomorphisms
	\[ \rho : G \longrightarrow \GL_n(A) \]
such that 
	\[ \xymatrix{
		G \ar[r]^-\rho \ar[d]_-{\rhobar} & \GL_n(A) \ar[d]^-{\mathrm{mod }I} \\
		\GL_n(\F) \ar[r] & \GL_n(A/I) }
	\]
commutes. Note that under the embedding $A \mapsto (A,\mathfrak{m}_A)$ of $\CNLO$ into $\mathrm{Top}_\mathcal{O}$ the restriction of $\mathcal{D}^\square$ on $\mathrm{Top}_\mathcal{O}$ to $\CNLO$ coincides with $\mathcal{D}^\square$ as defined in \ref{DefDefs}.

\begin{prop}\label{RepOnTop}

There is a $\CNLO$-algebra $R^\square$ such that $(R^\square,\mathfrak{m}_{R^\square})$ represents $\mathcal{D}^\square$ on $\mathrm{Top}_\mathcal{O}$. In particular $\mathcal{D}^\square$ is representable on $\CNLO$.

\end{prop}

\begin{proof}

Let $G' = \ker \rhobar$, $G'(p)$ its maximal pro-$p$ quotient, and $H$ the kernel of the natural surjection $G' \rightarrow G'(p)$. Note that $H$ is normal in $G$ as $G'$ is normal in $G$ and $H$ is a characteristic subgroup of $G'$. Since $G$ satisfies the $p$-finiteness condition, $G'/H$ is topologically finitely generated; hence, so is $G/H$. Fix a set of topological generators $\gamma_1,\ldots,\gamma_m$ of $G/H$. Let $F$ denote the free group on the set $\{\gamma_1,\ldots,\gamma_m\}$ and let $F^\wedge$ denote its profinite completion. We have a natural surjection $F^\wedge \rightarrow G/H$ and we denote by $K$ its kernel.

For each $\gamma_k$, let $[\rhobar(\gamma_k)] \in \GL_n(\mathcal{O})$ be the matrix whose entries are the Teichm\"{u}ller lifts of the entries of $\rhobar(\gamma_k)$. Consider the power series ring $\mathcal{O}[[a_{ij}^k]]$ where $1\le i,j\le n$ and $1\le k \le m$. We define a continuous homomorphism
 	\[ \varrho : F^\wedge \longrightarrow \GL_n(\mathcal{O}[[a_{ij}^k]])
	\]
by $\varrho(\gamma_k) = [\rhobar(\gamma_k)] + (a_{ij}^k)$. Let $J$ denote the ideal of $\mathcal{O}[[a_{ij}^k]]$ generated by the elements of the matrices $\varrho(r) - 1$ for all $r\in K$, and set $R^\square = \mathcal{O}[[a_{ij}^k]]/J$. Then the pushforward of $\varrho$ along $\mathcal{O}[[a_{ij}^k]] \rightarrow R^\square$ defines a continuous homomorphism $G/H \rightarrow \GL_n(R^\square)$, and we let
	\[ \rho^\mathrm{univ} : G \longrightarrow \GL_n(R^\square)
	\]
be the continuous homomorphism given by precomposing with the surjection $G \rightarrow G/H$. Note that $R^\square$ is an object in $\CNLO$. We will now show that $(R^\square,\mathfrak{m}_{R^\square})$ represents $\mathcal{D}^\square$ on $\mathrm{Top}_\mathcal{O}$ and that $\rho^\mathrm{univ}$ is the universal lift.

Let $(A,I)$ be an object in $\mathrm{Top}_\mathcal{O}$ and let 
	\[ \rho : G \longrightarrow \GL_n(A)
	\]
be an element of $\mathcal{D}^\square(A,I)$. Since $G \rightarrow \GL_n(A) \rightarrow \GL_n(A/I)$ has kernel $G'$, and $I^n/I^{n+1}$ is $p$-torsion for all $n$, the morphism $\rho$ factors through $G/H$. The morphism $\rho$ is equivalent to giving matrices $X_k \in \GL_n(A)$ for each $1\le k \le m$, such that their reduction modulo $I$ is equal to $\rhobar(\gamma_k)$, and such that the induced homomorphism $F^\wedge \rightarrow \GL_n(A)$ is trivial on the subgroup $K$. By viewing $\rho$ as a specialization of $\varrho$, this is then equivalent to giving an $\calO$-algebra morphism $\mathcal{O}[[a_{ij}^k]] \rightarrow A$ whose kernel contains $J$ and such that the maximal ideal of $\mathcal{O}[[a_{ij}^k]]$ is mapped to $I$, i.e to give an $\mathcal{O}$-algebra morphism $\varphi : R^\square \rightarrow A$ such that $\varphi (\mathfrak{m}_{R^\square}) \subseteq I$, and we see that $\rho$ is the pushforward of $\rho^\mathrm{univ}$ under this map.
	\end{proof}
	
This proposition has the following immediate consequence. If $E'/E$ is a finite extension with ring of integers $\mathcal{O}'$ and residue field $\F'$, then $R^\square \otimes_\mathcal{O} \mathcal{O}'$ represents the lifting functor $\mathcal{D}^\square$ for $\rhobar \otimes_\F \F'$ on $\CNL_{\mathcal{O}'}$.

\subsubsection{}\label{EAlgs}

Let $E'/E$ be finite with ring of intergers $\mathcal{O}'$ and residue field $\F'$. Assume we are given a continuous homomorphism $\rho_{\mathcal{O}'} : G \rightarrow \GL_n(\mathcal{O}')$ such that
	\[ \xymatrix{
		G \ar[r]^-\rho \ar[d]_-{\rhobar} & \GL_n(\mathcal{O}') \ar[d] \\
		\GL_n(\F) \ar[r] & \GL_n(\F') }
	\]
commutes. Let $\rho_{E'}$ denote the induced morphism $\rho_{E'} : G \rightarrow \GL_n(E')$. Recall that $\mathrm{Ar}_{E'}$ is the category of Artinian local rings with residue field $E'$ topologized by their structure as a finite dimensional $E'$-vector spaces, and morphisms local $E'$-algebra morphisms. Let $\mathcal{D}_{\rho_{E'}}^\square$ denote the functor on $\mathrm{Ar}_{E'}$ that sends an object $B$ of $\mathrm{Ar}_{E'}$ to the set of continuous homomorphisms $\rho_B : G \rightarrow \mathrm{GL}_n(B)$ such that $\rho_B$ modulo $\mathfrak{m}_B$ is equal to $\rho_{E'}$. The following proposition will be useful in determining the generic fibre of the local deformation ring in \S\ref{LocalDefsp} and the argument is due to Kisin \cite{KisinOverConvFM}*{Proposition 9.5}.

\begin{prop}\label{EAlgRep}

For any $B$ in $\mathrm{Ar}_{E'}$ and $\rho_B \in \mathcal{D}_{\rho_{E'}}^\square$ as in \ref{EAlgs}, there is a unique $\mathcal{O}$-algebra morphism $R^\square \rightarrow B$ such that $\rho_B$ is the pushforward of $\rho^\mathrm{univ}$ via this morphism.

\end{prop}

\begin{proof}

Choose a surjection $E'[[x_1,\ldots,x_n]] \rightarrow B$ and let $A_0$ denote the image of $\mathcal{O}'[[x_1,\ldots,x_n]]$ under this surjection. Since the representation $\rho_{E'}$ takes values in $\GL_n(\mathcal{O}')$, the representation $\rho_B$ takes values in $\GL_n(A_0+\mathfrak{m}_B)$. Let $\mathfrak{n}= A_0 \cap \mathfrak{m}_B$. Since $B$ is Artinian, we can define for any $m\ge 1$, $A_m = A_0 + \sum_{j=1}^\infty p^{-mj}\mathfrak{n}^j$. For each $m\ge 1$, $A_m$ is a $\CNL_{\mathcal{O}'}$-algebra subring of $B$. Note that $A_0 + \mathfrak{m}_B = \cup_{m\ge 0}^\infty A_m$. Since $G$ is compact, a standard Baire Category argument implies that $\rho_B$ takes values in $A_m$ for some $m$. By \ref{RepOnTop}, there is a unique $\CNLO$-morphism $R^\square \rightarrow A_m$ such that $G \rightarrow \GL_n(A_m)$ is the pushforward of $\rho^\mathrm{univ}$ via $R^\square\rightarrow A_m$.

It remains to show uniqueness. Let $\phi,\phi' : R^\square \rightarrow B$ be two such $\mathcal{O}$-algebra morphisms. Since the image of $R^\square$ under either of these maps lies in $A_0 + \mathfrak{m}_B = \cup_{m\ge 0} A_m$ and $R^\square$ is compact, again by Baire Category its image via either $\phi$ or $\phi'$ must lie in some $A_m$ for $m$ sufficiently large. By the universal property of $R^\square$ on $\mathrm{Top}_{\mathcal{O}}$, we must have $\phi = \phi'$.
	\end{proof}

%% file: DefTheory/LocalDefsp.tex
Throughout this subsection, $F_v$ will denote a finite extension of $\Q_p$, and $G_v = \Gal(\overline{\Q}_p/F_v)$. We assume that $E$ contains all embedding of $F_v$ into an algebraic closure of $E$. Our construction of the local deformation ring, as well as the analysis of its generic fibre follows \cite{GeraghtyOrdinary}*{\S 3}.

For a given ring $A$, we call an $A$-submodule $L\subset A^2$ a \textit{line} if both $L$ and $A^2/L$ are projective of rank one.

\subsubsection{}\label{localdefpstart}

We fix a continuous homomorphism
	\[ \overline{\rho} : G_v \longrightarrow \GL_2(\F) \]
and a continuous character $\psi : G_v \rightarrow \calO^\times$ such that $\det \overline{\rho} = \overline{\psi\epsilon_p}$. Let $V_\F$ denote the representation space of $\rhobar$. We will assume throughout this subsection that there is some line $L_\F$ in $V_\F$ that is stable by the action of $G_v$. Let $L_\F$ be one such line and let $\overline{\chi}$ denote the character of $G_v$ giving the action on $V_\F/L_\F$. Note that the choice of $\overline{\chi}$ is unique unless $V_\F$ is the direct sum of two distinct characters. In this case we simply make a choice of one of these characters.

We let $\mathcal{D}_v^\square$ and $\mathcal{D}_v^{\square,\psi}$ denote the functor of lifts of $\rhobar$ and the subfunctor consisting of lifts with determinant $\psi\epsilon_p$, repsectively. We denote the corresponding representing objects by $R_v^\square$ and $R_v^{\square,\psi}$, respectively. 

Let $G_v^\mathrm{ab}$ denote the abelianization of $G_v$ and $G_v^\mathrm{ab}(p)$ the maximal pro-$p$ quotient of the abelianization. Set $\Lambda(G_v) = \mathcal{O}[[G_v^\mathrm{ab}(p)]]$. Note that $G_v^\mathrm{ab}(p)$ is isomorphic to $\mu_{p^s}\times\Z_p^{d+1}$, where $d=[F_v:\Q_p]$ and $\mu_{p^s}$ is the group of $p$-power roots of unity in $F_v$. So $\Lambda(G_v)$ has $p^s$ minimal primes, corresponding to the $p^s$ distinct $\mathcal{O}$-valued characters of $\mu_{p^s}$ (recall we have assumed $E$ contains all embeddings of $F_v$ into $\overline{E}$), and its quotient by any of these minimal primes is isomorphic to a power series over $\mathcal{O}$ in $d+1$ variables. 

Let $\tilde{\chi}$ denote the Teichm\"{u}ller lift of $\overline{\chi}$. If $A$ is a $\CNLO$-algebra and $\chi : G_v \rightarrow A^\times$ is a continuous character, then writing $\chi = \tilde{\chi}\chi'$ with $\chi'$ factoring through the natural projection $G_v \rightarrow G_v^{\mathrm{ab}}(p)$, we see that $\Lambda(G_v)$ represents the set valued functor that assigns to each $\CNLO$-algebra $A$ the set of continuous characters $\{\chi_A : G_v\rightarrow A^\times\}$ that lift $\overline{\chi}$. We let $\chi^\mathrm{univ} : G_v \rightarrow \Lambda(G_v)$ denote the universal $\Lambda(G_v)$-valued character lifting $\overline{\chi}$.

We will need to consider quotients of $\Lambda(G_v)$ by its minimal primes in order to ensure that our local lifting ring is a domain. Recall $\mu_{p^s}$ is identified with the the $p$-power torsion subgroup of $G_v^\mathrm{ab}$. Fix a character $\eta : \mu_{p^s} \rightarrow \mathcal{O}^\times$, and let $\mathfrak{q}_\eta$ denote the corresponding minimal prime of $\Lambda(G_v)$. We set $\Lambda(G_v,\eta) = \Lambda(G_v)/\mathfrak{q}_\eta$, and let $\chi_\eta^\mathrm{univ}$ denote the character obtained by composing $\chi^\mathrm{univ}$ with the natural surjection $\Lambda(G_v)\rightarrow \Lambda(G_v,\eta)$. Then $\Lambda(G_v,\eta)$ represents the functor that assigns to each $\CNLO$-algebra $A$ the set of characters $\{\chi_A : G_v \rightarrow A^\times\}$ that lift $\overline{\chi}$ and whose restriction to the $p$-power torsion subgroup of $G_v^\mathrm{ab}$ is equal to $\eta$.

Set $R_{\Lambda(G_v,\eta)}^{\square,\psi} = R_v^{\square,\psi}\hat{\otimes}_\mathcal{O} \Lambda(G_v,\eta)$, and consider $\mathbb{P}^1_\mathcal{O} \otimes_\mathcal{O} R_{\Lambda(G_v,\eta)}^{\square,\psi}$. If $A$ is an $\calO$-algebra, then an $A$-point of this scheme is a triple $(\alpha,\beta,L)$ where $\alpha : R_v^{\square,\psi} \rightarrow A$ and $\beta:\Lambda(G_v,\eta)\rightarrow A$ are $\mathcal{O}$-algebra morphisms, and $L$ is a line in $A^2$. By pushing forward the universal $R_v^{\square,\psi}$ valued lift via $\alpha$, we get a homomorphism $\rho_A: G_v \rightarrow \GL_2(A)$, and by pushing forward $\chi_\eta^\mathrm{univ}$, we get a character $\chi_A : G_v\rightarrow A^\times$. We define a functor $\mathcal{DL}^{\square,\psi,\eta}$ on the category of $\calO$-algebras by letting $\mathcal{DL}^{\square,\psi}(A)$ be the subset of such triples $(\alpha,\beta,L)$ such that $\rho_A$ leaves $L$ invariant and $G_v$ acts on $A^2/L$ via $\chi_A$.

\begin{lem}\label{ProjRep}

$\mathcal{DL}^{\square,\psi}$ is represented by a closed subscheme $\mathscr{L}$ of $\bbP^1_\mathcal{O} \otimes_\mathcal{O} R_{\Lambda(G_v,\eta)}^{\square,\psi}$.

\end{lem}

\begin{proof}

For ease of notation, set $R = R_{\Lambda(G_v,\eta)}^{\square,\psi}$. Let $\phi$ denote the tautological morphism
	\[ \phi : \mathcal{O}_{\bbP^1_R}^2 \longrightarrow \mathcal{O}_{\bbP^1_R}(1), \]
and let $\mathcal{L}$ denote its kernel. Let $U=\Spec A$ be an open affine subset of $\bbP^1_R$ such that $\mathcal{O}_{\bbP^1_R}(1)(U)$ is free of rank one over $A$. Fix a generator $e$ of $\mathcal{O}_{\bbP^1_R}(1)(U)$. Let
	\[ \mathcal{I}(U) = \{ r \in A : \text{there is some } \sigma \in G_v \text{ and } x \in A^2	\text{ with } \phi(\rho_A(\sigma) x - \chi_A(\sigma)x) = r e \}.
	 \]
Then $\mathcal{I}(U)$ is an ideal in $A$ and does not depend on the choice of $e$. It is easy to see that $\mathcal{I}$ defines a sheaf of ideals of $\bbP^1_R$, and that $\mathcal{DL}^{\square,\psi}$ coincides with the functor of points of the closed subscheme of $\bbP^1_R$ defined by $\mathcal{I}$. 
\end{proof}

\subsubsection{}\label{triliftringdef}

We now let $R_{\Lambda(G_v,\eta)}^{\triangle,\psi}$ denote the quotient of $R_{\Lambda(G_v,\eta)}^{\square,\psi}$ by the kernel of the homomorphism
	\[ R_{\Lambda(G_v,\eta)}^{\square,\psi} \longrightarrow \mathcal{O}_{\mathscr{L}}(\mathscr{L}),
	\]
i.e. $R_{\Lambda(G_v,\eta)}^{\triangle,\psi}$ is the affine algebra of the scheme theoretic image of $\mathscr{L}$ in $\Spec R_{\Lambda(G_v,\eta)}^{\square,\psi}$.
		
\begin{prop}\label{TriLiftRing}
Let $E'/E$ be a finite extension with ring of integers $\mathcal{O}'$. Let $\rho :G_v\rightarrow \GL_2(\mathcal{O}')$ be a lift of $\rhobar$ with determinant $\psi\epsilon_p$, and let $\chi : G_v \rightarrow (\calO')^\times$ be a character lifting $\chibar$ whose restriction to the $p$-power torsion subgroup of $G_v^\mathrm{ab}$ is equal to $\eta$. The point of $R_{\Lambda(G_v,\eta)}^{\square,\psi}$ determined by the pair $(\rho,\chi)$ factors through $R_{\Lambda(G_v,\eta)}^{\triangle,\psi}$ if and only if there is a $G_v$-stable line in $(\mathcal{O}')^2$ such that $G_v$ acts on the quotient via $\chi$.
\end{prop}

\begin{proof} Let $f$ denote the morphism $\mathscr{L}\rightarrow \Spec R_{\Lambda(G_v,\eta)}^{\square,\psi}$. The point determined by $(\rho,\chi)$ satisfies the conclusion if and only if it is in the image of $\mathscr{L}[1/p]$. Since $f$ is proper, so is $f[1/p]$, and the topological image of $f[1/p]$ is equal to scheme theoretic image of $f[1/p]$, which is $\Spec R_{\Lambda(G_v,\eta)}^{\triangle,\psi}[1/p]$.  \end{proof}

In order to determine the structure of $R_{\Lambda(G_v,\eta)}^{\triangle,\psi}$ more precisely, we will relate it to $\mathscr{L}$ via the following lemma.

\begin{lem}\label{opensetiso}

Let $Z$ denote the closed subscheme of $\Lambda(G_v,\eta)$ defined by $(\chi_\eta^\mathrm{univ})^2 = \psi\epsilon_p$, and let $V$ denote its complement. The map
	\[ \mathscr{L} \times_{\Spec \Lambda(G_v,\eta)} V \longrightarrow 
	\Spec R_{\Lambda(G_v,\eta)}^{\triangle,\psi} \times_{\Spec \Lambda(G_v,\eta)} V
	\]
is an isomorphism.

\end{lem}

\begin{proof}

Since $V$ is an open subscheme of $\Spec \Lambda(G_v,\eta)$ and scheme theoretic image commutes with flat base change,
	\[ \mathscr{L} \times_{\Spec \Lambda(G_v,\eta)} V \longrightarrow 
	\Spec R_{\Lambda(G_v,\eta)}^{\triangle,\psi} \times_{\Spec \Lambda(G_v,\eta)} V
	\]
has injective structural morphism, so to prove it is an isomorphism it suffices to show it is a closed immersion. To show this, it suffices to show that if $A$ is a local ring and $(\rho_A,\chi_A,L_A)\in (\mathscr{L} \times_{\Spec \Lambda(G_v,\eta)} V)(A)$ is an $A$-point, the line $L_A$ is unique and is defined over $B$, where $B$ is the image of $R_{\Lambda(G_v,\eta)}^{\square,\psi}$ in $A$ under the morphism determined by $(\rho_A,\chi_A)$. Indeed, this implies that the fibres are all singletons and that the corresponding maps on local rings are surjective.

Let $(\rho_A,\chi_A,L_A) \in (\mathscr{L} \times_{\Spec \Lambda(G_v,\eta)} V)(A)$, with $A$ a local ring. Take $\sigma\in G_v$ such that  $\chi_A^2(\sigma) \ne \psi\epsilon_p(\sigma) \mod \mathfrak{m}_A$. Consider the matrix $M = \rho_A(\sigma) - \psi\epsilon_p\chi_A^{-1}(\sigma)$. Since $G_v$ acts on $L_A$ via $\psi\epsilon_p\chi_A^{-1}$ we see that $\det M = 0$. But, by our assumption on $\sigma$, its reduction modulo the maximal ideal of $A$ has rank $1$, hence one of the entries of the matrix $M$ is a unit. This implies that the line $L_A$ is unique and its projective coordinates can be defined using the entries of $M$, which are elements in the image of  $R_{\Lambda(G_v,\eta)}^{\square,\psi}$ in $A$.
	\end{proof}

\subsubsection{}\label{borelliftmodp}

Let $\mathrm{B}_2$ denote the Borel subgroup of upper triangular matrices in $\GL_2$. Fix a continuous homomorphism $\overline{\varrho} : G_v \rightarrow \mathrm{B}_2(\F)$ with $\det \overline{\varrho} = \overline{\psi\epsilon_p}$ and such that
	\[ \overline{\varrho} = \left(\begin{array}{cc} \ast & \ast \\ & \overline{\chi} \end{array} \right)
	\]
with $\overline{\chi}$ our fixed character. Define a functor $\mathcal{D}_{\overline{\varrho}}^{\mathrm{Bor},\psi}$ on $\CNLO$ by letting letting $\mathcal{D}_{\overline{\varrho}}^{\mathrm{Bor},\psi}(A)$ be the set of continuous morphisms $\varrho_A : G_v \rightarrow \mathrm{B}_2(A)$ that reduce to $\overline{\varrho}$ modulo $\mathfrak{m}_A$, have determinant $\psi\epsilon_p$, and such that, writing
	\[ \varrho_A = \left( \begin{array}{cc} \chi_1 & \ast \\ & \chi_2 \end{array}\right),
	\]
$\chi_2$ coincides with $\eta$ on the $p$-power torsion subgroup of $G_v^\mathrm{ab}$. 

Before proceeding with our analysis of $\mathcal{D}_{\overline{\varrho}}^{\mathrm{Bor},\psi}$, we record a lemma, which is \cite{MatsumuraCRT}*{Exercise 16.10}, that will be useful in the proof of \ref{borelliftringmodp} below.

\begin{lem}\label{MatsEx}

Let $R$ be a local Noetherian ring and let $r_1,\ldots, r_n \in R$, $n\ge 1$. If $r_1,\ldots,r_n$ generate a prime ideal of height $n$, then $R$ is a domain, and for any $1\le i\le r$, $r_1,\ldots,r_i$ generates a prime ideal of height $i$.

\end{lem}

\begin{proof}

We induct on $n$. Assume $r\in R$ generates a prime ideal of height $1$. Take a minimal prime $\mathfrak{q} \subset rR$. For any $x \in \mathfrak{q}$ we have $x = ry$ for some $y$. Then $r \notin\mathfrak{q}$ implies that $y\in \mathfrak{q}$, and $r\mathfrak{q} = \mathfrak{q}$. Nakayama's lemma implies $\mathfrak{q} = 0$.

For $n >1$, the image of $r_n$ in $R/(r_1,\ldots,r_{n-1})$ generates a prime ideal of height $1$, cf. \cite{MatsumuraCRT}*{Theorem 13.6 (iii)}, so the $n=1$ case implies that $r_1,\ldots,r_{n-1}$ generates a prime ideal of height $n-1$, and the lemma follows by induction.
	\end{proof}

\begin{prop}\label{borelliftringmodp}
	\begin{enumerate}
	\item $\mathcal{D}_{\overline{\varrho}}^{\mathrm{Bor},\psi}$ is representable by a $\CNLO$-algebra $R_{\overline{\varrho}}^{\mathrm{Bor},\psi}$. 
	\item Assume that $\overline{\varrho}$ has (possibly trivial) $p$-power order image. $R_{\overline{\varrho}}^{\mathrm{Bor},\psi}$ is a complete intersection of dimension $3+2[F_v:\Q_p]$.
	\item Assume that the image of $\overline{\varrho}$ is either trivial or has order $p$, and that if $p=2$, then either $F_v$ contains a primitive $4$-th root of unity or $[F_v:\Q_2] \ge 3$. Both $R_{\overline{\varrho}}^{\mathrm{Bor},\psi}$ and $R_{\overline{\varrho}}^{\mathrm{Bor},\psi}\otimes_\mathcal{O} \F$ are domains.
	\end{enumerate}
\end{prop}

\begin{proof}

The proof of representability in the general case is proved exactly as with $R_v^\square$ in \ref{RepOnTop}. We leave the details to the reader and henceforth assume that $\varrho$ has (possibly trivial) $p$-power order image.

For any $\sigma \in G_v$, we let $[\overline{\varrho}(\sigma)] \in \GL_2(\mathcal{O})$ denote the matrix whose entries are the Teichm\"{u}ller lifts of the entries of $\overline{\varrho}(\sigma)$. Let $G_v(p)$ be the maximal pro-$p$ quotient of $G_v$. Our assumption on the image of $\overline{\varrho}$ implies that for any $\varrho_A \in \mathcal{D}_{\overline{\varrho}}^{\mathrm{Bor},\psi}(A)$, $\varrho_A$ factors through $G_v(p)$. Let $\widetilde{\psi\epsilon_p}$ be the Teichm\"{u}ller lift of $\overline{\psi\epsilon_p}$, and set $\phi = (\psi\epsilon_p)(\widetilde{\psi\epsilon_p})^{-1}$.

If $F_v$ does not contain a $p$-th root of unity, then $G_v(p)$ is a free pro-$p$ group of rank $m=1+[F_v:\Q_p]$, cf. \cite{NSWCohomNumFields}*{Theorem 7.5.8}, and the $p$-part of the torsion subgroup of $G_v^\mathrm{ab}$ (and hence $\eta$) is trivial. Fix a set of generators $\gamma_1,\ldots,\gamma_m$ on which $G_v(p)$ is free, and define a lift
	\[ \varrho^\mathrm{univ} : G_v \rightarrow 
	\mathrm{B}_2(\mathcal{O}[[a_1,\ldots,a_m,b_1,\ldots,b_m]])
	\]
by
	\[ \varrho^{\mathrm{univ}}(\gamma_i) = [\overline{\varrho}(\gamma_i)]\left(\begin{array}{cc}
	\phi(\gamma_i)(1+a_i) & b_i \\ & (1+a_i)^{-1} \end{array}\right).
	\]
Any lift $\varrho_A \in \mathcal{D}_{\overline{\varrho}}^{\mathrm{Bor},\psi}(A)$ is a specialization of $\varrho^\mathrm{univ}$ via a unique $\CNLO$-morphism $\mathcal{O}[[a_1,\ldots,a_m,b_1,\ldots,b_m]] \rightarrow A$, and we have $R_{\overline{\varrho}}^{\mathrm{Bor},\psi} = \mathcal{O}[[a_1,\ldots,a_m,b_1,\ldots,b_m]]$. It has dimension $3+2[F_v:\Q_p]$, and both it and its reduction modulo $\varpi_E$ are domains. This proves both (2) and (3) in this case.

We now assume that $F_v$ contains a primitive $p$-th root of unity. Let $\mu_{p^s}$ denote the group of $p$-power roots of unity in $F_v$. Then $G_v(p)^\mathrm{ab}$ can be presented as the free pro-$p$ group on generators $\gamma_1,\ldots,\gamma_m$, with $m = 2+[F_v:\Q_p]$ modulo a single relation $k$. The shape of $k$ is divided into four subcases.
\begin{enumerate}
	\item[(a)] If $p^s>2$,
	\[
	 k = \gamma_1^{p^s}(\gamma_1,\gamma_2)(\gamma_3,\gamma_4)\cdots(\gamma_{m-1},
	\gamma_m),
	\]
where $(\gamma_i,\gamma_{i+1}) = \gamma_i^{-1}\gamma_{i+1}^{-1}\gamma_i\gamma_{i+1}$.
	\item[(b)] If $p^s = 2$ and $[F_v:\Q_2]$ is odd,
	\[
	 k = \gamma_1^2\gamma_2^4(\gamma_2,\gamma_3)\cdots(\gamma_{m-1},\gamma_m).
	\]
	\item If $p^s = 2$, $[F_v:\Q_2]$ is even, and the image of $\epsilon_2: G_v \rightarrow \Z_2^\times$ is procyclic,
	\[
	 k = \gamma_1^{2+2^f}(\gamma_1,\gamma_2)\cdots(\gamma_{m-1},\gamma_m),
	\]
for some $f\ge 2$.
	\item If $p^s = 2$, $[F_v:\Q_2]$ is even, and the image of $\epsilon_2: G_v \rightarrow \Z_p^\times$ is not procyclic,
	\[
	k = \gamma_1^2(\gamma_1,\gamma_2)\gamma_3^{2^f}(\gamma_3,\gamma_4)\cdots	
	(\gamma_{m-1},\gamma_m),
	\]
for some $f\ge 2$.
\end{enumerate}
Part (a) is due to Demu\u{s}kin, cf. \cite{NSWCohomNumFields}*{Theorem 7.5.9}, part (b) to Serre \cite{SerreProp}*{Corollaire 4.4}, and parts (c) and (d) to Labute \cite{LabuteDemuskin}*{Theorem 9}. We note that in \cite{SerreProp}, Serre uses a different convention for the commutator but either choice is valid, cf. \cite{NSWCohomNumFields}*{pg. 359}.
	
We note that the image of $\gamma_1,\ldots,\gamma_m$ in $G_v^\mathrm{ab}(p)$ are generators and are subject to the single relation $\gamma_1^{p^s} = 1$ in case (a), $\gamma_1^2\gamma_2^4 = 1$ in case (b), $\gamma_1^{2+2^f} = 1$ in case (c), and $\gamma_1^2\gamma_3^{2^f}=1$ in case (d). From this we see that the torsion subgroup of $G_v^\mathrm{ab}(p)$ is generated by the image of $\gamma_1$ in cases (a) and (c), by $\gamma_1\gamma_2^2$ in case (b), and by $\gamma_1\gamma_3^{2^{f-1}}$ in case (d). Let $F^\wedge$ be the free group pro-$p$ group on the set $\{\gamma_1,\ldots,\gamma_m\}$. Set $B = \mathcal{O}[[a_2,\ldots,a_m,b_1,\ldots,b_m]]$ and define
	\[ \varrho_{F^\wedge} : F^\wedge \longrightarrow \mathrm{B}_2(B)
	\]
as follows. In cases, (a) and (c) we set	
	\[ \varrho_{F^\wedge}(\gamma_1) =[\overline{\varrho}(\gamma_1)]
	 \left(\begin{array}{cc} \phi\eta^{-1}(\gamma_1) & b_1 \\ 
		& \eta(\gamma_1) \end{array}\right).
	\]
In case (b), we set
	\[ \varrho_{F^\wedge}(\gamma_1) =[\overline{\varrho}(\gamma_1)]
	 \left(\begin{array}{cc} \phi(\gamma_1)
	 \eta^{-1}(\gamma_1\gamma_2^2)(1+a_2)^{-2} & b_1 \\ 
		& \eta(\gamma_1\gamma_2^2)(1+a_2)^2 \end{array}\right).
	\]
In case (d), we set
	\[ \varrho_{F^\wedge}(\gamma_1) =[\overline{\varrho}(\gamma_1)]
	 \left(\begin{array}{cc} \phi(\gamma_1)
	 \eta^{-1}(\gamma_1\gamma_3^{2^{f-1}})(1+a_3)^{-2^{f-1}} & b_1 \\ 
		& \eta(\gamma_1\gamma_3^{2^{f-1}})(1+a_3)^{2^{f-1}} \end{array}\right).
	\]
And in all cases, let	
	\[ \varrho_{F^\wedge}(\gamma_i) =[\overline{\varrho}(\gamma_i)]
	 \left(\begin{array}{cc} \phi(\gamma_i)(1+a_i) & b_i \\ 
	 & (1+a_i)^{-1} \end{array}\right),
		\]
for $2\le i\le m$. It is easy to check that in all cases we have
	\[ \varrho_{F^\wedge}(k) - 1
	= \left(\begin{array}{cc} 1 & r \\ & 1
		\end{array}\right).
	\]
with $r$ a non-zero element in $\mathfrak{m}_B$. For any $\varrho_A \in \mathcal{D}_{\overline{\varrho}}^{\mathrm{Bor},\psi}(A)$, $\varrho_A$ is the pushforward of a unique $\CNLO$-morphism $B\rightarrow A$ that contains $r$ in its kernel. From this it follows that $R_\varrho^{\mathrm{Bor},\psi} \cong B/rB$ and the universal lift $\varrho^\mathrm{univ}$ is the pushforward of $\varrho_{F^\wedge}$ via the natural surjection. This proves (2).

We now prove (3). Let $\varpi_E$ be a uniformizer of $E$. Using \ref{MatsEx}, to show that $R_\varrho^{\mathrm{Bor},\psi} \cong B/(r)$ and $R_\varrho^{\mathrm{Bor},\psi}\otimes_\mathcal{O} \F \cong B/(r,\varpi_E)$ are domains, it suffices to show that $B/(r,\varpi_E,a_2,\ldots,a_{m-2},b_1,\ldots,b_{m-2})$ is a domain of dimension $3$. Let $r_0$ denote the image of $r$ in $B/(\varpi_E,a_2,\ldots,a_{m-2},b_1,\ldots,b_{m-2}) \cong \F[[a_{m-1},a_m,b_{m-1},b_m]]$. We are thus reduced to showing that $r_0$ is irreducible in $\F[[a_{m-1},a_m,b_{m-1},b_m]]$. 

Let $\varrho_0$ denote the pushforward of $\varrho_{F^\wedge}$ to $\F[[a_{m-1},a_m,b_{m-1},b_m]]$. Note that $m\ge 3$, and if $p=2$, since we are assuming either $F_v$ contains a primitive $4$-th root of unity or $[F_v:\Q_2]\ge 3$, if we are in case (b) or (d) above we must have $m\ge 5$. Then, in all cases, $r_0$ is given by
	\[ \varrho_0(\gamma_{m-1})^{-1}\varrho_0(\gamma_m)^{-1}\varrho_0(\gamma_{m-1})
	\varrho_0(\gamma_m) = \left(\begin{array}{cc} 1 & r_0 \\ & 1 \end{array}\right).
	\]
Our assumption on the image of $\overline{\varrho}$ implies that at most one of $\overline{\varrho}(\gamma_{m-1})$ and $\overline{\varrho}(\gamma_m)$ is non-trivial and that if it is non-trivial, then it is unipotent. Our assumptions also imply $\psi\epsilon_p$ is trivial mod $\varpi_E$. Define $\beta_i$, for $i=m-1,m$, by
	\[ \varrho_0(\gamma_i) = \left(\begin{array}{cc} 1+a_i & \beta_i \\ & (1+a_i)^{-1} 
	\end{array}\right) = 
	\overline{\varrho}(\gamma_i)\left(\begin{array}{cc} 1+a_i & b_i \\
	& (1+a_i)^{-1} \end{array} \right).
	\]
A straightforward computation shows that
	\[ r_0 = (1+a_{m-1})^{-1}(1+a_m)^{-1}\big(\beta_m ((1+a_{m-1})-(1+a_{m-1})^{-1}) - 
	\beta_{m-1}((1+a_m)-(1+a_m)^{-1})\big).
	\]
To finish the proof, it suffices to show that
	\begin{equation}\label{therelation}
	 r_1 = \beta_m ((1+a_{m-1})-(1+a_{m-1})^{-1}) - \beta_{m-1}((1+a_m)-(1+a_m)^{-1})
	\end{equation}
is irreducible. To do this we use the following easy but useful fact we leaned from \cite{KunzPlanAlgCurv}*{pg. 164}: if $K$ is a field and $f \in K[[x_1,\ldots,x_n]]$ is reducible, then for any grading $\deg(x_i) = n_i>0$, the lowest degree term of $f$ is reducible in $K[x_1,\ldots,x_n]$.

First consider the case that $\overline{\varrho}(\gamma_{m-1})$ is nontrivial, and write
	\[ \overline{\varrho}(\gamma_{m-1}) = \left(\begin{array}{cc} 1 & \alpha \\ & 1\end{array}\right),
	\]
with $\alpha \in \F^\times$, so that $\beta_{m-1} = \alpha(1+a_{m-1})^{-1} + b_{m-1}$. Note then $\beta_m = b_m$, since the image of $\overline{\varrho}$ has order $p$. If $p \ne 2$, we use the grading $\deg(a_i)=\deg(b_i)=1$, and \eqref{therelation} becomes
	\[ 2\alpha a_{m} + \text{ higher order terms}
	\]
and $2\alpha a_{m-1}$ is irreducible in $\F[a_{m-1},a_m,b_{m-1},b_m]$. If $p = 2$, we use the grading $\deg(a_{m-1})=\deg(b_{m-1}) = 1$ and $\deg(a_m) = \deg(b_m) = 2$. Then \eqref{therelation} becomes
	\[ -\alpha a_{m}^2 + b_{m}a_{m-1}^2 + \text{ higher order terms}
	\]
and $-\alpha a_{m-1}^2 + b_{m-1}a_m^2$ is irreducible in $\F[a_{m-1},a_m,b_{m-1},b_m]$. The case when $\overline{\varrho}(\gamma_m)$ is nontrivial is symmetric.

Now assume $\overline{\varrho}(\gamma_{m-1}) = \overline{\varrho}(\gamma_m) = 1$. We use the standard grading $\deg(a_i)=\deg(b_i) = 1$, and \eqref{therelation} becomes
	\[ 2a_{m-1}b_m - 2a_m b_{m-1} + \text{ higher order terms}
	\]
if $p \ne 2$, and
	\[ - a_{m-1}^2b_m + a_m^2b_{m-1} + \text{ higher order terms}
	\]
if $p=2$. In either case the leading term is irreducible in $\F[a_{m-1},a_m,b_{m-1},b_m]$. This concludes the proof of (3), and of the proposition.
	\end{proof}
	
Writing the universal lift $\varrho^\mathrm{univ} : G_v \rightarrow \mathrm{B}_2(R_{\overline{\varrho}}^{\mathrm{Bor},\psi})$ as
	\[ \varrho^\mathrm{univ} = \left(\begin{array}{cc} \chi_1^\mathrm{Bor} & \ast \\ & 
	\chi_2^\mathrm{Bor} \end{array}\right).
	\]
The character $\chi_2^\mathrm{Bor}$ gives a $\CNLO$-morphism $\Lambda(G_v,\eta) \rightarrow R_{\overline{\varrho}}^{\mathrm{Bor},\psi}$.
	
\subsubsection{}\label{closedpoints}

Let $\mathscr{L}$ be as in \ref{ProjRep}, and let $x$ be a closed point of $\mathscr{L}$. Then $x$ is simply a choice of $\rhobar$-stable line $L_x$ in the representation space of $\rhobar$ such that $G_v$ acts via $\overline{\chi}$ on $V_\F/L_x$. 

Consider the set valued functor $\mathcal{DL}_x^{\square,\psi}$ on $\CNLO$ that sends a $\CNLO$-algebra $A$ to the set of pairs $(\rho_A, L_A)$, where $\rho_A$ is a lift of $\rhobar$ to $\GL_2(A)$ with determinant $\psi\epsilon_p$, and $L_A$ is a $G_v$-stable line in $A^2$ lifting $L_x$ such that the action of the $p$-part of the torsion subgroup of $G_v^\mathrm{ab}$ on $A^2/L_x$ is given by $\eta$. Let $\mathcal{O}_{\mathscr{L},x}^\wedge$ denote the completion of the local ring of $\mathscr{L}$ at $x$. Note that the natural map
	\[ \Spec \mathcal{O}_{\mathscr{L},x}^\wedge \longrightarrow \mathscr{L} 
	\]
yields a lift $\rho_x^\wedge: G_v \rightarrow \GL_2(\mathcal{O}_{\mathscr{L},x}^\wedge)$ of $\rhobar$, and a $G_v$-stable line $L_x^\wedge$ lifting $L_x$. The following lemma is immediate.

\begin{lem}\label{complocrep}

$\mathcal{O}_{\mathscr{L},x}^\wedge$ represents $\mathcal{DL}_x^{\square,\psi}$ with universal object $(\rho_x^\wedge,L_x^\wedge)$.

\end{lem}
	
For any ring $A$, we will denote by $L_A^\mathrm{std}$ the $A$-line in $A^2$ fixed by the upper-triangular matrices. Let $x =(\overline{\rho},\overline{\chi},L_x)$ be a closed point of $\mathscr{L}$. Take $g \in \GL_2(\mathcal{O})$ such that $\overline{g}L_x = L_{\F}^\mathrm{std}$. Note that $\overline{g}\rhobar\overline{g}^{-1}$ is upper triangular, and we have the functor $\mathcal{D}_{\overline{g}\rhobar\overline{g}^{-1}}^{\mathrm{Bor},\psi}$ on $\CNLO$ as in \ref{borelliftmodp}, which is represented by $R_{g\rho_x g^{-1}}^{\mathrm{Bor},\psi}$ as in \ref{borelliftringmodp}.

\begin{lem}\label{smoverborelmodp}

There is an isomorphism $\mathcal{O}_{\mathscr{L},x}^\wedge \cong R_{\overline{g}\overline{\rho} \overline{g}^{-1}}^{\mathrm{Bor},\psi}[[z]]$ of $\Lambda(G_v,\eta)$-algebras.

\end{lem}

\begin{proof}

For ease of notation set $R = R_{\overline{g}\rhobar \overline{g}^{-1}}^{\mathrm{Bor},\psi}$ and let $\varrho$ denote its universal lift. Define a lift $\rho_{R[[z]]} : G_v \rightarrow R[[z]]$ of $\rhobar$ by
	\[
	\rho_{R[[z]]}(\sigma) = g^{-1} \left(\begin{array}{cc} 1 & \\ z & 1\end{array}\right)
	\varrho(\sigma) \left(\begin{array}{rc} 1 & \\ -z & 1 \end{array}\right)g,
	\]
an let $L_{R[[z]]}$ denote the $\rho_{R[[z]]}$-stable line $g^{-1}\left(\begin{array}{cc} 1 & \\ z & 1 \end{array}\right)L_{R}^\mathrm{std}$. Note that $(\rho_{R[[z]]},L_{R[[z]]})$ determines a local $\Lambda(G_v,\eta)$-algebra morphism $\mathcal{O}_{\mathscr{L},x}^\wedge\rightarrow R[[z]]$. Given any $\CNLO$-algebra $A$, and $(\rho_A,L_A) \in \mathcal{DL}_x^{\square,\psi}(A)$, there is a unique $c\in\mathfrak{m}_A$ such that
	\[
	L_A = g^{-1}\left(\begin{array}{cc} 1 & \\ c & 1 \end{array} \right)L_A^\mathrm{std}.
	\]
We then have a unique local $\mathcal{O}$-algebra morphism $R[[z]] \rightarrow A$, which is also a morphism of $\Lambda(G_v,\eta)$-algebras, such that $(\rho_{R[[z]]},L_{R[[z]]})$ specializes to $(\rho_A,L_A)$. This shows that $\mathcal{O}_{\mathscr{L},x}^\wedge \rightarrow R[[z]]$ is an isomorphism of $\Lambda(G_v,\eta)$-algebras.
	\end{proof}
	
\begin{prop}\label{triliftringdomain}

Assume that the image of $\rho$ is either trivial or has order $p$. If $p=2$, assume further that $F_v$ contains a primitive $4$th root of unity or $[F_v:\Q_2]\ge 3$. 
\begin{enumerate}
	\item $R_{\Lambda(G_v,\eta)}^{\triangle,\psi}$ is an $\mathcal{O}$-flat domain of dimension $4+2[F_v:\Q_p]$. 
	\item Let $Z$ denote the closed subscheme of $\Spec \Lambda(G_v,\eta)$ defined by $(\chi_\eta^\mathrm{univ})^2 = \psi\epsilon_p$, and let $V$ denote its complement. The scheme
	\[
	(\Spec R_{\Lambda(G_v,\eta)}^{\triangle,\psi} \times_{\Spec \Lambda(G_v,\eta)} V)
	\otimes_\mathcal{O} \F
	\]
is integral.
\end{enumerate}

\end{prop}

\begin{proof}

Note that the special fibre of $\mathscr{L}$ over $R_{\Lambda(G_v,\eta)}^{\triangle,\psi}$ is isomorphic to $\mathbb{P}^1_\F$ if $\rhobar$ is trivial, and is a point otherwise, so $\mathscr{L}$ is connected. By \ref{smoverborelmodp} and \ref{borelliftringmodp}, the completed local ring of $\mathscr{L}$ at any closed point is a domain of dimension $4+2[F_v:\Q_p]$, and its reduction modulo the maximal ideal of $\mathcal{O}$ is also a domain. We conclude that both $\mathscr{L}$ and $\mathscr{L}\otimes_\mathcal{O} \F$ are integral. This implies that $R_{\Lambda(G_v,\eta)}^{\triangle,\psi}$ is a domain as well as part (2) of the proposition, by \ref{opensetiso}. It is not hard to see that $\Spec R_{\Lambda(G_v,\eta)}^{\triangle,\psi} \times_{\Spec \Lambda(G_v,\eta)} V$ and $\mathscr{L}\times_{\Spec \Lambda(G_v,\eta)} V$ are each of codimension $1$, so $\dim R_{\Lambda(G_v,\eta)}^{\triangle,\psi} = \dim \mathscr{L} = 4+2[F_v:\Q_p]$.

\end{proof}

\subsubsection{}\label{borelliftschar0}

Fix a finite extension $E'/E$ with ring of integers $\mathcal{O}'$, and a continuous homomorphism $\varrho: G_v \rightarrow \mathrm{B}_2(E')$ with $\det \varrho = \psi\epsilon_p$. We do not assume that $\varrho$ is a lift of $\rhobar$. Define a functor $\mathcal{D}_\varrho^{\mathrm{Bor},\psi}$ on $\mathrm{Ar}_{E'}$ by letting letting $\mathcal{D}_{\varrho}^{\mathrm{Bor},\psi}(A)$ be the set of continuous morphisms $\varrho_A : G_v \rightarrow \mathrm{B}_2(A)$ that reduce to $\varrho$ modulo $\mathfrak{m}_A$ and have determinant $\psi\epsilon_p$.

\begin{prop}\label{char0borelliftring}

The functor $\mathcal{D}_\varrho^{\mathrm{Bor},\psi}$ is pro-represented by a complete, Noetherian, local $E'$-algebra $R_\varrho^{\mathrm{Bor},\psi}$. If $\varrho \ne \left(\begin{array}{cc} \chi \epsilon_p & \\ & \chi \end{array}\right)$ for some character $\chi$, then $R_\varrho^{\mathrm{Bor},\psi}$ is formally smooth over $E'$ of dimension $2+2[F_v:\Q_p]$.

\end{prop}

\begin{proof}
	We prove representability in the same way as \ref{RepOnTop} and \ref{borelliftringmodp}. There is a finite index subgroup $G'$ of $G_v$ such that $\varrho(G')$ is pro-$p$. Let $G'(p)$ be the maximal pro-$p$ quotient of $G'$ and let $H$ be the kernel of the natural surjection $G'\rightarrow G'(p)$. We have that $H$ is normal in $G_v$, that $G_v/H$ is topologically finitely generated, and that for any $\varrho_A \in \mathcal{D}_\varrho^{\mathrm{Bor},\psi}$, $A$ in $\mathrm{Ar}_{E'}$, $\varrho_A$ factors through $G_v/H$. Fix a set of topological generators $\gamma_1,\ldots,\gamma_m$ for $G_v/H$. Let $F$ be the free group on $\{\gamma_1,\ldots,\gamma_m\}$ and let $F^\wedge$ denote its profinite completion. Let $K$ denote the kernel of the natural surjection $F^\wedge \rightarrow G_v/H$. Define a continuous homomorphism
	\[
	\varrho_{F^\wedge} : F^\wedge \longrightarrow
	 \mathrm{B}_2(E'[[a_1,\ldots,a_m,b_1,\ldots,b_m]])
	 \]
by
	\[
	\varrho_{F^\wedge}(\gamma_i) = \varrho(\gamma_i)\left(\begin{array}{cc} 1+a_i & b_i \\
	& (1+a_i)^{-1} \end{array}\right)
	\]
Now let $J$ denote the ideal of $E'[[a_1,\ldots,a_m,b_1,\ldots,b_m]]$ generated by the entries of the of the matrices $\varrho_{F^\wedge}(k)-1$, for $k\in K$. Set $R_\varrho^{\mathrm{Bor},\psi} = E'[[a_1,\ldots,a_m,b_1,\ldots,b_m]]/J$. The pushforward of $\varrho_{R^\wedge}$ along the surjection $E'[[a_1,\ldots,a_m,b_1,\ldots,b_m]] \rightarrow R_\varrho^{\mathrm{Bor},\psi}$ gives a continuous morphism $G_v/H \rightarrow \mathrm{B}_2(R_\varrho^{\mathrm{Bor},\psi})$, and we let 
	\[ \varrho^\mathrm{univ} : G_v \longrightarrow \mathrm{B}_2(R_\varrho^{\mathrm{Bor},\psi})
	\]
denote the composite of this morphism with the natural surjection $G_v \rightarrow G_v/H$. It is easy to see that $R_\varrho^{\mathrm{Bor},\psi}$ pro-represents $\mathcal{D}_\varrho^{\mathrm{Bor},\psi}$ with universal object $\varrho^\mathrm{univ}$.

Let $E'[\varepsilon] = E'[[x]]/(x^2)$. Let $\mathfrak{b}$ denote the subspace of upper triangular matrices in $\mathrm{M}_{2\times 2}(E')$ with $G_v$-action given by $\sigma X = \varrho(\sigma)X\varrho(\sigma)^{-1}$, and let $\mathfrak{b}^0$ denote its trace zero subspace. The tangent space of $R_\varrho^{\mathrm{Bor},\psi}$ is given by $\mathcal{D}_\varrho^{\mathrm{Bor},\psi}(E'[\varepsilon])$. Given $\varrho_{E'[\varepsilon]}\in \mathcal{D}_\varrho^{\mathrm{Bor},\psi}(E'[\varepsilon])$, write $\varrho_{E'[\varepsilon]}(\sigma)$ as
	\[ \varrho_{E'[\varepsilon]}(\sigma) = (1+\varepsilon c(\sigma))\varrho(\sigma),
	\]
for each $\sigma\in G_v$, with $c(\sigma) \in \mathfrak{b}$. It is easily checked that $\varrho_{E'[\varepsilon]}$ is a homomorphism with $\det \varrho_{E'[\varepsilon]} = \psi\epsilon_p = \det\varrho$ if and only if $c\in Z^1(G_v,\mathfrak{b}^0)$, the space of $1$-cocycles of $G_v$ with coefficients in $\mathfrak{b}^0$. This determines an isomorphism of $E'$-vector spaces $\mathcal{D}_\varrho^{\mathrm{Bor},\psi}(E'[\varepsilon]) \cong Z^1(G_v,\mathfrak{b}^0)$. Now
	\begin{align*}
	\dim Z^1(G_v,\mathfrak{b}^0) &= \dim H^1(G_v,\mathfrak{b}^0) 
		- \dim H^0(G_v,\mathfrak{b}^0) + 2\\
		& = 2[F_v:\Q_p] + 2 + H^2(G_v,\mathfrak{b}^0),
	\end{align*}
by Euler-Poincar\'{e} characteristic.

Now fix a minimal presentation
	\[ 0 \longrightarrow J \longrightarrow A \longrightarrow R_\varrho^{\mathrm{Bor},\psi}
	\longrightarrow 0,
	\]
with $A$ a powerseries over $E'$ in $\dim Z^1(G_v,\mathfrak{b}^0)$ variables. As in \cite{MazurDefGalRep}*{\S 1.6}, we will show that
	\[ \dim J/\mathfrak{m}_A J \le \dim H^2(G_v,\mathfrak{b}^0).
	\]
Let $A_n = A/\mathfrak{m}_A^n$, $R_n = R_\varrho^{\mathrm{Bor},\psi}/\mathfrak{m}_{R_\varrho^{\mathrm{Bor},\psi}}^n$, and $J_n = \ker( A_n \rightarrow R_n)$. Then for $n$ sufficiently large the natural map $J/\mathfrak{m}_A J \rightarrow J_n/\mathfrak{m}_{A_n}J_n$ is an isomorphism of $E'$ vector spaces, and it suffices to show
	\[
	\dim J_n/\mathfrak{m}_{A_n}J_n \le \dim H^2(G_v,\mathfrak{b}^0).
	\]
Choose a continuous set theoretic lifting
	\[ \tilde{\varrho} : G_v \longrightarrow \mathrm{B}_2(A_n)
	\]
such that $\det\tilde{\varrho}(\sigma) = \psi\epsilon_p(\sigma)$ for all $\sigma\in G_v$, and such that its pushforward to $R_n$ is $\varrho^\mathrm{univ}$ mod $\mathfrak{m}_{R_\varrho^{\mathrm{Bor},\psi}}^n$. To see that this is possible, note that we can take a continuous $E'$-vector space section $s: R_n \rightarrow A_n$ of the quotient $A_n \rightarrow R_n$. This determines a continuous set-theoretic section $\mathrm{B}_2(R_n) \rightarrow \mathrm{B}_2(A_n)$, and we let $\tilde{\varrho}'$ be the composite of the map $G_v \rightarrow \mathrm{B}_2(R_n)$ with this section. Since the map $\sigma \mapsto \psi\epsilon_p(\sigma)\det\tilde{\varrho}'(\sigma)^{-1}$ is continuous on $G_v$, so is the function $\tilde{\varrho} : G_v \rightarrow \mathrm{B}_2(A_n)$ defined by
	\[ \tilde{\varrho}(\sigma) = \tilde{\varrho}'(\sigma)\left(\begin{array}{cc}
		\psi\epsilon_p(\sigma)\det\tilde{\varrho}'(\sigma)^{-1} & \\ & 1 \end{array}\right),
	\]
which is the desired set theoretic lifting.
Define the continuous function
	\[ c : G_v \times G_v \longrightarrow \mathfrak{b}^0 \otimes_{E'} J_n/\mathfrak{m}_{A_n}J_n
	\]
by $c(\sigma_1,\sigma_2) = \tilde{\varrho}(\sigma_1\sigma_2)^{-1}\tilde{\varrho}(\sigma_2)^{-1}\tilde{\varrho}(\sigma_1)^{-1}$. It can be checked that $c$ is a $2$-cocyle of $G_v$ with values in $\mathfrak{b}^0\otimes_{E'} J_n/\mathfrak{m}_{A_n}J_n$, and that is image $[c]$ in $H^2(G_v,\mathfrak{b}^0\otimes_{E'}J_n/\mathfrak{m}_{A_n}J_n) \cong H^2(G_v,\mathfrak{b}^0)\otimes_{E'} J_n/\mathfrak{m}_{A_n}J_n$ does not depend on the choice of $\tilde{\varrho}$. We get a natural map
	\[ (J_n/\mathfrak{m}_{A_n}J_n)^\ast \longrightarrow H^2(G_v,\mathfrak{b}^0)
	\]
by $\lambda \mapsto (1\otimes\lambda)([c])$. We will show that this map is injective. Take non-zero $\lambda \in (J_n/\mathfrak{m}_{A_n}J_n)^\ast$ such that $(1\otimes\lambda)([c]) = 0$. Let $A_n'$ denote the quotient of $A_n$ by the kernel of $\lambda$. We have an exact sequence
	\begin{equation}\label{obstructionseqn}
	0 \longrightarrow E' \longrightarrow A_n' \longrightarrow R_n \longrightarrow 0
	\end{equation}
and the obstruction class $[c]$ vanishes for this extension. This implies that there is a continuous homomorphism
	\[ \varrho_{A_n'} : G_v \longrightarrow \mathrm{B}_2(A_n')
	\]
with determinant $\psi\epsilon_p$, and whose pushforward to $R_n$ is $\varrho^\mathrm{univ}$ modulo $\mathfrak{m}_{R_\varrho^{\mathrm{Bor},\psi}}^n$. Then there is a local $E'$-algebra morphism $R_\varrho^{\mathrm{Bor},\psi} \rightarrow A_n'$ inducing $\varrho_{A_n'}$. Since $\mathfrak{m}_{A_n'}^n = 0$, this map factors through $R_n$ and induces a section of \eqref{obstructionseqn}. But this contradicts the fact that $A_n' \rightarrow R_n$ induces an isomorphism on tangent spaces.

We deduce that there is a presentation
	\[ R_\varrho^{\mathrm{Bor},\psi} \cong E'[[x_1,\ldots,x_g]]/(f_1,\ldots,f_r)
	\]
where $g = 2+ 2[F_v:\Q_p] + \dim H^2(G_v,\mathfrak{b}^0)$ and $r \le H^2(G_v,\mathfrak{b}^0)$. In particular, we conclude that $\dim R_\varrho^{\mathrm{Bor},\psi} \ge 2+2[F_v:\Q_p]$. Let $\mathfrak{u}$ be subspace of $\Ad^0$ consisting of upper triangular unipotent matrices. Note that $\mathfrak{u}$ is $G_v$-stable and the trace pairing on $\Ad^0$ induces a $G_v$-equivariant isomorphism $(\mathfrak{b}^0)^\ast \cong \Ad^0/\mathfrak{u}$. Local Tate duality gives an isomorphism $H^2(G_v,\mathfrak{b}^0) \cong H^0(G_v,(\Ad^0/\mathfrak{u})(1))$. It is easily checked that this latter space is non-zero if and only if
	\[
	\varrho = \left(\begin{array}{cc} \epsilon_p \chi & \\ & \chi \end{array}\right)
	\]
for some character $\chi$, in which case it has dimension $1$.
	\end{proof}
	
\subsubsection{}\label{ClosedPtsOfG}

Let $\mathscr{L}$ be as in \ref{ProjRep}. Take a closed point $x \in\mathscr{L}[1/p]$ with residue field $E'$. Note that since $\mathscr{L}$ is finite type over $R_{\Lambda(G_v,\eta)}^{\square,\psi}$, part (2) of \ref{OPointsOfR} implies that $E'/E$ is finite. Let $\rho_x$ and $\chi_x$ denote the pushforward of $\rho^\mathrm{univ}$ and $\chi_\eta^\mathrm{univ}$, repsectively, to $E'$, and let $L_x$ denote the fixed $G_v$-stable line in $(E')^2$ such that $G_v$ acts on $(E')^2/L_x$ via $\chi_x$. 

Consider the set valued functor $\mathcal{DL}_x^{\square,\psi}$ on $\mathrm{Ar}_{E'}$ that sends an $\mathrm{Ar}_{E'}$-algebra $A$ to the set of pairs $(\rho_A, L_A)$, where $\rho_A$ is a lift of $\rho_x$ to $\GL_2(A)$ with determinant $\psi\epsilon_p$, and $L_A$ is a $G_v$-stable line in $A^2$ lifting $L_x$. Let $\mathcal{O}_{\mathscr{L},x}^\wedge$ denote the completion of the local ring of $\mathscr{L}$ at $x$. The natural map
	\[ \Spec \mathcal{O}_{\mathscr{L},x}^\wedge \longrightarrow \mathscr{L}
	\]
yields a lift $\rho_x^\wedge: G_v \rightarrow \GL_2(\mathcal{O}_{\mathscr{L},x}^\wedge)$ of $\rho_x$, and a $G_v$-stable line $L_x^\wedge$ lifting $L_x$.

\begin{lem}\label{complocrep}

$\mathcal{O}_{\mathscr{L},x}^\wedge$ pro-represents $\mathcal{DL}_x^{\square,\psi}$ with universal object $(\rho_x^\wedge,L_x^\wedge)$.

\end{lem}

\begin{proof}

Let $B$ be an $\mathrm{Ar}_{E'}$-algebra and let $(\rho_B,L_B) \in \mathcal{DL}_x^{\square,\psi}(B)$. Denote by $\chi_B$ the character giving the $G_v$-action on $B^2/L_B$. By \ref{EAlgRep}, there are unique continuous $\mathcal{O}$-algebra morphisms $R_v^\square \rightarrow B$ and $\Lambda(G_v) \rightarrow B$ such that $\rho_B$ is the pushforward of $\rho^\mathrm{univ}$, and $\chi_B$ is the pushforward of $\chi^\mathrm{univ}$. Since the $\rho_B$ has determinant $\psi\epsilon_p$, and the restriction of $\chi_B$ to the torsion subgroup of $G_v^\mathrm{ab}$ is equal to $\eta$, we get a unique continuous $\mathcal{O}$-algebra morphism $R_{\Lambda(G_v,\eta)}^{\square,\psi} \rightarrow B$ which gives rise to $\rho_B$ and $\chi_B$. Since the line $L_B$ is stable under $\rho_B$, we obtain a unique morphism $y: \Spec B \rightarrow \mathscr{L}$ such that the image of the closed point of $\Spec B$ is $x$. As $B$ is Artinian, $y$ factors through $\Spec \mathcal{O}_{\mathscr{L},x}^\wedge \rightarrow \mathscr{L}$.
	\end{proof}
		
Let $x =(\rho_x,\chi_x,L_x)$ be a closed point of $\mathscr{L}[1/p]$, and denote by $E'$ its residue field. Take $g \in \GL_2(E')$ such that $gL_x = L_{E'}^\mathrm{std}$. Note that $g \rho_x g^{-1}$ is upper triangular, and we have the functor $\mathcal{D}_{g\rho_x g^{-1}}^{\mathrm{Bor},\psi}$ on $\mathrm{Ar}_{E'}$ as in \ref{borelliftschar0}, which is represented by $R_{g\rho_x g^{-1}}^{\mathrm{Bor},\psi}$ as in \ref{char0borelliftring}.

\begin{lem}\label{smoverborel}

There is an isomorphism $\mathcal{O}_{\mathscr{L},x}^\wedge \cong R_{g\rho_xg^{-1}}^{\mathrm{Bor},\psi}[[z]]$.

\end{lem}

\begin{proof}

For ease of notation set $R = R_{g\rho_x g^{-1}}^{\mathrm{Bor},\psi}$ and let $\varrho$ denote its universal lift. Define a lift $\rho_{R[[z]]} : G_v \rightarrow R[[z]]$ of $\rho_x$ by
	\[
	\rho_{R[[z]]}(\sigma) = g^{-1}\left(\begin{array}{cc} 1 & \\ z & 1\end{array}\right)
	\varrho(\sigma) \left(\begin{array}{rc} 1 & \\ -z & 1 \end{array}\right)g,
	\]
and we let $L_{R[[z]]}$ denote the $\rho_{R[[z]]}$-stable line $g^{-1}\left(\begin{array}{cc} 1 & \\ z & 1 \end{array}\right)L_{R}^\mathrm{std}$. Note that $(\rho_{R[[z]]},L_{R[[z]]})$ determines a local $E'$-algebra morphism $\mathcal{O}_{\mathscr{L},x}^\wedge\rightarrow R[[z]]$. Given any $\mathrm{Ar}_{E'}$-algebra $A$ and $(\rho_A,L_A) \in \mathcal{DL}_x^{\square,\psi}(A)$, there is a unique $c\in\mathfrak{m}_A$ such that
	\[
	L_A = g^{-1}\left(\begin{array}{cc} 1 & \\ c & 1 \end{array} \right)L_A^\mathrm{std}.
	\]
We then have a unique local $E'$-algebra morphism $R[[z]] \rightarrow A$, such that the specialization of $(\rho_{R[[z]]},L_{R[[z]]})$ is $(\rho_A,L_A)$. This shows that $\mathcal{O}_{\mathscr{L},x}^\wedge \rightarrow R[[z]]$ is an isomorphism.
	\end{proof}

\begin{prop}\label{triliftringsmpts}

Let $x = (\rho_x,\chi_x)$ be a closed point of $\Spec R_{\Lambda(G_v,\eta)}^{\triangle,\psi}[1/p]$. If $\chi_x^2 \ne \psi,\psi\epsilon_p$, then the local ring of $R_{\Lambda(G_v,\eta)}^{\triangle,\psi}$ at $x$ is formally smooth over $E$.

\end{prop}

\begin{proof}

Let $E'$ denote the residue field of $x$. Since $\chi_x^2 \ne \psi\epsilon_p$, \ref{opensetiso} implies there is a unique $E'$ point of $\mathscr{L}$, which we denote again by $x$, and an isomorphism $(R_{\Lambda(G_v,\eta)}^{\triangle,\psi})_x^\wedge \rightarrow \mathcal{O}_{\mathscr{L},x}^\wedge$. The result then follows from \ref{char0borelliftring} and \ref{smoverborel}.
	\end{proof}

%% file: DefTheory/LocalDefsNotp.tex
In this subsection, $F_v$ is either be a finite extension of $\Q_\ell$ with $\ell \ne p$, or $F_v = \R$. We fix a continuous homomorphism
	\[ \overline{\rho} : G_v \longrightarrow \GL_2(\F) \]
and a continuous character $\psi : G_v \rightarrow \calO^\times$ such that $\det \overline{\rho} = \overline{\psi\epsilon_p}$. We let $V_\F$ denote the representation space of $\rhobar$. We let $\mathcal{D}_v^\square$ and $\mathcal{D}_v^{\square,\psi}$ denote the functor of lifts of $\rhobar$ and the subfunctor consisting of lifts with determinant $\psi\epsilon_p$, repsectively. We denote the corresponding representing objects by $R_v^\square$ and $R_v^{\square,\psi}$, respectively.

\subsubsection{}\label{lnotpsec} 
Assume that $F_v$ is a finite extension of $\Q_\ell$ with $\ell \ne p$. Let $q$ denote the cardinality of the residue field of $F_v$. Let $\rho^\mathrm{univ}$ denote the universal lift to $R^\square_v$. 

\begin{prop}\label{FixedDetlp}
Let $\overline{R}_v^{\square,\psi}$ denote the quotient of $R_v^{\square,\psi}$ by its $p$-torsion. 
\begin{enumerate}
	\item $\overline{R}_v^{\square,\psi}$ is $\calO$-flat and equidimensional of relative dimension $3$.
	\item The set of unramified lifts of $\rhobar_v$ form an irreducible component of $\Spec \overline{R}_v^{\square,\psi}$, which is formally smooth over $\calO$. We denote the quotient by the corresponding minimal prime by $R_v^{\square,\psi,\mathrm{ur}}$.
	\item If $E'/E$ is finite and $x : \overline{R}_v^{\square,\psi} \rightarrow E$ is a continuous $\mathcal{O}$-algebra morphism such that $\ker(x)$ is contained in more than one irreducible component, then letting $\rho_x$ denote the corresponding lift, $\rho_x \cong \gamma_v\epsilon_p \oplus \gamma_v$ for some character $\gamma_v$ of $G_v$.
\end{enumerate}
\end{prop}

\begin{proof}
By definition $\overline{R}_v^{\square,\psi}$ is $p$-torsion free, hence is $\calO$-flat. The dimension of $\overline{R}_v^{\square,\psi}[1/p]$ is shown to be $3$ in \cite{Kisin2adic}*{Proposition 2.5.4}. It is shown in \cite{Kisin2adic}*{Proposition 2.5.3} that there is a quotient $R_v^{\square,\psi}\rightarrow R_v^{\square,\psi,\mathrm{ur}}$, such that a lift factors through $R_v^{\square,\psi,\mathrm{ur}}$ if and only if it is unramified, and that $R_v^{\square,\psi,\mathrm{ur}}$ is formally smooth over $\calO$ of relative dimension $3$. The quotient  $R_v^{\square,\psi}\rightarrow R_v^{\square,\psi,\mathrm{ur}}$ necessarily factors through $\overline{R}_v^{\square,\psi}$. Part (3) follows from the proof of \cite{Kisin2adic}*{Proposition 2.5.4}. More specifically, in the proof of \cite{Kisin2adic}*{Proposition 2.5.4}, it is shown that the completed local ring at $x$ is smooth unless $H^2(G_v,\Ad^0(\rho_x)) \ne 0$, where $\Ad^0(\rho_x)$ denotes the trace zero subspace of the adjoint representation of $\rho_x$, and it is easily checked that if $H^2(G_v,\Ad^0(\rho_x)) \ne 0$, then $\rho_x \cong \gamma_v\epsilon_p \oplus \gamma_v$ for some character $\gamma_v$ of $G_v$.
	\end{proof}

We now turn our attention to semistable lifts. We will need the following lemma, which is a variant of a result of Snowden \cite{SnowdenDefRing}.

\begin{lem}\label{lnotpCM}
Assume $\overline{\rho}$ is the trivial representation and that $q \equiv 1 \pmod p$. Let $J$ be the ideal in $R_v^\square$ generated by the equations
	\[ \tr\rho^\mathrm{univ}(\sigma) = \epsilon_p(\sigma)+1, \hspace{0.25cm} 
	\det\rho^\mathrm{univ}(\sigma) = \epsilon_p(\sigma),
	\]
for all $\sigma\in G_v$, as well as the entries of the matrix equations
	\[ (\rho^\mathrm{univ}(\sigma)-1)(\rho^\mathrm{univ}(\tau)-1) 
	= (\epsilon_p(\sigma)-1)(\rho_v^\mathrm{univ}(\tau) - 1) 
	\]
for all $\sigma,\tau\in G_v$. Then $(R_v^\square/J_\gamma)\otimes_\calO \F$ is a domain of dimension $3$.
\end{lem}

\begin{proof}
Set $R = (R_v^\square/J)\otimes_\mathcal{O} \F$, and let $\rho_R$ be the pushforward of $\rho^\mathrm{univ}$ to $R$. Since we have assumed $q\equiv 1 \pmod p$, modulo $\varpi_E$ the equations defining $J$ become
	\[ \tr \rho_R (\sigma) = 2 \hspace{0.25cm}\text{and}\hspace{0.25cm} \det\rho_R(\sigma) =1
	\]
for all $\sigma\in G_v$, and
	\[ (\rho_R(\sigma)-1)(\rho_R(\tau) -1) = 0
	\]
for all $\sigma,\tau\in G_v$. From this, it follows that $R$ represents the functor on $\CNL_\F$ that sends a $\CNL_\F$-algebra $A$ to the set of families of matrices $\{1+X_\sigma \in 1+\mathfrak{m}_A \mathrm{M}_{2\times 2}(A) : \sigma\in G_v\}$, such that 
	\[ \tr(1+X_\sigma) = 2, \hspace{0.25cm}\det(1+X_\sigma)=1,
	\hspace{0.25cm}X_\sigma X_\tau = 0
	\]
for all $\sigma,\tau\in G_v$. For any such family, the elements $1+X_\sigma$ commute and have order $p$ (or $0$), hence the map $G_v \mapsto 1+X_\sigma$ factors through the maximal abelian quotient of exponent $p$, which is a product of two cyclic groups of order $p$. From this, we see that $R$ represents the functor on $\CNL_\F$ that assigns to each $\CNL_\F$-algebra $A$, a pair of matrices $X,Y\in \mathfrak{m}_A\mathrm{M}_{2\times 2}(A)$ such that
	\[ \tr(1+X)=\tr(1+Y) = 2, \hspace{0.25cm}\det(1+X) = \det(1+Y)=1,
	\hspace{0.25cm}X^2=Y^2=XY=YX= 0.
	\]
	
Define a functor $F$ on the category of $\F$-algebas that assigns to an $\F$-algebra $A$ the set of pairs $(X,Y) \in \mathrm{M}_{2\times 2}(A)$ satisfying
	\[ \tr X = \tr Y = 0, \hspace{0.25cm} \det X = \det Y = 0, \hspace{0.25cm} 
	X^2 = Y^2 = XY = YX = 0
	\]
It is easy to see that $F$ is represented by $R_F = \F[a_1,b_1,c_1,a_2,b_2,c_2]/I$, where 
	\begin{equation}\label{definingideal}
	 I = (a_1^2 - b_1c_1, a_2^2 - b_2c_2, a_1a_2 + b_1c_2, a_1a_2 + b_2 c_1, a_1b_2 - a_2b_1, a_2c_1 - a_1c_2),
	\end{equation}
with universal pair $\left( \left( \begin{array}{cr} a_1 & b_1 \\ c_1 & -a_1 \end{array}\right), \left( \begin{array}{cr} a_2 & b_2 \\ c_2 & -a_2 \end{array} \right) \right)$. Let $x$ denote the $\F$-point of $\Spec R_F$ corresponding to the pair $(0,0)$. There is an isomorphism of functors on $\CNL_\F$
	\[ \Spf R \stackrel{\sim}{\longrightarrow} \Spf (R_F)_x^\wedge
	\]
given by $(1+X,1+Y) \mapsto (X,Y)$, and so we wish to show that $(R_F)_x^\wedge$ is a domain of dimension $3$. Note that
	\[ (R_F)_x^\wedge \cong \F[[a_1,b_1,c_1,a_2,b_2,c_2]]/I,
	\]
with $I$ as in \eqref{definingideal}. For any Noetherian local ring $A$, if its associated graded ring $\mathrm{gr}(A)$ is a domain, then so is $A$. Noting that
	\[ \mathrm{gr}((R_F)_x^\wedge) \cong \F[a_1,b_1,c_1,a_2,b_2,c_2]/I = R_F,
	\]
since $I$ is homogenous, we are reduced to showing $R_F$ is a domain of dimension $3$.

Note that $I$ is the homogenization in $\F[a_1,b_1,c_1,a_2,b_2,c_2]$ of the ideal
	\[ I' = (1 - b_1c_1, a_2^2 - b_2c_2, a_2 + b_1c_2, a_2 + b_2 c_1, b_2 - a_2b_1, a_2c_1 - c_2),
	\]
in $\F[b_1,c_1,a_2,b_2,c_2]$. By \cite{ZariskiSamuel}*{Chapter 7, \S 5, Theorem 17}, to show that $I$ it prime it suffices to show that $I'$ is prime. The map defined by $b_2 \mapsto a_2b_1$ and $c_2 \mapsto a_2c_1$ gives an isomorphism
	\[ \F[b_1,c_1,a_2,b_2,c_2]/I' \cong \F[b_1,c_1,a_2]/(1-b_1c_1),
	\]
which is a domain of dimension $2$. Hence, $R_F$ is a domain and has dimension $1 + \dim \F[b_1,c_1,a_2,b_2,c_2]/I' = 3$.	\end{proof}

\begin{prop}\label{SemiStlp}
Let $\gamma : G_v \rightarrow \calO^\times$ be a finitely ramified, continuous character such that $\gamma^2 = \psi$ (enlarging $\calO$ if necessary), and such that $\gamma(I_v)$ is prime to $p$. Assume $\overline{\rho}_v$ is an extension of $\overline{\gamma}$ by $\overline{\gamma\epsilon_p}$. 

There is an $\mathcal{O}$-flat reduced quotient $R_v^{\square,\gamma\text{-st}}$ of $R_v^{\square,\psi}$ of relative dimension $3$, such that if $E'/E$ is a finite extenstion, a continuous $\calO$-algebra morphism $x : R_v^{\square,\psi} \rightarrow E'$ factors through $R_v^{\square,\gamma\text{-st}}$ if and only if the corresponding lift $\rho_x$ is an extension of $\gamma$ by $\gamma\epsilon_p$. $R_v^{\square,\gamma\text{-st}}[1/p]$ is formally smooth over $E$, and $R_v^{\square,\gamma\text{-st}}\otimes_\mathcal{O} \F$ is a domain.
\end{prop}

\begin{proof} Except for the claim about the reduction mod $p$, this is proved in \cite{KisinFinFlat}*{Corollary 2.6.7} for $p>2$ and $\gamma$ unramified, and in \cite{KW2}*{Theorem 3.1} in the remaining cases.

If there is a unique line in $V_\F$ on which $G_v$ acs via $\overline{\gamma\epsilon_p}$, then the proof of \cite{KW2}*{Theorem 3.1} shows that $R^{\square,\gamma\text{-st}}_v$ is formally smooth over $\calO$, and so we we may assume that $\overline{\epsilon_p}$ is trivial, i.e. $q \equiv 1 \pmod p$, and that $G_v$ acts on $V_\F$ via the character $\overline{\gamma}$. Twisting lifts by $\gamma^{-1}$ yields an isomorphism of lifting functors $\mathcal{D}_{\rhobar}^\square \cong \mathcal{D}_{\rhobar\otimes\overline{\gamma}^{-1}}^\square$; hence, an isomorphism of their universal lifting rings. Part (3) of \ref{OPointsOfR} shows that this isomorphism of universal lifting rings yields an isomorphism $R_v^{\square,\gamma\text{-st}} \cong R_v^{\square,1\text{-st}}$, and so we may assume $\gamma = 1$ and $\rhobar$ is the trivial homomorphism.

Let $J$ be as in \ref{lnotpCM}. Let $X$ denote the set of points $x : R^{\square}_v \rightarrow \mathcal{O}'$, where $\calO'$ is the ring of integers in some finite extension $E'/E$, such that the induced lift $\rho_x$ is conjugate to an extension of $\epsilon_p$ by $1$. Set $J_X = \cap_{x\in X}\ker(x)$. By part (3) of \ref{OPointsOfR}, the surjection $R^\square_v \rightarrow R^{\square,\gamma\text{-st}}_v$ has kernel $J_X$. Since the equations defining $J$ hold for any $x\in X$, this surjection factors through $R_v^\square/J$. Let $\mathfrak{q}$ denote the kernel of $R^\square_v/J\rightarrow R_v^{\square,1\text{-st}}$. Note that $\mathfrak{q}$ is prime and $\varpi_E\notin\mathfrak{q}$. Since $R_v^{\square,1\text{-st}}$ is $\calO$-flat of relative dimension $3$, and $(R_v^\square/J)\otimes_\calO \F$ is a domain of dimension $3$ by \ref{lnotpCM}, the surjection
	\[ (R_v^\square/J) \otimes_\calO \F \longrightarrow R_v^{\square,1\text{-st}}
		\otimes_\calO \F\]
is an isomorphism. Then \ref{lnotpCM} implies $R_v^{\square,1\text{-st}}\otimes_\calO \F$ a domain.	\end{proof}

\subsubsection{}\label{LocalDefsInf}

Now assume that $F_v \cong \R$ and assume $\overline{\rho}$ is odd, i.e. $\det\overline{\rho}(c) = -1$ for $c$ complex conjugation.

\begin{prop}\label{OddLifts}
There is an $\mathcal{O}$-flat reduced quotient $R_v^{\square,-1}$ of $R_v^{\square}$, such that if $E'/E$ is finite, a continuous $\calO$-morphism $x : R_v^{\square} \rightarrow E'$ factors through $R_v^{\square,-1}$ if and only if $\rho_x$ is odd. $R_v^{\square,-1}$ is a complete intersection domain of dimension $3$,  $R_v^{\square,-1}[1/p]$ is formally smooth over $E$, and $R_v^{\square,-1}\otimes_\mathcal{O}\F$ is a domain.
\end{prop}

\begin{proof} \cite{KW2}*{Proposition 3.3} or \cite{Kisin2adic}*{Proposition 2.5.6} shows that $R_v^{\square,-1}$ is smooth over $\mathcal{O}$ if $\rhobar$ is non-trivial, and that if $\rhobar$ is trivial then
	\[ R_v^{\square,-1} \cong \mathcal{O}[[x,y,z]]/(x^2+2y+yz),
	\]
which is a complete intersubsection domain of dimension $3$. Lastly, $x^2 + yz$ is irreducible in $\F[[x,y,z]]$, so $R_v^{\square,-1}\otimes_\mathcal{O} \F$ is a domain.
	\end{proof}

%% file: DefTheory/GlobalDefs.tex
Let $F$ be a totally real field and let $S$ denote a finite set of places of $F$ containing all the infinite places as well as all the places above $p$. Denote by $F_{S}$ the maximal Galois extension of $F$ in $\Qbar$, unramified outside of $S$, and let $G_{F,S} = \Gal(F_{S}/F)$. Fix an absolutely irreducible, continuous representation
	\[ \overline{\rho} : G_{F,S} \longrightarrow \GL_2(\F), \]
and a continuous character $\psi : G_{F,S} \rightarrow \calO^\times$ such that $\overline{\psi\epsilon_p} = \det\overline{\rho}$. For each $v\in S$, let $F_v$ denote the completion of $F$ at $v$, and let $G_v = \Gal(\Qbar_\ell/F_v)$ where $\ell$ the residual characteristic of $v$, or $\infty$. Let $Q$ be some (possibly empty) set of places of $F$ disjoint from $S$.

Throughout this subsection, all completed tensor products are taken over $\mathcal{O}$ unless noted otherwise.

Let $R_{F,S\cup Q}$, respectively $R_{F,S\cup Q}^\psi$, denote the universal deformation ring, respectively universal deformation ring with determinant $\psi\epsilon_p$, of the $G_{F,S\cup Q}$-representation $\rhobar$. For each $v\in S$, let $R_v^{\square,\psi}$ denote the universal framed deformation ring for lifts of $\rhobar|_{G_v}$ with determinant $\psi\epsilon_p|_{G_v}$. Consider the set valued functor on $\mathrm{CNL}_\mathcal{O}$ that sends an object $A$ to $(V_A,\{\beta_v\}_{v\in S})$, where $V_A$ is a $G_{F,S\cup Q}$-deformation of $\rhobar$ to $A$, such that the determinant of $V_A|_{G_v}$ is equal to $\psi\epsilon_p|_{G_v}$ for all $v\in S$, and $\beta_v$ is a lift of $\overline{\beta}$ for each $v\in S$. This functor is representable and we denote the representing object by $R^\square_{F,S\cup Q}$. The subfunctor consisting of the tuples $(V_A,\{\beta_v\}_{v\in S})$ such that $\det V_A = \psi\epsilon_p$ is also representable and we denote the representing object by $R^{\square,\psi}_{F,S\cup Q}$. The forgetful functor $(V_A,\{\beta_v\}_{v\in S})\mapsto V_A$ gives a canonical maps $R_{F,S\cup Q}\rightarrow R^{\square}_{F,S\cup Q}$ and $R^\psi_{F,S\cup Q}\rightarrow R^{\square,\psi}_{F,S\cup Q}$, and in the latter case it is formally smooth of relative dimension $4|S|-1$, cf. \cite{KW2}*{Proposition 4.1}. Note that $R_{F,S\cup Q}^\psi$ is a quotient of $R_{F,S }$ and there is a canonical isomorphism $R_{F, S \cup Q}^{\square,\psi} \cong R_{F,S\cup Q}^\psi \otimes_{R_{F,S\cup Q}} R_{F, S\cup Q}^{\square}$.
		
The identity map $R_{F,S\cup Q}^{\square}\rightarrow R_{F,S\cup Q}^{\square}$ gives a universal object $(V^\mathrm{univ},\{\beta_v^\mathrm{univ}\}_{v\in S})$. For each $v\in S$, $(V^\mathrm{univ}|_{G_v},\beta_v^\mathrm{univ})$ determines a lift of $\rhobar|_{G_v}$ with determinant $\psi\epsilon_p|_{G_v}$, so we have a canonical morphism $R_v^{\square,\psi}\rightarrow R^{\square}_{F,S\cup Q}$. Letting $R^{\square,\psi}_S$ denote the completed tensor product $\hat{\otimes}_{v\in S} R^{\square,\psi}_v$, $R_{F,S\cup Q}^{\square}$ is canonically a $R^{\square,\psi}_S$ algebra. This also give an $R^{\square,\psi}_S$-algebra structure to $R_{F,S\cup Q}^{\square,\psi}$.

\subsubsection{}\label{GlobalDefPres}

Let $\Ad$ denote the space $\mathrm{M}_{2\times 2}(\F)$ with the adjoint $G_{F,S}$-action, and let $\Ad^0$ denote its trace zero subspace. In what follows we will use the following notation. Given a topological group $G$ such that $\Hom_{cts}(G,\F)$ is finite, and a  a finite $\F[G]$-module $M$, we denote by $M^\ast$ the $\F$-linear dual of $M$ with the induced $G$-action. For any $i\ge 0$, we denote by $h^i(G_{F,S},M)$ the $\F$-dimension of the cohomology group $H^i(G,M)$. If $G = G_{F,S}$, and $W$ is a finite set of places of $F$, we let $H_W^i(G_{F,S},M)$ denote the kernel of the restriction map
	\[ H^i(G_{F,S},M) \longrightarrow \prod_{v\in W} H^i(G_v,M) \]
and let $h_W^i(G_{F,S},M)$ denote its $\F$-dimension.

\begin{prop}\label{PresOverLoc}
There is a presentation
	\[ R_{F,S}^{\square,\psi} \cong
		 R_S^{\square,\psi}[[x_1,\ldots,x_g]]/(f_1,\ldots,f_r) \]
with
	\[ g-r \ge |S| - 1 - h^0(G_{F,S},(\Ad^0)^\ast(1)). \]
\end{prop}

\begin{proof}
Let
	\[ \phi:\mathfrak{m}_{R_S^{\square,\psi}}/((\mathfrak{m}_{R_S^{\square,\psi}})^2,\varpi_E)\longrightarrow
	\mathfrak{m}_{R_{F,S}^{\square,\psi}}/((\mathfrak{m}_{R_{F,S}^{\square,\psi}})^2,\varpi_E)\]
denote the map on reduced cotangent spaces induced from $R_S^{\square,\psi}\rightarrow R_{F,S}^{\square,\psi}$. By \cite{KisinModof2}*{Proposition 4.1.4}, we have a presentation
	\[ R_{F,S}^{\square,\psi} \cong R_S^{\square,\psi}[[x_1,\ldots,x_g]]/(f_1,\ldots,f_r) \]
where $g-r \ge \dim_\F\coker\phi - \dim_\F \ker\phi - h_S^2(G_{F,S},\Ad^0)$. One can show, cf. proof of \cite{KisinModof2}*{Lemma 4.1.5},
	\[ \dim_\F\mathfrak{m}_{R_{F,S}^{\square,\psi}}/((\mathfrak{m}_{R_{F,S}^{\square,\psi}})^2,\varpi_E)
	=4|S|+h^1(G_{F,S},\Ad^0)-h^0(G_{F,S},\Ad^0)-1\]
and
	\[ \dim_\F\mathfrak{m}_{R_S^{\square,\psi}}/((\mathfrak{m}_{R_S^{\square,\psi}})^2,\varpi_E)
	=4|S|+\sum_{v\in S}(h^1(G_v,\Ad^0)-h^0(G_v,\Ad^0)-1).\]
So,
	\begin{align}\dim_\F\mathrm{coker}\phi-\dim_\F\mathrm{ker}\phi
	&=\dim_\F\mathfrak{m}_{R_{F,S}^{\square,\psi}}/((\mathfrak{m}_{R_{F,S}^{\square,\psi}})^2,\varpi_E)
	-\dim_\F\mathfrak{m}_{R_S^{\square,\psi}}/((\mathfrak{m}_{R_S^{\square,\psi}})^2,\varpi_E)\notag\\
	&=|S|-1+h^1(G_{F,S},\Ad^0)-h^0(G_{F,S},\Ad^0)\notag\\
	&\hspace{1.0cm}-\sum_{v\in S}(h^1(G_v,\Ad^0)-h^0(G_v,\Ad^0)).
	\label{tangentdims}
	\end{align}
The Poitou-Tate sequence implies
	\[h^2_S(G_{F,S},\Ad^0)
	=h^2(G_{F,S},\Ad^0)-\sum_{v\in S}h^2(G_v,\Ad^0)+h^0(G_{F,S},(\Ad^0)^\ast(1)).\]
Combining this with (\ref{tangentdims}), we have
	\begin{align} g-r & \ge \dim_\F\mathrm{coker}\phi-\dim_\F\mathrm{ker}\phi
		-h_S^2(G_{F,S},\Ad^0)\notag\\
	& \ge |S|-1-h^0(G_{F,S},(\Ad^0)^*(1))-\chi(G_{F,S},\Ad^0)
	+\sum_{v\in S}\chi(G_v,\Ad^0), 
	\label{differencedim}
	\end{align}
where $\chi(G_{F,S},\Ad^0)$ and $\chi(G_v,\Ad^0)$ denote the global and local Euler characteristics, respectively, as $\F$-vector spaces. 

For $v|\infty$ we have $h^1(G_v,\Ad^0)=h^2(G_v,\Ad^0)$ since $G_v$ is cyclic,  and so $\chi(G_v,\Ad^0)=h^0(G_v,\Ad^0)$. For $v$ finite the local Euler characteristic formula gives $\chi(G_v,\Ad^0)=1$, when $v\nmid p$, and $\chi(G_v,\Ad^0)=-3[F_v:\Q]$, when $v|p$. The global Euler characteristic formula gives $\chi(G_{F,S},\Ad^0) = -3[F:\Q]+\sum_{v|\infty}h^0(G_v,\Ad^0)$. Equation (\ref{differencedim}) then becomes
	\[ g-r \ge |S| - 1 -h^0(G_{F,S}, (\Ad^0)^\ast(1)).\]
\end{proof}

We will use the above lemma to show that a certain quotient of $R_{F,S}^\psi$ (actually a quotient of $R_{F,S}$ tensored with an certain Iwasawa algebra) has an appropriate presentation in order to apply the connectivity result \ref{Connectivity}.

\subsubsection{}\label{QuotientPres}

Let $\Sigma$ be a fixed set of places not containing any places above $p$ or $\infty$. For each $v \in \Sigma$ fix a continuous character $\gamma_v : G_v \rightarrow \mathcal{O}^\times$ such that $\gamma_v(I_v)$ is finite and prime to $p$, and such that $\gamma_v^2 = \psi|_{G_v}$ (enlarging $\mathcal{O}$ if necessary). We further assume the following.
	\begin{itemize}
	\item[-] For $v|p$, $\rhobar|_{G_v} \cong \left(\begin{array}{cc} \ast & \ast \\ & \overline{\chi}_v \end{array} \right)$, where we fix the choice $\overline{\chi}_v$ in the case that $\rhobar|_{G_v}$ is the direct sum of two disctinct characters.
	\item[-] For $v\in \Sigma$, $\overline{\rho}|_{G_v} \cong
		\left(\begin{array}{cc} \overline{\gamma}_v\overline{\epsilon_p} & \ast \\
		& \overline{\gamma}_v \end{array} \right)$.
	\item[-] For any archimedean $v$, $\psi|_{G_v} = 1$.
	\item[-] $\rhobar$ is unramified outside $\Sigma\cup\{v|p\}\cup\{v|\infty\}$.
	\end{itemize}
Fix a finite set of places $S_\mathrm{ur}$ disjoint from $\Sigma\cup\{v|p\}\cup\{v|\infty\}$ and set $S = S_{\mathrm{ur}} \cup \Sigma\cup\{v|p\}\cup\{v|\infty\}$. For each $v|p$, let $\eta_v$ be a character of the torsion subgroup of $G_v^\mathrm{ab}(p)$. Then $\eta_v$ correpsonds to a unique minimal prime of $\Lambda(G_v)= \mathcal{O}[[G_v^\mathrm{ab}(p)]]$, which we denote by $\mathfrak{q}_{\eta_v}$. We denote the quotient $\Lambda(G_v)/\mathfrak{q}_{\eta_v}$ by $\Lambda(G_v,\eta_v)$. Note that giving the tuple $(\eta_v)_{v|p}$ is equivalent to giving a character $\eta$ on the torsion subgroup of $\prod_{v|p} G_v^\mathrm{ab}(p)$. For each $v \in S$, let $\overline{R}_v^{\square,\psi}$ denote the $\CNLO$-algebra given by
	\begin{itemize}
		\item[-] $\overline{R}_v^{\square,\psi} = R_{\Lambda(G_v,\eta_v)}^{\triangle,\psi}$ as in \ref{triliftringdef}, for $v|p$;
		\item[-] $\overline{R}_v^{\square,\psi} = R_v^{\square,\gamma\text{-st}}$ as in \ref{SemiStlp}, for $v \in \Sigma$,
		\item[-] $\overline{R}_v^{\square,\psi} = R_v^{\square,-1}$ as in \ref{OddLifts}, for $v|\infty$,
		\item[-] $\overline{R}_v^{\square,\psi} = R_v^{\square,\psi,\mathrm{ur}}$ as in part (2) of \ref{FixedDetlp}, if $v\in S_{\mathrm{ur}}$.
	\end{itemize}
Let $\overline{R}_S^{\square,\psi} = \hat{\otimes}_{v\in S} \overline{R}_v^{\square,\psi}$, $\Lambda(G_p) = \hat{\otimes}_{v|p}\Lambda(G_v)$, and $\Lambda(G_p,\eta) = \hat{\otimes}_{v|p}\Lambda(G_v,\eta_v)$. Note that $\overline{R}_S^{\square,\psi}$ is a quotient of
	\[ (\hat{\otimes}_{v|p} (R_v^{\square}\hat{\otimes} \Lambda(G_v,\eta_v)))
	\hat{\otimes}_{v\in S,v \nmid p} R_v^{\square,\psi} \cong
	(\hat{\otimes}_{v\in S} R_v^{\square})\hat{\otimes} \Lambda(G_p,\eta),
	 \]
and that $\Lambda(G_p,\eta)$ represents the functor on $\CNLO$ that sends a $\CNLO$-algebra $A$ to the set of tuples $(\chi_v)_{v|p}$, where each $\chi_v$ is an $A$-valued character of $G_v$ that reduces to $\overline{\chi}_v$ modulo the maximal ideal of $A$, and whose restriction to the $p$-power part of the torsion subgroup of $G_v^\mathrm{ab}$ is eqaul to $\eta_v$.

\begin{prop}\label{locdefring}

If $p=2$, assume that for each $v|2$, either $F_v$ contains a $4$-th root of unity or $[F_v:\Q_2] \ge 3$. Assume also that $\rhobar|_{G_v}$ is either the trivial representation or that its image has order $p$ for each $v|p$. Then $\overline{R}_S^{\square,\psi}$ is an $\mathcal{O}$-flat domain of dimension $1+3|S|+[F:\Q]$.

\end{prop}

\begin{proof}

By \ref{triliftringdomain}, \ref{SemiStlp}, \ref{OddLifts}, and part (2) of \ref{FixedDetlp}, each of the $\overline{R}_v^{\square,\psi}$ an $\mathcal{O}$-flat domain of relative dimension
	\begin{itemize}
		\item[-] $3+2[F_v:\Q_p]$ if $v|p$,
		\item[-] $3$ if $v \in \Sigma$,
		\item[-] $2$ if $v|\infty$,
		\item[-] $3$ if $v \in S_\mathrm{ur}$,
	\end{itemize}
and so $\overline{R}_S^{\square,\psi}$ is $\mathcal{O}$-flat of relative dimension
	\[ \sum_{v|p} 3+2[F_v:\Q_p] + \sum_{v\in\Sigma\cup S_{\mathrm{ur}}} 3 + \sum_{v|\infty} 2
	= 3|S|+[F:\Q_p].
	\]
To see that it is a domain, consider a finite extension $E'/E$ with ring of integers $\mathcal{O}'$ and residue field $\F'$. It follows from \ref{RepOnTop} that $R_v^\square\otimes_\mathcal{O} \mathcal{O}'$ is the universal lifting ring on $\CNL_{\mathcal{O}'}$ for the representation $\rhobar|_{G_v}\otimes_\F \F'$, and it follows easily that $\overline{R}_v^{\square,\psi}\otimes_\mathcal{O} \mathcal{O}'$ are the corresponding quotients on the category $\CNL_{\mathcal{O'}}$. Applying \ref{triliftringdomain}, \ref{SemiStlp}, and \ref{OddLifts} to $\overline{R}_v^{\square,\psi} \otimes_\mathcal{O} \mathcal{O}'$, we conclude that $\overline{R}_v^{\square,\psi}[1/p]$ is geometrically integral. Then $\overline{R}_S^{\square,\psi}$ is a domain by \ref{TensorCNLO}. 
	\end{proof}

\subsubsection{}\label{globalquotient}

Now fix a finite set of primes $Q$ disjoint from $S$. Letting $R_S^\square = \hat{\otimes}_{v\in S} R_v^\square$, define 
	\[ \overline{R}_{F,S\cup Q}^\square = R_{F,S\cup Q}^\square \hat{\otimes}_{R_S^\square}
		\overline{R}_S^{\square,\psi}
	\]
and
	\[ \overline{R}_{F,S\cup Q}^{\square,\psi} = R_{F,S\cup Q}^{\square,\psi} \hat{\otimes}_{R_S^\square}
		\overline{R}_S^{\square,\psi}.
	\]
Note that $\overline{R}_{F,S\cup Q}^\square$ and $\overline{R}_{F,S\cup Q}^{\square,\psi}$ are quotients of $R_{F,S\cup Q}^\square \hat{\otimes}\Lambda(G_p,\eta)$. We then define $\overline{R}_{F,S\cup Q}^\psi$ to be the image of $R_{F,S\cup Q}^\psi\hat{\otimes}\Lambda(G_p,\eta)$ under
	\[ R_{F,S\cup Q}^{\psi}\hat{\otimes}\Lambda(G_p,\eta) \longrightarrow 
	R_{F,S\cup Q}^{\square,\psi}\hat{\otimes}\Lambda(G_p,\eta)
	 \longrightarrow \overline{R}_{F,S\cup Q}^{\square,\psi}.
	\]
If $E'/E$ is finite with ring of integers $\mathcal{O}'$, a local $\calO$-algebra morphism $R_{F,S\cup Q}\hat{\otimes}\Lambda(G_p) \rightarrow \mathcal{O}'$ factors through $\overline{R}_{F,S\cup Q}^\psi$ if and only if the corresponding deformation $V_{\mathcal{O}'}$ and tuple of characters $(\chi_v)_{v|p}$ satisfies
\begin{itemize}
	\item[-] $\det V_{\mathcal{O}'} = \psi\epsilon_p$;
	\item[-] for each $v|p$, there is a $G_v$-stable line $L$ in $V_{\mathcal{O}'}$, and the action of $G_v$ on $V_{\mathcal{O'}}/L$ is given by $\chi_v$;
	\item[-] for each $v|p$, the restriction of $\chi_v$ to the $p$-power torsion subgroup of $G_v^\mathrm{ab}$ is equal to $\eta_v$;
	\item[-] for each $v \in \Sigma$, $V_{\mathcal{O}'}|_{G_v}$ is an extension of $\gamma_v$ by $\gamma_v\epsilon_p$;
	\item[-] for each archimedean $v$, $V_{\mathcal{O}'}|_{G_v}$ is not the trivial representation,
	\item[-] for each $v\in S_\mathrm{ur}$, $V_{\mathcal{O}'}|_{G_v}$ is unramified.
\end{itemize}
Note the second last condition is redundant, as it is implied by the first. However, we will later have to consider the $\overline{R}_S^{\square,\psi}$-map $\overline{R}_{F,S\cup Q}^\square \rightarrow \overline{R}_{F,S\cup Q}^{\square,\psi}$, and the oddness condition is not forced on $\overline{R}_{F,S\cup Q}^\square$. Also note that if we had omitted the primes $S_\mathrm{ur}$ from $S$ entirely, the resulting ring $\overline{R}_{F,S\cup Q}^\psi$ would have been the same. We include them because it will be useful later in \S \ref{LocRequalsT} to ensure that the local framed deformation ring $\overline{R}_S^{\square,\psi}$ surjects onto a certain Hecke algebra, as well as to ensure that a certain group action is free.

If $A$ is a $\CNLO$-algebra, and $x,x'\in \Spf (R_{F,S\cup Q}^{\square,\psi}\hat{\otimes}\Lambda(G_p,\eta))(A)$ are two $A$-points with the same image in $\Spf (R_{F,S\cup Q}^\psi\hat{\otimes}\Lambda(G_p,\eta))(A)$, i.e. give rise to the same deformation, then $x \in \Spf \overline{R}_{F,S\cup Q}^{\square,\psi}(A)$ if and only if $x' \in \Spf \overline{R}_{F,S\cup Q}^{\square,\psi}(A)$. It follows that $\overline{R}_{F,S\cup Q}^\psi \rightarrow \overline{R}_{F,S\cup Q}^{\square,\psi}$ is formally smooth of relative dimension $4|S| -1$. 

The following proposition will allow us to invoke \ref{Connectivity}, which is crucial to the Sinner-Wiles strategy. 

\begin{prop}\label{CMPres}
Assume that $\rhobar|_{G_v}$ is either the trivial representation or that its image has order $p$ for each $v|p$. If $p=2$, we also assume that for each $v|2$, either $F_v$ contains a $4$th root of unity or $[F_v:\Q_2] \ge 3$.  There is a presentation
	\[ \overline{R}_{F,S}^\psi \cong A/(f_1,\ldots,f_m) \]
with $A$ a domain, and $\dim A - m \ge 1 + [F:\Q] - h^0(G_{F,S},(\Ad^0)^\ast(1))$.
\end{prop}

\begin{proof}
Since $\overline{R}_{F,S}^\psi \rightarrow \overline{R}_{F,S}^{\square,\psi}$ is formally smooth of relative dimension $4|S| -1$, $\overline{R}_{F,S}^{\square,\psi}$ is isomorphic to a power series over $\overline{R}_{F,S}^\psi$ in $4|S|-1$ variables, and it suffices to show that there is a presentation
	\[ \overline{R}_{F,S}^{\square,\psi} \cong A/(f_1,\ldots,f_k) \]
with $A$ a domain and $\dim A - k \ge 4|S| + [F:\Q] - h^0(G_{F,S}, (\Ad^0)^\ast(1))$. By \ref{locdefring}, $\overline{R}_S^{\square,\psi}$ is a domain of dimension $1+3|S|+[F:\Q]$. The presentation in \ref{PresOverLoc} yields a presentation
	\[ 
		\overline{R}_{F,S}^{\square,\psi} \cong 
		\overline{R}_S^{\square,\psi}[[x_1,\ldots,x_g]]/(f_1,\ldots,f_r)
	\]
with $g - r \ge |S| - 1 - h^0(G_{F,S},(\Ad^0)^\ast(1))$. Take $A = \overline{R}_S^{\square,\psi}[[x_1,\ldots,x_g]]$. \end{proof}

\subsubsection{}\label{GlobalDefTwists}

Let $Q$ be a set of places of $F$ disjoint from $S$. Let $F_Q^S$ denote the maximal abelian pro-$p$ extension of $F$ unramified outside of $Q$ and split at primes in $S$. Note that $F_Q^S/F$ is finite, since $S$ contains all the primes above $p$. Let $G_Q = \Gal(F_Q^S/F)$ and let $G_Q^\ast$ denote the diagonalizable $\CNLO$-group as in \ref{diaggroups}. For any $\CNLO$-algebra $A$, $G_Q^\ast(A)$ is the subgroup $\Hom(G_Q,A^\times)$ that reduce to the trivial morphism modulo $\mathfrak{m}_A$.

There is an action of $G_Q^\ast$ on $\Spf R_{F,S\cup Q}^{\square} $ given as follows. Let $A$ be a $\CNLO$-algebra, and let $V_A$ be a deformation of $V_\F$ to $A$. If $\chi \in G_Q^\ast(A)$, viewing $\chi$ as a character of $G_{F,S\cup Q}$, we have a deformation $V_A\otimes\chi$. In this way we get an action of $G_Q^\ast$ on $\Spf R_{F,S\cup Q} $. This action extends to an action on $\Spf (R_{F,S\cup Q}^\square\hat{\otimes}\Lambda(G_p,\eta))$ by $(V_A, \{\beta_v\}_{v\in S},\{\chi_v\}_{v|p}) \mapsto (V_A\otimes \chi, \{\beta_v\}_{v\in S},\{\chi_v\}_{v|p})$. For any $v \in S$, since $F_Q^S$ is split at $v$, the lift of $\rhobar|_{G_v}$ given by $V_A|_{G_v}$ and $\beta_v$ is equal to the lift given by $(V_A\otimes\chi)|_{G_v}$ and $\beta_v$. Thus, the action of $G_Q^\ast$ commutes with the morphism $\Spf (R_{F,S\cup Q}^\square\hat{\otimes}\Lambda(G_p,\eta)) \rightarrow \Spf (R_S^{\square,\psi}\hat{\otimes}\Lambda(G_p,\eta))$. Letting $\overline{R}_S^{\square,\psi}$ and $\overline{R}_{F,S\cup Q}^\square$ be as in \ref{QuotientPres}, we get an action of $G_Q^\ast$ on $\Spf \overline{R}_{F,S\cup Q}^\square$ that commutes with $\Spf \overline{R}_{F,S\cup Q}^\square \rightarrow \Spf \overline{R}_S^{\square,\psi}$. Note that the map $(V_A,\{\beta_v\}_{v \in S},\{\chi_v\}_{v|p}) \mapsto \det V_A (\psi\epsilon_p)^{-1}$ determines a morphism $\delta_Q : \Spf \overline{R}_{F,S\cup Q}^\square \rightarrow G_Q^\ast$, such that for any $\CNL_\mathcal{O}$-algebra $A$, $g\in G_Q^\ast(A)$ and $x\in \Spf\overline{R}_{F,S\cup Q}^\square(A)$, we have $\delta_Q(gx) = g^2\delta_Q(x)$. The natural surjection $\overline{R}_{F,S\cup Q}^\square \rightarrow \overline{R}_{F,S\cup Q}^{\square,\psi}$ identifies $\Spf\overline{R}_{F,S\cup Q}^{\square,\psi}$ with the closed formal sub-scheme of $\Spf\overline{R}_{F,S\cup Q}^\square$ given by $\delta_Q = 1$.

If $p=2$, we denote by $G_{Q,2}^\ast$ the $2$-torsion subgroup of $G_Q^\ast$, i.e. the diagonalizable $\CNLO$-group $(G_Q/2 G_Q)^\ast$. For $\chi \in G_{Q,2}^\ast(A)$, we have $\chi^2 = 1$, hence the action of $G_Q^\ast$ on $\Spf R_{F,S\cup Q}^\square$, respectively on $\Spf \overline{R}_{F,S\cup Q}^\square$, induces an action of $G_{Q,2}^\ast$ on $\Spf R_{F,S\cup Q}^{\square,\psi}$, respectively on $\Spf \overline{R}_{F,S\cup Q}^{\square,\psi}$.

If $p=2$ and $\overline{\rho}$ has solvable image, then there is a unique quadratic extension $L/F$ such that $\rhobar|_{G_L}$ is abelian, since in this case the image of $\rhobar$ has order twice an odd number.

\begin{lem}\label{FreeTwistAction}
Let $p=2$. Assume that if $\rhobar$ has solvable image, then some $v\in S$ does not split in $L/F$, where $L$ is the unique quadratic extension for which $\rhobar|_{G_L}$ is abelian. Then if $Q$ is a set of places disjoint from $S$, the action of $G_Q^\ast$ on $\Spf \overline{R}_{F,S\cup Q}^\square$ is free.
\end{lem}

\begin{proof}
This is essentially \cite{KW2}*{Lemma 5.1}. It suffices to show that the action of $G_Q^\ast$ on $\Spf\overline{R}_{F,S\cup Q}$ is free. If $\rhobar$ is non-solvable, then its projective image is isomorphic to $\SL_2(\F_{2^r})$, for some $r\ge 1$. If $\rhobar$ is solvable, then by our assumption on $S$ and the fact that $\rhobar(G_L)$ has odd order, we see that the fixed field of the kernel of $\rhobar$ and $F_Q^S$ are disjoint. In either case, if $\chi$ is a non-trivial element of $G_Q^\ast(A)$, we can find $g\in G_{F,S\cup Q}$ such that $\chi(g) \ne 1$ and $\tr\rhobar(g)\ne 0$. Then, if $V_A$ is a deformation of $\rhobar$ to $A$, we have $\chi(g)\tr \rho_A(g) \ne \tr \rho_A(g)$, for any $\rho_A$ in the deformation class of $V_A$,  since $\tr \rho_A(g)$ is a unit. \end{proof}

%% file: DefTheory/Dihedral.tex
Fix continuous and absolutely irreducible
	\[ \overline{\rho} : G_{F,S} \rightarrow \GL_2(\F)
	\]
and let $V_\F$ denote its representation space. Let $A$ be a $\CNLO$-algebra and $V_A$ a deformation of $V_\F$. We say $V_A$ is $L$-dihedral, for a quadratic extension $L/F$, if $G_L$ acts on $V_A$ through an abelian quotient, but $G_F$ does not. We say $V_A$ is \textit{dihedral} if it is $L$-dihedral for some quadratic $L/F$. Note that $V_A$ is $L$-dihedral only if $V_\F$ is $L$-dihedral. If $A$ is a domain with field of fractions $K$, then if $V_A$ is $L$-dihedral and $V_A \otimes_A \overline{K}$ is irreducible, for $\overline{K}$ an algebraic closure of $K$, then one can show there is a character $\chi : G_L \rightarrow A^\times$ such that $V_A \cong \mathrm{Ind}_{G_L}^{G_F}\chi$. A subgroup $H$ of $\GL_2(\F)$ is called \textit{dihedral} if it's image in $\PGL_2(\F)$ is isomorphic to a dihedral group. It is easy to check that $V_\F$ is dihedral if and only if $\rhobar(G_F)$ is dihedral. From this we see that if $p=2$, there is a unique quadratic $L/F$ such that $V_\F$ is $L$-dihedral, as the order of $\rhobar(G_F)$ is twice an odd number.

\subsubsection{}\label{ZarDenseSec} We first establish some criteria for determining when the intersection of the image of a deformation with $\SL_2$ is Zariski dense.

\begin{lem}\label{CharPolyF}
Let $A$ be a $\CNLO$-algebra domain of characteristic $p$ with fraction field $K$. Note that $A$ is canonically an $\F$-algebra. An element $g\in \GL_n(A)$ has finite order eigenvalues if and only if its characteristic polynomial has coefficients in $\F$.
\end{lem}

\begin{proof}
Let $\overline{K}$ be an algebraic closure of $K$, and let $\overline{\F}$ be the algebraic closure of $\F$ in $\overline{K}$. If $g$ has finite order eigenvalues in $\overline{K}^\times$, then the characteristic polynomail of $g$ has coefficients in $\overline{\F}\cap A = \F$. The other direction is clear.	\end{proof}

\begin{lem}\label{InfOrdEigs}
Let $A$ be a $\CNLO$-algebra domain of characteristic $p$ with fraction field $K$, and let $\overline{K}$ be an algebraic closure of $K$. Let $V_A$ be a deformation of $V_\F$ to $A$ such that the map $\im \rho_A \rightarrow\im \rhobar $ has non-trivial kernel. The image of $\rho_A$ contains an element with an infinite order eigenvalue in $\overline{K}^\times$.
\end{lem}

\begin{proof}
Take nontrivial $g_1\in\ker(\im \rho_A \rightarrow\im \rhobar)$. If $g_1$ has finite order eigenvalues then it's eigenvalues must be $1$, and $g_1$ is unipotent. There is a basis of $V_A$ such that
\[
g_1=\left(\begin{array}{cc}1&x\\&1\end{array}\right)
\]
with $0 \ne x\in\mathfrak{m}_A$. Since $V_{\F}$ is irreducible, there is $g_2\in\im\rho_A$ such that, with respect to our fixed basis,
\[
g_2=\left(\begin{array}{cc}a&b\\c&d\end{array}\right)
\]
with $c\in A^\times$. If $\tr g_2\in\F$, then $\tr g_1g_2=cx + a+d=cx+ \tr g_2\notin\F$, since $\tr g_2 \in \F$ and $0\ne cx \in \mathfrak{m}_A$. We have thus shown that there is $g\in \im \rho_A$ with $\tr g \notin\F$. The lemma then follows from \ref{CharPolyF}. \end{proof}

\begin{lem}\label{FinIndexIrred} Let $A$ be a $\CNLO$-algebra domain and let $V_A$ be a non-dihedral deformation of $V_\F$ such that for some $\sigma\in G_F$, $\rho_A(\sigma)$ has eigenvalues whose ratio is not a root of unity. Then for any finite index subgroup $H$ of $G_F$, $\rho_A|_H$ is absolutely irreducible.
\end{lem}

\begin{proof}
It suffices to consider the case when $H$ is an open normal subgroup of $G_F$. Let $K$ denote the fraction field of $A$ and let $\overline{K}$ be an algebraic closure. Assume there is an $H$-stable subspace $W$ of $V_A\otimes_A{\overline{K}}$. Since $H$ is normal in $G_F$, $gW$ is also $H$ invariant for any $g\in G_F$. Take $n\ge 1$ such that $\sigma^n\in H$. Then, for each $g\in G_F$, $gW$ is an eigenspace for $\rho_A(\sigma^n)$. Since the ratio of the eigenvalues of $\rho_A(\sigma)$ is not a root of unity, the eigenvalues of $\rho_A(\sigma^n)$ are distinct. Hence, any $g\in G_F$ must permute the two one dimensional eigenspaces for $\rho_A(\sigma^n)$. So, either $W$ is $G_F$ stable or there is an index two subgroup $N$ of $G_F$ such that $W$ is $N$ stable and $G_F/N$ interchanges the two eigenspaces for $\rho_A(\sigma^n)$. In the first case $\rho_A$ is reducible over $\overline{K}$, contradicting irreducibility of $V_\F \otimes_\F \overline{\F}$, and in the second it is dihedral, contradicting our assumptions on $V_A$.
\end{proof}

\begin{prop}\label{ImZarDense}
Let $A$ be a $\CNLO$-algebra domain with fraction field $K$ of characteristic $p$ and let $V_A$ be a deformation of $V_\F$ with finite order determinant. Fix a basis of $V_A$ and let $\Gamma$ denote the image of $G_F$ in $\GL_2(A)$ with respect to this basis. If $\Gamma \rightarrow \GL_2(\F)$ has nontrivial kernel then $\Gamma\cap\SL_2(K)$ is Zariski dense in ${\SL_2}_{/K}$.
\end{prop}

\begin{proof}
By \ref{InfOrdEigs}, there is some $g\in\im\rho_A$ with infinite order eigenvalues. Since $\det\rho_A$ is finite order, the ratio of the eigenvalues of $g$ is not a root of unity, and $V_A$ satisfies the assumptions of \ref{FinIndexIrred}. So, for any finite index subgroup $H$ of $G_F$, $\rho_A|_H$ is absolutely irreducible.

Let $\Gamma^1=\Gamma\cap\SL_2(K)$ and assume $\Gamma^1$ is not Zariski dense in ${\SL_2}_{/K}$. Let $\overline{\Gamma^1}$ denote the Zariski closure of $\Gamma^1$ in ${\SL_2}_{/K}$ and $(\overline{\Gamma^1})^0$ its connected component at the identity. Our assumption implies $\dim(\overline{\Gamma^1})^0\le 2$, and so $(\overline{\Gamma^1})^0$ is solvable. Then $(\overline{\Gamma^1})^0$ acts reducibly on $V_A\otimes_A\overline{K}$, where $\overline{K}$ is an algebraic closure of $K$. Since the determinant of $V_A$ is finite, $\Gamma^1\cap(\overline{\Gamma^1})^0$ is finite index in $\Gamma^1$, so there is a finite index subgroup $H$ of $G_F$, such that $\rho_A|_H$ acts reducibly on $V_A\otimes_A\overline{K}$, a contradiction.
\end{proof}

We note that the assumption that $\Gamma \rightarrow \GL_2(\F)$ has nontrivial kernel is satisfied if $V_\F$ is dihedral and $V_A$ is non-dihedral.

\begin{cor}\label{DihedCrit}

Assume $V_\F$ is $L$-dihedral. Let $A$ be $\CNLO$-algebra domain of characteristic $p$, and let $V_A$ be a non-dihedral deformation of $V_\F$ with finite order determinant. Then there is $\sigma\in G_F\smallsetminus G_L$ such that $\rho_A(\sigma)$ has infinite order.

\end{cor}

\begin{proof}

Let $K$ denote the fraction field of $A$. Fix a basis of $V_A$ and let $\Gamma$ be the image of $\rho_A$ in $\GL_2(A)$ with respect to this basis. For each $(a,b)\in \F\times\F^\times$, let $X_{a,b}$ be the subvariety ${\GL_2}_{/K}$ defined by $\tr g = a$ and $\det g = b$, and let $X$ be the union of the $X_{a,b}$ for each $(a,b)\in \F\times\F^\times$. Note that each $X_{a,b}$, and hence $X$, has dimension $2$. Note also that if $g\in \GL_2(A)$ is of finite order, then $g \in X(K)$, by \ref{CharPolyF}. 

Assume that $\rho_A(\sigma)$ is finite order whenever $\sigma \in G_F\smallsetminus G_L$. Let $H=\rho_A(G_L)\subseteq\Gamma$, and let $g=\rho(\sigma)$ for some $\sigma\in G_F\smallsetminus G_L$. By assumption, every element of $gH$ is finite order, so $gH\subseteq X(K)$. Then $\Gamma=H\cup gH\subseteq gX(K)\cup X(K)$, so the Zariski closure of $\Gamma$ is contained in $gX\cup X$. Since $X$ has dimension $2$, so does $gX\cup X$. Since the Zariski closure of $\Gamma$ is contained in a $2$-dimensional subvariety, $V_A$ must be dihedral, by \ref{ImZarDense}. \end{proof}

\subsubsection{}\label{DihedLocus}

Let $L/F$ be a quadratic extension such that $\overline{\rho}$ is $L$-dihedral. Write $\rhobar=\Ind_{G_L}^{G_F}\overline{\chi}$, for $\overline{\chi}:G_L\rightarrow\F^\times$, and let $\tilde{\chi}:G_L\rightarrow\mathcal{O}^\times$ denote the Teichmuller lift of $\chibar$. Let $L_S^\mathrm{ab}/L$ denote the maximal abelian pro-$p$ extension of $L$ unramified outside $S$. Set $R_{L\text{-di}}=\mathcal{O}[[\Gal(L_S^\mathrm{ab}/L)]]$ and let $\Psi:G_L\rightarrow R_{L\text{-di}}^\times$ be the canonical character. We have an $L$-dihedral deformation of $V_\F$ to $R_{L\text{-di}}$ given by $\Ind_{G_L}^{G_F}\tilde{\chi}\Psi$. It is easy to see this deformation is universal for $L$-dihedral deformations, and that there is a surjection $R_{F,S}\rightarrow R_{L\text{-di}}$, hence the locus of all $L$-dihedral points in $\Spec R_{F,S}$ is closed. As there are only finitely many quadratic extensions $L/F$ such that $\overline{\rho}$ is $L$-dihedral, the locus of all dihedral points in $\Spec R_{F,S}$ is closed. The same is then true for any quotient of $R_{F,S}$.

The following two lemmas record some properties of dihedral deformations and will be used later to ensure certain deformations are non-dihedral.

\begin{lem}\label{NotDihedDim}

Let $L/F$ be a quadratic such that $\rhobar$ is $L$-dihedral. Denote by $L_S^-$ the maximal Galois subextension of $L_S^{\mathrm{ab}}/L$ such that the nontrivial element of $\Gal(L/F)$ acts on $\Gal(L_S^-/L)$ as $-1$. Let $R_{F,S} \rightarrow A$ be a surjection with kernel containing $\varpi_E$, and let $V_A$ denote the corresponding deformation. If $\det V_A$ is the Teichmuller lift of $\det V_\F$ to $A^\times$, and $V_A$ is $L$-dihedral, then 
	\[
		\dim A \le \mathrm{rk}_{\Z_p}\Gal(L_S^-/L).
	\]
\end{lem}

\begin{proof}

Set $R_{L\text{-di}}^-=\mathcal{O}[[\Gal(L_S^-/L)]]$ and let $\Psi^{-}:G_L\rightarrow (R_{L\text{-di}}^-)^\times$ be the canonical character. We have an $L$-dihedral deformation of $V_\F$ to $R_{L\text{-di}}^-$ given by $\Ind_{G_L}^{G_F}\tilde{\chi}\Psi$. This deformation is universal for $L$-dihedral deformations whose determinant is the Teichmuller lift of $\det V_\F$. There is a surjection $R_{F,S}\rightarrow R_{L\text{-di}}$, and our assumptions imply that the surjection $R_{F,S} \rightarrow A$ factors through
	\[
		R_{F,S}\longrightarrow R_{L\text{-di}}^-\longrightarrow\F[[\Gal(L_S^-/L)]],
	\]
from which the result follows.  \end{proof}

The following lemma is taken directly from \cite{SWcorrection}*{Lemma 2.2.1}. We include it for ease of reference later.

\begin{lem}\label{NotDihedSplit}
Let $L/F$ be quadratic such that $\rhobar$ is $L$-dihedral. Assume there is some $v|p$ in $F$ that does not split in $L$. Let $A$ be a $\CNL_\mathcal{O}$-algebra domain, and let $V_A$ be a deformation of $V_\F$ to $A$. If $V_A$ is dihedral and there are characters $\chi_1,\chi_2 : G_v \rightarrow A^\times$ such that
	\[ \tr \rho_A(\sigma) = \chi_1(\sigma) + \chi_2(\sigma)
	\]
for all $\sigma\in G_v$ (and any $\rho_A$ in the deformation class of $V_A$), then $\chi_1/\chi_2$ had order at most two.\end{lem}

\begin{proof}

Let $w$ denote the unique place in $L$ above $v$. Note that $G_w$ is index two in $G_v$, and we can find $\sigma \in G_v$ such that $\sigma$ generates $\Gal(L/F)$. By the theory of pseudo-representations, we see that $(V_A|_{G_v})^\mathrm{ss} = \chi_1 \oplus \chi_2$.

By assumption, there is a character $\chi: G_L \rightarrow A^\times$ such that $V_A \cong \Ind_{G_L}^{G_F} \chi$. Note that $V_A|_{G_L} \cong \chi \oplus \chi'$, where $\chi'$ denotes the conjugate of $\chi$ by $\sigma$. Replacing $\chi$ by $\chi'$, if necessary, we can assume $\chi|_{G_w} = \chi_1|_{G_w}$ and $\chi'|_{G_w} = \chi_2|_{G_w}$. But since $\chi_1$ and $\chi_2$ are characters of $G_v$, we have
	\[ \chi_1|_{G_w} = \chi_1^\sigma|_{G_w} = \chi'|_{G_w} = \chi_2|_{G_w}, \]
so $\chi_1/\chi_2$ factors through $\Gal(L_w/F_v)$.  
	\end{proof}

%% file: HidaFam/HidaFamIntro.tex
\section{Modular forms}\label{ModForms}

In this section we recall Hida's theory of $p$-adic Hecke algebras in the case of a totally definite quaternion algebra over a totally real field. We recall some facts about the associated Galois representations and define the particular Hecke algebras and Hecke modules that will be used in the patching argument in \S \ref{LocRequalsT}.

In the first subsection we recall the definition of modular forms on a totally definite quaternion algebra over a totally real field and their connection with cuspidal automorphic representations of $\GL_2$. This is all standard, except that we allow the possibility of non-compact subgroups as in \cite{KW2}*{\S 7}. We do this for the following reason. If we have a fixed central character for our space of modular forms, and $f$ is a eigenform with Iwahori level at $v$, then then there are two possibilities for the Hecke eigenvalue at $v$. When $p$ is odd, the choice for the eigenvalue is uniquely determined modulo $p$, and so the local representation of $G_v$ associated to $f$ is uniquely determined by the reduction of $f$ modulo $p$. When $p=2$ this is not the case, and in order to ensure that our local deformation ring at $v$ is a domain, we only want to allow lifts of a fixed mod $p$ eigenform that have a specified Iwahori level eigenvalue.

In the next subsection we recall the definition of the (finite level) nearly-ordinary Hecke algebras and the corresponding nearly-ordinary subspace of modular forms.

In the following subsection we construct the universal nearly-ordinary Hecke algebra and a certain avatar of $p$-adic Hida families. This is done following \cite{HidaNO}, except that we deal only with the totally definite case, and so things become much simpler. We also allow certain non-compact subgroups, but this does not pose an issue, provided one assumes that the level is small enough at some fixed place to guarantee a standard neatness condition.

In the next subsection we recall properties of Galois representations associated to our quaternionic modular forms and the construction of large Galois representations with values in our universal Hecke algebra using pseudo-representations and a theorem of Nyssen and Rouquier.

The final subsection will deal with augmenting the level of our Hecke modules at auxiliary primes, necessary for the patching argument in \S \ref{LocRequalsT}. Normally, one augments the level at primes $v$ that are congruent to $1$ modulo $p$, and such that the fixed residual representation is unramified at $v$ with the Frobenius having distinct eigenvalues. This is done to ensure there are no lifts of the fixed residual representation that are Steinberg at $v$. As in \cite{SWreducible} and \cite{SWirreducible}, it will be necessary for us augment the level at places $v$ for which the residual representation does not have distinct eigenvalues. Due to this, we cannot ensure that there are no Steinberg-at-$v$ lifts. However, we do show that any such are annihilated by a particular element, cf. \ref{NoSteinberg}, which allows us to prove a control theorem for these auxiliary primes that is necessary for the patching argument in \S \ref{LocRequalsT}. 

We now introduce some notation and assumptions that will be used throughout this section. We denote by $F\subset\Qbar$ a totally real field, $\mathcal{O}_F$ its ring of integers, and $\A_F$ its ring of adeles. If $S$ is a finite set of places of $F$, we let $\A_{F,S}$ denote $\prod_{v\in S} F_v$ and $\A_F^S$ denote $\prod_{v\notin S}' F_v$. If $w$ is a rational place, we will write $\A_{F,w}$ and $\A_F^w$ instead of $\A_{F,\{v|w\}}$ and $\A_F^{\{v|w\}}$; in particular, $\A_F^\infty$ denotes the ring of finite adeles.

Recall we have fixed algebraic closures $\Qbar_p$ of $\Q_p$, as well as embeddings $\Qbar\hookrightarrow\Qbar_p$ and $\Qbar\hookrightarrow\C$. Let $J_F$ denote the set of embeddings $F\hookrightarrow\Qbar$. Via our fixed embeddings of $\Qbar$ into $\Qbar_p$ and $\C$, we view $J_F$ as the set of embeddings of $F$ into $\Qbar_p$ and $\C$, respectively. Let $E/\Q_p$ be a finite extension containing all embeddings of $F\hookrightarrow\Qbar_p$. Let $\mathcal{O}$ denote the ring of integers of $E$. Given an element $z \in F \otimes_\Q \Q_p\cong \prod_{v|p} F_v$, and $\mathbf{k}\in \Z^{J_F}$, we let $z^\mathbf{k}$ denote $\prod_{\tau\in J_F}\tau(z)^{k_\tau} \in E$.  We call a pair $\kappa = (\mathbf{k},\mathbf{w})\in\Z^{J_F}\times\Z^{J_F}$ an \textit{algebraic weight} if $k_\tau\ge 2$ for all $\tau$ and $k_\tau+2w_\tau$ is independent of $\tau\in J_F$. For each finite place $v$ of $F$, we let $\mathfrak{m}_v$ denote the maximal ideal of $\mathcal{O}_{F_v}$ and $k_v$ the residue field.

We will let $\mathcal{U}_p = (\mathcal{O}_F \otimes_\Z \Z_p)^\times \cong \prod_{v|p} \mathcal{O}_{F_v}^\times$. For any $a \ge 1$, we let
	\[ \mathcal{U}_p^a = \{ (x_v)_{v|p} \in \mathcal{U}_p :
	 x_v \equiv 1 \;\mathrm{mod}\; \mathfrak{m}_v^a \text{ for each }v|p\}.
	 \]
For $a\ge 0$ we set $\Lambda(\mathcal{U}_p^a) = \mathcal{O}[[\mathcal{U}_p^a]]$. Note that $\Lambda(\mathcal{U}_p^a)$ has dimension $1+[F:\Q]$, is local if $a\ge 1$, and is isomorphic to a power series over $\mathcal{O}$ if $a$ is sufficiently large. 

For $v$ a finite place of $F$, and $a \ge 0$, we let $\mathrm{Iw}(v^a)$, respectively $\mathrm{Iw}_1(v^a)$, denote the subgroup of $\GL_2(\mathcal{O}_{F_v})$ that are upper triangular, respectively upper triangular unipotent, modulo $\mathfrak{m}_v^a$. Given integers $b \ge a \ge 0$, wet let $\mathrm{Iw}(v^{a,b}) = \mathrm{Iw}_1(v^a) \cap \mathrm{Iw}(v^b)$.

For a topological $\mathcal{O}$-module $M$, we denote by $M^\vee$ the Pontryagin dual of $M$, i.e. $M^\vee = \Hom_\mathcal{O}(M,E/\mathcal{O})$. This assignment is functorial and $f \mapsto  f^\vee$ gives an isomorphism $\End_\mathcal{O}(M)\cong \End_\mathcal{O}(M^\vee)$. In particular, if a commutative ring $R$ acts on $M$, then it acts naturally on $M^\vee$ by $(r\phi)(m) = \phi(rm)$. Pontryagin duality is not an exact functor, but if we restrict to the subcategory of locally compact Hausdorff $\mathcal{O}$-modules and strict morphisms, i.e. i.e morphisms $f:M\rightarrow M'$ such that the induced map $M/\ker f\rightarrow \im f$ is an isomorphism of topological $\mathcal{O}$-modules, then if 
	\[ M_1 \stackrel{f}{\longrightarrow} M_1 \stackrel{g}{\longrightarrow} M_3 \]
is exact and $g(M_2)$ is closed in $M_3$,
	\[ M_3^\vee \stackrel{g^\vee}{\longrightarrow} M_2^\vee 
	\stackrel{f^\vee}{\longrightarrow} M_3^\vee \]
is also exact, cf. \cite{MilneADT}*{Proposition 0.20}. We will always be in such a situation in what follows, so we will occasionally refer to the ``exactness of Pontryagin duality" without further comment. 

\subsection{Quaternionic cusp forms}\label{QuatForms}

\input{HidaFam/ModForms}

\subsection{Nearly ordinary Hecke algebras}\label{HeckeAlg}

\input{HidaFam/HeckeAlg}

\subsection{Universal nearly ordinary Hecke algebras}\label{UnivHeckeAlg}

\input{HidaFam/UnivHeckeAlg}

\subsection{Modular Galois representations}\label{ModGalReps}

\input{HidaFam/ModGalReps}

\subsection{Auxiliary primes and freeness}\label{ModAuxPrimes}

\input{HidaFam/ModAuxPrimes}

%% file: HidaFam/ModForms.tex


Let $D$ denote a quaternion algebra with centre $F$ ramified at all infinite places and split at all places above $p$. Denote by $\Sigma$ the set of finite places at which $D$ ramifies. Let $\nu_D$ denote the reduced norm of $D$. Fix a maximal order $\mathcal{O}_D$ of $D$ and, for each finite place $v$ at which $D$ is split, an isomorphism $\mathcal{O}_D\otimes_{\mathcal{O}_F}\mathcal{O}_{F_v}\cong M_2(\mathcal{O}_{F_v})$ of $\mathcal{O}_{F_v}$-algebras.  This determines an isomorphism $D_v^\times\cong\GL_2(F_v)$ sending $(\calO_D)_v^\times$ to $\GL_2(\calO_{F_v})$. Using this, we identify $(D\otimes_F \A_F^{\infty,\Sigma})^\times$ with $\GL_2(\A_F^{\infty,\Sigma})$. We also fix a locally algebraic character $\psi : F^\times \backslash (\A_F^\infty)^\times \rightarrow \calO^\times$, and for each $v \in \Sigma$, unramified characters $\gamma_v : F_v^\times \rightarrow \calO^\times$ such that $\gamma_v^2 = \psi|_{F_v^\times}$.

\subsubsection{}\label{QuatFormsDef}

Let $A$ be a topological $\mathcal{O}$-module. For each $\tau\in J_F$, we have an isomorphism $\mathcal{O}_D\otimes_{\mathcal{O}_F,\tau}\mathcal{O}\cong M_2(\mathcal{O})$. Via this isomorphism, given $\tau\in J_F$, $k_\tau\ge 2$ and $w_\tau\in\Z$, we can view
	\[\mathrm{Sym}^{k_\tau-2}A^2\otimes_\mathcal{O}{\det}^{w_\tau}\mathcal{O}^2\]
as an $\mathcal{O}_D\otimes_\Z\Z_p \cong \mathrm{M}_{2 \times 2}(\mathcal{O}_F \otimes_\Z \Z_p)$-module, which we denote by $W_{k_\tau,w_\tau}(A)$. Given an algebraic weight $\kappa=(\bold{k},\bold{w})\in \Z^{J_F}\times \Z^{J_F}$, we set $W_\kappa(A)=\otimes_{\tau\in J_F}W_{k_\tau,w_\tau}(A)$. Note $W_{\kappa}(A) \cong W_{\kappa}(\calO)\otimes_\mathcal{O}A$. Concretely we can view $W_{\kappa}(\mathcal{O})$ as the space of $\mathcal{O}$-linear combinations on the monomials
	\[ \prod_{\tau\in J_F} X_\tau^{k_\tau - 2 - j_\tau} Y_\tau^{j_\tau} \]
for $0\le j_\tau \le k_\tau-2$, where the action of $\mathcal{O}_D\otimes_\Z\Z_p\cong\mathrm{M}_{2\times 2}(\mathcal{O}_F\otimes_\Z\Z_p)$ is given by
	\[	\left(\begin{array}{cc} a_p & b_p \\ c_p & d_p \end{array}\right)
		\prod_{\tau\in J_F} X_\tau^{k_\tau - 2 - j_\tau} Y_\tau^{j_\tau} 
		= (a_pd_p-b_pc_p)^\mathbf{w}
		\prod_{\tau\in J_F} (\tau(a_p) X_\tau+\tau(c_p) Y_\tau)^{k_\tau - 2 - j_\tau} 
		(\tau(b_p) X_\tau + \tau(d_p) Y_\tau)^{j_\tau}.	\]
 
 Following \cite{KW2}, we will consider some non-compact open subgroups $U = \prod_v U_v$ of $(D\otimes_F \A_F^\infty)^\times$ in order to ensure that certain local deformation rings are domains. This is only necessary when $p=2$, which is our principal focus, but we do this for all $p$ simply to avoid cases and to treat all primes at once. To facilitate this, we extend the action of $(\mathcal{O}_D \otimes_\Z \Z_p)^\times$ on $W_\kappa(A)$ to an action of
 	\[ (D \otimes_F \A_F^{\infty,p})^\times \times (\mathcal{O}_D \otimes_\Z \Z_p)^\times
	\cong \prod_{v \nmid p} D_v \times \prod_{v|p} (\mathcal{O}_D)_v^\times \]
by letting $D_v$ act via $\gamma_v^{-1} \circ \nu_D$ if $v \in \Sigma$, and trivially if $v \notin \Sigma\cup \{w|p\}$.

Let $\Sigma' \subseteq \Sigma$. We will call an open subgroup $U=\prod_v U_v$ of $(D\otimes_F \A_F^\infty)^\times$ a $(\Sigma'\subseteq\Sigma)$-\textit{open subgroup} if
	\begin{itemize}
	\item[-] $U_v\subseteq \GL_2(\mathcal{O}_{F_v})$ for $v \notin \Sigma$,
	\item[-] $U_v = D_v$ for $v \in \Sigma'$, and
	\item[-] $U_v = (\mathcal{O}_{D})_v^\times$ for $v \in \Sigma\smallsetminus\Sigma'$. 
	\end{itemize} 
Let $\kappa$ be an algebraic weight and let $U$ be a $(\Sigma'\subset\Sigma)$-open subgroup. We let $S_{\kappa,\psi}(U,A)$ denote the space of functions
	\[
	f:D^\times\backslash(D\otimes_F\A_F^\infty)^\times\longrightarrow W_{\kappa}(A)
	\]
such that  $f(xu)=u^{-1}f(g)$ for all $u\in U$ and $f(zx)=\psi(z)f(x)$ for all $z\in(\A_F^\infty)^\times$. For this space to be nonzero, there must be a submodule of $A$ on which $\psi(z) = z_p^{2-\bold{k}-\bold{w}}$ for all $z\in U\cap(\A_F^\infty)^\times$. In the case that $\kappa=((k,\ldots,k),(0,\ldots,0))$ for some $k\ge 2$, we will also write $S_{k,\psi}(U,A)$ for $S_{\kappa,\psi}(U,A)$. Note that for $v \in \Sigma'$ we have $f(xg_v) = \gamma_v \circ \nu_D (g_v) f(x)$ for all $g_v \in D_v$.

Choose $t_1,\ldots,t_n\in (D\otimes_F\A_F^\infty)^\times$ such that $(D\otimes_F\A_F^\infty)^\times=\sqcup_{i=1}^n D^\times t_i U(\A_F^\infty)^\times$. Then the map $f\mapsto (f(t_1),\ldots,f(t_n))$ defines an isomorphism of $\mathcal{O}$-modules
	\[ S_{\kappa,\psi}(U,A) \stackrel{\sim}{\longrightarrow} \bigoplus_{i=1}^n 
	W_{\kappa}(A)^{(U(\A_F^\infty)^\times\cap t_i^{-1}D^\times t_i)/F^\times}
	\]
which is also an isomorphism of $A$-modules if $A$ is an $\mathcal{O}$-algebra. If for any $t\in (D\otimes_F\A_F^\infty)^\times$, we have
	\[ (U(\A_F^\infty)^\times\cap t^{-1}D^\times t)/F^\times = 1,
	\]
then $A\mapsto S_{\kappa,\psi}(U,A)$ is an exact functor and $S_{\kappa,\psi}(U,A)\cong S_{\kappa,\psi}(U,\mathcal{O})\otimes_\mathcal{O}A$. Note, however, that if $A$ is $\calO$-flat, in particular if $A=E$, we still have $S_{\kappa,\psi}(U,A)\cong S_{\kappa,\psi}(U,\calO)\otimes_\calO A$ without any assumption on $U$.

\subsubsection{}\label{IsotropyGropus}

We record a few lemmas regarding the structure of the isotropy groups $(U(\A_F^\infty)^\times \cap t^{-1}D^\times t)/F^\times$. Let $D^1$ denote the subgroup of $D^\times$ of elements of reduced norm $1$, and let $U' = U \cap \calO_D^\times$. Note that $U'$ is open compact. Set $V = \prod_{w < \infty} \calO_{F_w}^\times$. Then the reduced norm gives an exact sequence 
	\begin{equation}\label{rednormexact}
		0 \longrightarrow (UV \cap t^{-1} D^1 t)/\{\pm 1\} \longrightarrow 
		(U(\A_F^\infty)^\times \cap t^{-1}D^\times t)
		/F^\times \longrightarrow (((A_F^\infty)^\times)^2 V \cap F^\times)/(F^\times)^2, 
	\end{equation}
and there is an exact sequence
	\begin{equation}\label{classgroupexact}
		0 \longrightarrow \mathcal{O}_F^\times /(\mathcal{O}_F^\times)^2 \longrightarrow 
		(((A_F^\infty)^\times)^2 V \cap F^\times)/(F^\times)^2 \longrightarrow \mathrm{Cl}[2] 
		\longrightarrow 0
	\end{equation}
with $\mathrm{Cl}$ the class group of $\mathcal{O}_F$. The first two of the following three lemmas are taken directly from \cite{KW2}*{\S7}.

\begin{lem}\label{ExpIsotropy}

Let $U$ be a $(\Sigma'\subseteq\Sigma)$-open subgroup of $(D \otimes_F \A_F^\infty)^\times$. Let $w$ be a finite place not above $p$ at which $D$ is split. Let $N_w$ be the cardinality of $\GL_2(k_w)$.

The exponent of a Sylow $p$-subgroup of $(U(\A_F^\infty)^\times \cap t^{-1}D^\times t)/F^\times$ divides $4N_w$.

\end{lem}

\begin{proof}

First assume that $U$ is compact. Then $UV \cap t^{-1}D^1 t$ is discrete and compact, hence finite, and injects into $U_wV_w$. Note that any finite subgroup of $U_wV_w$ of exponent a power of $p$ injects into $\GL_2(k_w)$ under reduction modulo $\mathfrak{m}_w$. By \eqref{classgroupexact}, $(((A_F^\infty)^\times)^2 V \cap F^\times)/(F^\times)^2$ has exponent $2$, and so any $p$-Sylow subgroup of $(U(\A_F^\infty)^\times \cap t^{-1}D^\times t)/F^\times$ has exponent dividing $2N_w$.

Now assume that $U$ is not compact. Let $U'$ denote the maximal compact open subgroup of $U$. Since $U(\A_F^\infty)^\times /U'(\A_F^\infty)^\times$ is a finite group of exponent $2$, the same holds for the quotient of $(U(\A_F^\infty)^\times \cap t^{-1} D^\times t) / F^\times$ by $(U'(\A_F^\infty)^\times \cap t^{-1} D^\times t)/F^\times$. From what we know in the compact case, we get that the exponent of any $p$-Sylow subgroup of $(U(\A_F^\infty)^\times \cap t^{-1} D^\times t)/F^\times$ divides $4N_w$.	\end{proof}

Let $v$ be a finite place of $F$ at which $D$ is split. If $U_v \supseteq \mathrm{Iw}(v)$, then any character $\chi_v : k_v^\times \rightarrow \mathcal{O}^\times$ can be viewed as a character of $U$ by
	\[ \chi_v \left( \begin{array}{cc} a & b \\ c & d \end{array} \right) = \chi_v (a_v d_v^{-1} \;\mathrm{mod}\; 
	\mathfrak{m}_v).
	\]
Since $\chi_v$ is trivial on $U \cap (\A_F^\infty)^\times$, we can extend $\chi_v$ to a character on $U(\A_F^\infty)^\times$ by letting it be trivial on $(\A_F^\infty)^\times$.

\begin{lem}\label{CharKillIsotropy}

Let $Q$ be some fixed set of finite places of $F$ at which $D$ is split. Fix another finite place $w \notin Q$ at which $D$ is split, and let $N_w$ be the cardinality of $\GL_2(k_w)$. Let $U$ be a $(\Sigma'\subseteq\Sigma)$-open subgroup of $(D \otimes_F \A_F^\infty)^\times$ such that $U_v \supseteq \mathrm{Iw}(v)$ for every $v \in Q$. Assume that for every $v\in Q$, the order of the $p$-subgroup of $k_v^\times$ is divisible by the $p$-part of  $2p(4N_w)$.

There is a character $\chi = \prod_{v\in Q} \chi_v : \prod_{v \in Q} k_v^\times \rightarrow \mathcal{O}^\times$ such that, 
\begin{enumerate}
	\item each $\chi_v$ is non-trivial of order a power of $p$, and has order $\ge 4$ if $p=2$;
	\item viewing $\chi$ as a character of $U(\A_F^\infty)^\times$ as above, $\chi$ annihilates $(U(\A_F^\infty)^\times \cap t^{-1} D^\times t)/F^\times$ for any $t \in (D \otimes_F \A_F^\infty)^\times$.
\end{enumerate}

\end{lem}

\begin{proof}

Our hypothesis implies that there is a character $\chi'  = \prod_{v\in Q} \chi_v' : \prod_{v \in Q} k_v^\times \rightarrow \mathcal{O}^\times$ of order a power of $p$, with each $\chi'_v$ of order divisible by the $p$-part of $2p(4N_w)$. We then set $\chi_v = (\chi_v')^{4N_w}$, and $\chi = \prod_{v\in Q}\chi_v$. Note that the order of $\chi$ is divisible by the $p$-part of $2p$. Since the exponent of a $p$-Sylow subgroup of $(U(\A_F^\infty)^\times \cap t^{-1}D^\times t)/F^\times$ divides the $p$-part of $4N_w$ by \ref{ExpIsotropy}, $\chi$ is trivial on $(U(\A_F^\infty)^\times \cap t^{-1} D^\times t)/F^\times$ for any $t \in (D \otimes_F \A_F^\infty)^\times$.  \end{proof}

\begin{lem}\label{SuffSmallSub}
Let $v$ be a finite place of $F$ at which $D$ is split, and let $U$ be a $(\Sigma'\subseteq\Sigma)$-open subgroup of $(D \otimes_F \A_F^\infty)^\times$. There is some $n\ge 1$ such that if $U_v \subseteq \mathrm{Iw}_1(v^n)$ then $(U(\A_F^\infty)^\times \cap t^{-1}D^\times t)/F^\times = 1$ for any $t \in (D \otimes_F \A_F^\infty)^\times$.
\end{lem}

\begin{proof}

Since $D^\times \backslash (D \otimes_F \A_F^\infty)^\times /U(\A_F^\infty)^\times$ is finite, it suffices to show the existence of such an $n\ge 1$ for fixed $t \in (D \otimes_F \A_F^\infty)^\times$. Let $\mu$ denote the set of all roots of unity $\zeta \in \Qbar$ such that $[F(\zeta):F]\le 2$. Take $k\ge 1$ sufficiently large such that for any $\zeta \in \mu$ with $\zeta \ne \pm 1$, we have $\zeta + \zeta^{-1} \not\equiv \pm 2 \mod \mathfrak{m}_v^k$. We take $n$ such that $n \ge 2k$ and such that $2 \notin\mathfrak{m}_v^n$.

Let $U'$ denote the maximal compact subgroup of $U$, $D^1$ denote the subgroup of $D^\times$ of elements of reduced norm $1$, and set $V = \prod_{w < \infty} \calO_{F_w}^\times$. First consider the subgroup $(U'V \cap t^{-1} D^1 t)/\{\pm 1\}$ of $(U' (\A_F^\infty)^\times \cap t^{-1}D^\times t)/F^\times$. Note  $U'V \cap t^{-1} D^1 t$ is finite, since it is compact and discrete. Take $uz \in U'V \cap t^{-1}D^1 t$ with $u \in U'$ and $z \in V$. Since $uz \in t^{-1}D^1 t$ has finite order, $\tr(uz) = \zeta + \zeta^{-1}$ for some $\zeta \in \mu$. Our assumption on $U_v$ implies $\nu_D(u_v) \in 1 + \mathfrak{m}_v^n$, and so $1 = \nu_D (u)z^2$ implies that $z_v \in \pm 1+\mathfrak{m}_v^k$, by choice of $n$. Then
	\[ \zeta + \zeta^{-1} = \tr(uz) = \tr(u_v z_v) = z_v \tr(u_v) \in \pm 2 + \mathfrak{m}_v^k, \]
which implies $\zeta = \pm 1$, by choice of $k$, and thus that $(U'V \cap t^{-1}D^1 t)/\{\pm 1\}$ is trivial. 

Since $(U'V \cap t^{-1}D^1 t)/\{\pm 1\}$ is trivial, the reduced norm induces and injection
	\[  (U'(\A_F^\infty)^\times \cap t^{-1} D^\times t)/F^\times \longrightarrow 
	(V((\A_F^\infty)^\times)^2 \cap F^\times)/(F^\times)^2,\]
and by \eqref{classgroupexact}, $(U'(\A_F^\infty)^\times \cap t^{-1}D^\times t)/F^\times$ is a finite $2$-group. Since $U(\A_F^\infty)^\times/U'(\A_F^\infty)^\times$ is a finite group of exponent $2$, $(U(\A_F^\infty)^\times \cap t^{-1}D^\times t)/F^\times$ is also a finite $2$-group. It thus suffices to show that if $g \in U(\A_F^\infty)^\times \cap t^{-1}D^\times t$ is such that $g^2 \in F^\times$, then $g \in F^\times$. If $g^2 \in F^\times$ and $g \notin F^\times$, then $\tr g = 0$. But $U(\A_F^\infty)^\times \cap t^{-1}D^\times t$ injects into $U_v F_v^\times$, and since $2 \notin \mathfrak{m}_v^n$, elements of $U_vF_v^\times$ have nonzero trace. \end{proof}

\subsubsection{}\label{OpsOnCuspForms}

Set $U' = U \cap \calO_D^\times$. For $g \in (D \otimes_F \A_F^{\infty,\Sigma})$, if $U'gU' = \sqcup_i g_iU'$, we have $UgU = \sqcup_i g_iU$. If $g\in (D\otimes_F\A_F^{\infty,p})^\times \times (\calO_D \otimes_\Z \Z_p)^\times$ there is a double coset operator
	\[ [UgU]:S_{\kappa,\psi}(U,A)\longrightarrow S_{\kappa,\psi}(V,A) \]
given by $([UgU]f)(x)=\sum_i g_{i}f(xg_i)$. If $A$ is an $E$ vector space, then this double coset operator is defined for any $g \in (D\otimes_F \A_F^\infty)^\times$.

If $V\subseteq U$ is another $(\Sigma'\subset\Sigma)$-open subgroup with $V$ normal in $U$, then the group $\Delta = U(\A_F^\infty)^\times/V(\A_F^\infty)^\times$ acts on $S_{\kappa,\psi}(V,A)$ by $(\langle \delta \rangle f)(x) = u_\delta f(xu_\delta)$, where $u_\delta$ is any lift to $U$ of $\delta\in U(\A_F^\infty)^\times/V(\A_F^\infty)^\times$. We will have several occasions to use the following lemmas.

\begin{lem}\label{freeovergroupring}

Let $U$ and $V$ be $(\Sigma'\subseteq\Sigma)$-open subgroups of $(D\otimes_F\A_F^\infty)^\times$ with $V$ normal in $U$. If for all $t\in (D \otimes_F \A_F^\infty)^\times$ we have $(U(\A_F^\infty)^\times \cap t^{-1}D^\times t)/F^\times = 1$, then $S_{\kappa,\psi}(V,\mathcal{O})$ is a free $\mathcal{O}[\Delta]$-module.
	\end{lem}

\begin{proof}
Choose $\{t_1,\ldots,t_n\}\subset (D\otimes_F\A_F^\infty)^\times$ such that $(D\otimes_F\A_F^\infty)^\times=\sqcup_{i=1}^n D^\times t_i U(\A_F^\infty)^\times$. Our assumption on $U$ implies we have an $\mathcal{O}$-algebra isomorphism
	\begin{align*}
		S_{\kappa,\psi}(U,\mathcal{O}) &\stackrel{\sim}{\longrightarrow} \bigoplus_{i=1}^n 
	W_{\kappa}\\
	f & \longmapsto (f(t_1),\ldots,f(t_n)).
	\end{align*}
For each $\delta \in \Delta$, choose a representative $u_\delta\in U$. Then, $(D\otimes_F\A_F^\infty)^\times=\sqcup_{i=1}^n \sqcup_{\delta \in \Delta} D^\times t_iu_\delta V(\A_F^\infty)^\times$ and there is an isomorphism of $\mathcal{O}[\Delta]$-modules
	\begin{align*}
		S_{\kappa,\psi}(V,\mathcal{O}) &\stackrel{\sim}{\longrightarrow} \bigoplus_{i=1}^n
	W_{\kappa}\otimes_\mathcal{O}\mathcal{O}[\Delta]\\
	f & \longmapsto \big(\sum_{\delta\in\Delta} u_{\delta}f(t_i)\otimes \delta^{-1}\big)_{1\le i \le n}.
	\end{align*}
From which it follows that $S_{\kappa,\psi}(V,\mathcal{O})$ is a free $\mathcal{O}[\Delta]$-module. 		\end{proof}

\begin{lem}\label{FreePontDual}

Let $U$ and $V$ be $(\Sigma'\subseteq\Sigma)$-open subgroups of $(D\otimes_F\A_F^\infty)^\times$ with $V$ normal in $U$ and $\Delta = U(\A_F^\infty)^\times/V(\A_F^\infty)^\times$ abelian. Assume that for all $t\in (D \otimes_F \A_F^\infty)^\times$, we have $(U(\A_F^\infty)^\times \cap t^{-1}D^\times t)/F^\times = 1$.

 $S_{\kappa,\psi}(V,E/\mathcal{O})^\vee$ is a free $\mathcal{O}[\Delta]$-module, and for $\iota : S_{\kappa,\psi}(U,E/\mathcal{O})\rightarrow S_{\kappa,\psi}(V,E/\mathcal{O})$ the natural inclusion, $\iota^\vee$ defines an isomorphism from the $\Delta$-coinvariants of $S_{\kappa,\psi}(V,E/\calO)^\vee$ to $S_{\kappa,\psi}(U,E/\calO)^\vee$.

\end{lem}

\begin{proof}

Let $\mathfrak{a}_\Delta$ denote the augmentation ideal of $\mathcal{O}[\Delta]$. We have
	\begin{align*} S_{\kappa,\psi}(V,E/\mathcal{O})^\vee/\mathfrak{a}_\Delta &\cong
	\Hom_\mathcal{O}( S_{\kappa,\psi}(V,E/\mathcal{O})[\mathfrak{a}_\Delta], E/\mathcal{O})\\
	&\cong \Hom_\mathcal{O}( S_{\kappa,\psi}(U,E/\mathcal{O}), E/\mathcal{O})\\
	& = S_{\kappa,\psi}(U,E/\mathcal{O})^\vee,	\end{align*}
which is the second part of the lemma.

Give $\Hom_\mathcal{O}(\mathcal{O}[\Delta],\mathcal{O})$ the $\mathcal{O}[\Delta]$-module structure $(\delta\lambda)(r) = \lambda(\delta r)$. There is an isomorphism of $\mathcal{O}[\Delta]$-modules $\mathcal{O}[\Delta]\cong\Hom_\mathcal{O}(\mathcal{O}[\Delta],\calO)$ defined by sending $\delta\in\Delta$ to $(\delta^{-1})^\ast$, where $(\delta^{-1})^\ast$ takes value $1$ on $\delta^{-1}$ and $0$ on every other $\gamma\in\Delta$. Hence, \ref{freeovergroupring} implies that $\Hom_\calO(S_{\kappa,\psi}(V,\calO),\calO)$ is a free $\calO[\Delta]$-module. By our assumption on $U$, $S_{\kappa,\psi}(V,E/\mathcal{O})\cong S_{\kappa,\psi}(V,\mathcal{O})\otimes_\mathcal{O} E/\mathcal{O}$. Moreover, since $S_{\kappa,\psi}(V,\mathcal{O})$ is free over $\mathcal{O}$, we have an $\mathcal{O}$-module isomorphism
	\[ \Hom_\mathcal{O}(S_{\kappa,\psi}(V,\mathcal{O})\otimes_\mathcal{O} E/\mathcal{O},
		E/\mathcal{O} ) \cong \Hom_\mathcal{O}(S_{\kappa,\psi}(V,\mathcal{O}),\mathcal{O}) \]
given by $\phi\mapsto \phi\otimes 1$. This map is $\mathcal{O}[\Delta]$-equivariant, so the freeness of $\Hom_\calO(S_{\kappa,\psi}(V,\calO),\calO)$ over $\calO[\Delta]$ implies that of $S_{\kappa,\psi}(V,E/\mathcal{O})^\vee$.
	\end{proof}

\subsubsection{}\label{JLsubsubsection}

We finish this subsubsection by recalling the connection between the $\calO$-modules $S_{\kappa,\psi}(U,\calO)$ and cuspidal automorphic representations of $\GL_2(\A_F)$. We say an irreducible cuspidal automorphic representation $\pi$ of $\GL_2(\A_F)$ is \textit{regular algebraic} if there is an algebraic weight $\kappa = (\bold{k},\bold{w})$ such that for each $\tau\in I$, letting $v$ denote the corresponding infinite place coming from our fixed embedding $\Qbar\rightarrow\C$, $\pi_v$ is the discrete series representation with lowest weight $k_\tau-1$ and central character $z\mapsto \mathrm{sgn}(z)^{k_\tau} \lvert z \rvert^{2-k_\tau-2w_\tau}$. Recall that if $V$ is an open compact subgroup of $\GL_2(\A_F^\infty)$ and $g\in\GL_2(\A_F^\infty)$, there is a double coset operator $[VgV]$ on $\pi^V$ given by
	\[ [VgV]w = \sum_i g_i w \]
if $VgV = \sqcup_i g_i V$.

Set $U' = U \cap \calO_D^\times$. Assume that $\psi(z) = z_p^{2-\bold{k}-2\bold{w}}$ for all $z\in U' \cap(\A_F^\infty)^\times$. Fix an isomorphism $\Qbar_p \cong \C$ that extends our fixed embeddings $\Qbar \rightarrow \Qbar_p$ and $\Qbar \rightarrow \C$. Define a character $\psi_\C : F^\times \backslash \A_F^\times \rightarrow \C^\times$ by $\psi_\C(z) = \psi(z) z_p^{\mathbf{k} + 2\mathbf{w} -2} z_\infty^{2 - \mathbf{k} -2 \mathbf{w}}$. We also define characters $\gamma_{v,\C} : F_v^\times \rightarrow \C^\times$ for each $v \in \Sigma$ simply via the isomorphism $\overline{\Q}_p \cong \C$. Note that we can view
	\[ W_{\kappa}(\C) = \otimes_{\tau \in J_F} \Sym^{k_\tau -2} \C^2 \otimes \det{}^{w_\tau}\C^2 \]
as a representation of $D_\infty^\times = (D\otimes_F \R)^\times$. 

Let $\mathfrak{Aut}^D_{\kappa,\psi_\C,\Sigma'}$ denote the set of all irreducible automorphic representations $\pi$ of $D$ with central character $\psi_\C$ such that $\pi_\infty \cong W_\kappa(\C)^\ast$ and $\pi_v \cong \gamma_v \circ \nu_D$ for each $v \in \Sigma'$. For $\pi \in \mathfrak{Aut}^D_{\kappa,\psi_\C,\Sigma'}$ and $g \in (D \otimes_\F \A_F^\infty)^\times$, there is a double coset operator $[U'gU']$ on $\pi^{U'}$ defined in the same way as the $\GL_2$-case. Let $C^\infty_{\psi_\C,\Sigma'}(D^\times \backslash (D\otimes_F\A_F)^\times /U',\C)$ denote the space of smooth functions
	\[ \phi: D^\times \backslash (D\otimes_F\A_F)^\times /U' \longrightarrow \C \]
such that $\phi(zx) = \psi_\C(z)\phi(x)$ for all $z\in \A_F^\infty$ and $\phi(xg_v) = \gamma_{v,\C}\circ \nu_D(g_v) \phi(x)$ for all $g_v \in D_v^\times$ with $v\in \Sigma'$. Then
	\[ \bigoplus_{\mathfrak{Aut}_{\kappa,\psi_\infty}^D} \pi^{U'} = 
		\Hom_{D_\infty^\times} (W_\kappa(\C)^\ast, 
		C^\infty_{\psi_\C,\Sigma'}(D^\times \backslash (D\otimes_F\A_F)^\times /U', \C)) \]
and the map $f \mapsto (\lambda \mapsto (x \mapsto \lambda( x_\infty^{-1} x_p f(x))))$ defines an isomorphism
	\[ \mathfrak{A}: S_{\kappa,\psi}(U,E)\otimes_E \C \stackrel{\sim}{\longrightarrow} 
		\Hom_{D_\infty^\times} (W_\kappa(\C)^\ast, C^\infty_{\psi_\C}(D^\times \backslash
		(D\otimes_F\A_F)^\times /U')) \]
such that $\mathfrak{A}\circ [UgU] = [U'gU'] \circ \mathfrak{A}$ for all $g\in (D\otimes_F \A_F^\infty)^\times$. Applying the theorem of Jacquet-Langlands and Shimizu to $\mathfrak{Aut}^D_{\kappa,\psi_\C,\Sigma'}$ we obtain the following.

\begin{prop}\label{JacquetLanglands}
Define an open compact subgroup $V$ of $\GL_2(\A_F^\infty)$ by $V_v = U_v$ for $v\notin \Sigma$ and $V_v = \mathrm{Iw}(v)$ for $v\in \Sigma$. Let $\Pi_{\kappa,\psi_\C,\Sigma'}^{\Sigma}$ denote the set of all irreducible cuspidal automorphic representations $\pi$ of $\GL_2(\A_F)$ such that
\begin{itemize}
	\item[-] $\pi$ has central character $\psi_\C$,
	\item[-] $\pi$ is regular algebraic of weight $\kappa$
	\item[-] $\pi_v$ is square integrable for each $v \in \Sigma$,
	\item[-] $\pi_v \cong (\gamma_{v,\C} \circ \det) \otimes \mathrm{St}$ for each $v \in \Sigma'$, where $\mathrm{St}$ denotes the Steinberg representation.
\end{itemize}
There is a surjection of $\C$-vector spaces
	\[ \mathrm{JL} : S_{\kappa,\psi}(U,E)\otimes_E \C \longrightarrow 
		\bigoplus_{\Pi_{\kappa,\psi_\C,\Sigma'}^{\Sigma}} (\pi^\infty)^V	\]
such that
\begin{enumerate}
	\item for $g\in (D\otimes_F\A_F^{\infty,\Sigma})^\times \cong \GL_2(\A_F^{\infty,\Sigma})$, we have
	\[ \mathrm{JL} \circ [UgU] = [VgV] \circ \mathrm{JL};	\]
	\item $\mathrm{JL}$ is an isomorphism unless $\kappa = ((2,\ldots,2),\bold{w})$, in which case the kernel of $\mathrm{JL}$ consists of the functions that factor through the reduced norm.
\end{enumerate}
\end{prop}

%% file: HidaFam/HeckeAlg.tex

Keep the assumptions and notation of the previous subsubsection. We further assume that for each $v|p$, there is $n\ge 0$ such that $U_v \supseteq \mathrm{Iw}_1(v^n)$
		
\subsubsection{}\label{HeckeOps}

Recall that $\Sigma$ is the set of finite places at which $D$ ramifies. Let $S = \Sigma \cup \{v|p\} \cup \{v : U_v \ne (\mathcal{O}_D)_v^\times \}$. For any $v\notin S$ and uniformizer $\varpi_v$ at $v$, the double cosets
	\[ U \left(\!\begin{array}{cc} \varpi_v & \\ & 1 \end{array}\! \right) U \hspace{0.25cm}
		\text{and}\hspace{0.25cm}
	U\left( \! \begin{array}{cc} \varpi_v & \\ & \varpi_v \end{array} \! \right) U  \]
do not depend on the choice of $\varpi_v$. We define operators $T_v$ and $S_v$ on $S_{\kappa,\psi}(U,A)$ by setting
	\[T_v f = \left[U\left(\!\begin{array}{cc}\varpi_v&\\&1\end{array}\!\right)U\right]f
	 \hspace{0.25cm}\text{and}\hspace{0.25cm} 
	 S_vf = \left[U\left(\!\begin{array}{cc} \varpi_v & \\ & \varpi_v \end{array}\!\right) U \right]f,	\]
Note that $S_v$ is simply multiplication by $\psi(\varpi_v)$. If $V\subseteq U$ is another $(\Sigma'\subset\Sigma)$-open subgroup, the natural inclusion $S_{\kappa,\psi}(U,A)\rightarrow S_{\kappa,\psi}(V,A)$ is equivariant for the $T_v$ and $S_v$ such that $V_v = \GL_2(\mathcal{O}_{F_v})$.

For each $v|p$ we fix, once and for all, an element $\varpi_v \in F$ such that $\varpi_v$ is a uniformizer for $F_v$ and lies in $\calO_{F_w}^\times$ for all $w|p$ with $w\ne v$. We choose our uniformizers for $F_v$ in this way because, following Hida, we modify the usual Hecke operators at places above $p$ in order to define the nearly ordinary subspace of $S_{\kappa,\psi}(U,\calO)$. This modification will involve a multiplication by a power of $\varpi_v$, and having $\varpi_v$ belong to $F$ allows us to compare the nearly ordinary subspace of $S_{\kappa,\psi}(U,\calO)$ with nearly ordinary cuspidal automorphic representations of $\GL_2(\A_F)$.

Note first that for any $g\in\mathrm{M}_{2\times 2}(\mathcal{O}_F\otimes_\Z \Z_p)$ with $\det(g)\ne 0$, the endomorphism $\tau(\det(g))^{-w_\tau}g$ of $W_{k_\tau,w_\tau}(E)$ stabilizes the $\mathcal{O}$-lattice $W_{k_\tau,w_\tau}(\mathcal{O})$, hence defines an endomorphism of $W_{k_\tau,w_\tau}(A)$ for any $\mathcal{O}$-module $A$. For each $v|p$, we then define the operator $T_{\varpi_v}$ on $S_{\kappa,\psi}(U,A)$ by	
	\[ T_{\varpi_v} f=\varpi_v^{-\bold{w}}\left[U\left(\!\begin{array}{cc}\varpi_v&\\&1\end{array}\!\right)U\right]f,
	\]
Note that $T_{\varpi_v}$ depends on the choice of $\varpi_v$ as we are not assuming $U_v=\GL_2(\mathcal{O}_{F_v})$. If $V\subset U$ is another $(\Sigma'\subseteq\Sigma)$-open subgroup and $\mathrm{Iw}_1(v^n) \subseteq V_v \subseteq U_v \subseteq \mathrm{Iw}(v)$, for some $n\ge 1$, the natural inclusion $S_{\kappa,\psi}(U,A)\rightarrow S_{\kappa,\psi}(V,A)$ is a equivariant for $T_{\varpi_v}$.

For $y\in \mathcal{U}_p = (\mathcal{O}_F\otimes_\Z \Z_p)^\times$, we define two related operators. First we define $\langle y \rangle$ by 
	\[ \langle y\rangle f=
		\left[U\left(\!\begin{array}{cc}y&\\&1\end{array}\!\right)U\right] f. \]
We then define a normalized version $\langle y \rangle^{\mathrm{no}}$ by
	\[ \langle y \rangle^\mathrm{no} = y^{-\mathbf{w}} \langle y \rangle.\]
The point of introducing the normalized version is that later we will define an isomorphism between certain spaces of modular forms of different weights, and this isomorphism will not be $\langle y \rangle$-equivariant, but will be $\langle y \rangle^{\mathrm{no}}$-equivariant. Note that an $\calO$-subalgebra of $\End_\calO (S_{\kappa,\psi}(U,\calO))$ containing $\langle y \rangle$ also contains $\langle y \rangle^{\mathrm{no}}$ and vice-versa. If $V\subseteq U$ is another $(\Sigma'\subseteq\Sigma)$-open subgroup then the natural inclusion $S_{\kappa,\psi}(U,A)\rightarrow S_{\kappa,\psi}(V,A)$ respects the $\langle y \rangle$ and $\langle y \rangle^{\mathrm{no}}$-actions on each space.

If $A$ is a commutative $\mathcal{O}$-algebra, we denote by $\bold{T}_{\kappa,\psi}(U,A)$ the $A$-subalgebra of $\End_A(S_{\kappa,\psi}(U,A))$ generated by $T_v$ for each $v\notin S$, $T_{\varpi_v}$ for each $v|p$, and $\langle y\rangle$ (equivalently $\langle y \rangle^{\mathrm{no}}$) for each $y \in \mathcal{U}_p$. If $A$ is not a commutative $\mathcal{O}$-algebra, we denote by $\bold{T}_{\kappa,\psi}(U,A)$ the $\mathcal{O}$-subalgebra of $\End_{\mathcal{O}}(S_{\kappa,\psi}(U,A))$ generated by the aforementioned Hecke operators. If $A$ is a finite $\mathcal{O}$-module, or if $A=E/\mathcal{O}$, then $\bold{T}_{\kappa,\psi}(U,A)$ is a finite commutative $\mathcal{O}$-algebra.

\subsubsection{}\label{NearOrd}

Let $A$ be one of the following: $\mathcal{O}$, a finite quotient of $\mathcal{O}$, $E/\mathcal{O}$, or a finite submodule of $E/\mathcal{O}$. We call a maximal ideal $\mathfrak{m}$ of $\bold{T}_{\kappa,\psi}(U,A)$ \textit{nearly ordinary} if the image of $T_{\varpi_v}$ in $\bold{T}_{\kappa,\psi}(U,A)/\mathfrak{m}$ is non-zero for each $v|p$. Note that since $\bold{T}_{\kappa,\psi}(U,A)$ is finite over $\mathcal{O}$, we have a decomposition
	\[ \bold{T}_{\kappa,\psi}(U,A)=\prod_{\mathfrak{m}}\bold{T}_{\kappa,\psi}(U,A)_{\mathfrak{m}} \]
where the product runs over the set of maximal ideals of $\bold{T}_{\kappa,\psi}(U,A)$. We define the \textit{nearly ordinary Hecke algebra} (of weight $\bold{k}$, character $\psi$ and level $U$ with coefficients in $A$) by
	\[ \bold{T}_{\kappa,\psi}^{\mathrm{no}}(U,A)=\prod_{\mathfrak{m}\text{ no}}\bold{T}_{\kappa,\psi}(U,A)_\mathfrak{m}\]
where the product runs over all nearly ordinary maximal ideals of $\bold{T}_{\kappa,\psi}(U,A)$. Note that the projection 
	\[ \bold{T}_{\kappa,\psi}(U,A)\longrightarrow \bold{T}_{\kappa,\psi}^{\mathrm{no}}(U,A)\]
corresponds to an idempotent $e_\mathrm{H}$ in $\bold{T}_{\kappa,\psi}(U,A)$, and it is known as \textit{Hida's idempotent}. Letting
	\[ T_p = \prod_{v|p}T_{\varpi_v}, \]
it can be checked that
	\[ e_\mathrm{H} = \lim_{n\rightarrow\infty}T_p^{n!}.\]
We also define the nearly ordinary subspace of $S_{\kappa,\psi}(U,A)$ by
	\[ S_{\kappa,\psi}^{\mathrm{no}}(U,A) = e_\mathrm{H} S_{\kappa,\psi}(U,A).	\]
	
\subsubsection{}\label{NearOrdHilbert}

Let $\pi$ be a cuspidal automorphic representation of $\GL_2(\A_F)$ of weight $\kappa$. For each $v|p$, since $\varpi_v$ is an element of $F$, we can make sense of the operator
	\[ T_{\varpi_v} = \varpi_v^{-\mathbf{w}} \left[ V \left(\!\begin{array}{cc} \varpi_v & \\ & 1 \end{array} 
		\!\right) V \right]. \]
on $\pi^V$ for any compact open subgroup $V$, viewing $J_F$ as the set of embeddings $F \rightarrow \C$. We say $\pi$ is $p$\textit{-nearly ordinary} (or just \textit{nearly ordinary} if $p$ is clear from the context) if there is a open compact subgroup $V$ of $\GL_2(\A_F)$ and a non-zero vector $x \in \pi^V$ such that for each $v|p$, $T_{\varpi_v}$ acts on $x$ by an element of $\Qbar$ that is a unit in the ring of integers of $\Qbar_p$. Note this does in general depend on our choice of embedding $\overline{\Q} \rightarrow \overline{\Q}_p$. If $\pi$ is nearly ordinary, then for each $v|p$, $\pi_v$ is either a principal series representation $\pi(\eta_v |\;|_v^{1/2},\mu_v |\;|_v^{1/2})$ or a special representation $\sigma(\eta_v |\;|_v^{1/2}, \eta_v |\;|_v^{-1/2})$ with $\varpi_v^{-\mathbf{w}}\eta_v(\varpi_v)$ an element of $\Qbar$ that is a unit in $\Qbar_p$, cf. \cite{HidaNOGalRep}*{Corollary 2.2}. If $\pi$ is nearly ordinary and $\pi^V \ne 0$, then there is in fact a unique line in $\pi_v^{V_v}$ on which $T_{\varpi_v}$ acts via a unit in $\overline{\Q}_p$. 

Using the notation of \ref{JacquetLanglands}, if we let $\Pi_{\kappa,\psi_\C,\Sigma'}^{\Sigma,\mathrm{no}}$ denote the subset of $\Pi_{\kappa,\psi_\C,\Sigma'}^\Sigma$ consisting of nearly ordinary representations, then the map $\mathrm{JL}$ of \ref{JacquetLanglands} restricts to
	\[ \mathrm{JL} : S_{\kappa,\psi}^{\mathrm{no}}(U,E)\otimes_E \C
		\longrightarrow \bigoplus_{\pi \in \Pi_{\kappa,\psi_\C,\Sigma'}^{\Sigma,\mathrm{no}}}
		\pi^V. \]
If $U_v \subseteq \mathrm{Iw}(v)$ for each $v|p$, then any function factoring through the reduced norm of $D$ is not nearly-ordinary, and so in this case $\mathrm{JL}$ is an isomorphism on nearly ordinary subspaces. Using this we can identify $\mathbf{T}_{\kappa,\psi}^{\mathrm{no}}(U,E)$ with a subspace of endomorphism of $\oplus \, \pi^V$. 

Take $\pi \in \Pi_{\kappa,\psi_\C,\Sigma'}^{\Sigma,\mathrm{no}}$ with $\pi^V \ne 0$, and fix a nonzero vector $x = \otimes_v x_v \in \pi^V$, on which $T_{\varpi_v}$ acts via a unit in $\Qbar_p$ for each $v|p$. The line generated by $x$ is $\bold{T}_{\kappa,\psi}^{\mathrm{no}}(U,E)$-stable, and if $x' = \otimes_v x_v$ is another such vector, the $\bold{T}_{\kappa,\psi}^{\mathrm{no}}(U,E)$-eigensystem given by $x$ is equal to that of $x'$. It follows that we have an $E$-algebra injection
	\[ \mathbf{T}_{\kappa,\psi}^{\mathrm{no}}(U,E) \longrightarrow \prod_{\pi} \Qbar_p \]
where the product is taken over all $\pi \in \Pi_{\kappa,\psi_\C,\Sigma'}^{\Sigma,\mathrm{no}}$ with $\pi^V \ne 0$. In particular $\mathbf{T}_{\kappa,\psi}^{\mathrm{no}}(U,\calO)$ is reduced.

%% file: HidaFam/UnivHeckeAlg.tex

Keep all assumptions and notation of the previous subsection. We will further assume henceforth that $A$ is one of the following: $\mathcal{O}$, some finite quotient of $\mathcal{O}$, $E/\mathcal{O}$, or some finite submodule of $E/\mathcal{O}$. We will also assume henceforth that $U_v=\mathrm{Iw}(v)$ for all $v|p$. 

For integers $b\ge a\ge 0$ we let $U(p^{a,b})$ denote the open subgroup of $U$ given by $U(p^{a,b})_v = U_v$ if $v\nmid p$, and $U(p^{a,b})_v = U_v \cap \mathrm{Iw}(v^{a,b})$ for $v|p$. Note that $U(p^{a,a})$ is not the subgroup of $U$ consisting of elements that are upper triangular unipotent modulo $p^a$, but rather modulo $\varpi_p^a$, where $\varpi_p$ is the finite id\`{e}le with $(\varpi_p)_v = \varpi_v$ for $v|p$ and equal to $1$ elsewhere. 
		
\subsubsection{}\label{VertControlSec}

For $y\in \mathcal{U}_p$, $\left(\!\begin{array}{cc}y&\\&1\end{array}\!\right)$ normalizes $U(p^{a,b})$, so
	\[ \langle y \rangle f = \left(\!\begin{array}{cc}y&\\&1\end{array}\!\right) f.
	\]
Since we have fixed a central character $\psi$, the natural inclusion $S_{\kappa,\psi}(U(p^{0,b}),A)\rightarrow S_{\kappa,\psi}(U(p^{a,b}),A)$ identifies $S_{\kappa,\psi}(U(p^{0,b}),A)$ with the $\mathcal{U}_p$-invariants of $S_{\kappa,\psi}(U(p^{a,b}),A)$. Similarly for $S_{\kappa,\psi}^\mathrm{no}(U(p^{0,b}),A)\rightarrow S_{\kappa,\psi}^\mathrm{no}(U(p^{a,b}),A)$.

\begin{lem}\label{IwahoriAtp}
For any $a\ge 0$ and $b\ge\max\{a,1\}$, the inclusion $S_{\kappa,\psi}^\mathrm{no}(U(p^{a,b}),A)\rightarrow S_{\kappa,\psi}^\mathrm{no}(U(p^{a,b+1}),A)$ is an isomorphism.
\end{lem}

\begin{proof}

Recall that $T_p=\prod_{v|p}T_{\varpi_v}$. Then for $b\ge 1$, the action of $T_p$ on $S_{\kappa,\psi}(U(p^{a,b}),A)$ is given by the double coset operator 
	\[ \left[ U(p^{a,b}) \left(\! \begin{array}{cc} \varpi_p & \\ & 1 \end{array} \!\right) U(p^{a,b}) \right]
		\]
where $\varpi_p \in (\A_F^\infty)^\times$ is the id\`{e}le equal to $\varpi_v$ at $v|p$ and $1$ elsewhere. It easy to check that, for $b\ge 1$,
	\[	U(p^{a,b+1}) \left(\! \begin{array}{cc} \varpi_p & \\ & 1 \end{array} \!\right) U(p^{a,b+1}) =
		U(p^{a,b}) \left(\! \begin{array}{cc} \varpi_p & \\ & 1 \end{array} \!\right) U(p^{a,b+1}).
	\]
Hence $T_p S_{\kappa,\psi}(U(p^{a,b+1}),A) \subseteq S_{\kappa,\psi}(U(p^{a,b}),A)$, where we have identified $S_{\kappa,\psi}(U(p^{a,b}),A)$ with its image in $S_{\kappa,\psi}(U(p^{a,b+1}),A)$. It follows that $e_\mathrm{H} S_{\kappa,\psi}(U(p^{a,b+1}),A)=e_\mathrm{H} S_{\kappa,\psi}(U(p^{a,b}),A)$ for any $a\ge 0$ and $b\ge\max\{a,1\}$. 	\end{proof}

For $b\ge a$, the Hecke-equivariant injections $S_{\kappa,\psi}(U(p^{a,a}),A)\rightarrow S_{\kappa,\psi}(U(p^{b,b}),A)$ and $S_{\kappa,\psi}^\mathrm{no}(U(p^{a,a}),A)\rightarrow S_{\kappa,\psi}^\mathrm{no}(U(p^{b,b}),A)$ induce surjections $\bold{T}_{\kappa,\psi}(U(p^{b,b}),A)\rightarrow \bold{T}_{\kappa,\psi}(U(p^{a,a}),A)$ and $\bold{T}_{\kappa,\psi}^\mathrm{no}(U(p^{b,b}),A)\rightarrow \bold{T}_{\kappa,\psi}^\mathrm{no}(U(p^{a,a}),A)$. We define
	\[ \bold{T}_{\kappa,\psi}^\mathrm{no}(U(p^\infty),A)=
	\varprojlim_a\bold{T}_{\kappa,\psi}^\mathrm{no}(U(p^{a,a}),A), \]
and
	\[ S_{\kappa,\psi}^{\mathrm{no}}(U(p^\infty),E/\mathcal{O})=
	\varinjlim_a S_{\kappa,\psi}^{\mathrm{no}}(U(p^{a,a}),E/\mathcal{O}).\]
We have a faithful action of $\bold{T}_{\kappa,\psi}^\mathrm{no}(U(p^\infty),E/\mathcal{O})$ on $S_{\kappa,\psi}^\mathrm{no}(U(p^\infty),E/\mathcal{O})$.

If $M$ is a topological $\mathcal{O}$ module, then Pontryagin duality $M\mapsto M^\vee=\Hom_\mathcal{O}(M,E/\mathcal{O})$ induces an isomorphism $\End_\mathcal{O}(M)\cong\End_{\mathcal{O}}(M^\vee)$. Hence, we have a faithful action of $\bold{T}_{\kappa,\psi}^\mathrm{no}(U(p^{a,a}),E/\mathcal{O})$ on $S_{\kappa,\psi}^\mathrm{no}(U(p^{a,a}),E/\mathcal{O})^\vee$. Take $a_0\ge 1$ large enough so that $(U(p^{a_0,a_0})(\A_F^\infty)^\times \cap t^{-1}D^\times t)/F^\times = 1$ for all $t\in (D\otimes_F \A_F^\infty)^\times$, which is possible by \ref{SuffSmallSub}. Then for any $a\ge a_0$, $S_{\kappa,\psi}^\mathrm{no}(U(p^{a,a}),\mathcal{O})$ is a finite free $\mathcal{O}$-module and
	\[ S_{\kappa,\psi}^\mathrm{no}(U(p^{a,a}),E/\mathcal{O})\cong 
	S_{\kappa,\psi}^\mathrm{no}(U(p^{a,a}),\mathcal{O}) \otimes_\mathcal{O}E/\mathcal{O},	\]
hence there is an isomorphism
	\[ \Hom_\mathcal{O}(S_{\kappa,\psi}^\mathrm{no}(U(p^{a,a}),\mathcal{O},\mathcal{O})\cong
	\Hom_\mathcal{O}(S_{\kappa,\psi}^\mathrm{no}(U(p^{a,a}),E/\mathcal{O}),E/\mathcal{O}). \]
This isomorphism is equivariant for the action of the the operators $T_v$, with $v\notin S$, $T_{\varpi_v}$, with $v|p$, and $\langle y\rangle^\mathrm{no}$, with $y\in \mathcal{U}_p$, on each side. We then have an isomorphism, for any $a\ge a_0$, 
	\[ \bold{T}_{\kappa,\psi}^\mathrm{no}(U(p^{a,a}),\mathcal{O})\cong 
	\bold{T}_{\kappa,\psi}^\mathrm{no}(U(p^{a,a}),E/\mathcal{O}) \]
which sends $T_v$ to $T_v$, etc. And so there is a faithful action of $\bold{T}_{\kappa,\psi}^\mathrm{no}(U(p^{a,a}),\mathcal{O})$ on $S_{\kappa,\psi}^\mathrm{no}(U(p^{a,a}),E/\mathcal{O})^\vee$. We then have a faithful action of $\bold{T}_{\kappa,\psi}^{\mathrm{no}}(U(p^\infty),\mathcal{O})$ on
	\[ \varprojlim_{a\ge a_0}S_{\kappa,\psi}^\mathrm{no}(U(p^{a,a}),E/\mathcal{O})^\vee\cong
		\big(\varinjlim_{a\ge a_0}
		S_{\kappa,\psi}^\mathrm{no}(U(p^{a,a}),E/\mathcal{O})\big)^\vee = 
		S_{\kappa,\psi}^\mathrm{no}(U(p^\infty),E/\mathcal{O})^\vee.	\]

\noindent 

We view $\bold{T}_{\kappa,\psi}^\mathrm{no}(U(p^\infty),\mathcal{O})$ as a $\Lambda(\mathcal{U}_p) = \mathcal{O}[[\mathcal{U}_p]]$-algebra by letting $y\in \mathcal{U}_p$ acting via $\langle y \rangle^\mathrm{no}$. Denote by $\mathfrak{p}_{\kappa}$ the prime ideal in $\Lambda(\mathcal{U}_p)$ corresponding to the kernel of the $\mathcal{O}$-algebra homomorphism that sends $y\in \mathcal{U}_p$ to $y^{-\bold{w}}\in\mathcal{O}$. With some mild abuse of notation, we also denote by $\mathfrak{p}_\kappa$ it's pullback of to $\Lambda(\mathcal{U}_p^a) = \mathcal{O}[[\mathcal{U}_p^a]]$ for any $a\ge 0$. We then have a version of Hida's control theorem.

\begin{prop}\label{VertControl}

Let $a\ge 1$ be such that $(U(p^{a,a})(\A_F^\infty)^\times \cap t^{-1}D^\times t)/F^\times = 1$ for all $t\in (D\otimes_F \A_F^\infty)^\times$.

$S_{\kappa,\psi}^\mathrm{no}(U(p^\infty),E/\mathcal{O})^\vee$ is finite free over $\Lambda(\mathcal{U}_p^a)$ of rank equal to the $\mathcal{O}$-rank of $S_{\kappa,\psi}^\mathrm{no}(U(p^{a,a}),\mathcal{O})$. Moreover the natural surjection $S^\mathrm{no}_{\kappa,\psi}(U(p^\infty),E/\mathcal{O})^\vee \rightarrow S^\mathrm{no}_{\kappa,\psi}(U(p^{a,a}),E/\mathcal{O})^\vee$ has kernel $\mathfrak{p}_{\kappa} S^\mathrm{no}_{\kappa,\psi}(U(p^\infty),E/\mathcal{O})^\vee$.

\end{prop}

\begin{proof}

We define a different $\Lambda(\mathcal{U}_p^a)$-module structure on  $S_{\kappa,\psi}^\mathrm{no}(U(p^\infty),E/\calO)^\vee$, by letting $y \in \mathcal{U}_p$ act via $\langle y \rangle$. Note that the two $\Lambda(\mathcal{U}_p^a)$-module structures differ by an automorphism of $\Lambda(\mathcal{U}_p^a)$ that sends $\mathfrak{p}_\kappa$ to the augmentation ideal $\mathfrak{a}$ of $\Lambda(\mathcal{U}_p^a)$. It thus suffices to prove the lemma with this new $\Lambda(\mathcal{U}_p^a)$-module structure and $\mathfrak{p}_\kappa$ replaced by $\mathfrak{a}$.

The second part follows easily from \ref{IwahoriAtp},
	\begin{align*}
		S^\mathrm{no}_{\kappa,\psi}(U(p^\infty),E/\mathcal{O})^\vee / \mathfrak{a}
		S^\mathrm{no}_{\kappa,\psi}(U(p^\infty),E/\mathcal{O})^\vee & \cong
		\Hom_\mathcal{O}( S^\mathrm{no}_{\kappa,\psi}(U(p^\infty),
		E/\mathcal{O})^{\mathcal{U}_p^a} , E/\mathcal{O}) \\
		& \cong \big( \varinjlim_{b\ge a} S^\mathrm{no}_{\kappa,\psi} (U(p^{a,b}),E/\mathcal{O})
		 \big)^\vee \\ 
		& =S_{\kappa,\psi}^\mathrm{no}(U(p^{a,a}),E/\mathcal{O})^\vee.
	\end{align*}
For the first part, it suffices to show, by the second part, that for every $b\ge a$, $S^\mathrm{no}_{\kappa,\psi}(U(p^{b,b}),E/\mathcal{O})^\vee$ is a finite free $\mathcal{O}[\mathcal{U}_p^a/\mathcal{U}_p^b]$-module of rank equal to the $\mathcal{O}$-rank of $S^\mathrm{no}_{\kappa,\psi}(U(p^{a,a}),\mathcal{O})$. Since the projections
	\[ S^\mathrm{no}_{\kappa,\psi}(U(p^{b,b}),E/\mathcal{O})^\vee \longrightarrow 
		S^\mathrm{no}_{\kappa,\psi}(U(p^{a,a}),\mathcal{O})^\vee \]
are Hecke equivariant, and the direct summand of a free module over a local ring is free, it suffices to prove this without having applied Hida's idempotent. This follows from \ref{FreePontDual}. \end{proof}

\begin{cor}\label{HeckeFinite}
For any $a\ge 0$, both $S_{\kappa,\psi}^\mathrm{no}(U(p^\infty),E/\calO)^\vee$ and $\bold{T}_{\kappa,\psi}^\mathrm{no}(U(p^\infty),\mathcal{O})$ are finite over $\Lambda(\mathcal{U}_p^a)$.
\end{cor}

\begin{proof}

It suffices to show this for some $a\ge 0$. That it holds for $S_{\kappa,\psi}^\mathrm{no}(U(p^\infty),E/\mathcal{O})^\vee$ is part of \ref{VertControl}. Since there is an injection 
	\[	\bold{T}_{\kappa,\psi}(U(p^\infty),\mathcal{O}) \longrightarrow 
		\End_{\Lambda(\mathcal{U}_p^a)}( S_{\kappa,\psi}^\mathrm{no}(U(p^\infty),E/\mathcal{O})^\vee) \]
the same is true for $\bold{T}_{\kappa,\psi}(U(p^\infty),\mathcal{O})$.	\end{proof}

\subsubsection{}\label{WeightIndepSec}

There is a unique free rank one $\mathcal{O}$-submodule of $W_{\kappa}(\mathcal{O})$ on which the diagonal subgroup of $GL_2(\mathcal{O}_F\otimes_\Z\Z_p)$ acts via the character
	\[ \left(\!\begin{array}{cc} a_p & \\  & d_p \end{array} \!\right) \mapsto
		a_p^{\bold{w}}d_p^{\bold{k}+\bold{w}-2},	\]
which we will denote by $\chi_{\kappa}$. Fix a generator $v_l$ of this submodule. By abuse of notation we will also denote by $v_l$ its image in $W_{\kappa}(\mathcal{O}/\mathfrak{m}_\mathcal{O}^r)$ for any $r\ge 1$. Then, for $A=\mathcal{O}$ or $\mathcal{O}/\mathfrak{m}_\mathcal{O}^r$, the choice of $v_l$ gives a Borel equivariant map
	\[  W_{\kappa}(A)\longrightarrow A(\chi_{\kappa}). \]
If $A=\mathcal{O}/\mathfrak{m}_\mathcal{O}^r$ and $a\ge r$, this yields a $U(p^{a,a})$-equivariant map
	\[ \mathrm{pr}_\kappa : W_{\kappa}(A) \longrightarrow A. \]
Recall that, for $v \in \Sigma$, $U_v$ may act nontrivially on both sides of above. However, it does so via the same $\calO^\times$-valued character $\gamma^{-1} \circ \nu_D$.
	
\begin{lem}\label{KillHighWeights}
Let $a\ge r\ge 1$, and recall that $\varpi_p$ is the finite id\`{e}le defined by $(\varpi_p)_v=\varpi_v$ for $v|p$ and $(\varpi_p)_v=1$ for $v\nmid p$.
\begin{enumerate}
\item For any $v|p$ and $g\in U(p^{a,a})\left(\!\begin{array}{cc} \varpi_v & \\ & 1\end{array}\!\right) U(p^{a,a})$, we have $\varpi_v^{-\bold{w}}g_p v_l = v_l + w$ with $w\in\ker(\mathrm{pr}_\kappa)$.
\item For any $w\in \ker(\mathrm{pr}_\kappa)$ and $g\in U(p^{a,a})\left(\! \begin{array}{cc} \varpi_p^r & \\ & 1 \end{array} \!\right) U(p^{a,a})$, we have $\varpi_p^{-r\bold{w}} g_p w = 0$.
\end{enumerate}
\end{lem}

\begin{proof}

Let $\alpha$ be either $\varpi_v$ or $\varpi_p^r$. For $u\in U(p^{a,a})$, $u_p$ acts on $W_{\kappa}(\mathcal{O}/\mathfrak{m}_\mathcal{O}^r)$ through its image modulo $\mathfrak{m}_\mathcal{O}^r$, which is upper triangular unipotent. Hence, for any 
	\[ g\in U(p^{a,a})\left(\!\begin{array}{cc} \alpha & \\ & 1\end{array}\!\right) U(p^{a.a}) \]
we can assume
	\[ g = \left(\! \begin{array}{cc} \alpha & x \\ & 1 \end{array} \!\right)	\]
for some $x\in \A_F^\infty$ with $x_p \in \mathcal{O}_F \otimes_\Z \Z_p$.

Recall we have an isomorphism of $W_{\kappa}(\mathcal{O}/\mathfrak{m}_\mathcal{O}^r)$ with the space of $\mathcal{O}/\mathfrak{m}_\mathcal{O}^r$-linear combinations on the monomials
	\begin{equation}\label{Monomials}
	 \prod_{\tau \in J_F} X_\tau^{k_\tau - 2 - j_\tau}Y_\tau^{j_\tau}, \end{equation}
with each $0\le j_\tau \le k_\tau-2$, on which $\GL_2(\mathcal{O}\otimes_\Z\Z_p)$ acts by
	\[ \left(\! \begin{array}{cc} a_p & b_p \\ c_p & d_p \end{array} \!\right)
		\prod_{\tau \in J_F} X_\tau^{k_\tau - 2 - j_\tau}Y_\tau^{j_\tau}
		= (a_pd_p - b_p c_p)^{\bold{w}}
		\prod_{\tau \in J_F} (\tau(a_p) X_\tau + \tau(c_p) Y_\tau)^{k_\tau - 2 - j_\tau}
		(\tau(b_p) X_\tau + \tau(d_p) Y_\tau)^{j_\tau}.	\]
We see that $\ker(\mathrm{pr}_\kappa)$ is the span of all monomials \eqref{Monomials} with some $j_\tau < k_\tau - 2$ and we may assume $v_l=\prod_{J_F} Y_\tau^{k_\tau-2}$

For $\alpha = \varpi_v$ we have
	\[ \varpi_v^{-\bold{w}}\left(\!\begin{array}{cc} \varpi_v & x \\ & 1\end{array}\!\right)
	\prod_{\tau \in J_F} Y_\tau^{k_\tau - 2} = \prod_{\tau \in J_F} (\tau(x_p) X_\tau + Y_\tau)^{k_\tau-2}
	=v_l + w \]
with $w\in\ker(\mathrm{pr}_\kappa)$.

For $\alpha = \varpi_p^r$, since $\tau(\varpi_p^r)\in \mathfrak{m}_\mathcal{O}^r$ for any $\tau \in J_F$,
	\[ \varpi_p^{-r\bold{w}} \left(\! \begin{array}{cc} \varpi_p^r & x \\ & 1 \end{array} \!\right)
		\prod_{\tau \in J_F} X_\tau^{k_\tau - 2 - j_\tau}Y_\tau^{j_\tau} = 
		\prod_{\tau \in J_F} (\tau(\varpi_p^r) X_\tau)^{k_\tau - 2 - j_\tau}
		(\tau(x_p)X_\tau + Y_\tau)^{j_\tau} = 0,	\]
if some $j_\tau < k_\tau - 2$.		\end{proof}

If $a\ge r$, the map $\mathrm{pr}_\kappa : W_{\kappa}(\mathcal{O}/\mathfrak{m}_\mathcal{O}^r)\rightarrow \mathcal{O}/\mathfrak{m}_\mathcal{O}^r$ induces a morphism of $\mathcal{O}$-modules
	\[ \mathrm{pr}_\kappa^\ast : S_{\kappa,\psi}(U(p^{a,a}),\mathcal{O}/\mathfrak{m}_\mathcal{O}^r) \longrightarrow
		S_{2,\psi}(U(p^{a,a}),\mathcal{O}/\mathfrak{m}_\mathcal{O}^r). \]
	
\begin{prop}\label{WeightChangeFin}
For any algebraic weight $\kappa$ and $a\ge r\ge 1$, $\mathrm{pr}_\kappa^\ast$ is equivariant for all $T_v$ with $v\notin S$, $T_{\varpi_v}$ with $v|p$, and $\langle y \rangle^\mathrm{no}$ with $y\in \mathcal{U}_p$. Moreover $\mathrm{pr}_\kappa^\ast$ induces an isomorphism
	\[ S_{\kappa,\psi}^\mathrm{no}(U(p^{a,a}),\mathcal{O}/\mathfrak{m}_\mathcal{O}^r) \stackrel{\sim}{\longrightarrow}
		S_{2,\psi}^\mathrm{no}(U(p^{a,a}),\mathcal{O}/\mathfrak{m}_\mathcal{O}^r). \]
\end{prop}

\begin{proof}

It is easy to see that $\mathrm{pr}_\kappa$ is equivariant for $T_v$ with $v\notin S$ and $\langle y \rangle^\mathrm{no}$ for $y\in \mathcal{U}_p$. For $v|p$, writing
	\[ U(p^{a,a})\left(\!\begin{array}{cc} \varpi_v & \\ & 1 \end{array} \!\right) U(p^{a,a}) 
		= \bigsqcup_i g_i U(p^{a,a}) \]
we have, by part (1) of \ref{KillHighWeights},
	\[ \mathrm{pr}_\kappa (\varpi_v^{-\bold{w}}\sum_i g_if(xg_i)) 
		=\sum_i \Big(\prod_{v\in \Sigma'} \gamma_v^{-1}\circ \nu_D(g_{i,v})\Big)
		 \mathrm{pr}_\kappa (f(x g_i)) = \sum_i g_i \mathrm{pr}_\kappa(f(xg_i)).
	\]
So, $\mathrm{pr}_\kappa^\ast$ is equivariant for $T_{\varpi_v}$. 

The equivariance of $\mathrm{pr}_\kappa^\ast$ implies that it induces a morphism
	\[ S_{\kappa,\psi}^\mathrm{no}(U(p^{a,a}),\mathcal{O}/\mathfrak{m}_\mathcal{O}^r) \longrightarrow
		S_{2,\psi}^\mathrm{no}(U(p^{a,a}),\mathcal{O}/\mathfrak{m}_\mathcal{O}^r).	\]
Take $f \in S_{2,\psi}(U(p^{a,a}),\calO/\mathfrak{m}_\mathcal{O}^r)$. Write
	\[ U(p^{a,a}) \left(\!\begin{array}{cc} \varpi_p^r & \\ & 1 \end{array}\!\right) U(p^{a,a}) = 
		\bigsqcup_i g_i U(p^{a,a}), \]
and define a function $s_\kappa(f): D^\times \backslash (D\otimes_F\A_F^\infty)^\times \rightarrow W_\kappa(\mathcal{O}/\mathfrak{m}_\mathcal{O}^r)$ by
	\[ s_\kappa(f)(x) = \sum_i \varpi_p^{-r\bold{w}}g_if(xg_i)v_l. \]
We first show that $s(f)$ is independent of the choice of $\{g_i\}$. Indeed, choosing $u_i\in U(p^{a,a})$, we have
	\begin{equation}\label{IndOfgi1}
	 \sum_i \varpi_p^{-r\bold{w}} g_i u_i f(xg_iu_i)v_l 
	= \sum_i \varpi_p^{-r\bold{w}} g_i u_{i,p} f(xg_i)v_l \end{equation}
since $f\in S_{2,\psi}(U(p^{a,a}),\mathcal{O}/\mathfrak{m}_\mathcal{O}^r)$ implies $f(xg_iu_i) = (\prod_{v\in\Sigma'} \gamma_v \circ \nu_D(u_{i,v}))f(xg_i)$. For each $i$, since $u_{i,p}$ is upper-triangular unipotent mod $\varpi_p^a$, we have $u_{i,p}v_l = v_l + w_i$, with $w_i\in\ker(\mathrm{pr}_\kappa)$. Then \eqref{IndOfgi1} becomes
	\[ \sum_i \varpi_p^{-r\bold{w}} g_if(xg_i)v_l 
	+\sum_i \varpi_p^{-r\bold{w}} g_if(xg_i)w_i = 
	\sum_i \varpi_p^{-r\bold{w}} g_if(xg_i)v_l,\]
by part (2) of \ref{KillHighWeights}. So, $s(f)$ is independent of the choice of $\{g_i\}$.

Now we check $s_\kappa(f)\in S_{\kappa,\psi}(U(p^{a,a}),\mathcal{O}/\mathfrak{m}_\mathcal{O}^r)$. Set $F=s_\kappa(f)$. The fact that $F(xz)=\psi(z)F(x)$ for $z\in (\A_F^\infty)^\times$ is immediate. Take $u\in U(p^{a,a})$. Now
	\begin{equation}\label{UInv} 
	(uF)(x) = u \big( \sum_i \varpi_p^{-r\bold{w}} g_i f(xug_i) v_l \big) = 
		\sum_i \varpi_p^{-r\bold{w}} u g_i f(xug_i) v_l.	\end{equation}
For each $i$, we can write $ug_i = g_ju_j$, and \eqref{UInv} becomes
	\[ (uF)(x) = \sum_j \varpi_p^{-r\bold{w}} g_{j}u_{j} f(xg_ju_j)v_l = F(x),\]
since $s_\kappa$ is does not depend on the double coset decomposition.

For $f\in S_{2,\psi}(U(p^{a,a}),\mathcal{O}/\mathfrak{m}_\mathcal{O}^r)$, since $T_p^r = \prod_{v|p}T_{\varpi_v}^r$, part (1) of \ref{KillHighWeights} implies that
	\[ (\mathrm{pr}_\kappa\circ s_\kappa)(f) = T_p^r f.\]
Conversely, given $F\in S_{\kappa,\psi}(U(p^{a,a}),\mathcal{O}/\mathfrak{m}_\mathcal{O}^r)$, and writing $F(x) = f(x)v_l + F'(x)$ with $F'$ a function taking values in $\ker(\mathrm{pr}_\kappa)$, part (2) of \ref{KillHighWeights} implies that
	\[ ((s_\kappa \circ \mathrm{pr}_\kappa)(F))(x) = \sum_i \varpi_p^{-r\bold{w}} g_i
		f(xg_i)v_l =  \sum_i \varpi_p^{-r\bold{w}} g_i(f(xg_i)v_l + F'(xg_i)) 
		= (T_p^rF)(x). \]

If $f\in S_{2,\psi}^\mathrm{no}(U(p^{a,a}),\mathcal{O}/\mathfrak{m}_\mathcal{O}^r)$, then $s_\kappa(f)\in S_{\kappa,\psi}^\mathrm{no}(U(p^{a,a}),\mathcal{O}/\mathfrak{m}_\mathcal{O}^r)$. Indeed, since each space is finite, there is some $n\ge r$ such that $e_H = T_p^n$ on each space. Then
	\[ T_p^n s_\kappa(f) = (s_\kappa\circ\mathrm{pr}_\kappa)(T_p^{n-r} s_\kappa(f)) = 
	s_\kappa(T_p^{n-r}(\mathrm{pr}_\kappa\circ s_\kappa)(f)) = s_\kappa (T_p^n f)
	= s_\kappa(f).\]
Lastly, $\mathrm{pr}_\kappa$ and $s_\kappa$ restrict to morphisms between the nearly ordinary subspaces whose composites are automorphisms; so, $\mathrm{pr}_\kappa$ and $s_\kappa$ are isomorphisms on the nearly ordinary subspaces.		\end{proof}

Since $\mathfrak{m}_\mathcal{O}^{-r}/\mathcal{O}\cong\mathcal{O}/\mathfrak{m}_\mathcal{O}^r$, the following corollary follows from \ref{WeightChangeFin} upon taking direct limits and Pontryagin duals.

\begin{cor}\label{WeightChangeInf}
For any algebraic weight $\kappa$, there is an $\mathcal{O}$-module isomorphism 
	\[ S_{\kappa,\psi}^\mathrm{no}(U(p^\infty),E/\mathcal{O})^\vee\cong
		S_{2,\psi}^\mathrm{no}(U(p^\infty),E/\mathcal{O})^\vee,	\]
equivariant for all $T_v$ with $v\notin S$, $T_{\varpi_v}$ with $v|p$, and $\langle y \rangle^\mathrm{no}$ with $y\in \mathcal{U}_p$.

\end{cor}

\subsubsection{}\label{InfDefs} Henceforth we denote $S_{2,\psi}(U(p^\infty),E/\mathcal{O})^\vee$ by $S_\psi(U)$. By \ref{HeckeFinite}, $S_\psi(U)$ is a finite $\Lambda(\mathcal{U}_p)$-module that is free over $\Lambda(\mathcal{U}_p^a)$ for sufficiently large $a\ge 1$. We let $\bold{T}_\psi(U)$ denote the $\Lambda(\mathcal{U}_p)$-subalgebra of $\End_{\Lambda(\mathcal{U}_p)}(S_\psi(U))$ generated by $T_v$ for all $v\notin S$ and $T_{\varpi_v}$ for $v|p$. Note that $\bold{T}_\psi(U)$ is finite a finite $\Lambda(\mathcal{U}_p)$-algebra, and is reduced. The following corollary is immediate from \ref{WeightChangeInf}.

\begin{cor}\label{WeightChangeHecke}
Let $\kappa = (\bold{k},\bold{w})$ be an algebraic weight such that $U\cap(\A_F^\infty)^\times$ acts on $W_\kappa(\calO)$ via $\psi^{-1}$. We have an $\Lambda(\mathcal{U}_p)$-algebra isomorphism
	\[ \bold{T}_{\kappa,\psi}^\mathrm{no}(U(p^\infty),\mathcal{O})\cong 
		\bold{T}_\psi(U)	\]
identifying the $T_v$ for $v\notin S$, and the $T_{\varpi_v}$ for $v|p$
\end{cor}

Recall that if $\kappa$ is an algebraic weight, then for any $a\ge 1$ we denote by $\mathfrak{p}_\kappa$ the kernel of the $\mathcal{O}$-algebra morphism $\Lambda(\mathcal{U}_p^a)\rightarrow\mathcal{O}$ corresponding to the character $y\mapsto y^{-\bold{w}}$ of $\mathcal{U}_p^a$.

\begin{cor}\label{HorControl}
Let $\kappa=(\bold{k},\bold{w})$ be an algebraic weight such that $U\cap(\A_F^\infty)^\times$ acts on $W_\kappa(\calO)$ via $\psi^{-1}$. Let $a\ge 1$ be such that $(U(p^{a,a})(\A_F^\infty)^\times \cap t^{-1}D^\times t)/F^\times = 1$ for all $t \in (D \otimes_F \A_F^\infty)^\times$. 

The isomorphism $\bold{T}_\psi(U)\cong \bold{T}_{\kappa,\psi}(U(p^\infty),\mathcal{O})$ of \ref{WeightChangeHecke} combined with the natural projection $\bold{T}_{\kappa,\psi}(U(p^\infty),\mathcal{O})\rightarrow \bold{T}_{\kappa,\psi}(U(p^{a,a}),\mathcal{O})$ has kernel equal to the radical of $\mathfrak{p}_\kappa \bold{T}_\psi(U)$.
\end{cor}

\begin{proof} 
Combining \ref{WeightChangeHecke} with \ref{VertControl} yields an action of $\bold{T}_\psi(U)/\mathfrak{p}_\kappa\bold{T}_\psi(U)$ on $S_{\kappa,\psi}^\mathrm{no}(U(p^{a,a}),E/\mathcal{O})^\vee$ that factors through $\bold{T}_{\kappa,\psi}(U(p^{a,a}),\mathcal{O})$. Since $\bold{T}_{\kappa,\psi}(U(p^{a,a}),\mathcal{O})$ is reduced, it suffices now to show that any $\mathfrak{p} \in \Spec\bold{T}_\psi(U)$ containing $\mathfrak{p}_\kappa \bold{T}_\psi(U)$ is in the support of $S_{\kappa,\psi}^\mathrm{no}(U(p^{a,a}),E/\mathcal{O})^\vee$. But this follows from \ref{WeightChangeInf}, \ref{VertControl}, and Nakayama's Lemma, since
	\[ S_{\kappa,\psi}^\mathrm{no}(U(p^{a,a}),E/\mathcal{O})^\vee_\mathfrak{p} = 
	S_\psi(U)_\mathfrak{p}/\mathfrak{p}_\kappa S_\psi(U)_\mathfrak{p}.
	\]
\end{proof}

We say that a prime $\mathfrak{p}$ of $\bold{T}_\psi(U)$ is an \textit{arithmetic prime} if there is some $a\ge 1$ and some algebraic weight $\kappa$ such that $\mathfrak{p} \cap \Lambda(\mathcal{U}_p^a) = \mathfrak{p}_\kappa$.

\begin{cor}\label{ArPrimeZarDense}
	For any irreducible component $C$ of $\Spec \mathbf{T}_\psi(U)$, the set of arithmetic primes in $C$ is Zariski dense.
\end{cor}

\begin{proof}
	It is easy to see that the set of primes $\mathfrak{p} \in \Spec \Lambda(\mathcal{U}_p)$ such that $\mathfrak{p} \cap \Lambda(\mathcal{U}_p^a) = \mathfrak{p}_\kappa$ for some $a$ and some $\kappa$, is Zariski dense in $\Lambda(\mathcal{U}_p)$. The result now follows from the fact that $\Spec \bold{T}_\psi(U)$ is finite over $\Spec \Lambda(\mathcal{U}_p)$ by \ref{HeckeFinite}.
\end{proof}

%% file: HidaFam/ModGalReps.tex
We keep all assumptions and notations of the previous subsubsection. In particular $\kappa$ is an algebraic weight, $U$ is a $(\Sigma'\subseteq\Sigma)$-open subgroup of $(D \otimes_F \A_F)^\times$, and $S$ denotes the finite set of places containing $\Sigma$, all primes at which $U_v$ is not maximal compact, as well as all places above $p$ and $\infty$.

\subsubsection{}\label{SmallGalRepHilbert}

Let $\pi$ be a regular algebraic cuspidal automorphic representation of $\GL_2(\A_F)$. There is an absolutely irreducible representation
	\[ \rho_\pi : G_F \longrightarrow \GL_2(\Qbar_p)
	\]
such that $\rho_\pi|_{G_v}$ and $\pi_v$ satisfy the (suitably normalized) local Langlands correspondence for every place $v$ of $F$. The existence of such a $\rho_\pi$ was shown in \cite{TaylorHilbert} building on \cite{WilesOrdinary} and \cite{CarayolHilbert}, and an alternate construction is given in \cite{BlasiusRogawskiHilbert}. The compatibility with the local Langlands correspondence was shown for places away from $p$ in \cite{CarayolHilbert} and \cite{TaylorHilbert}, and for places above $p$ in \cite{BlasiusRogawskiHilbert}, \cite{SaitoHilbert} and \cite{SkinnerHilbert}.

In the case that $\pi$ is $p$-nearly ordinary, we can say more. Let $(\bold{k},\bold{w})$ denote the weight of $\pi$. Take an open compact subgroup $V \subset \GL_2(\A_F^\infty)$, with $V_v \subseteq \mathrm{Iw}(v)$ for each $v|p$, such that there is $0 \ne x \in \pi^V$ with $T_{\varpi_v}x = \alpha_vx$, where $\alpha_v \in \Qbar$ is a $p$-adic unit under out fixed embedding $\Qbar \rightarrow \Qbar_p$. Then the line generated by $x$ is also stable under $\langle y \rangle$ for every $y \in \mathcal{U}_p$, and the eigenvalues are algebraic. Letting $\chi_v'$ denote the resulting $\overline{\Q}$-valued character, we define a character $\chi_v : F_v^\times \rightarrow \Qbar_p^\times$ by $\chi_v(y) = \chi_v'(y)y^{-\bold{w}}$ and $\chi_v(\varpi_v) = \alpha_v$. It is shown in \cite{WilesOrdinary} and \cite{HidaNOGalRep} that
	\[ 
	\rho_\pi|_{G_v} \cong \left( \begin{array}{cc} \ast & \ast \\ & \chi_v \end{array} \right).
	\]
	
\subsubsection{}\label{SmallGalRepQuat}

Let $f \in S_{\kappa,\psi}^\mathrm{no}(U,\mathcal{O})$, be an eigenfunction for $\mathbf{T}_{\kappa,\psi}^\mathrm{no}(U,\mathcal{O})$, and let
	\[ 
	\lambda_f : \bold{T}_{\kappa,\psi}^\mathrm{no}(U,\calO) \longrightarrow \Qbar_p
	\]
denote the resulting homomorphism. The image of $f$ under the Jacquet-Langlands-Shimizu correspondence, c.f. \ref{JacquetLanglands}, generates an irreducible $\pi_f$. Let $\rho_f = \rho_{\pi_f}$, with $\rho_{\pi_f}$ as in \ref{SmallGalRepHilbert}. The discussion in \ref{SmallGalRepHilbert} applied to
	\[
	\rho_f : G_F \longrightarrow \GL_2(\A_F)
	\]
implies
	\begin{enumerate}
		\item $\rho_f$ is unramified outside $S$ and $\tr \rho_f(\Frob_v) = \lambda_f(T_v)$ for each $v \notin S$;
		\item $\det \rho_f = \psi\epsilon_p$, in particular $\det \rho_f (c) = -1$ for any choice $c$ of complex conjugation;
		\item for every $v \in \Sigma$, $\rho_f|_{G_v} \cong \left( \begin{array}{cc}
			\theta_v\epsilon_p & \ast \\ & \theta_v \end{array} \right)$ with $\theta_v$ an unramified character, and if $v \in \Sigma'$, $\theta_v = \gamma_v$;
		\item for every $v|p$, $\rho_f|_{G_v} \cong \left( \begin{array}{cc} \ast & \ast \\
			&\chi_v \end{array} \right)$, where, via the isomorphism of local class field theory,
			$\chi_v(\varpi_v) = \lambda_f(T_{\varpi_v})$, and $\chi_v(y) =
			 \lambda_f(\langle y \rangle^\mathrm{no})$ for all $y \in \mathcal{O}_{F_v}^\times$.
	\end{enumerate}

\subsubsection{}\label{BigGalRepSec}

We call an ideal $\mathfrak{a}$ of $\bold{T}_{\kappa,\psi}^\mathrm{no}(U,\mathcal{O})$ or $\bold{T}_\psi(U)$ \textit{Eisenstein} if there is an abelian extension $L/F$ such that for all but finitely many finite places $v$ of $F$ that split completely in $L$, we have $T_v-2\in\mathfrak{a}$.

Let $f$ be an eigenform for $\bold{T}_{\kappa,\psi}^{\mathrm{no}}(U,\mathcal{O})$, and let
	\[ 
	\lambda_f : \bold{T}_{\kappa,\psi}^\mathrm{no}(U,\mathcal{O})
		\longrightarrow \overline{\Q}_p,
	\]
denote the corresponding $\mathcal{O}$-algebra morphism. The kernel of $\lambda_f$ is contained in a unique maximal ideal $\mathfrak{m}$. Choosing a $\overline{\Z}_p$-lattice for $\rho_f$ and reducing modulo the maximal ideal of $\overline{\Z}_p$, we obtain a representation
	\[
		\overline{\rho}_\mathfrak{m} : G_{F,S} \longrightarrow \GL_2(\overline{\F}).
	\]
If $\mathfrak{m}$ is non-Eisenstein, $\overline{\rho}_\mathfrak{m}$ is irreducible and (up to isomorphism) $\overline{\rho}_\mathfrak{m}$ does not depend on the choice of $\overline{\Z}_p$-lattice in the representation space of $\rho_f$, nor on the choice of $f$.

Given the non-Eisenstein maximal ideal $\mathfrak{m}$ of $\mathbf{T}_\psi(U)$ and $v|p$, we define a character
	\[ \chi_{v,\mathfrak{m}}^\mathrm{univ} : G_v \longrightarrow 
	\mathbf{T}_\psi(U)_\mathfrak{m}^\times
	\]
by composing the isomorphism of class field theory with the character of $F_v^\times$ that sends $\varpi_v$ to $T_{\varpi_v}$, and on $\mathcal{O}_{F_v}^\times$ is equal to the canonical character
	\[ \mathcal{O}_{F_v}^\times \longrightarrow \Lambda(\mathcal{U}_p)^\times 
	\longrightarrow \mathbf{T}_\psi(U)_\mathfrak{m}^\times.
	\]

\begin{prop}\label{BigGalRep}
Let $\mathfrak{m}$ be a non-Eisenstein maximal ideal of $\bold{T}_\psi(U)$. There exists a continuous representation
	\[ \rho_{U,\frakm}:G_{F,S} \longrightarrow \GL_2(\bold{T}_\psi(U)_\frakm) \]
such that
\begin{enumerate}
	\item for any finite place $v\notin S$, $\tr \rho_{U,\frakm}(\Frob_v) = T_v$.
\end{enumerate}
Moreover, this representation satisfies
\begin{enumerate}
	\item[(2)] $\det \rho_{U,\frakm} = \psi\epsilon_p$,
	\item[(3)] for $v|p$ and $\sigma\in G_v$, $\tr\rho(\sigma) = \psi\epsilon_p(\chi_{v,\mathfrak{m}}^\mathrm{univ})^{-1}(\sigma)+\chi_{v,\mathfrak{m}}^\mathrm{univ}(\sigma)$.
\end{enumerate}
\end{prop}

\begin{proof}

Take an algebraic weight $\kappa$ and $a \ge 1$ such that $U\cap (\A_F^\infty)^\times$ acts on $W_\kappa(\calO)$ by $\psi^{-1}$ and such that for any $t \in (D \otimes_F \A_F^\infty)^\times$, $(U(p^{a,a})(\A_F^\infty)^\times \cap t^{-1}D^\times t)/F^\times = 1$. Then $\mathfrak{m}$ is the pullback of a non-Eisenstein maximal ideal of $\bold{T}_{\kappa,\psi}^\mathrm{no}(U(p^{a,a}),\calO)$. Let
	\[ 
	\lambda : \bold{T}_{\kappa,\psi}^\mathrm{no}(U(p^{a,a}),\calO)_\frakm \longrightarrow \Qbar_p 
	\]
by an $\mathcal{O}$-algebra morphism corresponding to an eigenform $f$, and let $\rho_f$ denote the corresponding representation as in \ref{SmallGalRepQuat}. Since $\tr \rho_\lambda(\Frob_v) = \lambda(T_v)$, for every $v\notin S$, the injection
	\[  \bold{T}_{\kappa,\psi}^\mathrm{no}(U(p^{a,a}),\calO)_\frakm \longrightarrow 
		\prod_{f} \Qbar_p,
	\]
c.f. \ref{NearOrdHilbert}, implies there is a is a pseudo-representation
	\[ r_a : G_{F,S} \longrightarrow \bold{T}_{\kappa,\psi}(U(p^{a,a}),\calO)_\frakm \]
with $r_a(\Frob_v) = T_v$ for every $v \notin S$. We also get a pseudo-representation
	\[
		r_b : G_{F,S} \longrightarrow \bold{T}_{\kappa,\psi}(U(p^{b,b}),\mathcal{O})_\mathfrak{m}
	\]
for every $b \ge a$, such that
	\[ 
	\xymatrix{ G_{F,S} \ar[r]^-{r_b} \ar[rd]_-{r_a} & \bold{T}_{\kappa,\psi}^\mathrm{no}
		(U(p^{b,b}),\calO)_\frakm \ar[d] \\
		&  \bold{T}_{\kappa,\psi}^\mathrm{no}(U(p^{a,a}),\calO)_\frakm } 
	\]
commutes. We then get a pseudo representation
	\[ 
	r = \varprojlim_{b\ge a} r_b : G_{F,S} \longrightarrow \bold{T}_{\psi}(U)_\frakm,
	\]
such that $r(\Frob_v) = T_v$ for any $v\notin S$. Since $r$ modulo $\frakm$ is the trace of an absolutely irreducible representation, namely $\rhobar_\mathfrak{m}$, a theorem of Nyssen and Rouqier \cite{NyssenPseudo}, \cite{RouquierPseudo}, implies that $r$ is the trace of a representation
	\[ \rho_{U,\frakm} : G_{F,S} \longrightarrow
	 \GL_2(\bold{T}_\psi(U)_\frakm), \]
and a theorem of Carayol \cite{CarayolAnneauLocal} implies this representation is unique. To see (2) and (3), note that the specialization of $\rho_{U,\frakm}$ at any arithmetic prime satisfies the corresponding properties, hence so does $\rho_{U,\frakm}$ by Zariski density of arithmetic primes, cf. \ref{ArPrimeZarDense}, and reducedness of $\mathbf{T}_\psi(U)_\mathfrak{m}$.
\end{proof}

Recall that $\bold{T}_\psi(U)$ is generated over $\Lambda(\mathcal{U}_p)$ be the operators $T_v$ for all $v\notin S$, and $T_{\varpi_v}$ for $v|p$, cf. \ref{InfDefs}. The following corollary shows that, after localizing $\bold{T}_\psi(U)$ at $\mathfrak{m}$, it suffices to use finitely many of the $T_v$ with $v$ outside of any finite set of places $S'\supseteq S$. It will be used in \S \ref{LocRequalsT}, cf. \ref{thesetS}.

\begin{cor}\label{Heckefingencor}

Let $S'$ be any finite set of places of $F$ containing $S$. Then there exist finite places $v_1,\ldots,v_k \notin S'$ such that $\mathbf{T}_\psi(U)_\mathfrak{m} = \Lambda(\mathcal{U}_p^1)[T_{v_1},\ldots,T_{v_k}][T_{\varpi_v}]_{v|p}$.

\end{cor}

\begin{proof}

Let $\mu$ denote the prime to $p$-torsion subgroup of $\mathcal{U}_p$. By definition, $\mathbf{T}_\psi(U)$ is generated over $\Lambda(\mathcal{U}_p^1)$ by the the operators $T_v$ for $v \notin S$, $T_{\varpi_v}$ for $v|p$, as well as $\langle y \rangle$ for $y\in \mu$. As we have assumed $E$ contains all embeddings $F_v \rightarrow \Qbar_p$, the projection $\mathbf{T}_\psi(U) \rightarrow \mathbf{T}_\psi(U)_\mathfrak{m}$ sends each $\langle y \rangle$ with $y\in \mu$, to elements of $\mathcal{O}$. The corollary then follow from \ref{HeckeFinite}, \ref{BigGalRep}, and Chebotarev density.
	\end{proof}

\subsubsection{}\label{BigGalRepLocal}

Let $\mathfrak{m}$ be a non-Eisenstein maximal ideal of $\bold{T}_\psi(U)$, and denote by 
	\[
		\overline{\rho}_\mathfrak{m} : G_{F,S} \longrightarrow \GL_2(\overline{\F})
	\]
the corresponding $\overline{F}$ representation. For $v|p$, let $G_v^\mathrm{ab}(p)$ be the maximal abelian pro-$p$ quotient of $G_v$. Let $\Lambda(G_v) = \mathcal{O}[[G_v^\mathrm{ab}(p)]]$, and let $\Lambda(G_p) = \hat{\otimes}_{v|p} \Lambda(G_v)$. Via local class field theory, we identify $\mathcal{U}_p^1$ with $\prod_{v|p} I_v^\mathrm{ab}(p)$, where $I_v^\mathrm{ab}(p)$ is the inertia subgroup of $G_v^\mathrm{ab}(2)$. Then $\Lambda(\mathcal{U}_p^1)$ is a subalgebra of $\Lambda(G_p)$, and $\Lambda(G_p)$ is isomorphic to a power series over $\Lambda(\mathcal{U}_p^1)$ in $|\{v|p\}|$ variables. Recall that for each $v|p$, we have a $\mathbf{T}_\psi(U)_\mathfrak{m}$-valued character $\chi_{v,\mathfrak{m}}^\mathrm{univ}$ of $G_v$ that sends $\varpi_v$ to $T_{\varpi_v}$. Hence, the $\Lambda(\mathcal{U}_p^1)$-algebra structure on $\mathbf{T}_\psi(U)_\mathfrak{m}$ extends to a $\Lambda(G_p)$-algebra structure. Fix an $\mathcal{O}$-valued character $\eta = (\eta_v)_{v|p}$ of the torsion subgroup of $\mathcal{U}_p^1$ (equivalently a character of the torsion subgroup of $\prod_{v|p}G_v^\mathrm{ab}(p)$). The character $\eta$ determines minimal primes of $\Lambda(\mathcal{U}_p^1)$ and $\Lambda(G_p)$, each of which denote by $\mathfrak{q}_\eta$. We set $\Lambda(\mathcal{U}_p^1,\eta) = \Lambda(\mathcal{U}_p^1)/\mathfrak{q}_\eta$, $\Lambda(G_p,\eta) = \Lambda(G_p)/\mathfrak{q}_\eta$ and $\mathbf{T}_\psi(U,\eta)_\mathfrak{m} = \mathbf{T}_\psi(U)_\mathfrak{m}/\mathfrak{q}_\eta$.

After enlarging $\mathcal{O}$, if necessary, we may assume that all eigenvalues of $\overline{\rho}_\mathfrak{m}$ are defined over $\F$. Let $R_{F,S}$ denote the universal deformation ring for $G_{F,S}$-deformations of $\overline{\rho}_\mathfrak{m}$, as in \ref{GlobalDefs}. By \ref{BigGalRep}, there is a local $\calO$-algebra morphism $R_{F,S} \rightarrow \bold{T}_\psi(U)_\mathfrak{m}$. We then get a local $\Lambda(G_p)$-algebra morphism $R_{F,S}\hat{\otimes}\Lambda(G_p) \rightarrow \mathbf{T}_\psi(U)_\mathfrak{m}$, which is surjective by \ref{Heckefingencor}. Let $\overline{R}_{F,S}^\psi$ denote the quotient of $R_{F,S}\hat{\otimes}\Lambda(G_p)$ defined in \ref{globalquotient} (with $S_\mathrm{ur} = \emptyset$). Recall that for a finite extension $E'/E$ with ring of integers $\mathcal{O}'$, a local $\calO$-algebra morphism $R_{F,S} \hat{\otimes}\Lambda(G_p) \rightarrow \mathcal{O}'$ factors through $\overline{R}_{F,S}^\psi$ if and only if the corresponding deformation $V_{\mathcal{O}'}$ and characters $(\chi_v)_{v|p}$ satisfy the following
\begin{itemize}
	\item[-] $\det V_{\mathcal{O}'} = \psi\epsilon_p$;
	\item[-] for each $v|p$, there is a $G_v$-stable line $L$ in $V_{\mathcal{O}'}$ such that $G_v$ acts on $V_{\mathcal{O}'}/L$ via $\chi_v$;
	\item[-] for each $v|p$, the restriction of $\chi_v$ to the torsion subgroup of $G_v^\mathrm{ab}(p)$ is equal to $\eta_v$;
	\item[-] for each $v \in \Sigma$, $V_{\mathcal{O}'}|_{G_v}$ is an extension of $\gamma_v$ by $\gamma_v\epsilon_p$.
\end{itemize}
		
\begin{prop}\label{BigGalLoc}
	Let $\mathfrak{m}$ and $\overline{R}_{F,S}^\psi$ be as above and assume $\Sigma' = \Sigma$. The $\Lambda(G_p)$-algebra morphism $R_{F,S}\hat{\otimes}\Lambda(G_p)\rightarrow\bold{T}_\psi(U,\eta)_\frakm$ factors through $\overline{R}_{F,S}^\psi$.
\end{prop}

\begin{proof}

Let $\mathfrak{p}$ be an arithmetic prime of $\bold{T}_\psi(U,\eta)_\mathfrak{m}$. The pushforward of the representation in \ref{BigGalRep} to $\bold{T}_\psi(U,\eta)_\mathfrak{m}/\mathfrak{p}$ is an integral model for some $\rho_f$ as in \ref{SmallGalRepQuat}. By \ref{SmallGalRepQuat}, the map
	\[
	R_{F,S}\hat{\otimes}\Lambda(G_p) \longrightarrow \bold{T}_\psi(U,\eta)_\mathfrak{m} 
	\longrightarrow \bold{T}_\psi(U,\eta)_\mathfrak{m}/\mathfrak{p}
	\]
factors through $\overline{R}_{F,S}^\psi$. The result now follows from the Zariski density of arithmetic primes, c.f. \ref{ArPrimeZarDense}, and the fact that $\bold{T}_\psi(U,\eta)_\mathfrak{m}$ is reduced.
\end{proof}

%% file: HidaFam/ModAuxPrimes.tex
In the patching argument of \S \ref{LocRequalsT} it is important to deepen the level $U$ at certain auxiliary primes. In this subsection we establish a number of lemmas regarding the relationship betweent our spaces of modular forms at these deeper levels and at the original level.

We keep the notations and assumptions of the previous subsections. In particular, $D$ is a totally definite quaternion algebra with centre $F$, $U$ is a $(\Sigma'\subseteq\Sigma)$-open subgroup of $(D \otimes_F \A_F^\infty)^\times$, $\psi : F^\times \backslash (\A_F^\infty)^\times \rightarrow \mathcal{O}^\times$ is a continuous character such that $\psi(z) = z_p^{2-\bold{k}-2\bold{w}}$ on $U\cap (\A_F^\infty)^\times$ for some algebraic weight $(\bold{k},\bold{w})$, and $S$ denotes the finite set of places at which either $D$ is ramified, $U_v \ne \GL_2(\mathcal{O}_{F_v})$, $v|p$, or $v|\infty$.

\subsubsection{}\label{DegenMap}

Fix a finite place $w \notin S$, and let $U'$ be the open subgroup of $U$ such that $U'_v = U_v$ if $v\ne w$ and $U_w = \mathrm{Iw}(w)$. Given $a\ge 0 $, we define a map
	\begin{align*}
		\xi^a : S_{2,\psi}(U(p^{a,a}),E/\calO)^2 &\longrightarrow 
			S_{2,\psi}(U'(p^{a,a}),E/\calO) \\
			(f,g) &\longmapsto f + \left(\begin{array}{cc} 1 & \\ & \varpi_w \end{array}\right)g
	\end{align*}
This map is equivariant for all Hecke operators outside $w$, and induces a map on the nearly ordinary subspaces, which we also denote by $\xi^a$. We then set
	\[ \xi = \varinjlim_{a\ge 1} \xi^a : 
		S_{2,\psi}^{\mathrm{no}}(U(p^\infty),E/\mathcal{O})^2 \longrightarrow 
		S_{2,\psi}^{\mathrm{no}}(U'(p^\infty),E/\mathcal{O})	\]
and $\xi^\vee : S_\psi(U')\rightarrow S_\psi(U)^2$ is its Pontryagin dual. These are both maps of $\Lambda(\mathcal{U}_p)$-modules and respect the action of $T_v$ for $v \notin S\cup \{w\}$, and $T_{\varpi_v}$ for $v|p$.

\begin{lem}\label{DegenAtNonEis}

Let $\frakm$ be a non-Eisenstein maximal ideal of $\bold{T}_\psi(U)$ and denote its pullback to $\bold{T}_\psi(U')$ also by $\frakm$. The localization of $\xi^\vee$ at $\frakm$, $S_\psi(U')_\frakm\rightarrow S_\psi(U)^2_\frakm$, is surjective.

\end{lem}

\begin{proof}

It suffices to show that
	\[
		\xi : S_{2,\psi}^\mathrm{no}(U(p^\infty),E/\mathcal{O})_\mathfrak{m}^2 
		\longrightarrow
		S_{2,\psi}^\mathrm{no}(U'(p^\infty),E/\mathcal{O})_\mathfrak{m}
	\]
is injective. For this it suffices to show that for any $a \ge r \ge 1$, that
	\[
	\xi^a : S_{2,\psi}^\mathrm{no}(U(p^{a,a}),\mathfrak{m}_\mathcal{O}^{-r}/\mathcal{O})_\mathfrak{m}^2
	\longrightarrow
	S_{2,\psi}^\mathrm{no}(U'(p^{a,a}),\mathfrak{m}_\mathcal{O}^{-r}/\mathcal{O})_\mathfrak{m}
	\]
is injective. If $(f,g)$ belongs to the kernel of
	\[
	\xi^a : S_{2,\psi}(U(p^{a,a}),\mathfrak{m}_\mathcal{O}^{-r}/\mathcal{O})^2
	\longrightarrow
	S_{2,\psi}(U'(p^{a,a}),\mathfrak{m}_\mathcal{O}^{-r}/\mathcal{O}),
	\]
then $f$ is invariant under $U(p^{a,a})\SL_2(F_w)$. Letting $(D \otimes_F \A_F^\infty)^1$ denote the subgroup of elements of reduced norm $1$, strong approximation implies that $f$ is invariant under $(D \otimes_F \A_F^\infty)^1$, hence $f$ factors through the reduced norm and is not in the support of $\mathfrak{m}$, since $\mathfrak{m}$ is non-Eisenstein. \end{proof}

For the remainder of this subsection fix $\mathfrak{p}\in\Spec\bold{T}_\psi(U)$ contained in a non-Eisenstein maximal ideal $\mathfrak{m}$, and denote by $\rho_\mathfrak{p}$ the $G_{F,S}$-representation into $\GL_2(\bold{T}_\psi(U)_\mathfrak{m}/\mathfrak{p})$ induced from \ref{BigGalRep}.  Note that this implies $\rho_\frakp$ is unramified at $w$. Denote again by $\frakp$ and $\frakm$ the pullbacks of $\frakp$ and $\frakm$ to $\bold{T}_\psi(U')$.
	
\begin{lem}\label{NoSteinberg}
Let $\sigma_w\in G_w$ be some lift of $\Frob_w$. Let
	\[ \rho_{U',\frakm}:G_F\longrightarrow\GL_2(\bold{T}_\psi(U')_\mathfrak{m}) \]
be as in \ref{BigGalRep}, and let $y = (\tr\rho_{U',\frakm}(\sigma_w))^2-\psi(\varpi_w)(1+\mathrm{Nm}(w))^2$.

We have $y(\ker \xi^\vee) = 0$. Moreover, if $p\in\frakp$, $\mathrm{Nm}(v) \equiv 1 \pmod p$, and $\rho_\mathfrak{p}(\Frob_w)$ has distinct eigenvalues, then $y \notin \mathfrak{p}$.
\end{lem}

\begin{proof}

Take $a\ge 1$ and an algebraic weight $\kappa$ such that 
\begin{itemize}
	\item[-] $(U(p^{a,a})(\A_F^\infty)^\times \cap t^{-1} D^\times t)/F^\times = 1$ for all $t \in (D\otimes_F \A_F^\infty)^\times$;
	\item[-] the action of $U(p^{a,a})\cap(\A_F^\infty)^\times$ on $W_\kappa(\calO)$ is given by $\psi^{-1}$;
	\item[-] $\frakm$ is the pullback of a maximal ideal of $\bold{T}_{\kappa,\psi}(U(p^{a,a}),\calO)$ under the projection of \ref{HorControl}.
\end{itemize}
For $A$ an $\calO$-module, let
	\begin{align*}
		\xi_{\kappa,A}^a:S_{\kappa,\psi}^\mathrm{no}(U(p^{a,a}),A)_\frakm^2&\longrightarrow
		 S_{\kappa,\psi}^\mathrm{no}(U'(p^{a,a}),A)_\frakm\\
		 (f,g) & \longmapsto f + \left(\begin{array}{cc} 1 & \\ & \varpi_w \end{array}\right)g.
	\end{align*}
We know that $\coker(\xi_{\kappa,\Qbar_p}^a)$ has a basis $\{f_i\}$ consisting of eigenforms which are new at $v$. Letting $\rho_{f_i}$ denote the Galois representation associated to $f_i$, local-global compatibility, cf. \cite{CarayolHilbert}, implies that
	\[ \rho_{f_i}|_{G_w} \cong \left( \begin{array}{cc} 
		\epsilon_p\chi_i & \ast \\ & \chi_i \end{array} \right)	\]
where $\chi_i$ is an unramified character of $G_w$ such that $\chi_i^2 = \psi|_{G_w}$. The eigenform $f_i$ defines an $\calO$-algebra morphism $\bold{T}_{\kappa,\psi}^\mathrm{no}(U'(p^{a.a}),\calO)_\frakm \rightarrow \Qbar_p$ such that, precomposing with the projection $\bold{T}_\psi(U')\rightarrow\bold{T}_{\kappa,\psi}^\mathrm{no}(U'(p^{a,a}),\calO)_\frakm$,
	\[ \xymatrix{
		G_F \ar[r]^-{\rho_{U',\frakm}} \ar[dr]_{\rho_{f_i}} & \GL_2(\bold{T}_\psi(U')_\frakm) \ar[d]
			\\	& \GL_2(\Qbar_p) }	\]
commutes. In particular, the image of $(\tr\rho_{U',\frakm}(\sigma_w))^2$ under this map is $\psi(\varpi_w)(1+\mathrm{Nm}(w))^2$, and $y(\coker(\xi_{\kappa,\Qbar_p}^a)) = 0$. Hence, $y(\coker(\xi_{\kappa,E}^a)) = 0$.

If $V$ is any $(\Sigma'\subseteq\Sigma)$-open subgroup with $(V(\A_F^\infty)^\times \cap t^{-1} D^\times t)/F^\times = 1$ for all $t \in (D\otimes_F \A_F^\infty)^\times$, we have $S_{\kappa,\psi}(V,E) \cong S_{\kappa,\psi}(V,\calO)\otimes_\calO E$ and $S_{\kappa,\psi}(V,E/\calO) \cong S_{\kappa,\psi}(V,\calO)\otimes_\calO E/\calO$, and so there is a natrual surjection $S_{\kappa,\psi}(V,E)\rightarrow S_{\kappa,\psi}(V,E/\calO)$. This yields a Hecke-equivariant commutative diagram
	\[ \xymatrix{
		S_{\kappa,\psi}^\mathrm{no}(U(p^{a,a}),E)_\frakm^2 \ar[r]^{\xi_{\kappa,E}^a} \ar[d]
		& S_{\kappa,\psi}^\mathrm{no}(U'(p^{a,a}),E)_\frakm \ar[r] \ar[d] & 
		\coker(\xi_{\kappa,E}^a) \ar[r] \ar[d] & 0 \\
		S_{\kappa,\psi}^\mathrm{no}(U(p^{a,a}),E/\calO)_\frakm^2 
		\ar[r]^{\xi_{\kappa,E/\calO}^a}
		& S_{\kappa,\psi}^\mathrm{no}(U'(p^{a,a}),E/\calO)_\frakm \ar[r] & 
		\coker(\xi_{\kappa,E/\calO}^a) \ar[r]
		& 0 }	\]
with exact rows. Since the first two vertical maps are surjections, so is the third and we deduce that $y(\coker(\xi_{\kappa,E/\calO}^a))=0$.

Then, $y(\coker(\varinjlim_{a}\xi_{\kappa,E/\calO}^a))=0$. Noting that $\xi_{\kappa,E/\calO}^a = \varinjlim_{r\ge 1} \xi_{\kappa,\mathfrak{m}_\calO^{-r}/\calO}^a$ and using the ismorphism of \ref{WeightChangeFin}, $\xi = \varinjlim_{a} \xi_{\kappa,E/\calO}^a$. In particular, $y(\coker(\xi))=0$. By exactness of Pontryagin duality, $y(\ker(\xi^\vee))=0$.

Assume $p\in\frakp$, $\Nm(w) \equiv 1 \pmod p$, and $y\in\frakp$. Then, since $\rho_{U,\mathfrak{m}}$ is unramified at $w$ and $\tr\rho_{U,\mathfrak{m}}(\sigma) = T_w$, our assumptions imply $T_w^2- 4\psi(\varpi_w)\in\frakp$. The characteristic polynomial of $\rho_\mathfrak{p}(\Frob_w)$ is $X^2 - T_wX+\psi(\Frob_w)\Nm(v) = X^2-T_wX+\psi(\varpi_w)$, which does not have distinct roots if $T_w^2  = 4\psi(\varpi_w)$ modulo $\mathfrak{p}$.  \end{proof}

\subsubsection{}\label{AuxPrimesModular}

Let $Q$ be a finite set of primes of $F$ disjoint from $S$ such that $\mathrm{Nm}(w) \equiv 1 \pmod p$ for each $w\in Q$. For each $w\in Q$, let $k_w$ denote the residue field of $F_w$ and let $\Delta_w$ be the maximal $p$-power quotient of $k_w^\times$. Set $\Delta_Q = \prod_{w\in Q}\Delta_w$. Define an open subgroup $U'$ of $U$ by $U'_v = U_v$ if $v \notin Q$, and $U_w' = \mathrm{Iw}(w)$ for $w \in Q$. We then define an open subgroup $U_Q$ of $U'$ by
	\[	U_Q = \left\{ \left( \begin{array}{cc} a & b \\c & d \end{array} \right) \in U': 
			a_wd_w^{-1} \mapsto 1 \text{ in } \Delta_w \text{ for each } w\in Q \right\}	
	\]

\begin{lem}\label{AuxPrimeControl}

Let $Q$ be as above. Let $a\ge 1$ be such that $(U(p^{a,a})(\A_F^\infty)^\times \cap t^{-1} D^\times t)/F^\times = 1$ for all $t \in (D\otimes_F \A_F^\infty)^\times$, and such that there is an algebraic weight $\kappa$ with $U(p^{a,a})\cap(\A_F^\infty)^\times$ acting on $W_\kappa(\calO)$ via $\psi^{-1}$.

$S_\psi(U_Q)$ is a free $\Lambda(\mathcal{U}_p^a)[\Delta_Q]$-module and the natural surjection $S_\psi(U_Q)\rightarrow S_\psi(U')$ has kernel $\mathfrak{a}_Q S_\psi(U_Q)$, where $\mathfrak{a}_Q$ is the $\Delta_Q$-augmentation ideal of $\Lambda(\mathcal{U}_p^a)[\Delta_Q]$. In particular the $\Lambda(\mathcal{U}_p^a)[\Delta_Q]$-rank of $S_\psi(U_Q)$ is eqaul to the $\Lambda(\mathcal{U}_p^a)$-rank of $S_\psi(U')$.

\end{lem}

\begin{proof}

Let $\kappa$ and $a\ge 1$ be as in the statement of the lemma. Take $b\ge a$. Applying \ref{FreePontDual} to the groups $U'(p^{a,a})\subseteq U'(p^{b,b})$, $U'(p^{a,a})\subset U_Q(p^{b,b})$ and $U'(p^{b,b})\subset U_Q(p^{b,b})$, we deduce that $S_{\kappa,\psi}(U'(p^{b,b}),E/\calO)^\vee$ and $S_{\kappa,\psi}(U_Q(p^{b,b}),E/\calO)^\vee$ are free over $\calO[\mathcal{U}_a/\mathcal{U}_b]$ and $\calO[\mathcal{U}_a/\mathcal{U}_b][\Delta_Q]$, respectively and that the natural surjection
	\[ S_{\kappa,\psi}(U_Q(p^{b,b}),E/\calO)^\vee \longrightarrow
		S_{\kappa,\psi}(U'(p^{b,b}),E/\calO)^\vee	\]
induces an isomorphism of $S_{\kappa,\psi}(U'(p^{b,b}),E/\calO)^\vee$ with the $\Delta_Q$ coinvariants of $S_{\kappa,\psi}(U_Q(p^{b,b}),E/\calO)^\vee$. Applying Hida's idempotent and passing to the limit over $b\ge a$ gives the result.	\end{proof}

\begin{lem}\label{AuxPrimeTrace}
Let $w\in Q$, and let $\sigma_w$ be a generator of the $p$-part of the tame inertia subgroup of $I_w$. Note that under $I_w \rightarrow \mathcal{O}_{F_w}^\times \rightarrow k_w^\times \rightarrow \Delta_w$, given by class field theory, $\sigma_w$ is mapped to a generator $\delta_w$ of $\Delta_w$. 

If $\mathfrak{m}$ is a non-Eisenstein maximal ideal of $\mathbf{T}_\psi(U_Q)$, and $\rho_{U_Q,\mathfrak{m}}$ denotes the representation in \ref{BigGalRep}, then $\tr \rho_{U_Q,\mathfrak{m}}(\sigma_w) = \delta_w + \delta_w^{-1}$.
\end{lem}

\begin{proof}

Let $\lambda : \mathbf{T}_\psi(U_Q)_\mathfrak{m} \rightarrow \Qbar_p$ denote an arithmetic point. By the definition of $(U_Q)_v$, the automorphic representation associated to $\lambda$ via Jacquet-Langlands is not cuspidal at $w$. Local global compatibility then shows that $\tr\rho_\lambda(\sigma_w) = \lambda(\delta_w + \delta_w^{-1})$. The result now follows from Zariski density of arithmetic points and the fact that $\mathbf{T}_\psi(U_Q)$ is reduced.	\end{proof}

Finally, we will need a lemma describing certain twists of $S_\psi(U)$ by characters of order $2$, as in \cite{KW2}*{\S 7.5}. Let $F_Q^S$ be the maximal $p$-power order abelian extension of $F$ that is unramified outside $Q$ and split at all primes in $S$. Let $G_Q^\ast(\mathcal{O})$ be the set of characters $\Gal(F_Q^S/F) \rightarrow \mathcal{O}^\times$ that reduce to the trivial character modulo $\mathfrak{m}_\mathcal{O}$. Since $S$ contains all infinite places and $F_Q^S$ is split at all places in $S$, we can view any $\chi \in G_Q^\ast(\mathcal{O})$ as a character of $(\mathbb{A}_F^\infty)^\times$. The following lemma is a slight variant of \cite{KW2}*{Proposition 7.6}.

\begin{lem}\label{modulartwists}
Assume $p=2$ and let $G_{Q,2}^\ast(\mathcal{O})$ be the $2$-torsion of $G_Q^\ast(\mathcal{O})$. There is an action $\phi \mapsto \phi_\chi$ of $G_{Q,2}^\ast(\mathcal{O})$ on $S_\psi(U_Q)$ such that 
\begin{itemize}
	\item[-] for any $v \notin S\cup Q$, $T_v \phi_\chi = \chi(\varpi_v)(T_v \phi)_\chi$,
	\item[-] for any $v|2$, $T_{\varpi_v} \phi_\chi = \chi(\varpi_v)(T_{\varpi_v} \phi)_\chi$,
	\item[-] for any $y \in \mathcal{U}_p$, $\langle y \rangle^\mathrm{no} \phi_\chi = \chi(y)(\langle y \rangle^\mathrm{no} \phi)_\chi$.
\end{itemize}
\end{lem}

\begin{proof} Take $a\ge r\ge 1$. For $\chi \in G_{Q,2}^\ast(\mathcal{O})$ and $f \in S_{2,\psi}(U_Q(p^{a,a}),\mathcal{O}/\mathfrak{m}_\mathcal{O}^r)$, we define
	\[ f_\chi : D^\times \backslash (D \otimes_F \A_F^\infty)^\times / U_Q(p^{a,a}) \longrightarrow \mathcal{O}/\mathfrak{m}_\mathcal{O}^r \]
by $f_\chi(g) = f(g)\chi(\nu_D(g))$, where $\nu_D$ is the reduced norm of $D$. Note that for any $z\in (\A_F^\infty)^\times$, $\chi(\nu_D(z)) = \chi(z)^2 = 1$, since $\chi$ has order $2$. Since $F_S^Q$ is unramified outside $Q$ and split at all places in $S$, for any $u \in U_Q(p^{a,a})$,
	\[ \chi(\nu_D(g)) = \prod_{w\in Q}\chi(\nu_D(u_w)).
	\]
For each $w \in Q$ and $u_w \in (U_Q(p^{a,a}))_w$ , the definition of $(U_Q)_w$ implies that the image of $\nu_D(u_w) = \det(u_w)$ in $k_w^\times$ is a square. Since $\chi$ has order two and the places in $Q$ have odd residual characteristic, we get
	\[ \chi(\nu_D(g)) = \prod_{w\in Q}\chi(\det(u_w)) = 1.
	\]
So, $f_\chi \in S_{2,\psi}(U(p^{a,a}),\mathcal{O}/\mathfrak{m}_\mathcal{O}^r)$. 

The definition of the operators $T_v$, $T_{\varpi_v}$, and $\langle y \rangle^\mathrm{no}$ as (normalized) double coset operators together with the fact that for any $h \in U_Q(p^{a,a})gU_Q(p^{a,a})$, we have $\chi(\nu_D(h)) = \chi(\nu_D(g))$, imply
\begin{enumerate}
	\item[(i)] for any $v \notin S\cup Q$, $T_v f_\chi = \chi(\varpi_v)(T_v f)_\chi$,
	\item[(ii)] for any $v|2$, $T_{\varpi_v} f_\chi = \chi(\varpi_v)(T_{\varpi_v} f)_\chi$,
	\item[(iii)] for any $y \in \mathcal{U}_p$, $\langle y \rangle^\mathrm{no} f_\chi = \chi(y)(\langle y \rangle^\mathrm{no} f)_\chi$.
\end{enumerate}
Since any $\chi \in G_{Q,2}^\ast(\mathcal{O})$ has order at most 2, (ii) gives $T_p^2 f_\chi = (T_p^2 f)_\chi$. So, the $G_{Q,2}^\ast(\mathcal{O})$-action commutes with Hida's idempotent and we have an induced action on 
	\[S_{2,\psi}^\mathrm{no}(U_Q(p^{a,a}),\mathcal{O}/\mathfrak{m}_\mathcal{O}^r)
		\cong S_{2,\psi}^\mathrm{no}(U_Q(p^{a,a}),\mathfrak{m}_\mathcal{O}^{-r}/\mathcal{O}),
	\]
and on
	\[ S_{2,\psi}^\mathrm{no}(U_Q(p^\infty),E/\mathcal{O}) = \varinjlim_r\varinjlim_a
	S_{2,\psi}^\mathrm{no}(U_Q(p^{a,a}),\mathfrak{m}_\mathcal{O}^{-r}/\mathcal{O}).
	\]
Letting $G_{Q,2}^\ast(\mathcal{O})$ act on
	\[ S_\psi(U_Q) = S_{2,\psi}^\mathrm{no}(U_Q(p^\infty),E/\mathcal{O})^\vee
	\]
by $\phi \mapsto \phi_\chi$, where $\phi_\chi$ is the function $\phi_\chi(f) = \phi(f_\chi)$ gives the result.	\end{proof}

%% file: AuxPrimes/AuxPrimesIntro.tex
\section{Galois Cohomology and Auxiliary Primes}\label{AuxPrimesSec}

Crucial to the patching method is the existence of so called Taylor-Wiles primes or auxiliary primes. The proof of their existence is the main result of this subsection.

The first subsection uses some of the lemmas proved in \S \ref{DihedralDefs} together with a result of Pink to prove that certain non-dihedral deformations to characteristic $p$ local fields have open image (up to finite index subfields). Using this, we then (mostly) compute the cohomology of the image acting on the adjoint representation. The results in the first subsection allows us to do this by explicit cocycle computation.

In the second subsection, we use the result from the previous one to show the existence of auxiliary primes analogous to those in \cite{KW2}*{Lemma 5.10}. As in \cite{SWreducible}*{\S 6}, some care has to be taken. In particular, it is not sufficient to compute the cohomology with coefficients in our local field, we must do the computations integrally, and we must make sure that the size of the torsion subgroups do not depend on the auxiliary primes chosen. This is due to the fact that when performing the patching in \S \ref{LocRequalsT}, we must consider finite quotients of our universal deformation ring and Hecke modules. In order to ensure that the limits of the resulting projective systems have the correct rank, we must ensure that the alluded to torsion subgroups do not grow. It is because of this that we must be careful to control ensure all error terms stay bounded in this subsection.

We recall and introduce some notation and assumptions that will be used throughout this section. $F\subset \Qbar$ is a totally real number field and $G_F=\Gal(\Qbar/F)$. For any extension $M/F$ inside $\Qbar$, let $G_M=\Gal(\Qbar/M)$. Let $K$ be a characteristic $2$ local field with ring of integers $A$ and residue field $\F$. Let $q$ denote the cardinality of $\F$. Let $\mathfrak{m}$ denote the maximal ideal of $A$ and let $\varpi$ be a fixed choice of uniformizer. Fix an algebraic closure $\overline{K}$ of $K$.

We fix a continuous $\rho:G_F\rightarrow\GL_2(A)$ satisfying:
\begin{enumerate}\item[A1] $\rho$ unramified outside a finite set of places $S$;
\item[A2] $\rho$ is not dihedral and $\im \rho \rightarrow \GL_2(\F)$ has nontrivial kernel;
\item[A3] $\det\rho$ is finite order;
\item[A4] the image of $\rho$ contains a non-trivial unipotent element.
\item[A5] if $\rhobar$ is $L$-dihedral, there is some $\tau_0\in G_F\smallsetminus G_L$ such that $\rho(\tau_0)$ has distinct infinite order $A$-rational eigenvalues.
\end{enumerate}

Let $V$ denote the free rank two $A$ module on which $G$ acts via $\rho$. Let $\Ad$ denote the space of endomorphisms of $V$ with the adjoint action of $G_F$, and $Z$ its centre. For any $A$-algebra $R$ (in particular $K$, $\overline{K}$, $\F$) we set $V_R=V\otimes_A R$, $\Ad_R=\Ad\otimes_A R$, and $Z_R=Z\otimes_A R$. For $m\ge 1$ we also write, for notational convenience, $\Ad_m$ and $Z_m$ for $\Ad_{A/\mathfrak{m}^m}=\Ad/\mathfrak{m}^m\Ad$ and $Z_{A/\mathfrak{m}^m}=Z/\mathfrak{m}^mZ$, respectively.

\input{AuxPrimes/TheImage}

\input{AuxPrimes/ImageCohom}

\input{AuxPrimes/AuxPrimes}

%% file: AuxPrimes/TheImage.tex
\subsection{The image}\label{Image}

The main result of this subsection is to establish an openness result, \ref{OpenSub}, on the image of a representation $\rho$ satisfying our assumptions A1-A4, and then to compute $H^1(\im \rho, \Ad)$. Set $\mathcal{G}=\im \rho$ and $\mathcal{G}^1=\mathcal{G}\cap\SL_2(V)$. By \ref{ImZarDense}, we know that $\mathcal{G}^1$ is Zariski dense in $\SL_{2/K}$.


\begin{lem}\label{PinkResult}
Let $\Gamma$ be a compact subgroup of $\SL_2(K)$, Zariski dense in ${\SL_2}_{/K}$. Then there is a subfield $K_0$ of $K$ with $K/K_0$ finite and a quaternion algebra $D$ over $K_0$, split over $K$, such that if $D^1$ denotes the algebraic group over $K_0$ defined by the norm one elements of $D$, there is an isomorphism $\tilde{\varphi}:D^1\times_{K_0}K\overset\sim\to{\SL_2}_{/K}$ with $\Gamma\subseteq\tilde{\varphi}(D^1(K_0))$ and such that both $\tilde{\varphi}^{-1}(\Gamma)$ and $\tilde{\varphi}^{-1}([\Gamma,\Gamma])$ are open in $D^1(K_0)$.
\end{lem}

\begin{proof}

Note that the openness of $\tilde{\varphi}^{-1}([\Gamma,\Gamma])$ implies that of $\tilde{\varphi}^{-1}(\Gamma)$.

Applying \cite{PinkCompactSubs}*{Theorem 0.2} to the image of $\Gamma$ in $\PGL_2(K)$, there is a finite index subfield $K_0$ of $K$, an absolutely simple adjoint group $H$ over $K_0$, and an isogeny $\varphi:H\times_{K_0}K\rightarrow{\PGL_2}_{/K}$ with nonvanishing derivative such that $\Gamma\subseteq\varphi(H(K_0))$ and the associated isogeny $\tilde{\varphi}:\tilde{H}\times_{K_0}K\rightarrow{\SL_2}_{/K}$ of simply connected covers  maps an open subgroup of $\tilde{H}(K_0)$ onto $[\Gamma,\Gamma]$.

Since $\PGL_2$ does not admit nonstandard isogenies and the derivative of $\varphi$ is nonzero, $\varphi$ is a central isogeny. Since $H$ is adjoint, $\varphi$ and $\tilde{\varphi}$ are isomorphisms. As all $K_0$ forms of ${\SL_2}_{/K}$ are inner, $\tilde{H}$ is the algebraic group defined by the norm one elements of some quaternion algebra $D$ defined over $K_0$ that splits over $K$. \end{proof}


\begin{prop}\label{OpenSub}
There is a finite index subfield $K_0\subseteq K$ and and $g\in\GL_2(K)$ such that both $g\mathcal{G}^1g^{-1}$ and $g[\mathcal{G}^1,\mathcal{G}^1]g^{-1}$ are open in $\SL_2(K_0)$.
\end{prop}

\begin{proof}
Since $\mathcal{G}^1$ is Zariski dense in ${\SL_2}_{/K}$ by \ref{ImZarDense}, we can apply \ref{PinkResult}. Let $D$ and $\tilde{\varphi}$ be as in \ref{PinkResult}. By assumption A4, $\mathcal{G}^1$ contains a non-trivial unipotent element, so $D$ splits over $K_0$. The lemma now follows from the fact that ${\SL_2}_{/K}$ does not have outer automorphisms. \end{proof}

%% file: AuxPrimes/ImageCohom.tex

\subsubsection{}\label{CohomIm}

Let $K_0$ be as in Proposition \ref{OpenSub}. Denote by $A_0$ it's ring of integers, $\mathfrak{m}_0$ its maximal ideal and $\varpi_0$ a choice of uniformizer. Let $e$ be the ramification index of $K/K_0$. The main goal of this subsection is to describe the cohomology group $H^1(\Gamma,\Ad)$ for $\Gamma$ an open subgroup of $\mathcal{G}=\im(\rho)$. We first record an easy lemma.


\begin{lem}\label{AdModZFixed}
Let $\Gamma$ be a Zariski dense compact subgroup of $\SL_2(A)$. Then $(\Ad/Z)^\Gamma=\{0\}$.
\end{lem}

\begin{proof}

Take $X\in\Ad$ such that the image of $X$ in $\Ad/Z$ is $\Gamma$-invariant. Let $W=KX+Z_K\subset \Ad_K$, and let $N$ be the kernel of $\Gamma\rightarrow\Aut(W)$. Then $X$ is an endomorphism of $V_K$ that commutes with the action of $N$. Since $0\subset Z\subseteq W$ is a $\Gamma$-stable filtration of $W$ and $\dim_K W\le 2$, $\Gamma/N$ is solvable. This together with the Zariski density of $\Gamma$ implies $N$ is Zariski dense; hence, $X\in Z$.  \end{proof}

Ideally one would want a proposition similar to \cite{SWreducible}*{Lemma 6.9}, which in our context would be to prove that $H^1(\Gamma,\Ad)$ is finite (they are actually more precise and consider not only the splitting field of $\im\rho$, but also ajoined all $p$-power roots of unity). This is not true in our case because of the presence of the centre in $\Ad$ and the fact that if $\rhobar$ is dihedral, the image of $\rho$ is pro-solvable. For example, if $\rhobar$ is dihedral and $L$ denotes the unique field from which $\rhobar$ is induced (it is unique since $p=2$), then we have an $A$-module of rank one inside $H^1(\mathcal{G},Z)$ given by the surjection $\mathcal{G}\rightarrow\Gal(L/F)$ composed with the map sending the nontrivial element of $\Gal(L/F)$ to any nonzero element in $Z$. It seems likely that the natural map $H^1(\mathcal{G},Z)\rightarrow H^1(\mathcal{G},\Ad)$, which is injective by \ref{AdModZFixed}, is surjective. We do not prove this, but we describe the cokernel (if it exists) in enough detail for our purposes in \S\ref{AuxPrimesSubSec}.

\begin{lem}\label{ImCohom}
Let $\Gamma$ be an open subgroup of $\mathcal{G}$. The $A$ rank of $\coker(H^1(\Gamma,Z)\rightarrow H^1(\Gamma,\Ad))$ is at most one. Moreover, if it is one there is a positive integer $N_0$ depending only on $\Gamma$ such that if $\gamma\in H^1(\Gamma,\Ad)$ maps to a non-torsion element of this cokernel, there is a cocycle $\kappa:\Gamma\rightarrow\Ad$ representing $\varpi^{N_0}\gamma$ such that for infinitely many $g\in\Gamma$ with distinct $A$-rational eigenvalues, $\kappa(g)\in Z\smallsetminus\{0\}$
\end{lem}

\begin{proof}
Identify $\Ad$ with $\mathrm{M}_{2\times 2}(A)$ using our fixed basis. Let $g\in\GL_2(K)$ be as in \ref{OpenSub}. Set $\Gamma'=g\Gamma g^{-1}$ and $\Ad'=g\Ad\subset \mathrm{M}_{2\times 2}(K)$.  We have isomorphisms $H^1(\Gamma',Z)\cong H^1(\Gamma,Z)$ and $H^1(\Gamma',\Ad')\cong H^1(\Gamma,\Ad)$ compatible with the maps $H^1(\Gamma,Z)\rightarrow H^1(\Gamma,\Ad)$ and $H^1(\Gamma',Z)\rightarrow H^1(\Gamma',\Ad')$. So, it suffices to prove the lemma for $H^1(\Gamma',Z)$ and $H^1(\Gamma',\Ad')$. Since $\Ad'$ is open compact in $\mathrm{M}_{2\times 2}(K)$, there is $l\ge 0$ such that $\varpi^l \mathrm{M}_{2\times 2}(A)\subseteq\Ad'\subseteq\varpi^{-l} \mathrm{M}_{2\times 2}(A)$. For $k\ge 1$, let $J_k$ be the principal congruence subgroup of level $k$ in $\SL_2(A_0)$, i.e. $J_k=(I+\varpi_0^k\mathrm{M}_{2\times 2}(A_0))\cap\SL_2(K_0)$. By \ref{OpenSub}, $\Gamma'$ contains an open subgroup of the form $J_k$ for some $k$. Set $N_0=3ke+4l$. We will show the following.

\begin{lem}\label{TheCocycle}
For any $\gamma\in H^1(\Gamma',\Ad')$, there is a cocycle $\kappa$ representing $\varpi^{N_0}\gamma$ such that, writing
	\[ \kappa\left(\begin{array}{cc} 1 & \varpi_0^{k+1} \\ & 1 \end{array} \right) = 
	\left( \begin{array}{cc} \ast & b \\ \ast & \ast \end{array} \right),
	\]
we have
 
\begin{enumerate}
\item letting $\kappa':\Gamma'\rightarrow\Ad'/Z$ denote the cocycle obtained by composing $\kappa$ with the projection $\Ad'\rightarrow\Ad'/Z$, the restriction $\kappa'|_{J_{3k}}$ is uniquely determined by $b$;
\item if $b=0$, then $\kappa(g) \in Z$ for all $g\in J_{3k}$;
\item if $b\ne 0$, then there are infinitely many diagonal $g\in J_k$ with $\kappa(g)\in Z\smallsetminus\{0\}$.
\end{enumerate}
\end{lem}

Granting \ref{TheCocycle} for now, we finish the proof of \ref{ImCohom}. Note that (1) of \ref{TheCocycle} implies that the cokernel of $H^1(\Gamma',Z)\rightarrow H^1(J_{3k},\Ad')$ has $A$-rank at most one. To see that this implies the same for $H^1(\Gamma',Z)\rightarrow H^1(\Gamma',\Ad')$, it suffices to show the kernels are the same. Let $N$ be an open normal subgroup of $\Gamma'$ contained in $J_{3k}$. Since $N$ is open in $\Gamma'$ it is Zariski dense in $\SL_{2/K}$, so is $g^{-1}Ng$, where $g\in\GL_2(K)$ is as above. Then \ref{AdModZFixed} implies $(\Ad'/Z)^N=(\Ad/Z)^{g^{-1}Ng}=\{0\}$. By inflation-restriction, this implies $H^1(\Gamma',\Ad'/Z)\rightarrow H^1(N,\Ad'/Z)$ is injective. Since the map $H^1(\Gamma',\Ad')\rightarrow H^1(N,\Ad'/Z)$ factors as
\[
H^1(\Gamma',\Ad')\longrightarrow H^1(\Gamma',\Ad'/Z)\longrightarrow H^1(J_{3k},\Ad'/Z)
\longrightarrow H^1(N,\Ad'/Z),
\]
we have
\begin{align*}
\ker(H^1(\Gamma',\Ad')\longrightarrow H^1(J_{3k},\Ad'/Z))&=\ker(H^1(\Gamma',\Ad')
\longrightarrow H^1(\Gamma',\Ad'/Z))\\
&=H^1(\Gamma',Z).
\end{align*}
The final statement of \ref{ImCohom} then follows from (2) and (3) of \ref{TheCocycle}. \end{proof}

The remainder of this subsection will be devoted to the proof of \ref{TheCocycle}, which will mostly comprise of somewhat laborious cocycle computations using relations between elements in $\SL_2(K_0)$. The proof of \ref{TheCocycle} will consist of the following four steps. 
\begin{enumerate}
	\item[\textit{Step 1.}] Show that for any $\gamma\in H^1(\Gamma',\Ad')$, there is a cocycle $\kappa$ representing $\varpi^{N_0}\gamma$ such that
\begin{enumerate}
\item[(i)] for any $\alpha \in 1+\mathfrak{m}_0^k$,
	\[ \kappa\left(\!\begin{array}{cc} \alpha & \\ & \alpha^{-1}\end{array}\!\right) = \left(\!\begin{array}{cc} a_\alpha & \\ & d_\alpha \end{array}\!\right) \]
for some $a_\alpha,d_\alpha$;
\item[(ii)] for any $x\in \mathfrak{m}_0^k$,
	\[ \kappa\left(\!\begin{array}{cc} 1 & x \\ & 1\end{array}\!\right) = \left(\!\begin{array}{cc} a_x & b_x \\ & a_x \end{array}\!\right) \;\text{and}\; \kappa\left(\!\begin{array}{cc} 1 & \\ x & 1 \end{array}\!\right) = \left(\!\begin{array}{cc} d_x & \\ c_x & d_x \end{array}\!\right), \]
for some $a_x,b_x,c_x,d_x$; moreover $b_{\varpi_0^k} = 0$.
\end{enumerate}

\item[\textit{Step 2.}] Show that the $a_\alpha$ and $d_\alpha$ in Step 1.(i) are equal, and that the $b_x$ and $c_x$ in Step 1.(ii) are equal.

\item[\textit{Step 3.}] Step 2 implies that $\kappa|_{J_k}$ mod $Z$ is uniquely determined by the function $x\mapsto b_x$. We then show that $\kappa|_{J_{3k}}$ is uniquely determined by the value $b_{\varpi_0^{k+1}}$.

\item[\textit{Step 4.}] Lastly, we show that if $b_{\varpi_0^{k+1}}\ne 0$, then
	\[ \kappa\left(\!\begin{array}{cc} 1+\varpi_0^n & \\ & (1+\varpi_0^n)^{-1}\end{array}\!\right) \in
	Z \smallsetminus \{0\},
	\]
for all sufficiently large, odd $n$.
\end{enumerate}

Before proceeding, we introduce some notation. For $\alpha\in 1+\mathfrak{m}_0^k$ and $x\in \mathfrak{m}_0^k$ set
\[t(\alpha)=\left(\begin{array}{cc}\alpha&\\&\alpha^{-1}\end{array}\right),\; u^+(x)=\left(\begin{array}{cc}1&x\\&1\end{array}\right),\;u^-(x)=\left(\begin{array}{cc}1&\\x&1\end{array}\right),\]
and define the subgroups
\[
T_k=\{t(\alpha)\in J_k\},\;U^+_k=\{u^+(x)\in J_k\},\;U^-_k=\{u^-(x)\in J_k\}.
\]

\textit{Step 1.} Fix $\gamma\in H^1(\Gamma',\Ad')$, with $\kappa_1:\Gamma'\rightarrow\Ad'$ a cocycle representing $\gamma$. Take $t(\alpha),t(\beta)\in T_k$ and write

\[\kappa_1(t(\alpha))=\left(\begin{array}{cc}a_\alpha&b_\alpha\\ c_\alpha&d_\alpha\end{array}\right)\text{ and }\kappa_1(t(\beta))=\left(\begin{array}{cc}a_\beta&b_\beta\\ c_\beta& d_\beta\end{array}\right).\]
Then using $t(\alpha)t(\beta)=t(\beta)t(\alpha)$ and the cocycle relation, we find that

\begin{equation}\label{bcequalzero}
(\alpha^2-1)b_\beta = (\beta^2-1)b_\alpha\text{ and }(\alpha^{-2}-1)c_\beta=(\beta^{-2}-1)c_\alpha.
\end{equation}
Now assume $\beta\in(1+\mathfrak{m}_0^k)\smallsetminus(1+\mathfrak{m}_0^{k+1})$. Since $\Ad'\subseteq\varpi^{-l}\mathrm{M}_{2\times 2}(A)$ we have $b_\beta,c_\beta\in\mathfrak{m}^{-l}$. Then, letting
\[
X=\left(\begin{array}{cc}&\frac{\varpi^{2ke+2l}}{1-\beta^2}b_\beta\\
\frac{\varpi^{2ke+2l}}{1-\beta^{-2}}c_\beta\end{array}\right),
\]
we have $X\in\varpi^l\mathrm{M}_{2\times 2}(A)\subseteq\Ad'$ and we can define a cocycle $\kappa_2:\Gamma'\rightarrow\Ad'$ by
\[
\kappa_2(g)=\varpi^{2ke+2l}\kappa_1(g)-gX+X.
\]
This cocycle represents the cohomology class $\varpi^{2ke+2l}\gamma$ and $\kappa_2(t(\beta))$ is diagonal. Using \eqref{bcequalzero} we see that

\begin{equation}\label{Tindiag}
\kappa_2(t(\alpha))\in\left\{\left(\begin{array}{cc}a&\\&d\end{array}\right)\right\}
\end{equation}
for all $t(\alpha)\in T_k$.

Before proceeding we prove a sub-lemma.

\begin{lem}\label{OrderTwo}
Let $\kappa':J_k\rightarrow\Ad'$ be a 1-cocyle and let $g=\left(\begin{array}{cc}w&x\\y&w\end{array}\right)\in J_k$ have order two. Then
\[\kappa'(g)=\left(\begin{array}{cc}a&b\\c&a\end{array}\right)\]
with $xc=yb$.
\end{lem}

\begin{proof}
First note that our assumptions on $g$ and $J_k$ impliy $w$ and at least one of $x,y$ are non-zero. Since $g$ has order two, the cocycle relation implies $g\kappa'(g)=\kappa'(g)$. This yields equations

\begin{align}w^2a+wyb+wxc+xyd&=a\label{eqna}\\
wxa+w^2b+x^2c+wxd&=b\label{eqnb}\\
wya+y^2b+w^2c+wyd&=c\label{eqnc}\\
xya+wyb+wxc+w^2d&=d\label{eqnd}
\end{align}
Using $w^2=1+xy$, equations \eqref{eqna} and \eqref{eqnd} both become

\begin{equation}\label{eqn1} xy(a+d)+w(xc+yb)=0,\end{equation}
equation \eqref{eqnb} becomes

\begin{equation}\label{eqn2} wx(a+d)+x(xc+yb)=0,\end{equation}
and equation \eqref{eqnc} becomes

\begin{equation}\label{eqn3} wy(a+d)+y(xc+yb)=0.\end{equation}
If $x=0$, then \eqref{eqn1} implies $b=0$ and \eqref{eqn3} implies $a=d$. If $x\ne0$, then \eqref{eqn2} implies $xc+by = w(a+d)$. Substituting this into \eqref{eqn1} we have $(w^2+xy)(a+d)=0$, i.e. $a=d$. Equation \eqref{eqn1} then implies $xc=yb$.

\end{proof}

By \ref{OrderTwo}, we can write

\[
\kappa_2(u^+(\varpi_0^k))=\left(\begin{array}{cc}a&b\\&a\end{array}\right).
\]
Let $u=\varpi^e\varpi_0^{-1}\in A^\times$. Since $\Ad'\subseteq\varpi^{-l}\mathrm{M}_{2\times 2}(A)$, $b\in\mathfrak{m}^{-l}$. So, letting
\[
X=\left(\begin{array}{cc}u^k\varpi^{2l}b&\\&0\end{array}\right),
\]
we have $X\in\varpi^l\mathrm{M}_{2\times 2}(A)\subseteq\Ad'$. We can then define a cocycle $\kappa:\Gamma\rightarrow\Ad'$ by
\[
\kappa(g)=\varpi^{ke+2l}\kappa_2(g)-gX+X.
\]
The cocycle $\kappa$ represents the cohomology class $\varpi^{3ke+4l}\gamma=\varpi^{N_0}\gamma$ and an easy computation shows $\kappa(u^+(\varpi_0^k))\in Z$. Note that since every $t(\alpha)\in T_k$ commutes with $X$, $\kappa(t(\alpha))=\varpi^{ke+2l}\kappa_2(t(\alpha))$ is diagonal for every $t(\alpha)\in T_k$, by \eqref{Tindiag}. Applying \ref{OrderTwo} to $\kappa$, we have
\begin{equation}\label{Utriang}
\kappa(u^+(x))=\left(\begin{array}{cc}a_x&b_x\\&a_x\end{array}\right)\quad\text{and}
\quad\kappa(u^-(x))=\left(\begin{array}{cc}d_x&\\c_x&d_x\end{array}\right),
\end{equation}
with $b_{\varpi_0^k}=0$. This completes Step 1.

\textit{Step 2.} We show that for any $\alpha\in1+\mathfrak{m}_0^k$, writing

\[
\kappa(t(\alpha))=\left(\begin{array}{cc}a_\alpha&\\&d_\alpha\end{array}\right),
\]
we have $a_\alpha=d_\alpha$, and that for any $x\in\mathfrak{m}_0^k$, writing $\kappa(u^+(x))$ and $\kappa(u^-(x))$ as in \eqref{Utriang}, we have $b_x=c_x$. This will be shown simultaneously by setting $\alpha=1+x$ and considering the relation $u^+(x)u^-(x)=t(\alpha)u^-(x)u^+(x)t(\alpha)$. 

Let $g=u^+(x)u^-(x)=t(\alpha)u^-(x)u^+(x)t(\alpha)$. Applying the cocycle relation to $u^+(x)u^-(x)$ we have

\[\kappa(g)=\left(\begin{array}{cc} a_x+d_x+xc_x&b_x+x^2c_x\\c_x&a_x+d_x+xc_x\end{array}\right)\]
Applying the cocycle relation to $t(\alpha)u^-(x)u^+(x)t(\alpha)$ we have

\[\kappa(g)=\left(\begin{array}{cc} a_x+d_x+xb_x+x^2(a_\alpha+d_\alpha)&
\alpha^2b_x+\alpha^2x(a_\alpha+d_\alpha)\\
\alpha^{-2}(c_x+x^2b_x) + x(a_\alpha+d_\alpha)&
a_x + d_x + xb_x + x^2(a_\alpha+d_\alpha)
\end{array}\right).\]
We may assume $x \ne 0$; so, comparing top left entries,

\begin{equation}\label{ceqn}
c_x=b_x+x(a_\alpha+d_\alpha).
\end{equation}
Comparing top right entries and using \eqref{ceqn}

\begin{align*}b_x+x^2c_x&=\alpha^2(b_x+x(a_\alpha+d_\alpha))\\
&=\alpha^2c_x=(1+x^2)c_x.
\end{align*}
Hence, $b_x=c_x$, and \eqref{ceqn} gives $a_\alpha=d_\alpha$. This completes Step 2.

We use the following notation in what follows. For $\alpha\in1+\mathfrak{m}_0^k$ and $x\in\mathfrak{m}_0^k$ write

\begin{equation}\label{notation}
\kappa(t(\alpha))=\left(\begin{array}{cc}a_\alpha&\\&a_\alpha\end{array}\right),\quad
\kappa(u^+(x))=\left(\begin{array}{cc}a_x&b_x\\&a_x\end{array}\right),\quad
\kappa(u^-(x))=\left(\begin{array}{cc}d_x&\\b_x&d_x\end{array}\right).
\end{equation}

\textit{Step 3.} Since $J_k=U_k^-T_kU_k^+$, $\kappa$ is uniquely determined modulo $Z$ by the values $b_x$, for $x\in \mathfrak{m}_0^k$. We will show that the values $b_x$ for $x\in\mathfrak{m}_0^{3k}$ are uniquely determined by $b_{\varpi_0^{k+1}}$.

First we establish some easy relations. For any $x\in\mathfrak{m}_0^k$, $U^+_k$ act trivially on $\kappa(u^+(x))$ and $U^-_k$ acts trivially on $\kappa(u^-(x))$. Hence, for any $x,y\in\mathfrak{m}_0^k$ we have

\begin{equation}\label{xplusy}
b_{x+y}=b_x+b_y,\quad a_{x+y}=a_x+a_y,\quad d_{x+y}=d_x+d_y.
\end{equation}
Let $\alpha\in1+\mathfrak{m}_0^k$ and $x\in\mathfrak{m}_0^k$. Using $u^+(\alpha^2x)=t(\alpha)u^+(x)t(\alpha)^{-1}$ and the cocycle relation we find

\[\kappa(u^+(\alpha^2x))=\left(\begin{array}{cc}a_x&\alpha^2b_x\\&a_x\end{array}\right),
\]
and similarly for $u^-(\alpha^2x)$. So

\begin{equation}\label{alphasquaredx}
b_{\alpha^2x}=\alpha^2b_x,\quad a_{\alpha^2x}=a_x,\quad d_{\alpha^2x}=d_x.
\end{equation}
Now, if $y\in\mathfrak{m}_0^k$, setting $\alpha=1+y$ and applyng \eqref{xplusy} and \eqref{alphasquaredx} to $u^+(y^2x)=u^+(x)u^+(\alpha^2x)$ we have

\begin{equation}\label{ysquaredx}
b_{y^2x}=y^2b_x,\quad a_{y^2x}=0,\quad d_{y^2x}=0.
\end{equation}

Let $\F_0$ denote the resdiue field of $A_0$. Take $z\in\F_0$ and $n\ge 3k$. Let $\sqrt{z}$ denote the unique element of $\F_0$ such that $(\sqrt{z})^2=z$. If $k$ and $n$ have the same parity, then setting $y =\sqrt{z}\varpi_0^{\frac{n-k}{2}}\in\mathfrak{m}_0^k$ in \eqref{ysquaredx} gives 

\[b_{z\varpi_0^n}=z\varpi_0^{n-k}b_{\varpi_0^k}=0.
\]
If $k$ and $n$ have different parity then, setting $y=\sqrt{z}\varpi_0^{\frac{n-k-1}{2}}\in\mathfrak{m}_0^k$ in \eqref{ysquaredx} gives

\[b_{z\varpi_0^n}=z\varpi_0^{n-k-1}b_{\varpi_0^{k+1}}.
\]
Combining these two equations with \eqref{xplusy} we know the value of $b_x$ for any $x\in\mathfrak{m}_{0}^{3k}$ of the form $x=z_{3k}\varpi_0^{3k}+\cdots z_n\varpi_0^n$ with $z_{3k},\ldots,z_n\in\F_0$. For arbitrary $x\in\mathfrak{m}_0^{3k}$ the value $b_x$ is then determined by the continuity of $\kappa$. This copmletes Step 3.

\textit{Step 4.} It remains to show that if $b_{\varpi_0^{k+1}}\ne 0$, then
	\[ \kappa\left(\!\begin{array}{cc} 1+\varpi_0^n & \\ & (1+\varpi_0^n)^{-1}\end{array}\!\right) \in
	Z \smallsetminus \{0\},
	\]
for all sufficiently large, odd $n$.

Take $x,y\in\mathfrak{m}^k_{0}$ and set $\alpha=1+xy$. We have the relation 

\[u^+(y)u^-(x)=u^-(\alpha^{-1}x)u^+(\alpha y)t(\alpha).
\]
The cocycle relation gives

\[\kappa(u^+(y)u^-(x))=\left(\begin{array}{cc}a_y+d_x+yb_x&b_y+y^2b_x\\
b_x&a_y+d_x+yb_x\end{array}\right),
\]
and

\[\kappa(u^-(\alpha^{-1}x)u^+(\alpha y)t(\alpha))=\left(\begin{array}{cc}
a_\alpha+a_{\alpha y}+d_{\alpha^{-1}x}+\alpha^{-1}xb_{\alpha y} & b_{\alpha y}\\
\alpha^{-2}x^2b_{\alpha y}+b_{\alpha^{-1}x} & 
a_\alpha+a_{\alpha y}+d_{\alpha^{-1}x}+\alpha^{-1}xb_{\alpha y}\end{array}\right),
\]
which, by comparing the diagonal entries, gives

\begin{equation}\label{diagsequal}
a_y+d_x+yb_x = a_\alpha+a_{\alpha y}+d_{\alpha^{-1}x}+\alpha^{-1}xb_{\alpha y}.
\end{equation}
Using \eqref{xplusy}, \eqref{alphasquaredx} and \eqref{ysquaredx}, we see

\[b_{\alpha y}=b_{y+y^2x}=b_y+y^2b_x,\quad a_{\alpha y}=a_{y+y^2x}=a_y,\quad
d_{\alpha^{-1}x}=d_{\alpha x}=d_{x+y^2x}=d_x.
\]
So, \eqref{diagsequal} becomes

\begin{align}
a_\alpha &= yb_x -\alpha^{-1}x(b_y+y^2b_x)\notag\\
&=\alpha^{-1}(yb_x-xb_y).\label{aalpha}
\end{align}
Let $n$ be an odd integer such that $n-k-1\ge 3k$. Since $n-k-1$ has the same parity as $k$ we know $b_{\varpi_0^{n-k-1}}=0$, and, putting $x=\varpi_0^{k+1}$ and $y=\varpi_0^{n-k-1}$ in \eqref{aalpha}, we have
\[
a_{1+\varpi_0^n}=(1+\varpi_0^n)^{-1}\varpi_0^{n-k-1}b_{\varpi_0^{k+1}},
\]
which is zero if and only if $b_{\varpi_0^{k+1}}$ is. In particular, if $b_{\varpi_0^{k+1}}\ne 0$ then there are infinitely many $n$ such that $\kappa(t(1+\varpi_0^n))\in Z\smallsetminus\{0\}$. This completes Step 4, and the proof of \ref{TheCocycle}. \hfill\qed

%% file: AuxPrimes/AuxPrimes.tex

\subsection{Auxiliary Primes}\label{AuxPrimesSubSec}

We now use the results in the previous subsection to prove the existence of auxiliary primes similar to \cite{KW2}*{Lemma 5.10}.

\subsubsection{}\label{KummerExt}

For each $n\ge 1$, let $F_n=F(\mu_{2^n})$. Let $\Q(\mu_{2^n})^+$ denote the maximal totally real subfield of $\Q(\mu_{2^n})$ and let $n_0$ be the largest integer $n\ge 2$ such that $\Q(\mu_{2^n})^+\subseteq F$. For each $n\ge1$ let $\zeta_n$ be a primitive $2^n$-th root of unity such that $\zeta_{n+1}^2=\zeta_n$. Let $x_n=\frac{1}{2}(\zeta_n+\zeta_n^{-1})$ and $y_n=\frac{1}{2}(x^n+1)$. For $n>n_0$ let $\tilde{F}_n=F_n(y_{n_0}^{1/2^n})$, $\tilde{y}_n\in F^\times/(F^\times)^{2^n}$ denote the image of $y_{n_0}$, and $\omega_n\in H^1(G_F,\mu_{2^n})$ be the image of $\tilde{y}_n$ under the Kummer map.


\begin{lem}\label{KummerExt}
Let $n>n_0$. We have
\begin{enumerate}
\item $F_{n_0}(y_{n_0}^{1/4})$ is a dihedral extension of $F$ of degree $8$;
\item the extension $\tilde{F}_n/F_n$ is cyclic of degree $2^{n-1}$, and its cyclic subextension of degree $2$ is $F_n(y_{n_0}^{1/4})$;
\item $\omega_n\in H^1(G_{F,S},\mu_{2^n})$, and its order is divisible by $2^{n-1}$;
\item any quadratic subextension of $\tilde{F}_n/F$ is contained in $F_n$;
\end{enumerate}
\end{lem}

\begin{proof} Parts (1), (2) and (3) are proved in \cite{KW2}*{Lemmas 5.8, 5.9}. Part (4) is a consequence of the first two. Indeed, if $M/F$ is a quadratic subextension of $\tilde{F}_n/F$, we have either $MF_n=F_n$ or $MF_n=F_n(y_{n_0}^{1/4})$, by part 2. Since $MF_n/F$ is abelian, part (1) implies $MF_n=F_n$. \end{proof}

For any $n\ge 1$, let $S_n$, resp. $\tilde{S}_n$, denote the places above $S$ in $F_n$, resp. $\tilde{F}_n$.


\begin{lem}\label{KerFn}
For any $n> n_0$,
\[\ker(H^1(G_{F,S},\Ad)\rightarrow H^1(G_{\tilde{F}_n,\tilde{S}_n},\Ad))
=\ker(H^1(G_{F,S},\Ad)\rightarrow H^1(G_{{F}_{n_0+1},{S}_{n_0+1}},\Ad))\cong A^2.
\]
\end{lem}

\begin{proof}

Let $\gamma\in\ker(H^1(G_{F,S},\Ad)\rightarrow H^1(G_{\tilde{F}_n,\tilde{S}_n},\Ad))$, and let $\gamma_{F_n}$ denote the image of $\gamma$ in $H^1(G_{F_n,S_n},\Ad)$. Since $\tilde{F}_n/F_n$ is Galois and $\Ad^{G_{\tilde{F}_n}}=Z$, we see that $\gamma_{F_n}\in H^1(\Gal(\tilde{F}_n/F_n),Z)$ by inflation-restriction. Since the map $H^1(\Gal(\tilde{F}_n/F_n),Z)\rightarrow H^1(G_{F_n},\Ad)$ factors through $H^1(G_{F_n},Z)$, $\gamma_{F_n}$ maps to zero in $H^1(G_{F_n,S_n},\Ad/Z)$. The commutative diagram

\[\xymatrix{
H^1(G_{F,S},Z) \ar[r] \ar[d] & H^1(G_{F,S},\Ad) \ar[r] \ar[d] & H^1(G_{F,S},\Ad/Z) \ar[d] \\
H^1(G_{F_n,S_n},Z) \ar[r] & H^1(G_{F_n,S_n},\Ad) \ar[r]& H^1(G_{F_n,S_n},\Ad/Z),}
\]
implies that the image of $\gamma$ in $H^1(G_{F,S},\Ad/Z)$ restricts to zero in $H^1(G_{F_n,S_n},\Ad/Z)$. By \ref{AdModZFixed}, $(\Ad/Z)^{G_{F_n}} = 0$, and so the inflation restriction exact sequence implies $H^1(G_{F,S},\Ad/Z)\rightarrow H^1(G_{F_n,S_n},\Ad/Z)$ is injective and the image of $\gamma$ in $H^1(G_{F,S},\Ad/Z)$ is zero. Hence, $\gamma$ is in the image of $H^1(G_{F,S},Z)$. Then $\gamma$ is a homomorphism from $G_{F,S}^\mathrm{ab}$ to $Z$ that is trivial on $G_{\tilde{F}_n,\tilde{S}_n}$. Part (4) of \ref{KummerExt} implies that $\gamma$ is trivial on $G_{F_n,S_n}$, and we have
\[\ker(H^1(G_{F,S},\Ad)\rightarrow H^1(G_{\tilde{F}_n,\tilde{S}_n},\Ad))
=\ker(H^1(G_{F,S},\Ad)\rightarrow H^1(G_{{F}_n,{S}_n})).
\]

For any $n> n_0$, $\Gal(F_n/F)$ is the product of $\Z/2\Z$ and a cyclic 2-group. Then
\begin{align*}
\ker(H^1(G_{F,S},\Ad)\rightarrow H^1(G_{{F}_n,{S}_n}))&=\Hom(\Gal(F_n/F),Z)\\
&=\Hom(\Gal(F_{n_0+1}/F),Z)\\
&=\ker(H^1(G_{F,S},\Ad)\rightarrow H^1(G_{{F}_{n_0+1},{S}_{n_0+1}})),
\end{align*}
and $\Hom(\Gal(F_{n_0+1}/F),Z)\cong A^2$.
\end{proof}

\subsubsection{}\label{n1Def}

We know, by \ref{PinkResult}, that $[\mathcal{G},\mathcal{G}]$ is open in $\mathcal{G}$. By \ref{PinkResult} again, we see that the commutator subgroup of the commutator subgroup of $\mathcal{G}$, i.e $[[\mathcal{G},\mathcal{G}],[\mathcal{G},\mathcal{G}]]$, is open in $[\mathcal{G},\mathcal{G}]$, hence also in $\mathcal{G}$. Since $\tilde{F}_n/F_n$ and $F_n/F$ are both abelian extension, it follows that $\rho|_{G_{\tilde{F}_n}}$ contains $[[\mathcal{G},\mathcal{G}],[\mathcal{G},\mathcal{G}]]$ for all $n\ge 1$. This implies that there is $n_1$, which we can assume is greater than $n_0$, and a finite extension $M/F$ such that if $F_\rho$ denotes the subfield of $\Qbar$ fixed by $\ker\rho$, we have $F_\rho\cap\tilde{F}_n=M$ and $\rho(G_{\tilde{F}_n})=\rho(G_M)$ for all $n\ge n_1$. Set $\Gamma=\rho(G_M)$. Since $\Gamma$ is open in $\mathcal{G}$, by \ref{OpenSub}, there is a finite index subfield $K_0$ of $K$ such that some conjugate of $\Gamma$ is open in $\SL_2(K_0)$. Let $A_0$ denote the ring of integers of $K_0$ and let $\varpi_0$ be a choice of uniformizer. For the remainder of this subsection we fix $k\ge 1$ such that $g\Gamma g^{-1}$ contains $\SL_2(A_0)\cap(I+\varpi_0^k M_2(A_0))$, for some $g \in \GL_2(K)$. Let $e$ be the ramification index of $K/K_0$ and write $\mathfrak{m}_0=\varpi_0A_0$.


\begin{lem}\label{SubofAd}
Let $B$ be a $\Gamma$-stable subgroup of $\Ad$ (not necessarily an $A$-module). Let $m\ge 1$ be such that $B\not\subseteq\varpi^m\Ad$. There is a non-negative integer $N_0$, independent of $B$, such that some $X\in B$ satisfies either
\begin{itemize}
\item[-] $\val(\tr X)<m+N_0$ or
\item[-] $X=zI+Y$ with $\val(z)<m+N_0$ and $Y\in\varpi^{\val(z)+1}\Ad$.
\end{itemize}
\end{lem}

\begin{proof}

Take $g\in\GL_2(K)$ such that $\Gamma'=g\Gamma g^{-1}$ contains $\SL_2(A_0)\cap(I+\varpi_0^k \mathrm{M}_{2\times 2}(A_0))$, with $k$ as above. Set $\Ad'=g\Ad\subset \mathrm{M}_{2\times 2}(K)$ and $B'=gB$. Then $B'$ is a $\Gamma'$ stable subgroup of $\Ad'$. There is some non-negative integer $l$ such that $\varpi^l\mathrm{M}_{2\times 2}(A)\subseteq\Ad'\subseteq\varpi^{-l}\mathrm{M}_{2\times 2}(A)$. Set $N_0=l+4ke$.

Since $B\not\subseteq\varpi^m\Ad$, there  is $X\in B$ such that, writing
\[
X=\left(\begin{array}{cc}a&b\\c&d\end{array}\right),
\]
one of $a,b,c,d$ is not in $\mathfrak{m}^m$. If $\val(\tr X)<m+N_0$ we are done, so we assume otherwise. Write
\[
X'=gX=\left(\begin{array}{cc}a'&b'\\c'&d'\end{array}\right).
\]
We know that $\val(a'+d')=\val(\tr X')=\val(\tr X)\ge m+N_0>m+l$. Then, if both $b',c'\in\mathfrak{m}^{m+l}$ we can write $X'=a'I+Y'$ with $Y'\in\varpi^{m+l}\mathrm{M}_{2\times 2}(A)\subseteq\varpi^m\Ad'$. Then $X=a'I+g^{-1}Y'$ and $g^{-1}Y'\in\varpi^m\Ad$. Since $X\notin\varpi^m\Ad$ we must have $a'\notin\mathfrak{m}^m$ and we can take $z=a'$ and $Y=g^{-1}Y'$. 

We are left with the case that one of $b',c'$ does not belong to $\mathfrak{m}^{m+l}$. Assume $b\notin\mathfrak{m}^{m+l}$. Since $B'$ is $\Gamma'$ stable, $B'$ contains
\begin{align}
\left(\begin{array}{cc}1+\varpi_0^k&\\&(1+\varpi_0^k)^{-1}\end{array}\right)
\left(\begin{array}{cc}a'&b'\\c'&d'\end{array}\right)
 \left(\begin{array}{cc}(1+\varpi_0^k)^{-1}&\\&1+\varpi_0^k\end{array}\right)-
\left(\begin{array}{cc}a'&b'\\c'&d'\end{array}\right)\notag\\
=\left(\begin{array}{cc}&\varpi_0^{2k}b'\\((1+\varpi_0^{2k})^{-1}-1)c'&\end{array}\right).
\label{B'1}
\end{align}
So, $B'$ also contains
\begin{align*}
\left(\begin{array}{cc}1&\\\varpi_0^k&1\end{array}\right)
\left(\begin{array}{cc}&\varpi_0^{2k}b'\\((1+\varpi_0^{2k})^{-1}-1)c'&\end{array}\right)
\left(\begin{array}{cc}1&\\\varpi_0^k&1\end{array}\right)-
\left(\begin{array}{cc}&\varpi_0^{2k}b'\\((1+\varpi_0^{2k})^{-1}-1)c'&\end{array}\right)\\
=\left(\begin{array}{cc}\varpi_0^{3k}b'&\\\varpi_0^{4k}b'&\varpi_0^{3k}b'\end{array}\right),
\end{align*}
as well as
\begin{align}
\left(\begin{array}{cc}(1+\varpi_0^k)^{-1}&\\&1+\varpi_0^k\end{array}\right)
\left(\begin{array}{cc}\varpi_0^{3k}b'&\\\varpi_0^{4k}b'&\varpi_0^{3k}b'\end{array}\right)
\left(\begin{array}{cc}1+\varpi_0^k&\\&(1+\varpi_0^k)^{-1}\end{array}\right)\notag\\
=\left(\begin{array}{cc}\varpi_0^{3k}b'&\\(\varpi_0^{4k}+\varpi_0^{6k})b'&\varpi_0^{3k}b'\end{array}\right).\label{B'2}
\end{align}
Then, using \eqref{B'1} again, $B'$ contains
\begin{align}
\left(\begin{array}{cc}1&\\\varpi_0^k+\varpi_0^{2k}&1\end{array}\right)
\left(\begin{array}{cc}&\varpi_0^{2k}b'\\((1+\varpi_0^{2k})^{-1}-1)c'&\end{array}\right)
\left(\begin{array}{cc}1&\\\varpi_0^k+\varpi_0^{2k}&1\end{array}\right)-
\left(\begin{array}{cc}&\varpi_0^{2k}b'\\((1+\varpi_0^{2k})^{-1}-1)c'&\end{array}\right)\notag\\
=\left(\begin{array}{cc}(\varpi_0^{3k}+\varpi_0^{4k})b'&\\(\varpi_0^{4k}+\varpi_0^{6k})b'&(\varpi_0^{3k}+\varpi_0^{4k})b'\end{array}\right).\label{B'3}
\end{align}
Subtracting \eqref{B'2} from \eqref{B'3} we have that $B'$ contains
\[
\left(\begin{array}{cc}\varpi_0^{4k}b'&\\&\varpi_0^{4k}b'\end{array}\right).
\]
Taking $z=\varpi_0^{4k}b'$, $g^{-1}(zI)=zI\in B$ and $\val(z)=\val(b')+4ke<m+l+4ke=m+N_0$. The case of $b'\in\mathfrak{m}^{m+l}$ but $c'\notin\mathfrak{m}^{m+l}$ is similar.
	\end{proof}


\begin{lem}\label{Centre}
Let $g\in\GL_2(V)$ have distinct $A$ rational eigenvalues $\alpha,\beta$ and set $w=\val(\alpha-\beta)$. For $z\in A$ non-zero, $Y\in\varpi^{\val(z)+w+1}\Ad$, and $m\ge 0$, if
\[
zI+Y\in (g-1)\Ad+\varpi^m\Ad
\]
then $m\le \val(z)+w$.
\end{lem}

\begin{proof}
Let $V_\alpha$ denote the $\alpha$ eigenspace for $g$, $e_\alpha$ a generator of $V_\alpha$ and $e_\beta$ be any element of $V$ mapping to a generator of $V/V_\alpha$. Write $ge_\beta=\beta e_\beta+x e_\alpha$ with $x\in A$. Note a splitting of 
\[
0\longrightarrow V_\alpha\rightarrow V\longrightarrow V/V_\alpha\longrightarrow 0,
\]
exists if and only if $\val(x)\ge\val(\alpha+\beta)$ and so we cannot, in general, assume $x=0$. Identify $\Ad$ with $\mathrm{M}_{2\times 2}(A)$ using the basis $\{e_\alpha,e_\beta\}$ of $V$. 

Assume there is $X\in\Ad$ such that $zI+Y\in (g-1)X+\varpi^m\Ad$. Set $j=\min\{m,\val(z)+w+1\}$. Then $zI-(g-1)X\in\varpi^j\Ad$. Writing 
\[
X=\left(\begin{array}{cc}a&b\\c&d\end{array}\right)
\]
we have 
\begin{align*}
(g-1)X &=\left(\begin{array}{cc}\alpha&x\\&\beta\end{array}\right)
\left(\begin{array}{cc}a&b\\c&d\end{array}\right)
\left(\begin{array}{cc}\alpha^{-1}&\alpha^{-1}\beta^{-1}x\\&\beta^{-1}\end{array}\right)
-\left(\begin{array}{cc}a&b\\c&d\end{array}\right)\\
&=\left(\begin{array}{cc} \alpha^{-1}xc&\ast\\(\alpha^{-1}\beta-1)c&\alpha^{-1}xc\end{array}\right).
\end{align*}
Then $zI-(g-1)X\in\varpi^j\Ad$ implies $(\alpha^{-1}\beta-1)c\in\mathfrak{m}^j$, so $\val(c)\ge j-w$. Then the upper left entry implies $z-\alpha^{-1}xc\in\mathfrak{m}^j$, so $\val(z)\ge j-w$. But, $j\le\val(z)+w$ implies, by definition of $j$, that $j=m$ and $m\le \val(z)+w$. \end{proof}

We record another lemma that will be useful below.

\begin{lem}\label{HensLemLem}

Let $\sigma_0 \in G_F$ be such that be such that $\rho(\sigma_0)$ has distinct $A$-rational eigenvalues $\alpha_0$ and $\beta_0$, and $\det \rho(\sigma_0) = 1$. Set $w = \val(\tr \rho(\sigma))$. If $\sigma \in G_F$ satisfies $\det \rho(\sigma) = 1$ and $\tr\rho(\sigma) - \tr\rho(\sigma_0) \in \mathfrak{m}^{2w+1}$, then $\rho(\sigma)$ has distinct $A$-rational eigenvalues and, denoting them by $\alpha$ and $\beta$, $\val(\alpha-\beta) = \val(\alpha_0-\beta_0)$.

\end{lem}

\begin{proof}

Let $f(t)$ be the characteristic polynomial of $\rho(\sigma)$, and note that $\val(\tr\rho(\sigma)) = w$. Since $\tr\rho(\sigma)-\tr\rho(\sigma_0) \in\mathfrak{m}^{2w}$,  we have
	\[ f(\alpha_0)= \alpha_0^2 - \tr\rho(\sigma)\alpha_0 + 1 = \alpha_0(\tr\rho(\sigma_0)-\tr\rho(\sigma)) \in\mathfrak{m}^{2w+1} =  (\tr\rho(\sigma))^2\mathfrak{m}= f'(\alpha_0)^2\mathfrak{m}.
		\]
Then $f(t)$ splits over $A$ by Hensel's Lemma, cf. \cite{BourbakiComAlg}*{Chapter III, \S 4, n${}^\circ 5$, Corollary 1 to Theorem 2}. The fact that $\val(\alpha-\beta) = \val(\alpha_0-\beta_0)$ follows from the fact that, since the characteristic is two, $\alpha-\beta = \tr \rho(\sigma)$ and $\alpha_0-\beta_0 = \tr \rho(\sigma_0)$.
	\end{proof}

\subsubsection{}\label{fixedsubmods}

We introduce some notation that will be used for the remainder of this section. Let $W$ be a finitely generated $A$ module. We denote by $W_\mathrm{tor}$ the torsion submodule of $W$ and by $W_\mathrm{free}$ the free $A$-module $W/W_\mathrm{tor}$.

We now fix certain submodules of $H^1(G_{F,S},\Ad)$ that will be used in the proof of auxiliary primes below. Let $F_\rho$ denote subfield of $\Qbar$ fixed by $\ker\rho$ and $M=F_\rho\cap \tilde{F}_{n_1}$. Recall that $n_1$ was chosen so that $F_\rho\cap\tilde{F}_n=M$ for any $n\ge n_1$. Let 
\[
h=\rk_AH^1(G_{F,S},\Ad)-\rk_A H^1(\Gal(\tilde{F}_{n_1}/F),Z)=\rk_A H^1(G_{F,S},\Ad)-2,
\]
and
\[
h_0=\rk_A(\im(H^1(G_{F,S},\Ad)\rightarrow H^1(G_M,\Ad))\cap H^1(\Gal(F_\rho/M),\Ad)).
\]
Set $\Gamma=\rho(G_M)$. We know by \ref{ImCohom} that the $A$-rank of the cokernel of the map $H^1(\Gamma,Z)\rightarrow H^1(\Gamma,\Ad)$ is at most one. For the remainder of this subsection we will assume that it is one and that there is some $\gamma\in H^1(G_{F,S},\Ad)$ whose image in $H^1(G_M,\Ad)$ lands in $H^1(\Gal(F_\rho/M),\Ad)=H^1(\Gamma,\Ad)$ and maps to a non-torsion element in $\coker(H^1(\Gamma,Z)\rightarrow H^1(\Gamma,\Ad))$. The case when every element of $H^1(G_{F,S},\Ad)$ that lands in $H^1(\Gamma,\Ad)$ maps to a torsion element of $\coker(H^1(\Gamma,Z)\rightarrow H^1(\Gamma,\Ad))$, which includes the case when the cokernel is torsion, is easier as there is one fewer ``type" of cohomology class to consider below (in particular one does not need case (b) of \ref{OpenSet}) and it will be obvious to the reader how to adjust the arguments.


We fix $W_1\subset\cdots\subset W_h$ of $H^1(G_{F,S},\Ad)$ such that
\begin{enumerate}
\item $W_i$ is free of rank $i$ for each $i$;
\item $W_h$ intersects the image of $H^1(\Gal(\tilde{F}_{n_1}/F),Z)$ trivially;
\item the image $W_{h_0}$ in $H^1(G_M,\Ad)$ is contained in $H^1(\Gal(F_\rho/M),\Ad)=H^1(\Gamma,\Ad)$;
\item the image of $W_{h_0-1}$ in $H^1(G_M,\Ad)$ is contained in $H^1(\Gal(F_\rho/M),Z)=H^1(\Gamma,Z)$.
\end{enumerate}
We note
\begin{itemize}

\item[-] for all $n\ge n_1$, the map $W_h\rightarrow \im(H^1(G_{F,S},\Ad)\rightarrow H^1(G_{\tilde{F}_n,\tilde{S}_n},\Ad))_\mathrm{free}$ is injective with finite cokernel by (1), (2), and \ref{KerFn};

\item[-] for all $n\ge n_1$, the image of $W_{h_0}$ under $H^1(G_{F,S},\Ad)\rightarrow H^1(G_{F_\rho\tilde{F}_n},\Ad)$ zero by (3);

\item[-] for all $n\ge n_1$, the map $W_h/W_{h_0}\rightarrow \im(H^1(G_{F,S},\Ad)\rightarrow H^1(G_{F_\rho\tilde{F}_n},\Ad))$ is injective by (2), (3) and the definition of $h_0$.

\end{itemize}

There are three different types of cohomology classes we will need to consider based on whether an element $\gamma\in W_h$ does not belongs to $W_{h_0}$, belongs to $W_{h_0}$ but not to $W_{h_0-1}$, or belongs to $W_{h_0-1}$. The main tool for guaranteeing the existence of auxiliary primes is the following lemma.


\begin{lem}\label{OpenSet}

There are non-negative integers $w$ and $N$ such that if $\gamma\in W_h$ and $s\ge 0$ are such that one of the following hold
\begin{enumerate}
\item[(a)] $\gamma\in W_{h_0-1}$ but $\gamma\notin\varpi^s W_{h_0-1}$,
\item[(b)] $\gamma\in W_{h_0}$ but $\gamma\notin\varpi^s W_{h_0}+W_{h_0-1}$,
\item[(c)] $\gamma\notin\varpi^s W_h+W_{h_0}$,
\end{enumerate}
 then for any $n\ge n_1$ there is a nonempty open set $U\subset G_{\tilde{F}_n}$ such that
\begin{enumerate}
\item for every $\sigma\in U$, $\rho(\sigma)$ has distinct $A$-rational eigenvalues $\alpha,\beta$ with $\val(\alpha-\beta)\le w$, and
\item for any cocycle $\kappa$ representing $\gamma$, $\sigma\in U$, and $m\ge 1$, if
\[
\varpi^j\kappa(\sigma)\in (\sigma-1)\Ad+\varpi^m\Ad
\]
then $j> m-s-N$.
\end{enumerate}
\end{lem}

\begin{proof}

We prove the three different cases (a), (b) and (c) separately in the next three sublemmas.


\begin{lem}\label{Typea}
There are non-negative integer $w_a$ and $N_a$ such that if $\gamma\in W_h$ and $s\ge 0$ are such that
\begin{enumerate}
\item[(c)] $\gamma\in W_{h_0-1}$ but $\gamma\notin\varpi^s W_{h_0-1}$,
\end{enumerate}
 then for any $n\ge n_1$ there is a nonempty open set $U\subset G_{\tilde{F}_n}$ such that
\begin{enumerate}
\item for every $\sigma\in U$, $\rho(\sigma)$ has distinct $A$-rational eigenvalues $\alpha,\beta$ with $\val(\alpha-\beta)\le w_a$, and
\item for any cocycle $\kappa$ representing $\gamma$, $\sigma\in U$, and $m\ge 1$, if
\[
\varpi^j\kappa(\sigma)\in (\sigma-1)\Ad+\varpi^m\Ad
\]
then $j> m-s-N_a$. 
\end{enumerate}
\end{lem}

\begin{proof}

If $\rhobar$ is $L$-dihedral, fix $\tau_0$ as in assumption A5 made at the very beginning of the section, i.e. $\tau_0 \in G_F \smallsetminus G_L$ and $\rho(\tau_0)$ has distinct infinite order $A$-rational eigenvalues, and set $w_a=\val(\tr\rho(\sigma))$. If $\rhobar$ is not $L$-dihedral, set $w_a = 0$.

The image of $W_{h_0-1}$ in $H^1(G_M,\Ad)$ is contained in $H^1(\Gamma,Z)$, and we identify it with it's image. Since $H^1(\Gamma,Z)$ is a separated $A$-module (in fact it is finitely generated), there is $N_1\ge 1$ such that for any $\delta\in W_{h_0-1}$ with $\delta\notin\varpi W_{h_0-1}$, $\delta\notin\varpi^{N_1}H^1(\Gamma,Z)$. et $N_a=N_1+w_a$.

Since $W_{h_0-1}\subseteq H^1(\Gamma,Z)$, there is a cocyle $\kappa_0$ representing $\gamma$ such that $\kappa_0$ is given by a continuous homomorphism $G_M\rightarrow\Gamma\rightarrow Z$. Let $n\ge n_1$. We will show below that there is some $\sigma_0\in G_{\tilde{F}_{n}}$ such that
\begin{enumerate}\item[(i)] $\rho(\sigma_0)$ has distinct $A$-rational eigenvalues and $\val(\tr\rho(\sigma))\le w_a$;
\item[(ii)] $\det\rho(\sigma_0)=1$;
\item[(iii)] $\kappa_0(\sigma_0)=zI\in Z$ with $z\notin\mathfrak{m}^{s+N_1}$.
\end{enumerate}
Granting the existence of such $\sigma_0$, we can define our open set $U$. Let $U$ be the non-empty open subset of $G_{\tilde{F}_n}$ consisting of elements $\sigma$ such that
\begin{itemize}
\item[-] $\tr\rho(\sigma)-\tr\rho(\sigma_0)\in\mathfrak{m}^{2w_a+1}$;
\item [-]$\det\rho(\sigma)=1$;
\item [-]$\kappa_0(\sigma)-\kappa_0(\sigma_0)\in\varpi^{s+N_a}\Ad$.
\end{itemize}

We first show that any element of the set $U$ satisfies the conclusion of the lemma. Part (1) of the lemma follows from \ref{HensLemLem}. To see that part (2) is satisfied, let $\kappa$ be any cocycle representing $\gamma$. Then 
\[
\varpi^j\kappa(\sigma)\in (\sigma-1)\Ad+\varpi^m\Ad
\]
implies
\[
\varpi^j\kappa_0(\sigma)\in (\sigma-1)\Ad+\varpi^m\Ad
\]
as $\kappa(\sigma)-\kappa_0(\sigma)\in (\sigma-1)\Ad$. Then, since $\kappa_0(\sigma)=zI+Y$ with $\val(z)<s+N_1$ and $Y\in\varpi^{\val(z)+w_a+1}\Ad$, \ref{Centre} gives $j> m-s-N_a$.

We now show the existence of $\sigma_0\in G_{\tilde{F}_n}$ satisfying (i), (ii) and (iii) above. By choice of $N_1$, $\gamma\notin\varpi^{s+N_1}H^1(\Gamma,Z)$ and so there is some $\tau\in G_{\tilde{F}_n}$ such that $\kappa_0(\tau)\notin\varpi^{s+N_1}Z$. 

First assume that $\rhobar$ is non-dihedral. Then by Dickson's classification of subgroups of $\PGL_2(\F)$, we must have $\rhobar(G_F) \supseteq \SL_2(\F')$ with $|\F'|\ge 4$. Since this group is simple, $\rhobar(G_{\tilde{F}_n}) \supseteq \SL_2(\F')$. Since $\Hom(\rhobar(G_{\tilde{F}_n}),Z) = 0$, we can find $\tau \in \ker(\rhobar|_{G_{\tilde{F}_n}})$ such that $\kappa_0(\tau)\notin\varpi^{s+N_1}Z$. Fix some $\sigma\in G_{\tilde{F}_n}$ such that $\rhobar(\sigma)$ has distinct eigenvalues and determinant $1$. If $\kappa_0(\sigma)\notin\varpi^{s+N_1}Z$ then we set $\sigma_0=\sigma$. Otherwise, we set $\sigma_0=\sigma\tau$. Note that $\rhobar(\sigma_0)=\rhobar(\sigma)$, so $\rho(\sigma_0)$ satisfies (i) and (ii), and $\kappa_0(\sigma_0)=\kappa_0(\sigma)+\kappa_0(\tau)\notin\varpi^{s+N_1}Z$, so $\sigma_0$ satisfies (iii).

Now we assume $\rhobar$ is dihedral and let $L$ denote the unique quadratic extension of $F$ for which $\rhobar$ is $L$-dihedral. Note that $\rhobar(G_L)$ has odd order. We first assume that there is some $\tau\in G_{L\tilde{F}_n}$ such that $\kappa_0(\tau)\notin\varpi^{s+N_1}Z$. Then, by replacing $\tau$ by $\tau^j$ for $j$ odd and sufficiently large, we can assume that $\tau\in\ker(\rhobar)$. Let $\sigma$ be any element of $G_{\tilde{F}_n}$ such that $\rhobar(\sigma)$ has distinct eigenvalues and determinant one. If $\kappa_0(\sigma)\notin\varpi^{s+N_1}Z$ then we set $\sigma_0=\sigma$. Otherwise, we set $\sigma_0=\sigma\tau$. Note that $\rhobar(\sigma_0)=\rhobar(\sigma)$, so $\rho(\sigma_0)$ satisfies (i) and (ii) and $\kappa_0(\sigma_0)=\kappa_0(\sigma)+\kappa_0(\tau)\notin\varpi^{s+N_1}Z$, so $\sigma_0$ satisfies (iii).

Now assume that all $\sigma\in G_{L\tilde{F}_n}$ have $\kappa_0(\sigma)\in\varpi^{s+N_1}Z$. In particular, this implies $L\not\subseteq\tilde{F}_n$. Let $\tau$ be such that $\kappa_0(\tau)\notin\varpi^{s+N_1}$. Then $\tau$ maps to the non-trivial element of $\Gal(L\tilde{F}_n/\tilde{F}_n)$. Let $\tau'\in G_{\tilde{F}_n}$ be any other element mapping to the non-trivial element of $\Gal(L\tilde{F}_n/\tilde{F}_n)$. By assumption $\kappa(\tau'\tau^{-1})\in\varpi^{s+N_1}Z$, so $\kappa(\tau')\notin\varpi^{s+N_a}Z$. Hence we may assume $\tau=\tau_0$, where $\tau_0$ is the element fixed at the beginning of the proof. Then $\sigma_0=\tau_0$ satisfies (i),(ii) and (iii) above. \end{proof}


\begin{lem}\label{Typeb}
There are non-negative integers $w_b$ and $N_b$ such that if $\gamma\in W_h$ and $s\ge 0$ are such that
\begin{enumerate}
\item[(b)] $\gamma\in W_{h_0}$ but $\gamma\notin\varpi^s W_{h_0}+W_{h_0-1}$,
\end{enumerate}
 then for any $n\ge n_1$ there is a nonempty open set $U\subset G_{\tilde{F}_n}$ such that
\begin{enumerate}
\item for every $\sigma\in U$, $\rho(\sigma)$ has distinct $A$-rational eigenvalues $\alpha,\beta$ with $\val(\alpha-\beta)\le w _b$, and
\item for any cocycle $\kappa$ representing $\gamma$, $\sigma\in U$, and $m\ge 1$, if
\[
\varpi^j\kappa(\sigma)\in (\sigma-1)\Ad+\varpi^m\Ad
\]
then $j> m-s-N_b$.
\end{enumerate}
\end{lem}

\begin{proof}

Fix $\delta\in H^1(\Gamma,\Ad)$ such that $\delta$ is a generator for $\coker(H^1(\Gamma,Z)\rightarrow H^1(\Gamma,\Ad))_\mathrm{free}$. By \ref{ImCohom}, there is $N_0 \ge 1$ and a cocycle $\kappa_0$ representing $\varpi^{N_0}\delta$ such that there are infinitely many $g\in\Gamma$ that have distinct $A$-rational eigenvalues with $\kappa_0(g)\in Z\smallsetminus\{0\}$. As $[\Gamma,\Gamma]$ is open in $\Gamma$, there is some $\sigma_0\in G_{\tilde{F}_n}$ such that 
\begin{enumerate}
\item[(i)] $\rho(\sigma_0)\in[\Gamma,\Gamma]$,
\item[(ii)] $\rho(\sigma_0)$ has distinct $A$-rational eigenvalues,
\item[(iii)] $\kappa_0(\sigma_0)\in Z$ but $\kappa(\sigma_0)\ne 0$.
\end{enumerate}
We now set $w_b=\val(\tr\rho(\sigma_0))$. Fix an element $\gamma'$ of $W_{h_0}$ that maps to a generator of $W_{h_0}/W_{h_0-1}$ and let $N_1\ge 0$ be such that $\gamma'-\varpi^{N_1}\delta$ maps to a torsion element in $\coker(H^1(\Gamma,Z)\rightarrow H^1(\Gamma,\Ad))$. Writing $\kappa_0(\sigma_0)=zI$ we set $N_b=N_1+\val(z)+w_b+1$.

Take $n\ge n_1$ and let $U$ be the non-empty open subset of $G_{\tilde{F}_n}$ consisting of elements $\sigma$ such that
\begin{itemize}
\item[-] $\tr\rho(\sigma)-\tr\rho(\sigma_0)\in\mathfrak{m}^{2w_b+1}$;
\item[-] $\rho(\sigma) \in [\Gamma,\Gamma]$;
\item[-] $\kappa_0(\sigma)-\kappa_0(\sigma_0)\in\varpi^{N_b}\Ad$.
\end{itemize}
We now check that every $\sigma\in U$ satisfies (1) and (2) of the Lemma.

Part (1) follows from \ref{HensLemLem} (note $\rho(\sigma)\in [\Gamma,\Gamma]$ implies $\det\rho(\sigma)=1$). Now let $\kappa$ be a cocycle representing $\gamma$ and let $\sigma\in U$. Let $m,j\ge 1$ and be such that 
\begin{equation}\label{TypebAssump1}
\varpi^j\kappa(\sigma)\in (\sigma-1)\Ad+\varpi^m\Ad.
\end{equation}
We have $\gamma=\varpi^{s'}\gamma'+\gamma''$, where $\gamma'$ is as above, $\gamma''\in W_{h_0-1}$ and $s'<s$. Let $l\ge 0$ be such that $\mathfrak{m}^l$ is the annihilator of $\coker(H^1(\Gamma,Z)\rightarrow H^1(\Gamma,\Ad))_\mathrm{tor}$. Then, by choice of $l$ and $N_1$ above,
\[
\varpi^l\gamma=\varpi^{s'+l}\gamma'+\varpi^l\gamma''=\varpi^{s'+l+N_1}\delta+\delta'
\]
with $\delta'\in H^1(\Gamma,Z)=\Hom(\Gamma^\mathrm{ab},Z)$. Equation \eqref{TypebAssump1} implies
\begin{equation}\label{TypebNewEq}
\varpi^{j+l+N_0}\kappa(\sigma)\in (\sigma-1)\Ad+\varpi^{m+l+N_0}\Ad.
\end{equation}
Now $\varpi^{j+l+N_0}\kappa$ represents the cohomology class $\varpi^{j+l+N_0}\gamma=\varpi^{j+s'+l+N_1+N_0}\delta+\varpi^{j+N_0}\delta'$. Since $\kappa_0$ is a cocyle representing $\delta$ and $\delta'(\sigma)=0$, as $\sigma\in[\Gamma,\Gamma]$, there is $X\in\Ad$ such that
\[
\varpi^{j+l+N_0}\kappa(\sigma)=\varpi^{j+s'+l+N_1+N_0}\kappa_0(\sigma)+(\sigma-1)X.
\]
This together with \eqref{TypebNewEq} yields
\begin{equation}\label{TypebEq}
\varpi^{j+s'+l+N_1+N_0}\kappa_0(\sigma)\in (\sigma-1)\Ad+\varpi^{m+l+N_0}\Ad.
\end{equation}
Writing $\kappa_0(\sigma)=zI+Y$ with $Y\in\varpi^{N_b}\Ad \subseteq \varpi^{\val(z)+1}\Ad$, \ref{Centre} and \eqref{TypebEq} imply
\[
m+l+N_0 \le j+s'+l+N_1+N_0+\val(z)+w_b
\]
hence
\[
j\ge m-s'-(N_1+\val(z)+w_b)>m-s-N_b,
\]
which is (2) of the lemma. \end{proof}


\begin{lem}\label{Typec}
There are non-negative integers $w_c$ and $N_c$ such that if $\gamma\in W_h$ and $s\ge 0$ are such that
\begin{enumerate}
\item[(c)] $\gamma\notin\varpi^s W_h+W_{h_0}$
\end{enumerate}
 then for any $n\ge n_1$ there is a nonempty open set $U\subset G_{\tilde{F}_n}$ such that
\begin{enumerate}
\item for every $\sigma\in U$, $\rho(\sigma)$ has distinct $A$-rational eigenvalues $\alpha,\beta$ with $\val(\alpha-\beta)\le w_c$, and
\item for any cocycle $\kappa$ representing $\gamma$, $\sigma\in U$, and $m\ge 1$, if
\[
\varpi^j\kappa(\sigma)\in (\sigma-1)\Ad+\varpi^m\Ad
\]
then $j> m-s-N_c$.
\end{enumerate}
\end{lem}

\begin{proof}

Since, for all $n\ge n_1$, $W_h/W_{h_0}$ injects into $H^1(G_{F_\rho\tilde{F}_n},\Ad)^{\Gal(F_\rho\tilde{F}_n/\tilde{F}_n)}$ and $\rho$ identifies $\Gal(F_\rho\tilde{F}_n/\tilde{F}_n)$ with $\Gamma$, we have, setting $\tilde{F}_\infty=\cup_{n\ge 1}\tilde{F}_n$, an injection
\[
W/W_{h_0}\longrightarrow H^1(G_{E\tilde{F}_\infty},\Ad)^\Gamma=
\Hom_\Gamma(G_{E\tilde{F}_\infty},\Ad),
\]
and we identify $W_h/W_{h_0}$ with it's image under this map. Since $\Hom_\Gamma(G_{E\tilde{F}_\infty},\Ad)$ is separated, there is $N_1\ge 0$ such that $\delta\notin\varpi W_h/W_{h_0}$ implies $\delta\notin\varpi^{N_1}H^1(G_{E\tilde{F}_\infty},\Ad)$. We set $N_c=N_0+N_1$ where $N_0$ is as in \ref{SubofAd}. We set $w_c=0$.

Since, for any two cocyles $\kappa$ and $\kappa'$ representing $\gamma$, $\kappa(\sigma)$ and $\kappa'(\sigma)$ differ by an element of $(\sigma-1)\Ad$, it suffices to show (2) holds for one particular choice of cocycle representing $\gamma$. We will show below that there is some $\sigma_0\in G_{\tilde{F}_\infty}$ and a cocycle $\kappa$ representing $\gamma$ such that
\begin{enumerate}
\item[(i)] $\rhobar(\sigma_0)$ has distinct eigenvalues and
\item[(ii)] either $\tr\kappa(\sigma_0)\notin\mathfrak{m}^{s+N_c}$ or $\kappa(\sigma_0)=zI+Y$ with $z\notin\mathfrak{m}^{s+N_c}$ and $Y\in\varpi^{\val(z)+1}\Ad$.
\end{enumerate}
Granting the existence of such a $\sigma_0$ and $\kappa$ we can define $U$ to be the non-empty open subset of $G_{\tilde{F}_n}$ consisting of $\sigma$ such that 
\begin{itemize}
\item[-] $\rhobar(\sigma)=\rhobar(\sigma_0)$ and 
\item[-] $\kappa(\sigma)-\kappa(\sigma_0)\in\varpi^{s+N_c}\Ad$.
\end{itemize}
Any $\rho(\sigma)$ with $\sigma\in U$ has distinct eigenvalues mod $\mathfrak{m}$, so $\sigma$ satisfies (1) by Hensel's Lemma. To see that the elements of $U$ satisfy (2) of the lemma first consider the case that $\tr\kappa(\sigma_0)\notin\mathfrak{m}^{s+N_c}$. Then, if $\sigma\in U$ we have $\tr\kappa(\sigma)\notin\mathfrak{m}^{s+N_c}$ and, since $\tr X=0$ for any $X\in(\sigma-1)\Ad$, $\varpi^j\kappa(\sigma)\in (\sigma-1)\Ad+\varpi^m\Ad$ implies $\varpi^j\tr\kappa(\sigma)\ge m$ so $j>m-s-N_c$. Now assume $\kappa(\sigma_0)=zI+Y$ with $z\notin\mathfrak{m}^{s+N_c}$ and $Y\in\varpi^{\val(z)+1}\Ad$. Then $\kappa(\sigma)=zI+Y'$ with $Y'\in\varpi^{\val(z)+1}\Ad$. The eigenvalues of $\rho(\sigma)$ are distinct mod $\mathfrak{m}$, so \ref{Centre} implies that if $\varpi^j\kappa(\sigma)\in (\sigma-1)\Ad+\varpi^m\Ad$, then $m\le j+\val(z)$, so $j> m-s-N_c$. 

It remains to show there exists some $\sigma_0\in G_{\tilde{F}_\infty}$ and some cocycle representing $\gamma$ satisfying (i) and (ii) above. First let $\kappa$ be any cocycle representing $\gamma$. Since $\tilde{F}_\infty/F$ is a pro-2 extension, there is $\sigma\in G_{\tilde{F}_\infty}$ such that $\rhobar(\sigma)$ has distinct eigenvalues. By Hensel's Lemma, $\rho(\sigma)$ has distinct $A$-rational eigenvalues and we denote them by $\alpha$ and $\beta$. Since $\alpha$ and $\beta$ are distinct mod $\mathfrak{m}$, we can find an eigenbasis of $V$ for $\rho(\sigma)$ and we identify $\Ad$ with $\mathrm{M}_{2\times 2}(A)$ using this basis. Write
\[
\kappa(\sigma)=\left(\!\begin{array}{cc}a&b\\c&d\end{array}\!\right).
\]
Since $\alpha/\beta-1$ and $\beta/\alpha-1$ are units, we can adjust $\kappa$ by the coboundary
\[
g\mapsto g\left(\begin{array}{cc}&(\frac{\alpha}{\beta}-1)^{-1}b\\ (\frac{\beta}{\alpha}-1)^{-1}c&
\end{array}\right)g^{-1} - 
\left(\begin{array}{cc}&(\frac{\alpha}{\beta}-1)^{-1}b\\ (\frac{\beta}{\alpha}-1)^{-1}c&
\end{array}\right),
\]
and assume $b=c=0$. If at least one of $a$, $d$ is not in $\mathfrak{m}^{s+N_c}$, then we have either $\tr\kappa(\sigma)=a+d\notin\mathfrak{m}^{s+N_c}$ or
\[
\kappa(\sigma)=\left(\begin{array}{cc}a&\\&a\end{array}\right)+\left(\begin{array}{cc}0&\\&a+d
\end{array}\right)
\]
with $a\notin\mathfrak{m}^{s+N_c}$ and $\val(a+d)>\val(a)$. In either case we take $\sigma_0=\sigma$. 

Now assume that both $a,d\in\mathfrak{m}^{s+N_c}$, i.e. $\kappa(\sigma)\in\varpi^{s+N_a}\Ad$. By the choice of $N_0$ and $s$ we know that the restriction of $\gamma$ to $H^1(G_{E\tilde{F}_\infty},\Ad)^{\Gal(E\tilde{F}_\infty/F)}=\Hom_\Gamma(G_{E\tilde{F}_\infty},\Ad)$ does not belong to $\varpi^{s+N_1}\Hom_\Gamma(G_{E\tilde{F}_\infty},\Ad)$. So $\kappa(G_{E\tilde{F}_\infty})$ is a $\Gamma$ stable subgroup of $\Ad$ which is not contained in $\varpi^{s+N_1}\Ad$. By \ref{SubofAd}, there is some $\tau\in G_{E\tilde{F}_\infty}$ such that either $\tr\kappa(\tau)\notin\mathfrak{m}^{s+N_0+N_1}=\mathfrak{m}^{s+N_c}$, or $\kappa(\tau)=zI+Y$ with $z\notin\mathfrak{m}^{s+N_0+N_1}=\mathfrak{m}^{s+N_c}$ and $Y\in\varpi^{\val(z)+1}\Ad$. Set $\sigma_0=\tau\sigma$. Since $\rho(\tau\sigma)=\rho(\sigma)$, $\sigma_0$ satisfies (i). Since $\kappa(\sigma_0)=\kappa(\tau)+\kappa(\sigma)$ and $\kappa(\sigma)\in\varpi^{s+N_c}\Ad$, $\sigma_0$ satisfies (ii) by choice of $\tau$. \end{proof}

Setting $w=\max\{w_a,w_b,w_c\}$ and $N=\max\{N_a,N_b,N_c\}$, \ref{OpenSet} follows from \ref{Typea}, \ref{Typeb}, and \ref{Typec}. 
	\end{proof}

\subsubsection{}\label{auxprimedisucssion}

We now apply \ref{OpenSet} to find our sets of auxiliary primes. Some care has to be taken in choosing these primes. We will need to consider Selmer groups with coefficients in $\Ad_m$, and we need to ensure that the dual Selmer groups have size $q^{2m}$, asymptotically in $m$. To this end, given a cohomology class $\gamma\in H^1(G_{F,S},\Ad)$, we not only need to find a prime $v$ of $F$ such that the image of $\gamma$ is non-torsion in $H^1(G_v,\Ad)$, but we need to ensure that it does not lie too ``deep" in $H^1(G_v,\Ad)$. The way we make sure the cohomology class does not lie too ``deep" in $H^1(G_v,\Ad)$ is by using property (2) of \ref{OpenSet}. There is a complication that arises here. In order to use property (2) effectively we need to make sure the value of $s$ in the assumption of the lemma stays bounded. 

Let us elaborate here. We remarked above that we require the dual Selmer groups with coefficients in $\Ad_m$ to have size asymptotic to $q^{2m}$. But, we also need to ensure that the error term in the asymptotic is bounded in a way that does not depend on the choice of primes. The way one usually constructs auxiliary primes, and the way we will do it here, is inductively. One first chooses a cohomology class $\gamma\in H^1(G_{F,S},\Ad)$ and then finds a prime $v_1$ that kills $\gamma$. Then we take $\gamma_2$ that lives in the dual Selmer group for the Selmer structure given by the single prime $\{v_1\}$ and find a prime $\{v_2\}$ that kills $\gamma_2$, etc. The problem is that if the value of $s$ for which $\gamma_2$ satisfies (a), (b), or (c) of \ref{OpenSet} depends on $v_1$, then the ``depth" for which $\gamma_2$ lies in the local cohomology group will depend on $v_1$. This will cause the error term in the asymptotic to depend on the choice of the auxiliary primes. We must be careful to avoid this in the proof of the following lemma.


\begin{lem}\label{ThePrimes}

There are non-negative integers $w$ and $N$ such that for each $n\ge n_1$ there is a set of primes $Q_n$ of $F$, of cardinality $h=\mathrm{rank}_A H^1(G_{F,S},\Ad)-2$, satisfying
\begin{enumerate}
\item for each $v\in Q_n$, $\rho$ is unramified at $v$ and $\rho(\Frob_v)$ has distinct $A$-rational eigenvalues $\alpha_v,\beta_v$ with  $\val(\alpha_v-\beta_v)\le w$;
\item each $v\in Q_n$ splits in the extension $\tilde{F}_n/F$;
\item if the image of $\gamma\in W_h$ under the map
\[
W_h\longrightarrow H^1(G_{F,S},\Ad)\longrightarrow\prod_{v\in Q_n}H^1(G_v,\Ad)\longrightarrow\prod_{v\in Q_n}H^1(G_v,\Ad)_\mathrm{free}
\]
lies in $\varpi^{(2^h-1)N}\prod_{v\in Q_n}H^1(G_v,\Ad)_\mathrm{free}$, then $\gamma\in \varpi W_h$.
\end{enumerate}

\end{lem}

\begin{proof}

Let $w$ and $N$ be as in \ref{OpenSet}. Fix elements $\gamma_1,\ldots,\gamma_r\in W_h$ such that for each $1\le i\le h$, $\{\gamma_1,\ldots,\gamma_i\}$ is a basis for $W_i$. We will inductively construct a set of primes $\{v_1,\ldots,v_i\}$, for $1\le i\le r$, of $F$ such that each $v_j$ satisfies (1) and (2) above as well as the following:
\begin{enumerate}
\item[(IND${}_i$)] if the image of $\gamma\in W_i$ under the map
\[
W_i\longrightarrow H^1(G_{F,S},\Ad)\longrightarrow\prod_{j=1}^i H^1(G_{v_j},\Ad)\longrightarrow\prod_{j=1}^i H^1(G_{v_j},\Ad)_\mathrm{free}
\]
lies in $\varpi^{(2^i-1)N}\prod_{j=1}^i H^1(G_{v_j},\Ad)_\mathrm{free}$, then $\gamma\in \varpi W_i$.
\end{enumerate}
Taking $Q_n=\{v_1,\ldots,v_r\}$ establishes the lemma. In what follows, given primes $v_1,\ldots, v_i$ of $F$, we will denote the map 
\[
H^1(G_{F,S},\Ad)\longrightarrow\prod_{j=1}^i H^1(G_{v_j},\Ad)\longrightarrow\prod_{j=1}^i H^1(G_{v_j},\Ad)_\mathrm{free}
\]
by  $\mathrm{res}_i$.

First take $i=1$. Then $W_1=A\gamma_1$ and $\gamma_1$ satisfies either (a) or (b) of \ref{OpenSet} (depending on whether $r_0=1$ or $r_0>1$) with $s=0$. Hence, there is an open subset $U$ of $G_{\tilde{F}_n,\tilde{S}_n}$ such that 
\begin{enumerate}
\item[(a)] for every $\sigma\in U$, $\rho(\sigma)$ has distinct $A$-rational eigenvalues $\alpha,\beta$ with $\val(\alpha-\beta)\le w$, and
\item[(b)]  for any cocycle $\kappa$ representing $\gamma_1$ and $\sigma\in U$,
\[
\varpi^j\kappa(\sigma)\in (\sigma-1)\Ad+\varpi^m\Ad.
\]
implies $j>m-N$.
\end{enumerate}
Viewing $U$ as a subset of $G_{F,S}$ and applying Chebotarev density we obtain a prime $v_1$ of $F$ satisfying (1) and (2) of the lemma. To see that (IND${}_1$) holds, take $\gamma=a\gamma_1\in W_1$ such that $\mathrm{res_1}(\gamma)\in \varpi^{N}H^1(G_{v_1},\Ad)_\mathrm{free}$. Take $l\ge 0$ such that $\varpi^l$ annihilates $H^1(G_{v_1},\Ad)_\mathrm{tor}$. Then the image of $\varpi^l\gamma$ in $H^1(G_{v_1},\Ad)$ lies in $\varpi^{l+N}H^1(G_{v_1},\Ad)$. So, for any choice of cocycle $\kappa$ representing $\gamma_1$, we have
\[
a\varpi^l\kappa(\Frob_v)\in (\Frob_v-1)\Ad+\varpi^{l+N}\Ad,
\]
and (b) implies $\val(a)+l> l$, i.e. $\gamma\in \varpi W_1$, which is (IND${}_1$).

Now assume, for $1\le i < r$, we have primes $v_1,\ldots, v_i$ of $F$ satisfying (1) and (2) of the proposition as well as (IND${}_i$). If there is no $\gamma\in W_{i+1}\smallsetminus\varpi W_{i+1}$  such that
\[
\mathrm{res}_{i}(\gamma)\in\varpi^{(2^{i+1}-1)N}\prod_{j=1}^iH^1(G_{v_j},\Ad)_\mathrm{free},
\]
then (IND${}_{i+1}$) is automatically satisfied for any choice of $v_{i+1}$. In this case we can apply \ref{OpenSet} to $\gamma_{i+1}$ with $s=0$, and we obtain a non-empty open set $U$ of $G_{\tilde{F}_n,\tilde{S}_n}$ to which we can apply Chebotarev density as in the $i=1$ case to we get a prime $v_{i+1}$ of $F$ satisfying (1) and (2) of the lemma.

Now assume there is some $\gamma\in W_{i+1}\smallsetminus\varpi W_{i+1}$ such that
\[
\mathrm{res}_{i}(\gamma)\in\varpi^{(2^{i+1}-1)N}\prod_{j=1}^iH^1(G_{v_j},\Ad)_\mathrm{free}.
\]
The idea is now to replace $\{\gamma_1,\ldots,\gamma_{i+1}\}$ with a basis for $W_{i+1}$ that includes $\gamma$ and then apply \ref{OpenSet} to $\gamma$.

Write $\gamma=a_1\gamma_1+\cdots+a_{i+1}\gamma_{i+1}$. Since $\gamma\notin\varpi W_{i+1}$, there is at least one $1\le j\le i+1$ such that $a_j$ is a unit. Let $j_0$ be the largest such index. Then $\{\gamma_1,\ldots,\gamma_{j_0-1},\gamma,\gamma_{j_0+1},\ldots,\gamma_{i+1}\}$ is a basis for $W_{i+1}$. We first show that (IND${}_i$) also holds with $W_i$ replaced by the $A$-span of $\{\gamma_1,\ldots,\gamma_{j_0-1},\gamma_{j_0+1},\ldots,\gamma_{i+1}\}$. Let $\gamma'=b_1\gamma_1+\cdots+b_{i+1}\gamma_{i+1}$ with $b_{j_0}=0$. We'll show that if
\[
\mathrm{res}_i(\gamma')\in \varpi^{(2^i-1)N}\prod_{j=1}^i H^1(G_{v_j},\Ad)_\mathrm{free},
\]
then $b_j\in\mathfrak{m}$ for all $1\le j\le i+1$.

First assume that $\val(b_{i+1})\le\val(a_{i+1})$. Then $\gamma-\frac{a_{i+1}}{b_{i+1}}\gamma'\in W_i$, and when written in terms of the basis $\{\gamma_1,\ldots,\gamma_i\}$, the $j_0$ coefficient is $a_{j_0}-\frac{a_{i+1}}{b_{i+1}}b_{j_0}=a_{j_0}$, a unit. Thus, $\gamma-\frac{a_{i+1}}{b_{i+1}}\gamma'\notin\varpi W_i$. But
\begin{equation}\label{ResAdd}
\mathrm{res}_i\left(\gamma-\frac{a_{i+1}}{b_{i+1}}\gamma'\right)=
\mathrm{res}_i(\gamma)-\frac{a_{i+1}}{b_{i+1}}\mathrm{res}_i(\gamma')\in
\varpi^{(2^i-1)N}\prod_{j=1}^i H^1(G_{v_j},\Ad)_\mathrm{free},
\end{equation}
contradicting (IND${}_i$). So we must have $\val(b_{i+1})>\val(a_{i+1})$. Then $\gamma'-\frac{b_{i+1}}{a_{i+1}}\gamma\in W_i$, and similar to \eqref{ResAdd}, we see that
\[
\mathrm{res}_i\left(\gamma'-\frac{b_{i+1}}{a_{i+1}}\gamma\right)\in
\varpi^{(2^i-1)N}\prod_{j=1}^i H^1(G_{v_j},\Ad)_\mathrm{free}.
\]
By (IND${}_i$), we have $\gamma'-\frac{b_{i+1}}{a_{i+1}}\gamma\in\varpi W_i$. Now $\gamma'-\frac{b_{i+1}}{a_{i+1}}\gamma=(b_1-\frac{b_{i+1}}{a_{i+1}}a_1)\gamma_1+
\cdots (b_i-\frac{b_{i+1}}{a_{i+1}}a_i)\gamma_i,$ so $b_j-\frac{b_{i+1}}{a_{i+1}}a_j\in\mathfrak{m}$ for each $1\le j\le i$. Then $\val(b_{i+1})>\val(a_{i+1})$ implies $b_j\in\mathfrak{m}$ for all $1\le j\le i+1$, which is what we wanted to show.

We let $\{\delta_1,\ldots,\delta_{i+1}\}=\{\gamma_1,\ldots,\gamma_{j_0-1},\gamma,\gamma_{j_0+1},\ldots,\gamma_{i+1}\}$, ordered so that $\delta_{i+1}=\gamma$. By the above claim, if
\begin{equation}\label{NewInd}
\mathrm{res}_i(b_1\delta_1+\cdots+b_i\delta_i)\in\varpi^{(2^i-1)N}\prod_{j=1}^i
H^1(G_{v_j},\Ad)_\mathrm{free},\text{   then   }\val(b_j)\ge 1\text{ for all }1\le j\le i.
\end{equation}
We wish to apply \ref{OpenSet} to $\delta_{i+1}$, but first we need a little more information. In particular we need to know the value of $s$ for which $\delta_{i+1}$ satisfies either (a), (b) or (c) of \ref{OpenSet}. Recall we have written $\delta_{i+1}=a_1\gamma_1+\cdots+a_{i+1}\gamma_{i+1}$. We have $\val(a_{i+1})<(2^i-1)N$. Indeed, if $a_{i+1}$ is not a unit then, $j_0\le i$ and (IND${}_i$) implies
\[
\mathrm{res}_i(a_1\gamma_1+\cdots+a_i\gamma_i)\notin\varpi^{(2^i-1)N}\prod_{j=1}^i
H^1(G_{v_j},\Ad)_\mathrm{free}.\]
Then
\[
\mathrm{res}_i(a_1\gamma_1+\cdots+a_{i+1}\gamma_{i+1})\in
\varpi^{(2^i-1)N}\prod_{j=1}^iH^1(G_{v_j},\Ad)_\mathrm{free},
\]
gives
\[
\mathrm{res}_i(a_{i+1}\gamma_{i+1})\notin
\varpi^{(2^i-1)N}\prod_{j=1}^iH^1(G_{v_j},\Ad)_\mathrm{free},
\]
hence $\val(a_{i+1})<(2^i-1)N$.

By the way $\gamma_1,\ldots,\gamma_r$ were chosen, the above claim implies $\delta_{i+1}$ satisfies one of (a), (b) or (c) of \ref{OpenSet} with $s= (2^i-1)N$. Let $U$ be the open subset of $G_{\tilde{F}_n,\tilde{S}_n}$ given by applying \ref{OpenSet} to $\delta_{i+1}$. Applying Chebotarev density to $U$ we get a prime $v_{i+1}$ of $F$ that, as in explained in the case $i=1$, satisfies (1) and (2) of the lemma and such that for any cocycle $\kappa$ representing $\delta_{i+1}$, if
\[
\varpi^j\kappa(\Frob_{i+1})\in (\Frob_{i+1}-1)\Ad+\varpi^m\Ad
\]
then $j>m-(2^i-1)N-N=m-2^iN$. As explained in the $i=1$ case, this implies that the image of $\delta_{i+1}$ in $H^1(G_{v_{i+1}},\Ad)_\mathrm{free}$ does not belong to $\varpi^{2^iN}H^1(G_{v_{i+1}},\Ad)_\mathrm{free}$.

It remains to show (IND${}_{i+1}$) holds. Take $\delta\in W_{i+1}$ and assume that
\[
\mathrm{res}_{i+1}(\delta)\in\varpi^{(2^{i+1}-1)N}\prod_{j=1}^{i+1}H^1(G_{v_j},\Ad)_\mathrm{free}.
\]
Write $\delta=b_1\delta_1+\cdots+b_{i+1}\delta_{i+1}$. We will show $\val(b_j)\ge 1$ for all $1\le j\le i+1$. Since
\[
\mathrm{res}_i(\delta_{i+1})\in\varpi^{(2^{i+1}-1)N}\prod_{j=1}^i
H^1(G_{v_j},\Ad)_\mathrm{free}
\]
we have
\[
\mathrm{res}_i(b_1\delta_1+\cdots+b_i\delta_i)\in\varpi^{(2^{i+1}-1)N}\prod_{j=1}^i
H^1(G_{v_j},\Ad)_\mathrm{free}.
\]
By \eqref{NewInd}, we know that $\val(b_j)>(2^{i+1}-1)N-(2^i-1)N= 2^iN$ for each $1\le j\le i$. But then the image of $\varpi^{(2^i-1)N}(b_1\delta_1+\cdots+b_i\delta_i)$ in $H^1(G_{v_{i+1}},\Ad)_\mathrm{free}$ lands in $\varpi^{(2^{i+1}-1)N}H^1(G_{v_{i+1}},\Ad)_\mathrm{free}$. So, $\varpi^{(2^i-1)N}b_{i+1}\delta_{i+1}$ also maps to $\varpi^{(2^{i+1}-1)N}H^1(G_{v_{i+1}},\Ad)_\mathrm{free}$. Since the image of $\delta_{i+1}$ in $H^1(G_{v_{i+1}},\Ad)_\mathrm{free}$ does not belong to $\varpi^{2^iN}H^1(G_{v_{i+1}},\Ad)_\mathrm{free}$, we must have 
\[
(2^i-1)N+\val(b_{i+1})>(2^{i+1}-1)N-2^iN,
\]
which implies $\val(b_{i+1})\ge 1$. This establishes (IND${}_{i+1}$). \end{proof}

\subsubsection{}\label{cohomassympnotation}

If $V\subseteq W$ are finite sets of primes in $F$ and $M$ is an $A$-module with a continuous $G_F$ action which is unramified outside $W$, we denote by $H^1_V(G_{F,W},M)$ the subgroup of $H^1(G_{F,W},M)$ consisting of elements whose image in $\prod_{v\in V} H^1(G_v,M)$ under the restriction map is trivial.

We introduce some notation as in \cite{SWreducible}. For each $n,m\ge 1$, let $C_{n,m}$ and $D_{n,m}$ be positive integers. We write
\[
C_{n,m}\asymp D_{n,m}
\]
if there are constants $0<a<b$ such that
\[
a<\frac{C_{n,m}}{D_{n,m}}<b
\]
for all $n,m\ge 1$.


\begin{prop}\label{AuxPrimes}

For each $n\ge n_1$, there is a set of primes $Q_n$ of $F$, disjoint from $S$ and of cardinality $h=\mathrm{rank}_A H^1(G_{F,S},\Ad)-2$, such that
\begin{enumerate}
\item for $v\in Q_n$, $\rho$ is unramified at $v$, $\rho(\Frob_v)$ has distinct $A$-rational eigenvalues, and $\mathrm{val}(\tr \rho(\Frob_v)) < w$, with $w$ not depending on $n$ or $m$;
\item $v$ splits in the extension $\tilde{F}_n/F$;
\item for each $v\in Q_n$, $|H^0(G_v,\Ad_m)|\asymp q^{2m}$;
\item $|H_{Q_n}^1(G_{F,S\cup Q_n},\Ad_m)|\asymp q^{2m};$
\item letting $F_{Q_n}^S$ denote the maximal abelian extension of $F$ of degree a power of $2$ which is unramified outside $Q_n$ and split at primes in $S$, $G_n=\Gal(F_{Q_n}^S/F)$, we have $G_n/2^{n-2}G_n\cong(\Z/2^{n-2}\Z)^t$, with $t=2-|S|+|Q_n|$.
\end{enumerate}
\end{prop}

\begin{proof}

We let $Q_n$ be the set of primes given by \ref{ThePrimes}. Then (1) and (2) of the proposition are given by (1) and (2) of \ref{ThePrimes}. In particular, the bound on the valuation of the trace follows from (1) of \ref{ThePrimes}, since the characteristic is $2$.

Take $v\in Q_n$ and let $\alpha$ and $\beta$ denote the eigenvalues of $\rho(\Frob_v)$. Take $g\in\GL_2(K)$ such that
\[
g\rho(\Frob_v)g^{-1}=\left(\begin{array}{cc}\alpha&\\&\beta\end{array}\right).
\]
Identify $\Ad$ with $\mathrm{M}_{2\times 2}(A)$ using our fixed basis of $V$, and set $\Ad'=g\Ad\subset \mathrm{M}_{2\times 2}(K)$. Letting $\Frob_v$ act on $\Ad'$ via $g\rho(\Frob_v)g^{-1}$, and setting $\Ad'_m=\Ad'/\varpi^m\Ad'$, we see that $H^0(G_v,\Ad_m)\cong H^0(G_v,\Ad'_m)$. There is some $l\ge 0$ such that
\begin{equation}\label{Containment}
\varpi^l\Ad\subseteq \Ad'\subseteq \varpi^{-l}\Ad.
\end{equation}
For $\left(\begin{array}{cc}a&b\\c&d\end{array}\right)\in \mathrm{M}_{2\times 2}(K)$, 
\begin{equation}\label{FrobConj}
\left(\begin{array}{cc}\alpha&\\&\beta\end{array}\right)
\left(\begin{array}{cc}a&b\\c&d\end{array}\right)
\left(\begin{array}{cc}\alpha^{-1}&\\&\beta^{-1}\end{array}\right)
=\left(\begin{array}{cc}a&\alpha\beta^{-1}b\\ \alpha^{-1}\beta c&d\end{array}\right).
\end{equation}
This and $\Ad'\supseteq\varpi^l\Ad$ imply
\[
\left\{\left(\begin{array}{cc}a&\\&d\end{array}\right)+\varpi^m\Ad':a,d\in\mathfrak{m}^l
\right\}\big/\varpi^m\Ad'\subseteq H^0(G_v,\Ad'_m).
\]
Now $\varpi^m\Ad'\subseteq\varpi^{m-l}\Ad$ then implies
\begin{align*}
&\left|\left\{\left(\begin{array}{cc}a&\\&d\end{array}\right)+\varpi^m\Ad':a,d\in\mathfrak{m}^l
\right\}\big/\varpi^m\Ad'\right|\\
&\ge
\left|\left\{\left(\begin{array}{cc}a&\\&d\end{array}\right)+\varpi^{m-l}\Ad:a,d\in\mathfrak{m}^l
\right\}\big/\varpi^{m-l}\Ad\right|=q^{2m-4l},
\end{align*}
and
\begin{equation}\label{LowerBound}
|H^0(G_v,\Ad_m)|= |H^0(G_v,\Ad'_m)|\ge q^{2m-4l}.
\end{equation}
To get a lower bound, recall that by (1) of \ref{ThePrimes}, there is an integer $w$ that does not depend on $v$ such that $\val(\alpha\beta^{-1}-1)=\val(\alpha^{-1}\beta-1)\le w$. Using this, together with \eqref{FrobConj} and \eqref{Containment}, we have
\[
H^0(G_v,\Ad'_m)\subseteq\left\{\left(\begin{array}{cc}a&b\\c&d\end{array}\right)+\varpi^m\Ad':
a,d\in\mathfrak{m}^{-l}\;\text{and}\;b,c\in\mathfrak{m}^{m-l-w}\right\}\big/\varpi^m\Ad'.
\]
Now $\varpi^m\Ad'\supseteq\varpi^{m+l}\Ad$ then implies
\begin{align*}
&\left| \left\{\left(\begin{array}{cc}a&b\\c&d\end{array}\right)+\varpi^m\Ad':
a,d\in\mathfrak{m}^{-l}\;\text{and}\;b,c\in\mathfrak{m}^{m-l-w}\right\}\big/\varpi^m\Ad'\right|\\
&\le \left|\left\{\left(\begin{array}{cc}a&b\\c&d\end{array}\right)+\varpi^{m+l}\Ad:
a,d\in\mathfrak{m}^{-l}\;\text{and}\;b,c\in\mathfrak{m}^{m-l-w}\right\}\big/\varpi^{m+l}\Ad\right|
=q^{2m+8l+2w},
\end{align*}
hence
\begin{equation}\label{UpperBound}
|H^0(G_v,\Ad_m)|= |H^0(G_v,\Ad'_m)|\le q^{2m+8l+2w}.
\end{equation}
Since $l$ and $w$ do not depend on $n$ or $m$, \eqref{UpperBound} and \eqref{LowerBound} imply (3) of the proposition.

We now check (4) of the proposition. Since $\rho$ is unramified at each $v\in Q_n$, the injection 
\[
H^1_{Q_n}(G_{F,S\cup Q_n},\Ad)\longrightarrow H^1(G_{F,S\cup Q_n},\Ad)
\]
factors through $H^1(G_{F,S},\Ad)$. Similarly with $\Ad_m$ in place of $\Ad$. From the exact sequence
\begin{equation}\label{MultExact}
0\longrightarrow\Ad\xrightarrow{\varpi^m}\Ad\longrightarrow\Ad_m\longrightarrow 0
\end{equation}
we have
\[
0\longrightarrow H^1(G_{F,S},\Ad)/\varpi^m\longrightarrow H^1(G_{F,S},\Ad_m)
\longrightarrow H^2(G_{F,S},\Ad).
\]
Since the size of the torsion subgroups of $H^i(G_{F,S},\Ad)$ do not depend on $n$ or on $m$, we have
\begin{equation}\label{H1asymp}
|H^1(G_{F,S},\Ad_m)|\asymp |H^1(G_{F,S},\Ad)/\varpi^m|
\asymp | (A/\mathfrak{m}^m)^{h+2}|=q^{m(h+2)}.
\end{equation}
Consider our fixed submodule $W_h$ of $H^1(G_{F,S},\Ad)$. Say $\gamma\in W_h$ is such that it maps to $\varpi^m\prod_{v\in Q_n}H^1(G_v,\Ad)$. Then $\gamma$ maps to $\varpi^m\prod_{v\in Q_n}H^1(G_v,\Ad)_\mathrm{free}$. Writing $\gamma=\varpi^j\gamma'$ with $\gamma'\notin\varpi W_h$, part (3) of \ref{ThePrimes} implies $j> m-(2^h-1)N$. It follows that
\begin{equation}\label{WImSize}
|\im(W_h\rightarrow \prod_{v\in Q_n}H^1(G_v,\Ad)/\varpi^m)|\ge q^{h(m-(2^h-1)N)}.
\end{equation} 
Applying local cohomology to \eqref{MultExact} we have an injection
\[
0\longrightarrow \prod_{v\in Q_n}H^1(G_v,\Ad)/\varpi^m\longrightarrow\prod_{v\in Q_n}H^1(G_v,\Ad_m).
\]
Combining this with \eqref{WImSize} and the commutativity of
\[\xymatrix{
W_h \ar@{^{(}->}[r] & H^1(G_{F,S},\Ad) \ar[r]\ar[d] & \prod_{v\in Q_n} H^1(G_v,\Ad) \ar[d]\\
& H^1(G_{F,S},\Ad_m) \ar[r] & \prod_{v\in Q_n} H^1(G_v,\Ad_m)
}\]
we conclude
\begin{equation}\label{ImSize}
|\im(H^1(G_{F,S},\Ad_m)\rightarrow\prod_{v\in Q_n}H^1(G_v,\Ad_m))|\ge q^{h(m-(2^h-1)N)}.
\end{equation}
But, since each $v\in Q_n$ splits in $\tilde{F}_n$, we also have injections
\[
(A/\mathfrak{m}^m)^2\cong H^1(\Gal(\tilde{F}_n/F),Z_m)\longrightarrow H^1(\Gal(\tilde{F}_n/F),
(\Ad_m)^{G_{\tilde{F}_n}})\longrightarrow H^1_{Q_n}(G_{F,S\cup Q_n},
\Ad_m),
\]
so
\[
|H^1_{Q_n}(G_{F,S\cup Q_n},\Ad_m)|\ge q^{2m}.
\]
This combined with \eqref{H1asymp} and \eqref{ImSize} imply
\[
|H^1_{Q_n}(G_{F,S\cup Q_n},\Ad_m)|\asymp q^{2m},
\]
which is (4) of the proposition.

It remains to show (5). Using (3) of \ref{KummerExt}, it is shown in \cite{KW2} that (5) holds with
\[
t=\dim_{\F}H^1_{Q_n}(G_{F,S\cup Q_n},\F)-|S|+|Q_n|
\]
and so we only have to show $\dim_{\F}H^1_{Q_n}(G_{F,S\cup Q_n},\F)=2$. If $M$ is any $A$-module on which $G_F$ acts trivially, $H^1_{Q_n}(G_{F,S\cup Q_n},M)$ is the group of continuous homomorphisms from $\Gal(F_S^{Q_n}/F)$ to $M$, where $F_S^{Q_n}$ is the maximal abelian Galois extension of $F$ of exponent $2$, unramified outside $S$ and split at the primes in $Q_n$. Hence,
\[
\dim_\F H^1_{Q_n}(G_{F,S\cup Q_n},\F)=\rk_A H^1_{Q_n}(G_{F,S\cup Q_n},Z).
\]
 We have a series of injections
\[
H^1(\Gal(\tilde{F}_n/F),Z)\longrightarrow H^1_{Q_n}(G_{F,S\cup Q_n},Z)
\longrightarrow H^1_{Q_n}(G_{F,S\cup Q_n},\Ad),
\]
where the last inclusion comes from the fact that $(\Ad/Z)^{G_F}=\{0\}$. Since $\rk_A H^1(\Gal(\tilde{F}_n/F),Z)=2$, $\rk_AH^1_{Q_n}(G_{F,S\cup Q_n},Z)\ge 2$. Part (3) of \ref{ThePrimes} implies $W_h$ intersects $H^1_{Q_n}(G_{F,S\cup Q_n},\Ad)$ trivially, so $\rk_A H^1_{Q_n}(G_{F,S\cup Q_n},\Ad)\le \rk_A H^1(G_{F,S},\Ad)-\rk_A W_h=2$. Part (5) of the proposition now follows.
\end{proof}

%% file: Patching/PatchingIntro.tex
\section{Promodularity}\label{LocRequalsT}

The purpose of this section is to prove a certain $R^\mathrm{red} = \mathbf{T}$ theorem, where $R$ is a quotient of a universal deformation ring tensored with an Iwasawa algebra as in \S \ref{GlobalDefs}, and $\mathbf{T}$ is a quotient of the universal nearly ordinary Hecke algebra as in \S \ref{UnivHeckeAlg} by a minimal prime of the Iwasawa algebra.

In the first subsection, we state assumptions on our field and residual representation, recall notation and properties of the deformation rings, Hecke algebras, and Hecke modules, and we then state the localized ``$R=T$" theorem, cf. \ref{LocRredT}. An important technical point is \ref{localdefringnormal}, where we prove normality of certain localizations of our local deformation ring. This will be important in the patching argument because we will have to perform a completion after localizing at a dimension one prime. The normality implies that the completed local deformation ring is still a domain. Without this, it does not seem clear how to show the completed Hecke module is supported on the whole deformation ring.

In the next subsection, we state some reductions, introduce the auxiliary level data, and recall its relevant properties.

In the following subsection, we perform the patching argument to prove the localized ``R=T" theorem. The patching is carried out in a similar way as \cite{KW2}*{Proposition 9.3}, except that we must control ``error terms" generated from the fact that our auxiliary data is associated to a dimension one primes ideal, as opposed to the maximal ideal. After performing the patching we localize and complete the limiting objects at our fixed dimension one prime ideal. It is worth pointing out that we must perform the patching first and then the localization and completion second. This is due to the fact that the backbone of the patching argument is the pigeonhole principle, i.e. one has infinitely many finite objects, so a projective system can be extracted. The remainder of the argument is then still quite similar to \cite{KW2}*{Proposition 9.3}, due to the fact that we can ensure the completed local deformation ring remains a domain.

In the last subsection, we complete the proof of $R^\mathrm{red} = \mathbf{T}$ using the localized version together with our connectivity result \ref{Connectivity}. The argument is almost exactly the same as that of \cite{SWirreducible}*{Proposition 4.1}. 

Throughout this section we take $p=2$.

\input{Patching/LocRequalsT}

\input{Patching/RequalsT}

%% file: Patching/LocRequalsT.tex
\subsection{Notation and statement of the localized $R^\mathrm{red}=\mathbf{T}$ theorem}\label{NotationAndStatement}

\subsubsection{}\label{ResRepAssumptions}

Recall $F \subset \overline{\Q}$ denotes a totally real number field and $G_F = \Gal(\overline{\Q}/F)$. We assume that $[F:\Q]$ is even and that for each $v|2$, either $F_v$ contains a $4$-th root of unity or $[F_v:\Q_2]\ge 3$. For each place $v$ of $F$, we let $G_v = \Gal(\overline{F_v}/F)$. Let $E$ be a finite extension of $\Q_2$ with ring of integers $\mathcal{O}$ and residue field $\F$. We assume that for any $v|2$, the image of each embedding $F_v \rightarrow \overline{\Q}_2$ is contained in $E$. In what follows, all completed tensor products will be taken over $\mathcal{O}$ unless otherwise noted.

Fix an absolutely irreducible continuous representation
	\[
	\overline{\rho} : G_F \longrightarrow \GL_2(\F).
	\]
We assume that all eigenvalues of elements of $\overline{\rho}(G_F)$ lie in $\F$. We assume that for all $v|2$, $\rhobar|_{G_v}$ is trivial or has order $2$.
	
We fix a continuous character $\psi : F^\times \backslash (\A_F^\infty)^\times \rightarrow \calO^\times$ such that
	\begin{itemize}
		\item[-] $\psi$ is totally even and  unramified outside $\{v|2\}$;
		\item[-] on some open subgroup of $(\A_F^\infty)^\times$, $\psi(z) = \mathrm{Nm}_{F/\Q}(z_2)^{1-w}$ for some $w \in \Z$;
		\item[-] $\overline{\psi\epsilon_2} = \det\overline{\rho}$.
	\end{itemize}
	
Fix a finite set of finite places $\Sigma$ of $F$ of even cardinality not containing any places above $p$. For each $v \in \Sigma$, we fix unramified characters $\gamma_v : \Gal(\overline{F_v}/F_v) \rightarrow \mathcal{O}^\times$, and we assume
 	\begin{itemize}
	\item[-] for each $\in \Sigma$, $\overline{\rho}|_{G_v} \cong \left(\begin{array}{cc} \overline{\gamma_v} \overline{\epsilon_2}
	& \ast \\ & \overline{\gamma_v} \end{array} \right)$;
	\item[-] for each $v \in \Sigma$, $\gamma_v^2 = \psi|_{G_v}$;
	\item[-] $\overline{\rho}$ is unramified outside of $\Sigma\cup\{v|2\}\cup \{v|\infty\}$.
	\end{itemize}

We fix a finite place $v_0$ of $F$ disjoint from $\Sigma\cup\{v|2\}\cup\{v|\infty\}$. This place will be used to ensure a certain neatness property below.

\subsubsection{}\label{HeckeAssumptions}

Let $D$ denote the quaternion algebra with centre $F$, ramified at all archimedean places as well as all the places in $\Sigma$. Fix a maximal order $\mathcal{O}_D$ of $D$ and an algebraic weight $\kappa = (\mathbf{k},\mathbf{w})$ for $F$. Let $U$ be the open subgroup of $(D \otimes_F \A_F^\infty)^\times$ given by 
\begin{itemize}
	\item[-] $U_v = \mathrm{Iw}_1(v)$ for $v|2$;
	\item[-] $U_v = D_v^\times$ for $v \in \Sigma$;
	\item[-] $U_v = \GL_2(\mathcal{O}_{F_v})$ for $v$ not above $2$ and not in $\Sigma$.
\end{itemize}
We choose an open subgroup $U_0$ of $U$ by letting $(U_0)_v = U_v$ for $v \ne v_0$ and letting $(U_0)_{v_0} = \mathrm{Iw}_1(v_0^n)$ with $n$ sufficiently large so that $(U_0(\A_F^\infty)^\times \cap t^{-1}D t)/F^\times = 1$ for every $t \in (D \otimes_F \A_F^\infty)^\times$, cf. \ref{SuffSmallSub}.

We let $U$ act on $W_\kappa(\mathcal{O})$ as in \S\ref{QuatForms}. In particular, for $v\in \Sigma$, $D_v^\times$ acts on $W_\kappa(\mathcal{O})$ as $\gamma_v^{-1} \circ \nu_D$, where $\nu_D$ is the reduced norm of $D$. We assume that $U\cap (\A_F^\infty)^\times$ acts on $W_\kappa(\mathcal{O})$ via $\psi^{-1}$, and let $S_{\kappa,\psi}^\mathrm{no}(U,\mathcal{O})$ denote the corresponding nearly ordinary space of quaternionic modular forms, cf. \S\ref{QuatForms} and \S\ref{HeckeAlg}. We let $\mathbf{T}_{\kappa,\psi}^\mathrm{no}(U,\mathcal{O})$ denote the nearly ordinary Hecke algebra as in \S\ref{HeckeAlg}, and $\mathbf{T}_\psi(U)$ the universal nearly ordinary Hecke algebra as in \S\ref{UnivHeckeAlg}. We also let $S_\psi(U)$ be the universal family of nearly ordinary modular forms as in \S\ref{UnivHeckeAlg}, i.e. $S_\psi(U) = (\varinjlim_a S_{2,\psi}(U(p^{a,a}),E/\mathcal{O}))^\vee$. We have similar algebras and modules for $U_0$ in place of $U$.

Say we have a finite set of places $Q$ of $F$ disjoint from $S_0$. Note that $\mathrm{Nm}(v) \equiv 1 \pmod 2$ for each $v \in Q$. For each $v \in Q$, let $\Delta_v$ be the maximal $2$-power quotient of $k_v^\times$. We define $U_Q$ to be the open subgroup of $U_0$ given by $(U_Q)_v = (U_0)_v$ if $v \notin Q$, and for $v \in Q$
	\[
	(U_{Q})_v = \left\{ \left(\begin{array}{cc} a & b \\ c & d \end{array} \right) \in \mathrm{Iw}(v) :
	ad^{-1} \mapsto 1 \in \Delta_v \right\}.
	\]
We then define $S_{\kappa,\psi}^\mathrm{no}(U_Q,\mathcal{O})$, $\mathbf{T}_{\kappa,\psi}(U_Q,\mathcal{O})$, $\mathbf{T}_\psi(U_Q)$, and $S_\psi(U_Q)$ as before. Recall that for $V$ any of $U$, $U_0$, or $U_Q$, $\mathbf{T}_\psi(V)$ is a $\Lambda(\mathcal{U}_2^1)$-algebra, and the natural maps between them are $\Lambda(\mathcal{U}_2^1)$-algebra morphisms, where $\Lambda(\mathcal{U}_2^1) = \mathcal{O}[[\mathcal{U}_2^1]]$ with $\mathcal{U}_2^1 = \ker(\prod_{v|2} \mathcal{O}_{F_v}^\times \rightarrow \prod_{v|2} (\mathcal{O}_{F_v}/\varpi_v\mathcal{O}_{F_v})^\times)$.

We assume that there is some eigenform $f \in S_{\kappa,\psi}^\mathrm{no}(U,\mathcal{O})$ such that $\overline{\rho_f} \cong \overline{\rho}$, with $\rho_f$ the Galois representation as in \S\ref{ModGalReps}. Denote by $\mathfrak{m}$ the corresponding maximal ideal of $\mathbf{T}_{\kappa,\psi}^\mathrm{no}(U,\mathcal{O})$, and again denote by $\mathfrak{m}$ its pullback to any of $\mathbf{T}_\psi(U)$, $\mathbf{T}_\psi(U_0)$, $\mathbf{T}_\psi(U_Q)$. 

Recall we have $\Lambda(I_v)=\mathcal{O}[[I_v^\mathrm{ab}(2)]]$, where $I_v^\mathrm{ab}(2)$ is the inertia subgroup of $G_v^\mathrm{ab}(2)$, the maximal pro-$2$ quotient of the abelianization of $G_v$. We set $\Lambda(I_2) = \hat{\otimes}_{v|2} \Lambda(I_v)$. Local class field theory gives an isomorphism $\Lambda(I_2) \cong \Lambda(\mathcal{U}_2^1)$. Also recall, $\Lambda(G_2) = \hat{\otimes}_{v|2}\Lambda(G_v)$, where $\Lambda(G_v) = \mathcal{O}[[G_v^\mathrm{ab}(p)]]$. For $V$ any of $U$, $U_0$, or $U_Q$ as above, the $\Lambda(\mathcal{U}_2^1)$-algebra structure on $\mathbf{T}_\psi(U)_\mathfrak{m}$ extends to a $\Lambda(G_2)$-algebra structure, cf. \ref{BigGalRepLocal}. 

Let $\mu$ denote the torsion subgroup of $\mathcal{U}_2^1$. Since $E$ contains all embedding $F_v \rightarrow \Qbar_2$ for all $v|2$, the minimal primes of $\Lambda(\mathcal{U}_2^1)$ and $\Lambda(G_2)$ are in one to one correspondence with $\mathcal{O}^\times$ valued characters of $\mu$. We let $\eta$ be the character of $\mu$ given by our fixed eigenform $f$ above, and denote by $\mathfrak{q}_\eta$ the corresponding minimal primes of $\Lambda(\mathcal{U}_2^1)$ and $\Lambda(G_2)$. Set $\Lambda(\mathcal{U}_2^1,\eta) = \Lambda(\mathcal{U}_2^1)/\mathfrak{q}_\eta$ and $\Lambda(G_2,\eta) = \Lambda(G_2)/\mathfrak{q}_\eta$. Note that $\Lambda(\mathcal{U}_2^1,\eta)$ and $\Lambda(G_2,\eta)$ are isomorphic to power series rings over $\mathcal{O}$ in $[F:\Q]$ and $|\{v|2\}|+[F:\Q]$ variables, respectively. 

\subsubsection{}\label{thesetS}

We now specify some finite places of $F$ at which our deformation problem will be unramified. This may seem redundant, but is important for two reasons. The first is that below we will chose a dimension one characteristic $2$ prime ideal of the Hecke algebra and it will be important that (after a choice of framing) the local deformation ring surjects onto the Hecke algebra modulo the prime ideal in order to compare tangent spaces of the local and global deformation rings localized at our fixed prime ideal. The second is to guarantee the freeness of a certain group action, cf. \ref{FreeTwistAction}, which is necessary for the $2$-adic patching method. To these ends, first choose a finite set of places $\{v_1,\ldots,v_k\}$ disjoint from $\Sigma\cup \{v|2\}\cup\{v|\infty\}\cup\{v_0\}$, such that $\mathbf{T}_\psi(U,\eta)_\mathfrak{m} \cong \Lambda(\mathcal{U}_2^1)[T_{v_1},\ldots,T_{v_k}][T_{\varpi_v}]_{v|2}$, cf. \ref{Heckefingencor}. By enlarging $\{v_1,\ldots,v_k\}$ if necessary we can assume that if $\rhobar$ is dihedral, and $L$ denotes the unique quadratic extension of $F$ for which $\rhobar|_{G_L}$ is abelian, there is some $v_i \in \{v_1,\ldots,v_k\}$ that is inert in $L$. We set $S_\mathrm{ur} = \{v_1,\ldots,v_k\}$ and
	\[ S = \Sigma\cup\{v|2\}\cup\{v|\infty\}\cup S_\mathrm{ur}
	\]
We then set $S_0 = S\cup \{v_0\}$. 

\subsubsection{}\label{DefTheoryAssumptions}

Let $Q$ be a (possibly empty) set of places of $F$ disjoint from $S$. We let $R_{F,S\cup Q}^\psi$ denote the universal deformation ring for $G_{F,S\cup Q}$-deformations of $\overline{\rho}$ with determinant $\psi\epsilon_2$. We let $R_{F,S\cup Q}^{\square,\psi}$ denote the universal framed deformation ring for framed $G_{F,S\cup Q}$-deformations of $\overline{\rho}$ with determinant $\psi\epsilon_2$ and frames at places in $S$, cf. \S\ref{GlobalDefs}. We let $R_{F,S\cup Q}^\square$ denote the universal framed deformation ring for $G_{F,S\cup Q}$-deformations of $\overline{\rho}$ with frames at $S$, and with determinant $\psi\epsilon_2|_{G_v}$ for each $v \in S$, but not fixed globally, c.f. \S\ref{GlobalDefs}. We similarly define $R_{F,S_0\cup Q}^\psi$, $R_{F,S_0\cup Q}^{\square,\psi}$, and $R_{F,S_0 \cup Q}^\square$.

For each $v \in S_0$, we let $R_v^{\square,\psi}$ denote the universal lifting ring for lifts of $\overline{\rho}|_{G_v}$ with determinant $\psi\epsilon_2$. Set $R_S^{\square,\psi} = \hat{\otimes}_{v\in S} R_v^{\square,\psi}$ and $R_{S_0}^{\square,\psi} = \hat{\otimes}_{v\in S_0} R_v^{\square,\psi}$. For each $v\in S_0$, we define quotients of $\overline{R}_v^{\square,\psi}$ of $R_v^{\square,\psi}$ if $v\nmid p$, and of $R_v^{\square,\psi}\hat{\otimes}\Lambda(G_v,\eta_v)$ for $v|2$, as follows.
\begin{itemize}
	\item[-]  For $v|2$, $\overline{R}_v^{\square,\psi} = R_{\Lambda(G_v,\eta_v)}^{\triangle,\psi}$ with $R_{\Lambda(G_v,\eta_v)}^{\triangle,\psi}$ as in \ref{triliftringdef}. Since we are assuming $\rhobar|_{G_v}$ is either trivial or has order $2$ image, and that $F_v$ contains a $4$-th root of unity or satisfies $[F_v:\Q_2] \ge 3$, we have
	\begin{itemize}
		\item[-]  By \ref{triliftringdomain}, $\overline{R}_v^{\square,\psi}$ is a domain of relative $\mathcal{O}$-dimension $3+2[F_v:\Q_p]$. Moreover, if $Z_v$ denotes the closed subscheme of $\Spec \Lambda(G_v,\eta_v)$ defined by $(\chi_{\eta_v}^\mathrm{univ})^2 = \psi\epsilon_2$, and $V_v$ denotes its complement, 
		\[ (\Spec \overline{R}_v^{\square,\psi}\times_{\Spec \Lambda(G_v,\eta_v)} V_v)
		\otimes_\mathcal{O} \F
		\]
		 is integral.
		\item[-] By \ref{triliftringsmpts}, if $x = (\rho_x,\chi_x)$ is a closed point of $\Spec\overline{R}_v^{\square,\psi}[1/p]$ such that $\chi_x^2 \ne \psi$ or $\psi\epsilon_2$, then $\overline{R}_v^{\square,\psi}$ is formally smooth over $E$ at $x$.
			\end{itemize}
	\item[-] For $v\in \Sigma$, $\overline{R}_v^{\square,\psi} = R_v^{\square,\gamma\text{-st}}$ as in \ref{SemiStlp}. It is a domain of relative $\mathcal{O}$-dimension $3$, $\overline{R}_v^{\square,\psi}[1/p]$ is formally smooth over $E$, and $\overline{R}_v^{\square,\psi}\otimes_\mathcal{O} \F$ is a domain.
	\item[-] For $v|\infty$, $\overline{R}_v^{\square,\psi} = R_v^{\square,-1}$ as in \ref{OddLifts}. It is a domain of relative $\mathcal{O}$-dimension $2$, $\overline{R}_v^{\square,\psi}[1/p]$ is formally smooth over $E$, and $\overline{R}_v^{\square,\psi}\otimes_\mathcal{O} \F$ is a domain.
	\item[-] For $v\in S_\mathrm{ur}$, $\overline{R}_v^{\square,\psi} = R_v^{\square,\psi,\mathrm{ur}}$ as in (2) of \ref{FixedDetlp}. It is formally smooth over $\mathcal{O}$ of relative dimension $3$.
	\item[-] $\overline{R}_{v_0}^{\square,\psi}$ is as in \ref{FixedDetlp}. It is equidimensional and $\mathcal{O}$-flat of relative $\mathcal{O}$-dimension $3$. There is a minimal prime $\mathfrak{q}_\mathrm{ur}$ such that a lift factors through $\overline{R}_{v_0}^{\square,\psi}/\mathfrak{q}_\mathrm{ur}$ if and only if it is unramified. The quotient $\overline{R}_{v_0}^{\square,\psi}/\mathfrak{q}_\mathrm{ur}$ is formally smooth over $\mathcal{O}$.
\end{itemize}

We set $\overline{R}_S^{\square,\psi} = \hat{\otimes}_{v\in S} \overline{R}_v^{\square,\psi}$ and $\overline{R}_{S_0}^{\square,\psi} = \hat{\otimes}_{v\in S_0} \overline{R}_v^{\square,\psi}$

\begin{lem}\label{locdefringdim}

$\overline{R}_{S_0}^{\square,\psi}$ is $\mathcal{O}$-flat and equidimensional of relative dimension $3|S_0|+[F:\Q_p]$.
Any minimal prime of $\overline{R}_{S_0}^{\square,\psi}$ is of the form $\mathfrak{q}\overline{R}_{S_0}^{\square,\psi}$ with $\mathfrak{q}$ a minimal prime of $\overline{R}_{v_0}^{\square,\psi}$.

\end{lem}

\begin{proof}

By \ref{triliftringdomain}, \ref{FixedDetlp}, \ref{SemiStlp}, and \ref{OddLifts}, and (3) of \ref{TensorCNLO}, $\overline{R}_{S_0}^{\square,\psi}$ is $\mathcal{O}$-flat of relative dimension
	\[ \sum_{v|2} 3+2[F_v:\Q_p]  + \sum_{v|\infty} 2\
	+ \sum_{v\in\Sigma\cup S_{\mathrm{ur}}\cup\{v_0\}} 3
	= 3|S_0|+[F:\Q_p],
	\]
and the claim regarding minimal primes follows from part (4) of \ref{TensorCNLO}.
	\end{proof}

In particular, for $\mathfrak{q}_\mathrm{ur}$ the minimal prime of $R_{v_0}^{\square,\psi}$ corresponding to unramified lifts, \ref{locdefringdim} implies that $\mathfrak{q}_\mathrm{ur}$ generates a minimal prime of $\overline{R}_{S_0}^{\square,\psi}$ which we again denote by $\mathfrak{q}_\mathrm{ur}$.

Let $\chi_{\eta_v}^\mathrm{univ} : G_v \rightarrow \Lambda(G_v,\eta_v)^\times$ denote the universal $\Lambda(G_v,\eta_v)$-valued character.

\begin{lem}\label{localdefringnormal}

Let $Z_1$ denote the closed subscheme of $\Spec \Lambda(G_2,\eta)$ defined by $(\chi_{\eta_v}^\mathrm{univ})^2 = \psi$ for some $v|2$, and $Z_2$ the closed subscheme defined by $(\chi_{\eta_v}^\mathrm{univ})^2 = \psi\epsilon_2$ for some $v|2$. Let $U$ denote the complement of $Z_1\cup Z_2$, and let $f: \Spec \overline{R}_{S_0}^{\square,\psi}/\mathfrak{q}_\mathrm{ur} \rightarrow \Spec \Lambda(G_2,\eta)$ denote the natural morphism. If $\mathfrak{p} \in f^{-1}(U)$, then $(\overline{R}_{S_0}^{\square,\psi}/\mathfrak{q}_\mathrm{ur})_\mathfrak{p}$ is normal.

\end{lem}

\begin{proof}

For ease of notation, we set $R = \overline{R}_{S_0}^{\square,\psi}/\mathfrak{q}_\mathrm{ur}$, $R_v = \overline{R}_v^{\square,\psi}$ for $v\in S$, and $R_{v_0} = R_{v_0}^{\square,\psi}/\mathfrak{q}_\mathrm{ur}$. Note that $R \cong \hat{\otimes}_{v\in S_0} R_v$. We check Serre's conditions $(S_2)$ and $(R_1)$. Since $f^{-1}(U)$ is open it suffices to check 
\begin{enumerate}
	\item[(a)] for any $\mathfrak{p} \in f^{-1}(U)$, we have $\depth R_\mathfrak{p} \ge\min\{i,\mathrm{ht}\mathfrak{p}\}$ and
	\item[(b)] for any $\mathfrak{p} \in f^{-1}(U)$ with $\mathrm{ht}\mathfrak{p} = 1$, $R_\mathfrak{p}$ is regular.
\end{enumerate}

Take $\mathfrak{p} \in f^{-1}(U)$. First assume that $2 \notin \mathfrak{p}$. Since $R[1/2]$ is Jacobson, we can find a finite extension $E'/E$ and an $E'$ valued point $x: R[1/2] \rightarrow E'$ whose kernel lies in $f^{-1}(U)$ and contains $\mathfrak{p}$. Write $x = (x_v)$ for $x_v : R_v \rightarrow E'$, for each $v \in S_0$. Part (2) of \ref{FixedDetlp}, \ref{SemiStlp}, and \ref{OddLifts}, imply that $R_v$ is smooth over $E$ for each $v \in S_0$ with $v\nmid 2$. For $v|2$, since the kernel of $x$ lies in $f^{-1}(U)$, \ref{triliftringsmpts} implies that $R_v$ is smooth over $E$ at $x_v$.  Part (1) of \ref{TensorCNLO} shows that $R[1/2]$ is smooth over $E$ at $x$. Since $\mathfrak{p}$ is contained in the kernel of $x$, $\mathfrak{p}$ is regular. This establishes both (a) and (b) for primes in $f^{-1}(U)[1/2]$.

Now take $\mathfrak{p} \in f^{-1}(U)$ with $2 \in \mathfrak{p}$. Let $\varpi_E$ denote a uniformizer of $E$. Let $N$ denote the nilradical of $R/\varpi_E R$, and for each $v\in S_0$, let $N_v$ denote the nilradical of $R_v/\varpi_E R_v$. Part (2) of \ref{FixedDetlp}, \ref{SemiStlp}, and \ref{OddLifts}, imply that $N_v=0$ for each $v \in S_0$ with $v\nmid 2$. For $v|2$, \ref{triliftringdomain} implies that the support of $N_v$ in $\Spec \Lambda(G_v,\eta_v)$ is contained in the closed subscheme defined by $(\chi_{\eta_v}^\mathrm{univ})^2 = \psi\epsilon_2$. Then part (4) of \ref{TensorCNLO} implies that the support of $N$ is contained in $f^{-1}(Z_2)$, so $R_\mathfrak{p}/\varpi_E R_\mathfrak{p} \cong (R/\varpi_E R)_\mathfrak{p}$ is reduced. If $\mathrm{ht}\mathfrak{p} = 1$, then $R_\mathfrak{p}/\varpi_E R_\mathfrak{p}$ is a field, and $\mathfrak{p}$ is regular. If $\mathrm{ht}\mathfrak{p} >1$, then $R_\mathfrak{p}/\varpi_E R_\mathfrak{p}$ reduced and of dimension $\ge 1$ implies the existence of a non-zero divisor in its maximal ideal; hence, $\depth R_\mathfrak{p} \ge 2$.
	\end{proof}

As in \ref{QuotientPres}, we define quotients 
\begin{itemize}
	\item[-] $R_{F,S\cup Q}^{\square,\psi}\hat{\otimes}\Lambda(G_2,\eta) \rightarrow \overline{R}_{F,S\cup Q}^{\square,\psi}$,
	\item[-] $R_{F,S_0\cup Q}^{\square,\psi}\hat{\otimes}\Lambda(G_2,\eta) \rightarrow \overline{R}_{F,S_0\cup Q}^{\square,\psi}$, and
	\item[-]$R_{F,S_0\cup Q}^\square\hat{\otimes}\Lambda(G_2,\eta) \rightarrow \overline{R}_{F,S_0\cup Q}^\square$
\end{itemize}
 by letting 
 \begin{itemize}
 	\item[-] $\overline{R}_{F,S\cup Q}^{\square,\psi} = \overline{R}_S^{\square,\psi} \otimes_{R_S^{\square,\psi}} R_{F,S\cup Q}^{\square,\psi}$,
	\item[-] $\overline{R}_{F,S_0\cup Q}^{\square,\psi} = \overline{R}_{S_0}^{\square,\psi} \otimes_{R_{S_0}^{\square,\psi}} R_{F,S_0\cup Q}^{\square,\psi}$, and 
	\item[-] $\overline{R}_{F,S_0\cup Q}^\square = \overline{R}_{S_0}^{\square,\psi} \otimes_{R_{S_0}^{\square,\psi}} R_{F,S_0\cup Q}^\square$. 
\end{itemize}
We define a quotient $R_{F,S\cup Q}^\psi \hat{\otimes}\Lambda(G_2,\eta) \rightarrow \overline{R}_{F,S\cup Q}^\psi$ by letting $\overline{R}_{F,S\cup Q}^\psi$ be the image of $R_{F,S\cup Q}^\psi\hat{\otimes}\Lambda(G_2,\eta)$ under the natural map
	\[ R_{F,S\cup Q}^\psi\hat{\otimes}\Lambda(G_2,\eta)
	 \longrightarrow R_{F,S \cup Q}^{\square,\psi} \longrightarrow 
	\overline{R}_{F,S \cup Q}^{\square,\psi},
	\]
and similarly with $S_0$ in place of $S$. Note that if $E'/E$ is finite with ring of integers $\mathcal{O}_{E'}$, a local $\mathcal{O}$-algebra morphism $\overline{R}_{F,S_0}^{\square,\psi}\rightarrow \mathcal{O}_{E'}$ has kernel lying over $\mathfrak{q}_\mathrm{ur}\in \Spec R_{S_0}^{\square,\psi}$ if and only if the induced morphism $\overline{R}_{F,S_0}^\psi \rightarrow \mathcal{O}_{E'}$ factors through $\overline{R}_{F,S}^\psi$.

By \ref{BigGalRepLocal}, the existence of $f$ in \ref{HeckeAssumptions} yields surjective morphisms $\overline{R}_{F,S}^\psi \rightarrow \mathbf{T}_\psi(U,\eta)_\mathfrak{m}$, $\overline{R}_{F,S_0}^\psi \rightarrow \mathbf{T}_\psi(U_0,\eta)_\mathfrak{m}$, and $\overline{R}_{F,S_0\cup Q}^\psi \rightarrow \mathbf{T}_\psi(U_Q,\eta)_\mathfrak{m}$. These are all morphisms of $\Lambda(G_2,\eta)$-algebras, and the natural diagrams all commute.

If $R$ is either of $\overline{R}_{F,S\cup Q}^\psi$ or $\overline{R}_{F,S_0\cup Q}^\psi$, $A$ is a $\mathrm{CNL}_\mathcal{O}$-algebra, and $x \in \Spf R(A)$, we let $\rho_x$ denote the pushforward of  the universal deformation by $x$. If $\mathfrak{p} \in \Spec R$, we denote by $\rho_\mathfrak{p}$ the pushforward of the universal deformation by $R \rightarrow R/\mathfrak{p}$.

\subsubsection{}\label{SWPrimeDef}

As in \cite{SWreducible} and \cite{SWirreducible}, we say a prime $\mathfrak{p}\in\Spec\overline{R}^\psi_{F,S}$ is \textit{pro-modular} if it is the inverse image of a prime of $\textbf{T}_\psi(U,\eta)_\mathfrak{m}$. We say a closed subset of $\Spec\overline{R}^\psi_{F,S}$ is \textit{pro-modular} if every prime in it is pro-modular. Note that an irreducible component is pro-modular if and only if its corresponding minimal prime is pro-modular. We say a prime $\mathfrak{p}\in\Spec\overline{R}^\psi_{F,S}$ is \textit{nice} if
\begin{enumerate}
\item[(a)] $\mathfrak{p}$ is pro-modular;
\item[(b)] $\mathfrak{p}$ is dimension one and containing $2$;
\item[(c)] $\rho_\mathfrak{p}$ is absolutely irreducible and non-dihedral;
\item[(d)] for each $v|2$, the image of $\mathfrak{p}$ in $\Spec\Lambda(G_v,\eta_v)$ does not lie in the closed subscheme defined by $(\chi_{\eta_v}^\mathrm{univ})^2 = \psi\epsilon_2$;
\item[(e)] the image of $\rho_\mathfrak{p}$ contains a non-trivial unipotent element.
\end{enumerate}

\noindent We are now in a position to state the the localized ``$R=T$" theorem.

\begin{prop}\label{LocRredT} 

With the notation and assumptions as above, if $\mathfrak{p} \in \Spec \overline{R}_{F,S}^\psi$ is a nice prime, then every prime of $\overline{R}_{F,S}^\psi$ contained in $\mathfrak{p}$ is pro-modular.
\end{prop}

\subsection{The Setup}\label{the Setup}

The proof of \ref{LocRredT} will be carried out in a number of steps. Fix a nice prime $\mathfrak{p}$ and let $A =  \overline{R}_{F,S}^\psi/\mathfrak{p}$.

\subsubsection{}\label{StoS0}

We have a commutative diagram
	\[
	\xymatrix{
	\overline{R}_{F,S_0}^\psi \ar[r] \ar[d] & \mathbf{T}_\psi(U_0,\eta)_\mathfrak{m} 
	\ar[d]\\
	\overline{R}_{F,S}^\psi \ar[r] & \mathbf{T}_\psi(U,\eta)_\mathfrak{m} }
	\]
with all arrows surjective. Pull back $\mathfrak{p}$ to a prime $\mathfrak{p}_0$ of $\mathbf{T}_\psi(U_0,\eta)_\mathfrak{m}$, and denote again by $\mathfrak{p}_0$ its pullback to $\overline{R}_{F,S_0}^\psi$. Let $X$ denote the Zariski closure in $\Spec \overline{R}_{F,S_0}^\psi$ of the set of points $x \in \Spf\overline{R}_{F,S_0}^\psi(\mathcal{O}_{E'})$ whose corresponding deformations are unramified at $v_0$, as $E'$ ranges over all finite extensions of $E$. Note that $X$ is the image of $\Spec (\overline{R}_{F,S_0}^{\square,\psi}/\mathfrak{q}_\mathrm{ur})$ under $\Spec \overline{R}_{F,S_0}^{\square,\psi} \rightarrow \Spec\overline{R}_{F,S_0}^\psi$. From this it follows that $X$ is also equal to the image of $\Spec \overline{R}_{F,S}^\psi \rightarrow \Spec \overline{R}_{F,S_0}^\psi$. Consider the commutative diagram
	\[
	\xymatrix{
		\Spec \mathbf{T}(U,\eta)_\mathfrak{m} \ar[r] \ar[d] &
		\Spec \overline{R}_{F,S}^\psi \ar[d]\\
		\Spec \mathbf{T}(U_0,\eta)_\mathfrak{m}  \ar[r] &
		\Spec \overline{R}_{F,S_0}^\psi .	}
	\]
Let $\mathfrak{q}$ be a minimal prime of $\mathbf{T}_\psi(U_0,\eta)_\mathfrak{m}$ whose image in $\Spec \overline{R}_{F,S_0}^\psi$ lies in $X$. We know that any arithmetic prime contained in $\mathfrak{q}$ is in the image of $\Spec \mathbf{T}_\psi(U,\eta)_\mathfrak{m}$. By Zariski density of arithmetic primes, cf. \ref{ArPrimeZarDense}, $\mathfrak{q}$ must be in the image of $\Spec\mathbf{T}_\psi(U,\eta)_\mathfrak{m}$. It thus suffices to prove that any element of $X \cap \Spec (\overline{R}_{F,S_0}^\psi)_{\mathfrak{p}_0} \subset \Spec \overline{R}_{F,S_0}^\psi$ is in the image of $\Spec \mathbf{T}(U_0,\eta)_{\mathfrak{p}_0}$. 

Recall $\overline{R}_{F,S_0}^{\square,\psi}$ is isomorphic to a power series ring over $\overline{R}_{F,S_0}^\psi$ in $4|S_0|-1$ variables, cf. \ref{globalquotient}. Set $j = 4|S_0|-1$, and choose a presentation $\overline{R}_{F,S_0}^{\square,\psi} \cong \overline{R}_{F,S_0}^\psi[[y_1,\ldots,y_j]]$. Let $\overline{R}_{F,S_0}^{\square,\psi} \rightarrow \overline{R}_{F,S_0}^\psi$ be the map sending each $y_i$ to zero. Set $\mathbf{T}_\psi^\square(U_0,\eta)_\mathfrak{m} =  \overline{R}_{F,S_0}^{\square,\psi} \hat{\otimes}_{\overline{R}_{F,S_0}^\psi} \mathbf{T}_\psi(U_0,\eta)_\mathfrak{m}$. We have a map $\mathbf{T}_\psi^\square(U_0,\eta)_\mathfrak{m} \rightarrow \mathbf{T}_\psi(U_0,\eta)_\mathfrak{m}$, by sending each $y_i$ to zero. Pull back $\mathfrak{p}_0$ to ideals of $\mathbf{T}_\psi^\square(U_0,\eta)_\mathfrak{m}$ and $\overline{R}_{F,S}^{\square,\psi}$, and denote each again by $\mathfrak{p}_0$. In order to show that any element of $X \cap \Spec (\overline{R}_{F,S_0}^\psi)_{\mathfrak{p}_0} \subset \Spec \overline{R}_{F,S_0}^\psi$ is in the image of $\Spec \mathbf{T}(U_0,\eta)_{\mathfrak{p}_0}$, it suffices to show that any element of $\Spec(\overline{R}_{F,S_0}^{\square,\psi})_{\mathfrak{p}_0}$ lying over $\mathfrak{q}_\mathrm{ur} \in \Spec R_{S_0}^{\square,\psi}$ is in the image of $\Spec \mathbf{T}_\psi^\square(U_0,\eta)_{\mathfrak{p}_0} \rightarrow \Spec(\overline{R}_{F,S_0}^{\square,\psi})_{\mathfrak{p}_0}$.
	
Finally, recall that we have a faithful $\mathbf{T}_\psi(U_0,\eta)_\mathfrak{m}$-module $S_\psi(U_0,\eta)_\mathfrak{m}$. Set $S_\psi^\square(U_0,\eta)_\mathfrak{m} =  \overline{R}_{F,S_0}^{\square,\psi} \hat{\otimes}_{\overline{R}_{F,S_0}^\psi} S_\psi(U_0,\eta)_\mathfrak{m}$; this is a faithful $\mathbf{T}_\psi^\square(U_0,\eta)_\mathfrak{m}$-module. To prove \ref{LocRredT}, it suffices to show that if $\mathfrak{q} \in \Spec (\overline{R}_{F,S_0}^{\square,\psi})_{\mathfrak{p}_0}$ lies over $\mathfrak{q}_\mathrm{ur} \in \Spec \overline{R}_{S_0}^{\square,\psi}$, then $\mathfrak{q}$ is in the support of $S_\psi^\square(U_0,\eta)_\mathfrak{m}$. This is what we will prove.

\subsubsection{}\label{somereductions}

For ease of notation, set $\Lambda = \Lambda(\mathcal{U}_p^1,\eta)=\Lambda(I_p,\eta)$, $R_\mathrm{loc} =  \overline{R}_{S_0}^{\square,\psi}$, $R_0' = \overline{R}_{F,S_0}^\square $, $R_0 =  \overline{R}_{F,S_0}^{\square,\psi}$, and $M_0 = S_\psi^\square(U_0,\eta)_\mathfrak{m} $. Also set $B = \Lambda[[y_1,\ldots,y_j]]$. Let $\mathfrak{p}_\Lambda$ denote the pullback of $\mathfrak{p}$ to $\Lambda$. Note that $R_0'$ and $R_0$ are $B$-algebras, and the surjection $R_0'\rightarrow R_0$ is a $B$-algebra morphism. We denote by $\mathfrak{p}_0'$ the pullback of $\mathfrak{p}$ to $R_0'$, and $\mathfrak{p}_\mathrm{loc}$ its pullback to $R_\mathrm{loc}$. By our assumption on $(U_0)_{v_0}$, \ref{VertControl} and \ref{WeightChangeInf} imply that $S_\psi(U,\eta)_\mathfrak{m}$ is finite free over $\Lambda$, thus $M_0$ is finite free of the same rank over $B$.

\begin{lem}\label{UniqueComp}
There is a prime $\mathfrak{q}\in \Spec R_0$ contained in $\mathfrak{p}_0$ and in the support of $M_0$ such that the irreducible component of $\Spec R_\mathrm{loc}$ determined by $\mathfrak{q}_\mathrm{ur}$ is the unique irreducible component containing the image of $\mathfrak{q}$.
\end{lem}

\begin{proof}

By \ref{locdefringdim}, $\Spec R_\mathrm{loc} \rightarrow \Spec \overline{R}_{v_0}^{\square,\psi}$ induces a bijection on irreducible components. Hence, it suffices to show the existence of $\mathfrak{q}\in \Spec R_0$ contained in $\mathfrak{p}_0$ and in the support of $M_0$ such that the irreducible component of $\Spec \overline{R}_{v_0}^{\square,\psi}$ determined by $\mathfrak{q}_\mathrm{ur}$ is the unique irreducible component containing the image of $\mathfrak{q}$ under $\Spec R_0 \rightarrow \Spec \overline{R}_{v_0}^{\square,\psi}$.

Let $\mathfrak{q}$ be a minimal prime of $\mathbf{T}_\psi^\square(U)_\mathfrak{m}$ contained in $\mathfrak{p}$, and write $\mathfrak{q}$ again for its pullback to $R_0$. Choose an arithmetic $\mathcal{O}$-morphism $\lambda_f : \mathbf{T}_\psi^\square(U) \rightarrow \Qbar_p$ whose kernel contains $\mathfrak{q}$. Since $\rho_f$ is unramified at $v_0$, we know that the image of $\ker \lambda_f$ in $\Spec \overline{R}_{v_0}^{\square,\psi}$ lies in the irreducible component determined by $\mathfrak{q}_\mathrm{ur}$. If it was contained in more than one irreducible component, \ref{FixedDetlp} shows that  $\rho_f|_{G_{v_0}} \cong \gamma\epsilon_2 \oplus \gamma$, for some unramified character $\gamma$ of $G_v$. But using local-global compatibility, this contradicts the fact that the automorphic representation $\pi_f$ of $\GL_2(\A_F)$ corresponding to $f$ via the Jacquet-Langlands-Shimizu correspondence is generic at $v_0$. Since the irreducible component of $\overline{R}_{v_0}^{\square,\psi}$ determined by $\mathfrak{q}_\mathrm{ur}$ is the unique one containing $\ker x$, it is also the unique one containing $\mathfrak{q}$.		\end{proof}

\subsubsection{}\label{nthLevelData}

Note that the map $R_0 \rightarrow A$ determines a fixed tuple $(V_A, \{\chi_v\}_{v|2},\{\beta_{A,v}\}_{v\in S_0})$, with $V_A$ a deformation of $\overline{\rho}$ to $A$, $\chi_v$ an $A$ valued character of $G_v$ for each $v|2$, and $\beta_{A,v}$ a basis for $V_A$ reducing to our fixed basis of $\overline{\rho}$ for each $v\in S_0$.

We need to use the results of \S\ref{AuxPrimesSec}. Let $K$ be the fraction field of $A$, and let $A'$ be the integral closure of $A$ in $K$. If we view $\rho_{\mathfrak{p}}$ as taking values in $\GL_2(A')$, then all the assumptions of \S \ref{AuxPrimesSec} are satisfied (with $\rho = \rho_{\mathfrak{p}}$ and $A'$ in place of $A$), except possibly for the assumption A5, i.e. that in the case where $\rhobar$ is $L$-dihedral, there is some $\tau_0 \in G_F\smallsetminus G_L$ such that  $\rho_\mathfrak{p}(\tau_0)$ has distinct infinite order $A'$-rational eigenvalues (note that (d) of the definition of nice primes ensures that A2 holds even in the case that $\rhobar$ is not dihedral). Assuming $\overline{\rho}$ is dihedral, \ref{DihedCrit} implies there is some $\tau_0 \in G_F$ such that $\overline{\rho}(\tau_0)$ has order $2$, but $\rho_\mathfrak{p}(\tau_0)$ has infinite order. Since $2\in\mathfrak{p}$, $\det\rho_\mathfrak{p}$ is finite, and unipotent elements of $\rho_\mathfrak{p}(G_F)$ have finite order, we deduce that $\rho_\mathfrak{p}(\tau_0)$ has distinct infinite order eigenvalues. Since $\rhobar(\tau_0)$ does not have distinct eigenvalues, the eigenvalues of $\rho_\mathfrak{p}(\tau_0)$ lie in the ring of integers $A''$ of a quadratic ramified extension $K''/K$. The representation $\rho_\mathfrak{p} \otimes_A A''$ then satisfies the assumptions of \S\ref{AuxPrimesSec}. Applying the results of that section, we fix $n_1 \ge 1$ as in \ref{n1Def}, and for each $n \ge n_1$ we fix a finite set of finite places $Q_n$ as in \ref{AuxPrimes}. Recall that $|Q_n|$ has cardinality independent of $n$, and we denote it by $h$. For each $n \ge 1$, in the notation of \ref{HeckeAssumptions} and \ref{DefTheoryAssumptions}, we set
	\[
	R_n' = \overline{R}_{F,S_0\cup Q_n}^\square 
	\hspace{0.5cm}\text{and}\hspace{0.5cm}
	R_n = \overline{R}_{F,S_0\cup Q_n}^{\square,\psi} .
	\]
Fix a choice for the framing of $\overline{R}_{F,S_0\cup Q_n}^\square$ over $\overline{R}_{F,S_0\cup Q_n}$ compatible with the surjection $\overline{R}_{F,S_0\cup Q_n}^\square \rightarrow \overline{R}_{F,S_0}^\square$. This then gives a $B$-algebra structure to $R_n'$ and $R_n$ and we have a commutative diagram in $\CNL_\mathcal{O}$ 
	\[\xymatrix{
		R_\mathrm{loc} \ar[dr] \ar[r] & R_n' \ar[d] \ar[r] & R_n \ar[d] \\
		 & R_0' \ar[r] & R_0}
	\]
such that each is a morphism of $\Lambda$-algebras, and each of $R'_n \rightarrow R_n$, $R_n' \rightarrow R_0'$ and $R_n \rightarrow R_0$ are surjective morphisms of $B$-algebras. Let $\mathfrak{p}_n$ and $\mathfrak{p}_n'$ denote the pullbacks of $\mathfrak{p}$ to $R_n$ and $R_n'$, respectively.

\begin{lem}\label{genoverloc}
\begin{enumerate} \item Let $k = 1+2h$. For any $n\ge n_1$, there an injection of $A$-modules
	\[ A^k \longrightarrow \mathfrak{p}_n'/((\mathfrak{p}_n')^2 + \mathfrak{p}_\mathrm{loc})	\]
whose cokernel is finite of size bounded independent of $n$.
\item The minimal number of generators of the maximal ideal of $R_n$ is bounded independently of $n$.
\end{enumerate}
\end{lem}

\begin{proof}

We first prove (1). Recall $K$ is the fraction field of $A$. Let $\mathfrak{t}_n = \mathfrak{p}_n'/((\mathfrak{p}_n')^2 + \mathfrak{p}_\mathrm{loc})$. Take $x_1,\ldots,x_{k_n} \in \mathfrak{t}_n$ such that the images of $x_1,\ldots,x_{k_n}$ in $\mathfrak{t}_n/\mathfrak{m}_A$ are linearly independent and such that $x_1,\ldots,x_{k_n}$ form a basis for $\mathfrak{t}_n\otimes_A K$. This gives a map $\varphi : A^{k_n} \rightarrow \mathfrak{t}_n$. We will show that $k_n=k$ and $\coker(\varphi))$ is finite of size bounded independently of $n$.

Recall $A'$ denotes the integral closure of $A$ in $K$. Set $\mathfrak{t}_n' = \mathfrak{t}_n \otimes_A A'$.  I claim that $\mathfrak{t}_n \rightarrow \mathfrak{t}_n'$ is an injection. As both are finitely generated over $A$, it suffices to show injectivity modulo $\mathfrak{m}_A$, by Nakayama's Lemma. This follows from the fact that the composite
	\[ \mathfrak{t}_n/\mathfrak{m}_A \longrightarrow \mathfrak{t}_n'/\mathfrak{m}_A
	\longrightarrow \mathfrak{t}_n'/\mathfrak{m}_{A'} \]
is the isomorphism $\mathfrak{t}_n \otimes_A \F \cong (\mathfrak{t}_n \otimes_A A') \otimes_{A'} \F$.

Identifying $x_1,\ldots,x_{k_n}$ with their images in $\mathfrak{t}_n'$, we have the commutative diagram
	\[
	\xymatrix{
		0 \ar[r] & Ax_1+\cdots Ax_{k_n} \ar[r]^-\varphi \ar[d] & \mathfrak{t}_n \ar[r] \ar[d] & \coker(\varphi)
		\ar[r] \ar[d]^f & 0 \\
		0 \ar[r] & A'x_1+\cdots+A'x_{k_n} \ar[r]^-{\varphi'} & \mathfrak{t}_n' \ar[r] &
		\coker(\varphi') \ar[r] & 0 }
	\]
and the snake lemma gives an injection $\ker(f) \rightarrow (A'/A)^{k_n}$. If $k_n=k$, then $(A'/A)^{k_n}$ is finite of size bounded independently of $n$. So, it suffices to show the $A'$-rank of $\mathfrak{t}_n'$ is $k$, and that $\coker(\varphi')$ is finite of size bounded independently of $n$.

Recall that $A''$ is the ring of integers in a quadratic ramified extension $K''/K$ so that $\rho_\mathfrak{p}\otimes_A A''$ satisfies the assumptions of \S\ref{AuxPrimesSec}. Set $\mathfrak{t}_n'' = \mathfrak{t}_n' \otimes_{A'} A'' = \mathfrak{t}_n \otimes_A A''$. Since $A''$ is free of rank $2$ over $A'$, it suffices to show that the $A''$-rank of $\mathfrak{t}_n''$ is $k$, and that the cokernel of $\varphi'' : A''x_1+\cdots A''x_{k_n} \rightarrow \mathfrak{t}_n''$ is finite of order bounded independently of $n$.

We recall some notation from \ref{cohomassympnotation}. For each $n,m\ge 1$, given positive $C_{n,m}$ and $D_{n,m}$, we write
\[
C_{n,m}\asymp D_{n,m}
\]
if there are constants $0<a<b$ such that
\[
a<\frac{C_{n,m}}{D_{n,m}}<b
\]
for all $n,m\ge 1$. If $M$ is a finite abelian group with a continuous $G_{F,S_0\cup Q_n}$-action and $V \subset S_0\cup Q_n$, we let $H_V^i(G_{F,S_0\cup Q_n}, M)$ denote the subgroup of $H_V^i(G_{F,S_0\cup Q_n}, M)$ consisting of elements whose restriction to $H^i(G_v,M)$ is trivial for each $v \in V$. Let $\Ad$ denote the set of $2\times 2$ matrices over $A''$ with the adjoint $G_{F,S_0\cup Q_n}$-action. Set $\Ad_m = \Ad/\mathfrak{m}_{A''}^m\Ad$. Let $q = |\F|$. 

Since $\mathfrak{t}_n''$ is a finitely generated $A''$-module, we can write $\mathfrak{t}_n'' \cong (A'')^{k_n} \oplus T_n$, with $T_n$ finite. By our choice of $x_1,\ldots,x_{k_n}$, $\coker(\varphi'') \cong T_n$. Since $T_n$ is finite, $\Hom_{A''}(T_n,K''/A'') \cong T_n$, and 
	\[
	\Hom_{A''}(\mathfrak{t}_n'',K''/A'') \cong (K''/A'')^{k_n} \times T_n.
	\]
It then suffices to show that $\Hom_{A''}(\mathfrak{t}_n'',A''/\mathfrak{m}_{A''}^m) \asymp (q^m)^k$.
Let $\varpi_{K''}$ be a uniformizer for $K''$. We have
	\[
	\Hom_{A''}(\mathfrak{t}_n'',A''/\mathfrak{m}_{A''}^m) \cong
	\Hom_A (\mathfrak{t}_n, A''/\mathfrak{m}_{A''}^m) 
	 \]
and this latter space is identified with the set of $\phi \in \Hom_{R_\mathrm{loc}}(R_n', A''[\varepsilon]/(\varepsilon^2,\varpi_{K''}\varepsilon))$ such that
	\[
	R_n' \stackrel{\phi}{\longrightarrow} A''[\varepsilon]/(\varepsilon^2,\varpi_{K''}^m\varepsilon) \stackrel{\varepsilon\mapsto 0}{\longrightarrow} A''
	\]
is the map $R_n' \rightarrow A \rightarrow A''$. This Hom set is then identified with tuples $(V,\{\beta_v\}_{v\in S_0})$, where 
	\begin{itemize}
	\item[-] $V$ is a $G_{F,S_0\cup Q_n}$-deformation to $A''[\varepsilon]/(\varepsilon^2,\varpi_{K''}^m\varepsilon)$ that lifts the $A''$-deformation $V_A \otimes_A A''$, where $V_A$ is our fixed deformation corresponding to $\mathfrak{p}$;
	\item[-] each $\beta_v$ is a basis of $V$ lifting the basis of $V_A \otimes_A A''$ determined by our fixed basis $\beta_{A,v}$ of $V_A$,
	\end{itemize}
such that 
	\begin{itemize}
	\item[-] for each $v\in S_0$, the lift determined by $V|_{G_v}$ and the basis $\beta_v$ is equal to the lift given by $V_A|_{G_v} \otimes_A A''[\varepsilon]/(\varepsilon^2,\varpi_{K''}^m\varepsilon)$ and $\beta_{A,v}$.
	\end{itemize}
Note that there is no need to specify the characters of $G_v$ as they are determined by the fact that our morphisms are $R_\mathrm{loc}$-morphisms. The set of such tuples surjects onto the set of deformations lifting $V_A \otimes_A A''$ with fixed restriction to $G_v$ for each $v \in S_0$. A standard argument shows that this space of deformations is isomorphic to $H^1_{S_0}(G_{F,S_0\cup Q_n},\Ad_m)$. The fibre over any given deformation $V$ is given by sets of basis $\{\beta_v\}_{v\in S_0}$ reducing to our fixed set of bases modulo $\varepsilon$, such that the lift given by $V|_{G_v}$ and $\beta_v$ is independent of the element in the fibre, up to equivalence by automorphisms of $V$ reducing to the identity modulo $\varepsilon$ that commute with the $G_{F,S_0\cup Q_n}$-action.

Putting this all together, we have
	\begin{equation}\label{firstselmer}
	|\Hom_A(\mathfrak{t}_n,A''/\mathfrak{m}_{A''}^m)| = \frac{
	|H^1_{S_0}(G_{F,S_0\cup Q_n},\Ad_m)|\prod_{v\in S_0}|H^0(G_v,\Ad_m)|}
	{|H^0(G_F,\Ad_m)|}.
	\end{equation}
Since the trace pairing on $\Ad_m$ is perfect and the characteristic of $K''$ is $2$, the Pontryagin dual of $\Ad_m$ is $G_{F,S_0\cup Q_n}$-isomorphic to itself, and the dual Selmer group to $H^1_{S_0}(G_{F,S_0\cup Q_n},\Ad_m)$ is $H^1_{Q_n}(G_{F,S_0\cup Q_n},\Ad_m)$. Then, the Greenberg-Wiles formula, c.f. \cite{DDTFermat}*{Theorem 2.19}, together with \eqref{firstselmer} yields
	\[
	|\Hom_A(\mathfrak{t}_n,A''/\mathfrak{m}_{A''}^m)| = \frac{
	|H^1_{Q_n}(G_{F,S_0\cup Q_n},\Ad_m)|}{|H^0(G_F,\Ad_m)|}
	\prod_{v\in Q_n}\frac{|H^1(G_v,\Ad_m)|}{|H^0(G_v,\Ad_m)|}.
	\]
Local Tate duality and Euler-Poincar\'{e} characteristic give
	\begin{equation}\label{afterWilesform}
	|\Hom_A(\mathfrak{t}_n,A''/\mathfrak{m}_{A''}^m)| = \frac{
	|H^1_{Q_n}(G_{F,S_0\cup Q_n},\Ad_m)|}{|H^0(G_F,\Ad_m)|}
	\prod_{v\in Q_n} |H^0(G_v,\Ad_m)|.
	\end{equation}
Since $V_A\otimes_A A''$ is absolutely irreducible, $|H^0(G_F,\Ad_m)|\asymp q^m$. Then parts (3) and (4) of \ref{AuxPrimes}  with \eqref{afterWilesform} imply
	\[
	|\Hom_A(\mathfrak{t}_n,A''/\mathfrak{m}_{A''}^m)| \asymp (q^m)^{1+2|Q_n|} = q^{mk}
	\]
which is what was required to prove.

To see (2), it suffices to show the minimal number of generators of $R_n'$ over $R_\mathrm{loc}$ is bounded independently of $n$. Letting $\Ad_\F$ denote the space of $2\times 2$ matrices over $\F$ with the adjoint $G_{F,S_0\cup Q_n}$-action, a similar argument to above shows that the minimal number of generators is bounded above by 
	\[ \dim_\F H^1_{Q_n}(G_{F,S_0\cup Q_n},\Ad_\F) - \dim_\F H^0(G_F,\Ad_\F) +
	\sum_{v\in Q_n} \dim_\F H^0(G_v,\Ad_\F) \le \dim_\F H^1(G_{F,S_0},\Ad_\F) -1 +4h.
	\]	\end{proof}

\subsubsection{}\label{PatchingTwists}

We let $F_{Q_n}^{S_0}$ denote the maximal abelian $p$-extension unramified outside $Q_n$ and split at all primes in $S_0$. Set $G_n = \Gal(F_{Q_n}^{S_0}/F)$. Part (5) of \ref{AuxPrimes} shows that $G_n/2^{n-2}G_n \cong (\Z/2^{n-2}\Z)^t$ with $t = 2 - |S_0| + |Q_n|$. Let $G_n^\ast$ denote the diagonalizable $\mathcal{O}$-group associated to $G_n$ as in \ref{GlobalDefTwists}. Recall that an element of $G_n^\ast(A)$ is a character $\chi : G_n \rightarrow A^\times$ reducing to the trivial character modulo the maximal ideal of $A$.

By \ref{FreeTwistAction} and our assumptions on $S$, cf. \ref{thesetS}, there is a free action of $G_n^\ast$ on $\Spf R_n'$, and the morphism $\Spf R_n' \rightarrow \Spf R_\mathrm{loc}$ is constant on orbits. We also have an action of $G_{n,2}^\ast$, the $2$-torsion of $G_n^\ast$, on $\Spf R_n$, and a function $\delta_{Q_n} : \Spf R_n' \rightarrow G_n^\ast$ such that that $R_n'\rightarrow R_n$ identifies $\Spf R_n$ with the closed sub-formal-scheme of $\Spf R_n'$ defined by $\delta_{Q_n}=1$, and for any $\CNL_\mathcal{O}$-algebra $A$, $\alpha\in G_n^\ast(A)$, and $x \in \Spf R_n'(A)$, we have $\delta_{Q_n}(\alpha x) = \alpha^2\delta_{Q_n}(x)$, cf. \ref{GlobalDefTwists}. We also note that for any $\CNL_\mathcal{O}$-algebra $A$, the stucture map $\mathcal{O} \rightarrow A$ gives a group homomorphism $G_n^\ast(\mathcal{O}) \rightarrow G_n^\ast(A)$. In this way the finite group $G_n^\ast(\mathcal{O})$ acts on $\Spf R_n'$; hence, also on $R_n'$. Similarly, $G_{n,2}^\ast(\mathcal{O})$ acts on $R_n$.

\subsubsection{}\label{PatchingModules}

For each $n\ge n_1$, let $S_\psi(U_{Q_n},\eta)$ be as in \ref{HeckeAssumptions}. We set $M_n = \overline{R}_{F,S_0\cup Q_n}^{\square,\psi} \otimes_{\overline{R}_{F,S_0\cup Q_n}^\psi} S_\psi(U_{Q_n},\eta)_\mathfrak{m}$. Note that $M_n$ is an $R_n$-module and we have a surjection $M_n \rightarrow M_0^{2^h}$ of $R_n$-modules induced by the degeneracy map in \ref{DegenMap}. Here, $R_n$ acts on $M_0$ via the surjection $R_n \rightarrow R_0$.

Write $Q_n = \{v_1,\ldots,v_h\}$ and define two power series rings $B[[s_1,\ldots,s_h]]$ and $B[[t_1,\ldots,t_h]]$.  We view $B[[t_1,\ldots,t_h]]$ as a subring of $B[[s_1,\ldots,s_h]]$ by $t_i \mapsto (1+s_i)+(1+s_i)^{-1}-2$. For each $v_i\in Q_n$, fix a generator $\sigma_i$ of the $p$-part of the tame inertia group of $F_{v_i}$. The map $I_{v_i} \rightarrow \mathcal{O}_{F_{v_i}}^\times \rightarrow k_{v_i}^\times \rightarrow \Delta_{v_i}$ given by class field theory sends $\sigma_i$ to a generator $\delta_i$ of $\Delta_{v_i}$. Let $V$ be tautological deformation to $\overline{R}_{F,S_0\cup Q_n}$. We define a local $B$-algebra morphism $B[[t_1,\ldots,t_h]] \rightarrow R_n$ by sending $2+t_i$ to the trace of $\sigma_i$ acting on $V$. We also define a $B[[s_1,\ldots,s_h]]$-module structure on $M_n$ by letting $s_i$ act via the action of $\delta_i$ on $S_\psi(U_{Q_n})_\mathfrak{m}$. By \ref{AuxPrimeTrace}, the two $B[[t_1,\ldots,t_h]]$-module structures on $M_n$ given by $B[[s_1,\ldots,s_h]]$ and $R_n$ coincide. Note that under the surjection $R_n \rightarrow R_0$, each $t_i$ is mapped to zero.

We remark that the power series ring $B[[t_1,\ldots,t_h]]$ is introduced because we may not be able to define a $B[[s_1,\ldots,s_h]]$-algebra structure on $R_n$. Although $\rhobar$ is unramified at each $v \in Q_n$, its Frobenius eigenvalues may not be distinct, so the local lift to $R_n$ may not be split and we may not have a morphism $\Delta_v \rightarrow R_n^\times$. This is why we introduce this subring $B[[t_1,\ldots,t_h]]$ of ``traces".

\begin{lem}\label{MnProperties}

\begin{enumerate}
	\item There is $s\ge 1$, independent of $n$, such that letting $\mathfrak{b}_n$ denote the annihilator of $M_n$ in $B[[s_1,\ldots,s_h]]$, 
		\[ \mathfrak{b}_n \subseteq ((1+s_1)^{2^n}-1,\ldots,(1+s_h)^{2^n}-1)
		\]
		and $M_n$ is free over $B[[s_1,\ldots,s_h]]/\mathfrak{b}_n$ of rank $s$.
	\item $(s_1,\ldots,s_h)M_n \subseteq \ker (M_n \rightarrow M_0^{2^h})$. 
	\item There is $\lambda_M \in \Lambda$, $\lambda_M \notin \mathfrak{p}_{\Lambda}$ and independent of $n$ such that letting $N_n = \ker(M_n/(s_1,\ldots,s_h)M_n \rightarrow M_0^{2^h})$, $\lambda_M N_n\subseteq \mathfrak{p}_nN_n$.
	\item There is an action of the finite group $G_{n,2}^\ast(\mathcal{O})$ on $M_n$ such that for $\alpha\in G_{n,2}^\ast(\mathcal{O})$, $r \in R_n$, and $m\in M_n$, $r(\alpha m) = \alpha((\alpha r)m)$.
	\item Letting $G_{n,2}^\ast(\mathcal{O})$ act on $B[[s_1,\ldots,s_h]]$ by $\chi s_i = \chi(\delta_i)(1+s_i)-1$, we also have $s_i(\alpha m) = \alpha((\alpha s_i)m)$.
\end{enumerate}
\end{lem}

\begin{proof}
Parts (1) and (2) follow immediately from \ref{AuxPrimeControl} and our assumption on $v_0$. Parts (4) and (5) follow from \ref{modulartwists}.

We now show (3). By (1) of \ref{AuxPrimes}, there is $w$ independent of $n$ such that for any $v\in Q_n$, $\val_K(\rho_\mathfrak{p}(\Frob_v)) < w$, where $\val_K$ denotes the valuation of the fraction field $K$ of $A$, normalized to give a uniformizer valuation $1$. Let $\lambda_M \in \Lambda$ be any element whose image in $A$ is nonzero and has valuation at least $2w$. By \ref{NoSteinberg}, there is $y_v \in R_n$ lifting $(\tr \rho_\mathfrak{p}(\Frob_v))^2$ such that $y_vN_n = 0$. In particular $y_v(N_n/\mathfrak{p}_nN_n) = 0$. Since the valuation of $y_v$ modulo $\mathfrak{p}_n$ is bounded above by $2w$, we have $\lambda_M(N_n/\mathfrak{p}_n N_n) = 0$.	\end{proof}

\subsection{Patching and proof of \ref{LocRredT}}\label{Patching}

Denote by $\mathfrak{T} $ the formal $\CNLO$-torus $(\Z^t)^\ast$, where $t = 2 - |S_0| + |Q_n|$, cf. \ref{PatchingTwists}. For any $n\ge 1$, we let $\mathfrak{T}_{2^n}$ denote the $2^n$-torsion subgroup of $\mathfrak{T}$.

\begin{lem}\label{Rinf}

Let $k = 1+4h = 1+4|S_0|$ as in \ref{genoverloc}. Let $\CNL_{\Lambda}$ denote the full subcategory of $\CNL_\mathcal{O}$ consisting of $\Lambda$-algebras. We have 
\begin{itemize}
	\item[-] $\CNL_{\Lambda}$ objects $R_\infty'$, $R_\infty$, a power series $R_\mathrm{loc}[[x_1,\ldots,x_k]]$, and an $R_\infty \times B[[s_1,\ldots,s_h]]$-module $M_\infty$;
	\item[-] $\CNL_{\Lambda}$ morphisms $R_\mathrm{loc}[[x_1,\ldots,x_k]] \rightarrow R_\infty' \rightarrow R_\infty$, $R_\infty' \rightarrow R_0'$, $R_\infty \rightarrow R_0$, and $B[[t_1,\ldots,t_h]] \rightarrow R_\infty$, and a morphism of $R_\infty \times B[[s_1,\ldots,s_h]]$-modules $M_\infty \rightarrow M_0^{2^h}$;
\end{itemize}
such that the following hold.
\begin{enumerate}
	\item The diagrams
	\[
		\xymatrix{
		R_\mathrm{loc}[[x_1,\ldots,x_k]] \ar[r] & R_\infty' \ar[d] \ar[r] & R_\infty \ar[d] \\
		 R_\mathrm{loc} \ar[u] \ar[r] & R_0' \ar[r] & R_0}
		\hspace{1.0cm}\text{and}\hspace{1.0cm}
		\xymatrix{
		B[[t_1,\ldots,t_h]] \ar[r] & R_\infty \ar[d] \\
		B \ar[u] \ar[r] & R_0}
		\]
	both commute and the two $B[[t_1,\ldots,t_h]]$-module structures on $M_\infty$ (coming from $R_\infty$ and $B[[s_1,\ldots,s_h]]$) coincide.
	\item The morphisms $R_\infty' \rightarrow R_\infty$, $R_\infty' \rightarrow R_0'$, $R_\infty \rightarrow R_0$, and $M_\infty \rightarrow M_0^{2^h}$ are all surjections.
	\item $(t_1,\ldots, t_h)R_\infty \in \ker(R_\infty\rightarrow R_0)$ and $(s_1,\ldots,s_h)M_\infty \in \ker(M_\infty\rightarrow M_0^{2^h})$.
	\item Letting $N_\infty = \ker(M_\infty/(s_1,\ldots,s_h) \rightarrow M_0^{2^h})$ and letting $\mathfrak{p}_\infty$ be the pullback of $\mathfrak{p}_0$ to $R_\infty$, $N_\infty/\mathfrak{p}_\infty N_\infty$ is a torsion $A$-module.
	\item Letting $\mathfrak{p}_\infty'$ be the pullback of $\mathfrak{p}_0'$ to $R_\infty'$, $(\mathfrak{p}_\mathrm{loc},x_1,\ldots,x_k)R_\infty'\subseteq\mathfrak{p}_\infty'$ and $\mathfrak{p}_\infty'/((\mathfrak{p}_\infty')^2+(\mathfrak{p}_\mathrm{loc},x_1,\ldots,x_k)R_\infty')$ is a torsion $A$-module.
	\item There is a free action of $\mathfrak{T}$ on $\Spf R_\infty'$ and a map $\delta_\infty:\Spf R_\infty'\rightarrow\mathfrak{T}$ such that
		\begin{itemize}
		\item[--] the morphism $\Spf R_\infty'\rightarrow\Spf R_\mathrm{loc}$ is constant on orbits of $\mathfrak{T}$,
		\item[--] the closed immersion $\Spf R_\infty\rightarrow\Spf R_\infty'$ identifies $\Spf R_\infty$ with the closed subfunctor of points $x$ with $\delta_\infty(x) =1$ and
		\item[--] for any object $A$ in $\CNLO$, $x\in\Spf R_\infty'(A)$ and $\alpha\in\mathfrak{T}(A)$, we have $\delta_\infty(\alpha x)=\alpha^2\delta_\infty(x)$.
		\end{itemize}
	\item The subfunctor $\Spf R_\infty\rightarrow\Spf R_\infty'$ is stable under the action of $\mathfrak{T}_2$, and there is an action of $\mathfrak{T}_2(\mathcal{O})$ on $M_\infty$ satisfying the following compatibility condition: for any $m\in M_\infty$, $r\in R_\infty$, and $\alpha\in \mathfrak{T}_2(\mathcal{O})$ we have $r(\alpha m)=\alpha((\alpha r)m)$.
\end{enumerate}
\end{lem}

\begin{proof}

The proof is similar to the contsruction in \cite{KW2}*{Proposition 9.3}.

For a $\CNLO$-algebra $A$ with maximal ideal $\mathfrak{m}_A$ and $r\ge 1$, we let $\mathfrak{m}_A^{(r)}$ denote the ideal of $A$ generated by elements of $\mathfrak{m}_A$ that are $r$\textsuperscript{th} powers. Note that if $\mathfrak{m}_A$ can be generated by $g$ elements, then $\mathfrak{m}_A^{gr}\subseteq\mathfrak{m}_A^{(r)}$. Let $s\ge 1$ be as in (1) of \ref{MnProperties}. For each $m\ge 1$, set $r_m=sm(h+j)2^m$ and
	\[
	\mathfrak{c}_m=(\mathfrak{m}_\Lambda^m,y_1^{2^m},\ldots,y_j^{2^m},
	(1+s_1)^{2^m}-1,\ldots,(1+s_h)^{2^m}-1)\subset B[[s_1,\ldots,s_h]].
	\]
Let $\mathfrak{d}_m = \mathfrak{c}_m \cap B[[t_1,\ldots,t_h]]$. As in the proof of \cite{KisinFinFlat}*{Proposition 3.3.1} we can show that for any $m\ge 1$, $M_0/\mathfrak{c}_mM_0$ is an $R_0/(\mathfrak{d}_mR_0+\mathfrak{m}_{R_0}^{(r_m)})$-module. Similarly, for any $n\ge n_1$ and $m\ge 1$, we can show that $M_n/\mathfrak{c}_mM_n$ is an $R_n/(\mathfrak{d}_mR_n+\mathfrak{m}_{R_n}^{(r_m)})$-module.

Fix $\lambda_R \in \Lambda$ with non-zero image in $A$ such that $\lambda_R$ annihilates the cokernel of the map in (1) of \ref{genoverloc} for all $n\ge n_1$. Also let $\lambda_M$ be as in (3) of \ref{MnProperties}. 

Let $G_n'=G_n/2^{n-2}G_n$. As in \cite{KW2}, we fix a surjection $\Z^t\rightarrow G_n$ for each $n\ge n_0$. By \ref{PatchingTwists}, this induces an isomorphism $\Z/2^{n-2}\Z\cong G_n'$ and a closed embedding $G_n^\ast\rightarrow\mathfrak{T}$ which identifies $G_n'^\ast$ with $\mathfrak{T}_{2^{n-2}}$.

Let $g$ by such that the maximal ideal of $R_n$ is genererated by $\le g$ elements for all $n\ge n_1$, cf. (2) of \ref{genoverloc}.
	
We let $\CNL_{R_\mathrm{loc}}$ denote the full subcategory of $\CNL_\mathcal{O}$ consisting of $R_\mathrm{loc}$-algebras. We recall some notation of \ref{groupchunks}. We let $\CNL_\mathcal{O}^{[d]}$ and $\CNL_\mathcal{O}^{[d]}$ be the full subcategories of $\CNL_\mathcal{O}$ and $\CNL_{R_\mathrm{loc}}$, respectively, consisting of objects $A$ such that $\mathfrak{m}_A^d = 0$. Given a $\CNLO$-algebra $A$, we let $A^{[d]} = A/\mathfrak{m}_A^d$. Given $X = \Spf A$ in $\CNLO^\mathrm{op}$, we let $X^{[d]} = \Spf A^{[d]}$. For $m\ge 1$, a patching datum of level $m$, denoted $(D_m',D_m,L_m)$, consists of the following.
\begin{enumerate}
	\item[(a)] A surjective $\CNL_{R_\mathrm{loc}}$ morphism
		\[ D_m\longrightarrow R_0/(\mathfrak{d}_mR_0 + \mathfrak{m}_{R_0}^{(r_m)}) \]
		which is also a $B[[t_1,\ldots,t_h]]$-algebra morphism, such that $\mathfrak{m}_{D_m}^{(r_m)}=(0)$.
	\item[(b)] A object $D_m'$ of $\CNL_{R_\mathrm{loc}}^{[gr_m]}$ such that $\mathfrak{m}_{D_m'}$ can be generated by $g$ elements, a surjective $\CNL_{R_\mathrm{loc}}^{[gr_m]}$ morphism
		\[ D_m'\longrightarrow D_m, \]
	a free action of $\mathfrak{T}_{2^m}^{[gr_m]}$ on $D_m'$ such that the map $\Spf D_m'\rightarrow\Spf R_\mathrm{loc}$ is constant on orbits of $\mathfrak{T}_{2^m}^{[gr_m]}$, and a morphism $\delta_m:\Spf D_m'\rightarrow\mathfrak{T}$, such that $\delta_m(\alpha x)=\alpha^2\delta_m(x)$ for any $A$ in $\CNL_O^{[gr_m]}$, $x\in \Spf D_m'(A)$ and $\alpha\in\mathfrak{T}_{2^m}(A)$. Let $D_m''$ be the object in $\CNLO^{[gr_m]}$ representing the closed subfunctor of $\Spf D_m'$ given by points $x$ with $\delta_m(x)=1$. We further demand that the surjection $D_m'\rightarrow D_m$ factors through $D_m'\rightarrow D_m''$ and that the kernel of $D_m''\rightarrow D_m$ is contained in $\mathfrak{m}_{D_m''}^m$.
	\item[(c)] A $D_m \times B[[s_1,\ldots,s_h]]$-module $L_m$, which is finite free of rank $s$ over $B[[s_1,\ldots,s_h]]/\mathfrak{c}_m$ and a surjection of $D_m$-modules $L_m\rightarrow M_0^{2^h}/\mathfrak{c}_m M_0^{2^h}$ (the $D_m$-module structure on $M_0^{2^h}/\mathfrak{c}_m M_0^{2^h}$ is via the surjection $D_m\rightarrow R_0/(\mathfrak{d}_m+\mathfrak{m}_{R_0}^{(r_m)})$ in (a)), such that letting $\mathfrak{a}_m$ denote the inverse image in $D_m$ of the image of $\mathfrak{p}_0$ in $R_0/(\mathfrak{d}_mR_0+\mathfrak{m}_{R_0}^{(r_m)})$,
		\[ \lambda_M \ker(L_m /(s_1,\ldots,s_h) \rightarrow M_0^{2^h}/\mathfrak{c}_m M_0^{2^h}) \subseteq
		\mathfrak{a}_m \ker(L_m /(s_1,\ldots,s_h) \rightarrow M_0^{2^h}/\mathfrak{c}_m M_0^{2^h}).
		\]
	\item[(d)] A $\CNL_{R_\mathrm{loc}}$ morphism
		\[ R_\mathrm{loc}[[x_1,\ldots,x_k]]\longrightarrow D_m' \]
	such that, letting $\mathfrak{a}_m'$ denote the inverse image in $D_m'$ of the image of $\mathfrak{p}_0$ in $R_0/(\mathfrak{d}_mR_0+\mathfrak{m}_{R_0}^{(r_m)})$, we have $(\mathfrak{p}_\mathrm{loc},x_1,\ldots,x_k)D_m'\subseteq\mathfrak{a}_m'$ and
		\[ \lambda_R(\mathfrak{a}_m'/((\mathfrak{a}_m')^2+(\mathfrak{p}_\mathrm{loc},x_1,\ldots,x_k)D_m'))\subseteq\mathfrak{m}_{D_m'}^m.
		\]
\end{enumerate}

We define a morphism $(D'_m\!{}^{1},D_m^1,L_m^1)\rightarrow(D_m'\!{}^2,D_m^2,L_m^2)$ of patching data of level $m$ to be surjective $\CNL_{R_\mathrm{loc}}$ morphisms $D_m'\!{}^1\rightarrow D_m'\!{}^2$ and $D_m^1\rightarrow D_m^2$, and a surjection of $D_m^1$-modules $L_m^1\rightarrow L_m^2$, such that
	\begin{itemize}
	\item[--] $D_m^1\rightarrow D_m^2$ is a $B[[t_1,\ldots,t_h]]$-algebra morphism and is compatible with the surjections to $R_0/(\mathfrak{d}_m+\mathfrak{m}_{R_0}^{(r_m)})$ of (a);
	\item[--] $D_m'\!{}^1\rightarrow D_m'\!{}^2$ is compatible with the $\mathfrak{T}_{2^m}^{[gr_m]}$-action and with the morphisms $\delta_m^i:\Spf D_m'\!{}^i\rightarrow\mathfrak{T}$ from (b), and
	\[\xymatrix{ D_m'\!{}^1 \ar[r] \ar[d] & D_m^1\ar[d] \\ D_m'\!{}^2 \ar[r] & D_m^2 } \]
	commutes;
	\item[--] $D_m'\!{}^1\rightarrow D_m'\!{}^2$ is compatible with the morphisms $R_\mathrm{loc}[[x_1,\ldots,x_k]]\rightarrow D_m'\!{}^i$ from (d).
	\end{itemize}

Fix $n\ge n_1$ and $n_1\le m\le n$. We now show how to define a patching datum $(D_{m,n}',D_{m,n},L_{m,n})$ of level $m$, from $R_{n+2}'$, $R_{n+2}$, and $M_{n+2}$.

Set $D_{m,n}'=R'_{n+2}/\mathfrak{m}_{R'_{n+2}}^{gr_m}$, $D_{m,n}=R_{n+2}/(\mathfrak{d}_mR_{n+2}+\mathfrak{m}_{R_{n+2}}^{(r_m)})$, and $L_{m,n}=M_{n+2}/\mathfrak{c}_mM_{n+2}$. The surjection $R_n\rightarrow R_0$ is an $R_\mathrm{loc} \times B[[t_1,\ldots,t_h]]$-algebra morphism, so we get a morphism $D_{m,n}\rightarrow R_0/(\mathfrak{d}_mR_0+\mathfrak{m}_{R_0}^{(r_m)})$ satisfying the properties of (a) in the definition of a patching datum of level $m$.

As noted above, $L_{m,n}$ is a $D_{m,n}$-module and the surjection $M_n\rightarrow M_0^{2^h}$ of $R_n$-modules induces a surjection $L_{m,n}\rightarrow M_0^{2^h}/\mathfrak{c}_m M_0^{2^h}$ of $D_{m,n}$-modules. That $L_{m,n}$ satisfies the required properties follows from \ref{MnProperties}.

We now show part (b). By choice of $g$, $\mathfrak{m}_{D_{m,n}'}$ can be generated by $g$ elements. As noted above, we have $\mathfrak{m}_{R_n'}^{gr_m}\subseteq\mathfrak{m}_{R_n'}^{(r_m)}$, so the surjection $R_n'\rightarrow R_n$ induces a surjection $D_{m,n}'\rightarrow D_{m,n}$. From
	\[\mathfrak{T}_{2^m}\longrightarrow \mathfrak{T}_{2^n}\cong 
	G_{n+2}'^\ast\longrightarrow G_{n+2}^\ast, \]
the free action of $G_{n+2}^\ast$ on $\Spf R_{n+2}'$ yields a free action of $\mathfrak{T}_{2^m}$ on $\Spf R_{n+2}'$. This gives rise to a free group action chunk in $\CNLO^{[gr_m]}$, cf. \ref{groupchunks}, of $\mathfrak{T}_{2^m}^{[gr_m]}$ on $\Spf D_{m,n}'$. Now let $\delta:\Spf R_{n+2}'\rightarrow G_{n+2}^\ast$ be as in \ref{PatchingTwists}. With our fixed immersion $G_n^\ast\rightarrow\mathfrak{T}$ we define $\delta_m$ to be the composite
	\[\Spf D_{m,n}'\longrightarrow\Spf R_{n+2}'\longrightarrow G_{n+2}^\ast
	\longrightarrow\mathfrak{T}.\]
The fact that $\delta_m(\alpha x)=\alpha^2\delta_m(x)$ for any $A$ in $\CNL_O^{[gr_m]}$, $x\in \Spf D_{m,n}'(A)$ and $\alpha\in\mathfrak{T}_{2^m}(A)$ now follows from \ref{PatchingTwists}. Let $D_{m,n}''$ be the object in $\CNLO^{[gr_m]}$ representing the closed subfunctor of $\Spf D_{m,n}'$ given by points $x$ with $\delta_m(x)=1$. The surjection $D_{m,n}'\rightarrow D_{m,n}$ factors through $D_{m,n}'\rightarrow D_{m,n}''$. We also see that, since $\mathfrak{d}_mR_{n+2}'+\mathfrak{m}_{R_{n+2}'}^{(r_m)}\subseteq\mathfrak{m}_{R_{n+2}'}^m$, the kernel of $D_{m,n}''\rightarrow D_{m,n}$ is contained in $\mathfrak{m}_{D_{m,n}''}^m$. We have verified all the conditions of (b).

To realize part (d), first note that $\mathfrak{a}_m$ is the image in $D_{m,n}'$ of
	\[ \mathfrak{p}_{n+2}'+\mathfrak{d}_mR_{n+2}'+\mathfrak{m}_{R_{n+2}'}^{(r_m)}
	\subseteq \mathfrak{p}_{n+2}'+\mathfrak{m}_{R_{n+2}'}^m. \]
By \ref{genoverloc}, we can choose a $\CNL_{R_\mathrm{loc}}$ morphism
	\[ R_\mathrm{loc}[[x_1,\ldots,x_k]] \longrightarrow R_{n+2}' \]
such that $(\mathfrak{p}_\mathrm{loc},x_1,\ldots,x_k)R_{n+2}' \subseteq \mathfrak{p}_{n+2}'$, and such that $\mathfrak{p}_{n+2}'/((\mathfrak{p}_{n+2}')^2+(\mathfrak{p}_\mathrm{loc},x_1,\ldots,x_k)R_{n+2}')$ is annihilated by $\lambda_R$. This then induces the desired morphism of part (d).

For any $n> n_1$ and $n_1\le m<n$ we have natural surjections $D_{m+1,n}'\rightarrow D_{m,n}'$, $D_{m+1,n}\rightarrow D_{m,n}$ and $L_{m+1,n}\rightarrow L_{m,n}$ that induce an isomorphism
	\[ (D_{m+1,n}'/\mathfrak{m}_{D_{m+1,n}'}^{gr_m},D_{m+1,n}/(\mathfrak{d}_mD_{m+1,n}+
	\mathfrak{m}_{D_{m+1,n}}^{(r_m)}),L_{m+1,n}/\mathfrak{c}_mL_{m+1,n})\cong
	(D_{m,n}',D_{m,n},L_{m,n})\]
of patching data of level $m$. Since, for each $m\ge 1$ there are only finitely many isomorphism classes of patching data of level $m$, after extracting a subsequence of $(n)_{n\ge n_1}$ we can assume that $(D_{m,n}',D_{m,n},L_{m,n})\cong (D_{m,m}',D_{m,m},L_{m,m})$ for all $n\ge m$ and denoting this common isomorphism class by $(D_m',D_m,L_m)$, we get a projective system in $m\ge n_1$. Set $R_\infty'=\varprojlim D_m'$, $R_\infty=\varprojlim D_m$ and $M_\infty = \varprojlim L_m$. The fact that for each $m \ge n_1$ the maximal ideal of $D_m'$ can be generated by $g$ objects ensures that both $R_\infty'$ and $R_\infty$ are Noetherian.

Just as in \cite{KW2}*{Proposition 9.3} we get a commutative diagram in $\CNL_{R_\mathrm{loc}}$
	\[\xymatrix{ R_\infty' \ar[r] \ar[d] & R_\infty \ar[d] \\ R_0' \ar[r] &R_0 }\]
with $R_\infty\rightarrow R_0$ a $B[[t_1,\ldots,t_h]]$-algebra morphism, and an  $R_\infty \times B[[s_1,\ldots,s_h]]$-module morphism $M_\infty\rightarrow M_0^{2^h}$ satisfying parts (1), (2), (3), and (6) of the lemma. Parts (1), (2), and (3) all follow directly from the analogous statements at finite level. To see (6), first note that the free action of $\mathfrak{T}$ on $\Spf R_\infty'$ follows from \ref{goupchunkaction} and the fact that the group action chunks of $\mathfrak{T}_{2^m}^{[gr_m]}$ on $\Spf D_{m,n}'$ are free. The morphism $\delta_\infty$ is defined by the limit of the $\delta_m$. From the corresponding properties of $\delta_m$, it is immediate that the morphism $\Spf R_\infty'\rightarrow\Spf R_\mathrm{loc}$ is constant on orbits of $\mathfrak{T}$, and that for any object $A$ in $\CNLO$, $x\in\Spf R_\infty'(A)$ and $\alpha\in\mathfrak{T}(A)$, we have $\delta_\infty(\alpha x)=\alpha^2\delta_\infty(x)$. It remains to see that the closed immersion $\Spf R_\infty\rightarrow\Spf R_\infty'$ identifies $\Spf R_\infty$ with the closed subfunctor of points $x$ with $\delta_\infty(x) =1$. Note that at finite level, the subfunctor defined by $\delta_m = 1$, is represented by $D_m''$ and there is a surjection $D_m'' \rightarrow D_m$ with kernel contained in $\mathfrak{m}_{D_m''}^m$. Then the closed the closed subfunctor defined by $\delta_\infty = 1$ is $\varprojlim D_m'' \cong \varprojlim D_m = R_\infty$.

We now show (4) and (5). Let $\mathfrak{a}_m$ denote the ideal of $D_m$ as in (c) of patching datum. Because $\varprojlim$ is exact on systems of finite objects, we have
	\[ N_\infty/\mathfrak{c}_m N_\infty \cong \ker( L_m /(s_1,\ldots,s_h) \rightarrow M_0^{2^h}/\mathfrak{c}_m M_0^{2^h}).
	\]
Then, since $\mathfrak{p}_\infty = \varprojlim \mathfrak{a}_m$, we have $	\lambda_M N_\infty \subseteq \mathfrak{p}_\infty N_\infty$, which implies $N_\infty/\mathfrak{p}_\infty N_\infty$ is a torsion  $A$ module since  $\lambda_M$ has non-zero image in $A$.

The morphisms $R_\mathrm{loc}[[x_1,\ldots,x_k]]\rightarrow D_m'$ for each $m\ge 1$ yield a $\CNL_{R_\mathrm{loc}}$ morphism $R_\mathrm{loc}[[x_1,\ldots,x_k]]\rightarrow R_\infty'$. Let $\mathfrak{a}_m'$ be the ideal of $D_m$ as in part (d) of the definition of a patching datum of level $m$. Since $\mathfrak{p}_\infty'=\varprojlim\mathfrak{a}_m'$ and
	\[\lambda_R(\mathfrak{a}_m'/((\mathfrak{a}_m')^2 +(\mathfrak{p}_\mathrm{loc},x_1,\ldots,x_k)D_m'))\subseteq
	\mathfrak{m}_{D_m'}^m, \]
we have
	\[ \lambda_R(\mathfrak{p}_\infty'/((\mathfrak{p}_\infty')^2 +(\mathfrak{p}_\mathrm{loc},x_1,\ldots,x_k)R_\infty'))=(0),\]
which implies part (2) of the lemma since $\lambda_R$ has non-zero image in $A$.

It remains to show part (7) of the lemma. The fact that $\Spf R_\infty$ is stable under the action of $\mathfrak{T}_2$ follows immediately from part (6). Note that for $n> 2$, the isomorphism $G_n'^\ast\cong\mathfrak{T}_{p^{n-2}}$ induces an isomorphism $G_{n,2}^\ast\cong\mathfrak{T}_2$. If $m\ge n$ the action of $\mathfrak{T}_2(\mathcal{O})$ on $B[[s_1,\ldots,s_h]]$ as in \ref{MnProperties} stabilizes the ideal $\mathfrak{c}_m$ and the compatible actions of $\mathfrak{T}_2(\mathcal{O})$ on $R_{n+2}$ and $M_{n+2}$ descend to compatible actions on $D_{m,n}$ and $L_{m,n}$. Also note that for $\alpha\in\mathfrak{T}_2(\mathcal{O})$, if we let $\alpha'\in\mathfrak{T}_2^{[gr_m]}(\mathcal{O}/\mathfrak{m}_{\mathcal{O}}^{gr_m})$ denote the corresponding truncated group element, the diagram
	\[\xymatrix{ D_m' \ar[r] \ar[d]^{\alpha'} & D_m \ar[d]^\alpha \\ D_m' \ar[r] & D_m}\]
commutes. It follows that, after taking limits, the action of $\mathfrak{T}_2(\mathcal{O})$ on $M_\infty$ is compatible with the action of $\mathfrak{T}_2(\mathcal{O})$ on $R_\infty$. This establishes (7) and the lemma is proven. 
	\end{proof}

By part (1) of \ref{diaggroupaction}, we have a $\CNLO$-algebra $R_\infty^{\inv}$ that represents the orbits for the action of $\mathfrak{T}$ on $\Spf R_\infty'$,  and the morphism $R_\infty^{\inv}\rightarrow R_\infty'$ is formally smooth of relative dimension $t$. Note that as $\Spf R_\infty'\rightarrow\Spf R_\mathrm{loc}$ is constant on $\mathfrak{T}$-orbits, we have a $\CNLO$ morphism $R_\mathrm{loc}\rightarrow R^{\inv}_\infty$ and $R_\infty^{\inv}\rightarrow R_\infty'$ is a morphism of $R_\mathrm{loc}$-algebras. Part (2) of \ref{diaggroupaction} then shows that the map $R_\infty^{\inv}\rightarrow R_\infty$ makes $R_\infty$ a torsor over $R_\infty^{\inv}$ with group $\mathfrak{T}_2$. Let $\mathfrak{p}_\infty^{\inv}$ denote pullback to $R_\infty^\mathrm{inv}$ of $\mathfrak{p}_0$. Note that $R_\infty$ is finite over $R_\infty^{\inv}$, so $\dim R_\infty^{\inv}/\mathfrak{p}_\infty^{\inv}=\dim R_\infty/\mathfrak{p}_\infty=\dim R_0/\mathfrak{p}_0$, and $2\in\mathfrak{p}_\infty^{\inv}$. Let $\mathfrak{p}_1,\mathfrak{p}_2 \in \Spec R_\infty$ lie over $\mathfrak{p}_\infty^\mathrm{inv} \in \Spec R_\infty^{\inv}$. Choose a characteristic $2$ field $L$ with embeddings $R_\infty/\mathfrak{p}_1\rightarrow L$ and $R_\infty/\mathfrak{p}_2\rightarrow L$. Since $R_\infty$ is a $\mathfrak{T}_2$-torsor over $R_\infty^{\inv}$, there is $\alpha\in\mathfrak{T}_2(L)$ such that if $x$ denotes the point $R_\infty\rightarrow R_\infty/\mathfrak{p}_1\rightarrow L$ of $\Spf R_\infty(L)$, then $\alpha x$ is the point $R_\infty\rightarrow R_\infty/\mathfrak{p}_2\rightarrow L$. Since $L$ has characteristic $2$, $\mathfrak{T}_2(L)$ is trivial, and $\mathfrak{p}_1=\mathfrak{p}_2$. So, $\mathfrak{p}_\infty$ is the unique prime of $R_\infty$ above $\mathfrak{p}_\infty^{\inv}$. It follows that $(R_\infty)_{\mathfrak{p}_\infty^{\inv}}$ is local and the natural map $(R_\infty)_{\mathfrak{p}_\infty^{\inv}}\rightarrow (R_\infty)_{\mathfrak{p}_\infty}$ is an isomorphism. It also follows that the $\mathfrak{p}_\infty^{\inv}$-adic topology on $(R_\infty)_{\mathfrak{p}_\infty}$ is the same as the $\mathfrak{p}_\infty$-adic topology on $(R_\infty)_{\mathfrak{p}_\infty}$, since $(R_\infty)_{\mathfrak{p}_\infty}/\mathfrak{p}_\infty^{\inv}(R_\infty)_{\mathfrak{p}_\infty}$ is a finite local algebra over the field $(R_\infty^{\inv})_{\mathfrak{p}_\infty^{\inv}}/\mathfrak{p}_\infty^{\inv}$; hence is Artinian. Similar statements hold for any $R_\infty$-module, in particular the natural map of $(R_\infty^{\inv})_{\mathfrak{p}_\infty^{\inv}}$-modules $(M_\infty)_{\mathfrak{p}_\infty^{\inv}}\rightarrow (M_\infty)_{\mathfrak{p}_\infty}$ is an isomorphism and the topologies defined on $(M_\infty)_{\mathfrak{p}_\infty}$ by $\mathfrak{p}_\infty^{\inv}$ and $\mathfrak{p}_\infty$ are equivalent. Also note that this module is non-zero as the surjection $M_\infty\rightarrow M$ induces a surjection $(M_\infty)_{\mathfrak{p}_\infty}\rightarrow (M_0^{2^h})_{\mathfrak{p}}$.

\begin{lem}\label{RinfFaithful}
If $\mathfrak{q} \in \Spec (R_\infty)_{\mathfrak{p}_\infty}$ contains $\mathfrak{q}_{\mathrm{ur}}(R_\infty)_{\mathfrak{p}_\infty}$, then $\mathfrak{q}$ is in the support of $(M_\infty)_{\mathfrak{p}_\infty}$.
\end{lem}

\begin{proof} 

We first show that $\mathrm{Supp}_{(R_\infty)_{\mathfrak{p}_\infty}}(M_\infty)_{\mathfrak{p}_\infty}$ is a union of irreducible components. Since $M_\infty$ is finite free over $B[[s_1,\ldots,s_h]]$ and $B[[s_1,\ldots,s_h]]$ is finite free over $B[[t_1,\ldots,t_h]]$, we see that the $(\mathfrak{p}_{\Lambda},y_1,\ldots,y_j,t_1,\ldots,t_h)$-depth of $M_\infty$, as a $B[[t_1,\ldots,t_h]]$-module, is
	\[ \mathrm{ht}(\mathfrak{p}_{\Lambda},y_1,\ldots,y_j,t_1,\ldots,t_h) = 
		d + j+h. \]
Since the image of $(\mathfrak{p}_{\Lambda},y_1,\ldots,y_j,t_1,\ldots,t_h)$ in $R_\infty$ is contained in $\mathfrak{p}_\infty$, 
	\[ \depth_{(R_\infty)_{\mathfrak{p}_\infty}^\wedge}(M_\infty)_{\mathfrak{p}_\infty}^\wedge\ge\depth_{R_\infty}
	(\mathfrak{p}_\infty,M_\infty)
	\ge d+ j+h. \]
In particular, if $\mathfrak{q}$ is an associated prime of $(M_\infty)_{\mathfrak{p}_\infty}^\wedge$ in $\Spec (R_\infty)_{\mathfrak{p}_\infty}^\wedge$, we have
	\begin{equation}\label{DimLower}
	 \dim((R_\infty)_{\mathfrak{p}_\infty}^\wedge/\mathfrak{q})\ge d + j+h.
	 \end{equation}
	 
Consider the morphism $R_\mathrm{loc}[[x_1,\ldots,x_k]]\rightarrow R_\infty'$ of \ref{Rinf}. By part (3) of \ref{Rinf}, this map induces a surjection
	\[ (R_\mathrm{loc})_{\mathfrak{p}_\mathrm{loc}}^\wedge[[x_1,\ldots,x_k]]\longrightarrow
	(R_\infty')_{\mathfrak{p}_\infty'}^\wedge .\]
Since $R_\infty'$ is formally smooth over $R_\infty^{\inv}$ of relative dimension $t$, and $R_\infty^{\inv}/\mathfrak{p}_\infty^{\inv}\cong R_\infty'/\mathfrak{p}_\infty'$, there is an isomorphism $R_\infty'\cong R_\infty^{\inv}[[z_1,\ldots,z_t]]$ that sends $\mathfrak{p}_\infty'$ to $\mathfrak{p}_\infty^{\inv}+(z_1,\ldots,z_t)$. It follows that 
	\[
	(R_\infty')_{\mathfrak{p}_\infty'}^\wedge \cong 
	(R_\infty^{\inv})_{\mathfrak{p}_\infty^\mathrm{inv}}^\wedge[[z_1,\ldots,z_t]].
	\]
 We have
	\[ \dim (R_\infty)_{\mathfrak{p}_\infty}^\wedge 
	= \dim(R_\infty^\mathrm{inv})_{\mathfrak{p}_\infty^\mathrm{inv}}^\wedge
	 = \dim (R_\infty')_{\mathfrak{p}_\infty'}^\wedge - t \le 1+d+3|S_0|+k-t = d+j+h.
	\]
Combining this with \eqref{DimLower}, we see that the support of $(M_\infty)_{\mathfrak{p}_\infty}^\wedge$ in $\Spec (R_\infty)_{\mathfrak{p}_\infty}^\wedge$ is a union of irreducible components. Then \cite{BourbakiComAlg}*{Chapter 2, \S 4, Proposition 19} implies that $\mathrm{Supp}_{(R_\infty)_{\mathrm{p}_\infty}}(M_\infty)_{\mathfrak{p}_\infty}$ is a union of irreducible components. We are now reduced to showing that any minimal prime $\mathfrak{q}$ of $(R_\infty)_{\mathfrak{p}_\infty}$ containing $\mathfrak{q}_\mathrm{ur}(R_\infty)_{\mathfrak{p}_\infty}$ is in the support of $(M_\infty)_{\mathfrak{p}_\infty}$.

The commutativity of 
	\[ \xymatrix{
		(R_\mathrm{loc})_{\mathfrak{p}_\mathrm{loc}} \ar[r] \ar[rd] & 
		(R_\infty)_{\mathfrak{p}_\infty} \ar[d] \\
		&  (R_0)_{\mathfrak{p}_0}
	} \]
and \ref{UniqueComp} imply that there is an element of $\mathrm{Supp}_{(R_\infty)_{\mathfrak{p}_\infty}}(M_\infty)_{\mathfrak{p}_\infty}$ whose image image in $\Spec (R_\mathrm{loc})_{\mathfrak{p}_\mathrm{loc}}$ is contained in the irreducible component determined by $\mathfrak{q}_\mathrm{ur}$ and no other. Thus, there is a minimal prime $\mathfrak{q}$ of $(R_\infty)_{\mathfrak{p}_\infty}$ containing $\mathfrak{q}_\mathrm{ur}(R_\infty)_{\mathfrak{p}_\infty}$ and in the support of $(M_\infty)_{\mathfrak{p}_\infty}$. Set $\mathfrak{q}^\mathrm{inv} = \mathfrak{q}\cap (R_\infty^\mathrm{inv})_{\mathfrak{p}_\infty^\mathrm{inv}}$.

We now show that $\mathfrak{q}^\mathrm{inv} = \mathfrak{q}_\mathrm{ur}(R_\infty^\mathrm{inv})_{\mathfrak{p}_\infty^\mathrm{inv}}$; in particular, $\mathfrak{q}_\mathrm{ur}(R_\infty^\mathrm{inv})_{\mathfrak{p}_\infty^\mathrm{inv}}$ is prime. Since $(R_\infty^\mathrm{inv})_{\mathrm{p}_\infty^\mathrm{inv}} \rightarrow (R_\infty)_{\mathrm{p}_\infty}$ is finite, $\mathfrak{q}^\mathrm{inv} \in \mathrm{Supp}_{(R_\infty^\mathrm{inv})_{\mathrm{p}_\infty^\mathrm{inv}}}(M_\infty)_{\mathfrak{p}_\infty}$. Let $\mathfrak{Q}^\mathrm{inv}$ be a minimal prime of $(R_\infty^\mathrm{inv})_{\mathrm{p}_\infty^\mathrm{inv}}^\wedge$ above $\mathfrak{q}^\mathrm{inv}$. Using the isomorphism
	\[  
	(R_\infty')_{\mathfrak{p}_\infty'}^\wedge \cong 
	(R_\infty^{\inv})_{\mathfrak{p}_\infty^\mathrm{inv}}^\wedge[[z_1,\ldots,z_t]],
	\]
we see that $\mathfrak{Q}^\mathrm{inv}(R_\infty')_{\mathfrak{p}_\infty'}^\wedge$ is a minimal prime of $(R_\infty')_{\mathfrak{p}_\infty'}^\wedge$. Recall we have a surjection
	\[ (R_\mathrm{loc})_{\mathfrak{p}_\mathrm{loc}}^\wedge[[x_1,\ldots,x_k]]\longrightarrow
	(R_\infty')_{\mathfrak{p}_\infty'}^\wedge,
	\]
and that both of these rings have the same dimension. By \ref{localdefringnormal} and part (d) of the definition of nice prime, cf. \ref{SWPrimeDef}, $(R_\mathrm{loc}/\mathfrak{q}_\mathrm{ur})_{\mathfrak{p}_\mathrm{loc}}$ is normal; hence, so is $(R_\mathrm{loc}/\mathfrak{q}_\mathrm{ur})_{\mathfrak{p}_\mathrm{loc}}^\wedge$ by \cite{EGA4.2}*{Scholie 7.8.3 (vii)}. This implies that $\mathfrak{q}_\mathrm{ur}(R_\mathrm{loc})_{\mathfrak{p}_\mathrm{loc}}^\wedge$ is a minimal prime. The pullback of $\mathfrak{Q}^\mathrm{inv}$ to $(R_\mathrm{loc})_{\mathfrak{p}_\mathrm{loc}}^\wedge[[x_1,\ldots,x_k]]$ is a minimal prime that contains $\mathfrak{q}_\mathrm{ur}$, so must be equal to $\mathfrak{q}_\mathrm{ur}(R_\mathrm{loc})_{\mathfrak{p}_\mathrm{loc}}^\wedge[[x_1,\ldots,x_k]]$. From this we conclude that $\mathfrak{Q}^\mathrm{inv}(R_\infty')_{\mathfrak{p}_\infty'}^\wedge = \mathfrak{q}_\mathrm{ur}(R_\infty')_{\mathfrak{p}_\infty'}^\wedge$, that $\mathfrak{Q}^\mathrm{inv} = \mathfrak{q}_\mathrm{ur}(R_\infty^\mathrm{inv})_{\mathrm{p}_\infty^\mathrm{inv}}^\wedge$, and finally that $\mathfrak{q}^\mathrm{inv} = \mathfrak{q}_\mathrm{ur}(R_\infty^\mathrm{inv})_{\mathfrak{p}_\infty^\mathrm{inv}}$.		

We know that there is some minimal prime $\mathfrak{q}$ of $(R_\infty)_{\mathfrak{p}_\infty}$ above $\mathfrak{q}_\mathrm{ur}(R_\infty^\mathrm{inv})_{\mathfrak{p}_\infty^\mathrm{inv}}$ and in the support of $(M_\infty)_{\mathfrak{p}_\infty}$. Since $R_\infty$ is a $\mathfrak{T}_2$-torsor on $R_\infty^{\inv}$ and under this action $\mathfrak{p}_\infty$ is the unique prime above $\mathfrak{p}_\infty^\mathrm{inv}$, we see that $\mathfrak{T}_2$ acts transitively on the set of minimal primes in $(R_\infty)_{\mathfrak{p}_\infty}$ above $\mathfrak{q}_\mathrm{ur}(R_\infty^\mathrm{inv})_{\mathfrak{q}_\infty^\mathrm{inv}}$. We are reduced to showing that the support of $(M_\infty)_{\mathfrak{p}_\infty}$ is stable under this action. Let $\mathfrak{q}_1,\mathfrak{q}_2 \in \Spec (R_\infty)_{\mathfrak{p}_\infty}$ lie above $\mathfrak{q}_\mathrm{ur}(R_\infty^\mathrm{inv})_{\mathfrak{p}_\infty^\mathrm{inv}}$. If $A$ is a $\CNLO$-algebra domain, then the natural map $\mathfrak{T}_2(\mathcal{O})\rightarrow\mathfrak{T}_2(A)$ is surjective; hence, there is $\alpha\in\mathfrak{T}_2(\mathcal{O})$, such that the automorphism of $(R_\infty)_{\mathfrak{p}_\infty}$ induced by $\alpha$ sends $\mathfrak{q}_1$ to $\mathfrak{q}_2$. By part (5) of \ref{Rinf}, there is a compatible action of $\mathfrak{T}_2(\mathcal{O})$ on $M_\infty$. It follows that $\mathfrak{q}_1 \in \mathrm{Supp}_{(R_\infty)_{\mathfrak{p}_\infty}}(M_\infty)_{\mathfrak{p}_\infty}$ if and only if $\mathfrak{q}_2 \in \mathrm{Supp}_{(R_\infty)_{\mathfrak{p}_\infty}}(M_\infty)_{\mathfrak{p}_\infty}$, i.e. the support of $M_\infty$ is stable under the action of $\mathfrak{T}_2$.	 \end{proof}

\subsubsection{}\label{FinishPatching}

We can now complete the proof of Proposition \ref{LocRredT}. Let $\mathfrak{q} \in \Spec(R_0)_{\mathfrak{p}_0}$ contain $\mathfrak{q}_\mathrm{ur}(R_0)_{\mathfrak{p}_0}$. Pull $\mathfrak{q}$ back to a prime of $(R_\infty)_{\mathfrak{p}_\infty}$ and call it $\mathfrak{q}_\infty$. By \ref{RinfFaithful}, $\mathfrak{q}_\infty$ is in the support of $(M_\infty)_{\mathfrak{p}_\infty}$. Since $t_1,\ldots,t_h \in \ker ((R_\infty)_{\mathfrak{p}_\infty} \rightarrow (R_0)_{\mathfrak{p}_0})$, we see that $\mathfrak{q}_\infty/(t_1,\ldots,t_h) \in \Spec ((R_\infty)_{\mathfrak{p}_\infty}/(t_1,\ldots,t_h))$ is in the support of $(M_\infty)_{\mathfrak{p}_\infty}/(t_1,\ldots,t_h)$. Since $s_i^2 \in (t_1,\ldots,t_h)(M_\infty)_{\mathfrak{p}_\infty}$ for each $i$, we further deduce that $\mathfrak{q}_\infty/(t_1,\ldots,t_h)$ is in the support of $(M_\infty)_{\mathfrak{p}_\infty}/(s_1,\ldots,s_h)$ as an $(R_\infty)_{\mathfrak{p}_\infty}/(t_1,\ldots,t_h)$-module. But, by (3) and (4) of \ref{Rinf}, $(M_\infty)_{\mathfrak{p}_\infty}/(s_1,\ldots,s_h) \cong {(M_0^{2^h})}_{\mathfrak{p}_0}$. Then, since the action of $(R_\infty)_{\mathfrak{p}_\infty}/(t_1,\ldots,t_h)$ on ${(M_0^{2^h})}_{\mathfrak{p}_0}$ factors through $(R_0)_{\mathfrak{p}_0}$ and this map sends $\mathfrak{q}_\infty/(t_1,\ldots,t_h)$ to $\mathfrak{q}$, we conclude that $\mathfrak{q}$ is in the support of ${(M_0^{2^h})}_{\mathfrak{p}_0}$	 \hfill\qed

%% file: Patching/RequalsT.tex
\subsection{An $R^\mathrm{red} = \mathbf{T}$ theorem}\label{RredT}

Unless specifically noted otherwise, keep the assumptions and notations of the previous subsubsections. If $\rhobar$ is dihedral we let $L/F$ denote the unique quadratic extension of $F$ such that $\overline{\rho}|_{G_L}$ is abelian. Let $L_S^\mathrm{ab}$ denote the maximal pro-$2$ extension of $L$ unramified outside places above $S$. Let $L_S^-$ denote the maximal subextension of $L_S^\mathrm{ab}/L$ such that the nontrivial element of $\Gal(L/F)$ acts on $\Gal(L_S^-/L)$ by $-1$.

We further assume that
\begin{enumerate}
	\item[(i)] for each $v|2$, $[F_v:\Q_2] \ge 4$;
	\item[(i)] if $L/F$ is CM, then there is some $v|2$ in $F$ that does not split in $L$;
	\item[(ii)] if $L/F$ is not CM, then $\mathrm{rank}_{\Z_p} \Gal(L_S^-/F) < [F:\Q]-3$.
\end{enumerate}

\begin{lem}\label{thenondihedlem}
Take $\mathfrak{Q} \in \Spec \mathbf{T}_\psi(U,\eta)_\mathfrak{m}$ with $2\in \mathfrak{Q}$ such that $\dim (\mathbf{T}_\psi(U,\eta)_\mathfrak{m}/\mathfrak{Q}) \ge [F:\Q]-3$.
\begin{enumerate}
	\item The specialization $\chi_{\mathfrak{Q},v}$ of $\chi_{\eta_v}^\mathrm{univ}|_{I_v}$ at $\mathfrak{Q}$ is infinite order.
	\item The the specialization $\rho_\mathfrak{Q}$ of $\rho_{U,\mathfrak{m}}$ at $\mathfrak{Q}$ is non-dihedral.
\end{enumerate}
\end{lem}

\begin{proof}

Since $\mathbf{T}_\psi(U,\eta)_\mathfrak{m}$ is finite and torsion-free over $\Lambda(I_p,\eta) = \hat{\otimes}_{v|2} \Lambda(I_v,\eta_v)$, part (1) follows from the fact that $\dim (\Lambda(I_v,\eta_v)\otimes_\mathcal{O}\F) = [F_v:\Q_2] \ge 4$, and so $\mathfrak{Q}\cap \Lambda(I_v,\eta_v)$ defines a non-maximal prime ideal in $\Lambda(I_v,\eta_v)/(\varpi_E)$.

Part (2) is trivial if $\rhobar$ is not dihedral, so assume $\rhobar$ is $L$-dihedral. First assume that $L/F$ is CM. Let $v|2$ in $F$ be such that $v$ does not split in $L$. If $\rho_\mathfrak{Q}$ where dihedral, then part (3) of \ref{BigGalRep} together with \ref{NotDihedSplit} contradict the fact that $\chi_{\mathfrak{Q},v}$ is infinite order.

Now assume that $L/F$ is not CM. Let $B$ denote the subring of $\mathbf{T}_\psi(U,\eta)_\mathfrak{m}/\mathfrak{Q}$ generated by traces of $\rho_\mathfrak{Q}$, and note that $\mathbf{T}_\psi(U,\eta)_\mathfrak{m}/\mathfrak{Q}$ is integral over $B$. So, $\dim B \ge [F:\Q]-3$. We know that the image of $R_{F,S} \rightarrow \mathbf{T}_\psi(U,\eta)_\mathfrak{m}/\mathfrak{Q}$ contains $B$, and so has dimension $\ge [F:\Q]-3$. The result now follows from \ref{NotDihedDim} as we have assumed $\mathrm{rank}_{\Z_p} \Gal(L_S^-/F) < [F:\Q]-3$.
	\end{proof}

\begin{lem}\label{niceprimelem}
Let $X\subseteq \Spec \overline{R}_{F,S}^\psi$ be closed and pro-modular. If $\dim X \ge [F:\Q]-1$, then $X$ is contains a nice prime.
\end{lem}

\begin{proof}

Since $X$ is pro-modular, we have to show $X$ is contains a prime $\mathfrak{p}$ such that
\begin{enumerate}
	\item[(a)] $\mathfrak{p}$ has dimension one and contains 2;
	\item[(b)] $\rho_\mathfrak{p}$ is not dihedral;
	\item[(c)] for each $v|2$, the image of $\mathfrak{p}$ in $\Spec \Lambda(G_v,\eta_v)$ is not contained in $Z_v$, the closed subscheme of $\Spec \Lambda(G_v,\eta_v)$ defined by $(\chi_{\eta_v}^\mathrm{univ})^2 = \psi\epsilon_2$;
	\item[(d)] the image of $\rho_\mathfrak{p}$ contains a non-trivial unipotent element.
\end{enumerate}

Let $I$ be some ideal of $\overline{R}_{F,S}^\psi$ giving rise to $X$. Fix some $\sigma_0\in G_F$ such that $\rhobar(\sigma_0)$ has order $2$ and let $T \in \overline{R}_{F,S}^\psi$ denote the trace of $\sigma_0$ under the universal $\overline{R}_{F,S}^\psi$-deformation. Note that $T \in \mathfrak{m}_{\overline{R}_{F,S}^\psi}$ since $\rhobar(\sigma_0)$ has order $2$. Let $\mathfrak{Q}$ be a minimal prime of $\overline{R}_{F,S}^\psi/(I,\varpi_E,T)$. Then $\mathfrak{Q}$ is pro-modular and $\dim \overline{R}_{F,S}^\psi/\mathfrak{Q} \ge [F:\Q]-3$. 

By part (1) of \ref{thenondihedlem}, $\mathfrak{Q}\cap \Lambda(G_v,\eta_v)$ is not contained in $Z_v$ for any $v|2$, and by part (2) it is not dihedral. The diheral locus is closed, cf. \ref{DihedLocus}, so the locus of primes in $\Spec\overline{R}_{F,S}^\psi/\mathfrak{Q}$ satisfying (b) and (c) above is open and nonempty. Since $\Spec(\overline{R}_{F,S}^\psi/\mathfrak{Q})\smallsetminus \{\mathfrak{m}_{\overline{R}_{F,S}^\psi/\mathfrak{Q}}\}$ is Jacobson, this open set contains a dimension $1$ prime; let $\mathfrak{p}$ be such a prime. It remains to show that the image of $\rho_\mathfrak{p}$ contains a non-trivial unipotent element. Since $T\in \mathfrak{p}$, $\mathrm{tr}\rho_\mathfrak{p}(\sigma_0)=0$. Since $\varpi_E\in\mathfrak{p}$, $\det\rho_\mathfrak{p}(\sigma_0)=\det\rhobar(\sigma_0)=1$. From this we see that $\rho_\mathfrak{p}(\sigma_0)^2 = 1$. Since $\rhobar(\sigma_0)$ has order two, so does $\rho_\mathfrak{p}(\sigma_0)$, and it is unipotent.			\end{proof}

\begin{prop}\label{ProMod}
Every prime of $\overline{R}^\psi_{F,S}$ is pro-modular.
\end{prop}

\begin{proof}

This proof is essentially the same as \cite{SWirreducible}*{Proposition 4.1}. By applying \ref{niceprimelem} to any minimal prime of $\mathbf{T}_\psi(U,\eta)_\mathfrak{m}$, which has dimension $1+[F:\Q]$, we see that $\overline{R}_{F,S}^\psi$ contains a nice prime. By \ref{LocRredT}, any irreducible component  of $\Spec \overline{R}_{F,S}^\psi$ containing $\mathfrak{p}$ is pro-modular. Fix one such irreducible component $C$. Let $C'$ be any other irreducible component of $\Spec \overline{R}_{F,S}^\psi$. 

By \ref{CMPres}, there is a presentation $\overline{R}_{F,S}^\psi \cong A/(f_1,\ldots,f_r)$, with $A$ a complete local domain and $\dim A - r \ge 1+[F:\Q] - \dim_\F H^0(G_F,(\Ad_\F^0)^\ast(1))$. Since $\rhobar$ is absolutely irreducible, $\dim_\F H^0(G_F,(\Ad_\F^0)^\ast(1))=\dim_\F H^0(G_F,\Ad_\F/Z) \le 1$. Then \ref{Connectivity} implies that $\overline{R}_{F,S}^\psi$ is $[F:\Q]-1$-connected, so  there is a sequence of irreducible components
	\[ C = C_0, C_1,\ldots, C_n = C'\]
such that $\dim (C_i\cap C_{i+1}) \ge [F:\Q]-1$ for each $0\le i < n$. By \ref{niceprimelem}, $C_0\cap C_1$ contains a nice prime, and \ref{LocRredT} implies $C_1$ is pro-modular. Continuing in this way, we deduce that $C_i$ is pro-modular for each $0\le i\le n$; in particular, $C_n=C'$ is pro-modular. 	\end{proof}

%% file: MainThms/MainThmsIntro.tex
\section{The Main Theorem}\label{MainThms}

We are now in a position to prove the main theorem. Before doing so, in the first subsection we recall some congruences proved in \cite{KW2} and \cite{KisinFinFlat} that are necessary to be able to satisfy the assumptions of the $R^\mathrm{red} = \mathbf{T}$ theorem. We also prove a small lemma that shows the existence of ordinary lifts in the residually dihedral case. This is an application of a result of Wiles that allows one to insert a $p$-ordinary Hilbert modular form of parallel weight $1$ into a $p$-adic family.

In the second subsection we prove the main theorem. This is mostly routine, except that we must use some known cases of Leopoldt's conjecture in our base changes, to ensure the assumptions of \ref{ProMod} are met.

In the last subsection we show how the main theorem implies the corollary from the introduction on modularity of (certain) elliptic curves, and give some examples of elliptic curves satisfying the hypotheses.

\subsection{Congruences}\label{Congruences}

\input{MainThms/Congruences}

\subsection{Proof of the main theorem}\label{TheThms}

\input{MainThms/Thms}

\subsection{An application to elliptic curves}\label{Examples}

\input{MainThms/Examples}

%% file: MainThms/Congruences.tex
For simplicity we fix an isomorphism $\overline{\Q}_p \cong \C$ throughout this subsection. 

\subsubsection{}\label{congassumptions}

We fix a continuous absolutely irreducible
	\[ \overline{\rho} : G_F \longrightarrow \GL_2(\overline{\F}) \]
such that for each $v|p$, $\overline{\rho}|_{G_v}$ is reducible. Write
	\[ \overline{\rho}_{G_v} \cong \left(\begin{array}{cc} \overline{\chi}_v' & \ast \\
	& \overline{\chi}_v \end{array}\right).
	\]
Let $\overline{\chi}$ denote the tuple $\overline{\chi} = (\overline{\chi}_v)_{v|p}$. If we are given any finite field extension $F'/F$ we will still denote by $\overline{\chi}$ the tuple $(\overline{\chi}_v|_{G_w})_{w|v,v|p}$.

Given a cuspical automorphic representaion $\pi$ of $\GL_2(\A_F)$, we say that $\pi$ \textit{lifts} $\overline{\rho}$ if, letting $\rho_\pi$ be the representation as in \ref{SmallGalRepHilbert}, there is a $G_F$-stable lattice in the representation space of $\rho_\pi$ whose reduction is equal to $\overline{\rho}$ (after extension of scalars, if necessary). If $\pi$ is a $p$-nearly ordinary lift of $\rhobar$, we say that it is $\overline{\chi}$-\textit{good lift} if for each $v|p$ we have
	\[ 
	\rho_\pi|_{G_v} \cong \left( \begin{array}{cc} \ast & \ast \\ & \chi_v \end{array} \right)
	\]
with $\chi_v$ a lift of $\overline{\chi}_v$. Similarly, if $f$ is an eigenfunction in some $S_{\kappa,\psi}^\mathrm{no}(U,\mathcal{O})$ as in \ref{NearOrd}, and $\pi_f$ denotes the cuspidal automorphic representation of $\GL_2(\A_F)$ associated to it by \ref{JacquetLanglands}, we say $f$ \textit{lifts} $\overline{\rho}$ if $\pi_f$ does, and that $f$ is a $\overline{\chi}$-\textit{good lift} of $\overline{\rho}$ if $\pi_f$ is. We keep track of $\overline{\chi}$-good lifts in what follows. This isn't necessary for our applications to modularity of Galois representations (we will perform a base change to assume that $\overline{\rho}|_{G_v}$ is either trivial or unipotent for each $v|p$), but we do so because it entails relatively little extra work.

\begin{lem}\label{OrdinaryLift}

If $\overline{\rho}$ is dihedral, then it has a $\chi$-good $p$-nearly ordinary regular algebraic cuspidal lift.

\end{lem}

\begin{proof}

Let $\chi_v$ be the Teichm\"{u}ller lift of $\overline{\chi}_v$, and view it as taking values in $\Qbar^\times$. It suffices to prove the lemma in the case that $\overline{\chi}_v$ is unramified for each $v|p$. Otherwise, we take a finite order character $\theta : F^\times \backslash \A_F^\times \rightarrow \Qbar^\times$ such that $\theta|_{\mathcal{O}_{F_v}^\times} = \chi_v|_{\mathcal{O}_{F_v}^\times}$ for each $v|p$, cf. \cite{ArtinTate}*{\S 10, Theorem 5}. Then letting $\pi$ denote the resulting lift of $\overline{\rho} \otimes \overline{\theta}^{-1}$, the twist $\pi \otimes \theta$ will be the desired $\overline{\chi}$-good lift of $\overline{\rho}$.

Note that the reducibility of $\rhobar|_{G_v}$ together with the assumption that $\rhobar$ is dihedral implies that $\rhobar|_{G_v}$ is split or  $\overline{\chi}_v = \overline{\chi}_v'$ and $p=2$. If $p$ is odd, we let $\chi_v'$ be the Teichm\'{u}ller lift of $\overline{\chi}_v'$ for each $v|p$. If $p=2$, we let $\chi_v'$ be the Teichm\'{u}ller lift of $\overline{\chi}_v'$ for each $v|2$ such that $\rhobar|_{G_v}$ is split. If $p=2$ and $\rhobar|_{G_v}$ is nonsplit, define $\chi_v'$ as follows. Let $L/F$ denote the unique quadratic extension such that $\rhobar|_{G_L}$ is abelian, and let $w$ denote the unique prime above $v$ in $L$. Since $\chi_v$ has odd order, the fixed field of its kernel is disjoint from $L_w/F_v$, and we let $\chi_v'$ be the product of $\chi_v$ with the nontrivial character of $\Gal(L_w/F_v)$. In all the above cases, we view $\chi_v'$ as taking values in $\Qbar^\times$.

We first construct a totally odd dihedral representation $\rho_0 : G_F \rightarrow \GL_2(\Qbar)$ lifting $\overline{\rho}$. Write $\overline{\rho} = \Ind_{G_F}^{G_L}\overline{\chi}$ for a quadratic extension $L/F$ and $\overline{\chi} : G_L \rightarrow \F^\times$, enlarging $\F$ if necessary. Let $\chi: G_L \rightarrow \Qbar^\times$ denote the Teichmuller lift of $\chi$. If $\overline{\rho}(c) \ne 1$ for any choice $c$ of complex conjugation, we set $\rho_1 = \Ind_{G_L}^{G_F} \chi$. If there is some choice of complex conjugation at which $\overline{\rho}$ is trivial, we use a trick of Serre. Note that in this case $p = 2$. Let $\{\tau_1,\ldots,\tau_k\}$ denote the set of embeddings $F\hookrightarrow \R$ where $\overline{\rho}$ is trivial. Then each $\tau_i$ splits in $L/F$, and we write $\sigma_i$ and $\sigma_i'$ for the two embeddings above $\tau_i$. By \cite{ArtinTate}*{\S 10, Theorem 5}, there is a character $\xi: G_L \rightarrow \Qbar^\times$ of order either two or four, such that $\xi$ nontrivial at each $\{\sigma_1,\ldots,\sigma_k\}$ and trivial at each $\{\sigma_1',\ldots,\sigma_k'\}$ as well as at every place above $2$. Note that if $\xi'$ denotes the conjugate of $\xi$ by $\Gal(L/F)$, we have $\xi(c)\xi'(c) = -1$ if $c$ is the complex conjugation corresponding to any $\tau\in \{\tau_1,\ldots,\tau_k\}$. We set $\rho_1 = \Ind_{G_L}^{G_F} \chi\xi$. Since $\xi$ has order two or four, it is trivial mod $2$, and $\rho_1$ is a lift of $\overline{\rho}$. By choice of $\xi$, the lift $\rho_1$ is totally odd.  Also, because $\xi$ is trivial at any place above $2$ we have
	\[ \rho_1|_{G_v} \cong \left(\begin{array}{cc} \chi_v' & \\ & \chi_v \end{array} \right)
	\]
for $\chi_v'$ and $\chi_v$ as above.

A classical construction yields a cuspidal Hilbert modular newform $f_1$ of weight $((1,\ldots,1),(0,\ldots,0))$, such that $\rho_{f_1} \cong \rho_1$. Let $\pi_{f_1}$ denote the corresponding automorphic representation. For each $v|p$, the local representation $(\pi_{f_1})_v$ is the principal series $\pi(\chi_v',\chi_v)$. If $\chi_v'$ is unramified then the double coset operator
	\[ \left[\GL_2(\mathcal{O}_{F_v}) \left(\begin{array}{cc} \varpi_v & \\ & 1 \end{array}\right)
	\GL_2(\mathcal{O}_{F_v}) \right]
	\]
acts on $\pi(\chi_v',\chi_v)^{\GL_2(\mathcal{O}_{F_v})}$ via $\chi_v'(\varpi_v)+\chi_v(\varpi_v)$. In this case we replace $f_1$ with the $v$-stabilized eigenform on which the double coset operator
	\[ \left[ \mathrm{Iw}(v) \left(\begin{array}{cc} \varpi_v & \\ & 1 \end{array}\right) 
	\mathrm{Iw}(v) \right]
	\]
acts via $\chi_v(\varpi_v)$. We may then assume that $T_{\varpi_v} f_1 = \chi_v(\varpi_v)f_1$ for each $v|p$. Then \cite{WilesOrdinary}*{Theorem 3} allows us to insert $f_1$ in an ordinary $p$-adic analytic family and letting $f_k$ denote a classical specialization at some parallel weight $k\ge 2$, for each $v|p$ we have $T_{\varpi_v} f_k = \alpha_v f_k$ for some $\alpha_k$ congruent to $\overline{\chi}_v(\varpi_v)$. The automorphic representation generated by this $f_k$ is then a $\overline{\chi}$-good lift of $\overline{\rho}$.
	\end{proof}

\subsubsection{}\label{thecongruences}

Let $\kappa = (\mathbf{k},\mathbf{w})$ be an algebraic weight and let $\psi : F^\times \backslash \A_F^\times \rightarrow \mathcal{O}^\times$ be a continuous character such that $\psi(z) = z_p^{2-\mathbf{k}-2\mathbf{w}}$ on some open subgroup of $\A_F^\times$. We denote by $\psi_\C$ the character $\psi_\C : F^\times \backslash \A_F^\times \rightarrow \C^\times$ given by $\psi_\C(z) = \psi(z)z_p^{\mathbf{k}+2\mathbf{w}-2}z_\infty^{2-\mathbf{k}-2\mathbf{w}}$, using our fixed isomorphism $\C \cong \overline{\Q}_p$. Conversely, given a character $\psi_\C : F^\times \backslash \A_F^\times \rightarrow \C^\times$ such that $\psi_\C(z) = z_\infty^{2-\mathbf{k}-2\mathbf{w}}$ on some open subgroup of $\A_F^\times$, we let $\psi$ denote the character $\psi : F^\times \backslash \A_F^\times \rightarrow \mathcal{O}^\times$ (enlarging $\mathcal{O}$ if necessary) given by $\psi(z) = \psi_\C(z)z_\infty^{\mathbf{k}+2\mathbf{w}-2}z_p^{2-\mathbf{k}+2\mathbf{w}}$.

In what follows we can always ensure that the base changes performed are disjoint from any fixed finite extension $K/F$ as follows. Let $K'$ denote the normal closure of $K$. Let $\{v\}$ be a finite set of places of $F$, unramified in $K'$ and such that every conjugacy class in $\Gal(K'/F)$ contains one such $\mathrm{Frob}_v$. Note that this set can always be chosen to be disjoint from any given finite set of places of $F$. We then demand that our extension $F'/F$ splits at all places in $\{v\}$. We will not repeat this argument in each of the lemmas below.

We will call an extension $F'/F$ a $p$-\textit{split allowable base change} if 
	\begin{itemize}
		\item[-] $F'/F$ is finite of even degree, solvable, and totally real;
		\item[-] $F'/F$ is disjoint from $\overline{F}^{\ker \overline{\rho}}$;
		\item[-] any $v|p$ in $F$ splits completely in $F'$.
	\end{itemize}
For a given quadratic extension $L/F$, we say that an extension $F'/F$ is an $L$-\textit{allowable base change} if 
	\begin{itemize}
		\item[-] $F'/F$ is finite of even degree, solvable, and totally real;
		\item[-] $F'/F$ is disjoint from $\overline{F}^{\ker \overline{\rho}}$;
		\item[-] if $\rhobar|_{G_L}$ is abelian, and $v|p$ in $F$ does not split in $L$, then no $w|v$ in $F'$ splits in $F'L$.
	\end{itemize}
Clearly, for any quadratic $L/F$, a $p$-split allowable base change is an $L$-allowable base change. In order to show that our automorphic representations have the desired central character, we are forced to perform base changes that are (possibly) ramified at places above $p$. However, in order to apply the results of \S \ref{RredT} in the case that $\overline{\rho}$ is induced from a quadratic CM extension $L/F$, we need to ensure that $\overline{\rho}|_{G_{F'}}$ still satisfies the assumptions of that section; in particular, that there is some $v'|2$ in $F'$ such that $\overline{\rho}|_{G_{v'}}$ is nonslplit. This is why we introduce the notion of an $L$-allowable base change.

\begin{lem}\label{WeightChange}

Let $\pi$ be a regular algebraic cuspidal automorphic representation of $\GL_2(\A_F)$ of weight $\kappa$ and central character $\psi_\C$. Assume that $\pi$ is $p$-nearly ordinary and a $\overline{\chi}$-good lift of $\overline{\rho}$. Let $\kappa' = (\mathbf{k}',\mathbf{w}')$ be another algebraic weight and let $\psi'_\C : F^\times \backslash \A_F^\times \rightarrow \C^\times$ be such that $\psi'_\C(z) =  z_\infty^{2-\mathbf{k}'-2\mathbf{w}'}$ on some open subgroup of $\A_F^\times$.  Let $L/F$ be some fixed quadratic extension.

Assume that $\psi \cong \psi'$ modulo the maximal ideal of $\mathcal{O}$. Then, there is an $L$-allowable base change $F'/F$ and a $p$-nearly ordinary regular algebraic cuspidal automorphic representation $\pi'$ of $\GL_2(\A_{F'})$ such that
\begin{itemize}
	\item[-] the central character of $\pi'$ is $\mathrm{Nm}_{F'/F}\circ \psi_\C'$;
	\item[-] if $w$ is an archimedean place of $F'$ that extends the embedding $\tau : F \hookrightarrow \R$, then $\pi'_w$ is discrete series of lowest weight $k_\tau'-1$ and central character $z_v \mapsto \mathrm{sgn}(z_v)^{k_\tau'} \lvert z_v \rvert^{2-k_\tau'-2w_\tau'}$;
	\item[-] for $v|p$, we have $(\pi'_v)^{\mathrm{Iw}_1(v)} \ne 0$;
	\item[-] $\pi'$ is a $\overline{\chi}$-good lift of $\rhobar$.
\end{itemize}
\end{lem}

\begin{proof}

Take some $v|p$. If $v$ does not split in $L$, we let $L_v$ be the completion of $L$ at the unique prime above $v$. Let $\mathcal{U}_v^1$ be the maximal pro-$p$ subgroup of $\mathcal{O}_{F_v}^\times$, and recall that $J_{F_v}$ denotes the set of embeddings $F_v \hookrightarrow E$. Let $\phi_v$ be the finite order $\mathcal{O}$-valued character of $\mathcal{U}_v^1$ given by 
	\[ \phi_v(z_v) = \psi'(z_v)\prod_{\tau \in J_{F_v}} z_v^{2-k_\tau' - 2w_\tau'}.
	\]
We want a finite totally ramified extension $F_v'/F_v$ of order a power of $p$ such that $\phi_v\circ\mathrm{Nm}_{F_v'/F_v}$ is trivial. However, if $v$ does not split in $L$, we further require this extension is disjoint from $L_v/F_v$. If $p$ is odd or if $L_v/F_v$ is unramified, this is automatic. In the case that $L_v/F_v$ is ramified, there is a choice of uniformizer $\varpi_v$ in $F_v$ such that $\varpi_v$ is not a norm from $L_v$. By class field theory, the subgroup of $F_v^\times$ generated by the kernel of $\phi_v$, the group of prime-to-$p$ roots of unity, and the element $\varpi_v$ then gives the desired extension. We fix such an $F_v'/F_v$ for each $v|p$. We now let $F'/F$ be a finite solvable, totally real extension of even degree such that for any $v|p$ in $F$ and $w|v$ in $F'$, the completion of $F'$ at $w$ is our fixed extension $F_v'/F_v$. Note that by choice of the $F_v'$, we have
\begin{itemize}
	\item[-] $\psi'( z_p)\circ\mathrm{Nm}_{F'/F} = z_p^{2-\mathbf{k}'-2\mathbf{w}'}$ for every $z_p \in (\mathcal{U}_p')^1$, where $(\mathcal{U}_p')^1$ is the maximal pro-$p$ subgroup of $(\mathcal{O}_{F'} \otimes_\Z \Z_p)^\times$;
	\item[-] if $\rhobar|_{G_L}$ is abelian and $v|p$ in $F$ does not split in $L$, then no $w|v$ in $F'$ splits in $F'L$.
\end{itemize}

We again denote by $\kappa$ and $\kappa'$ the weights of $F'$ obtained from $\kappa$ and $\kappa'$ in the obvious way. We also again write $\psi_\C$ and $\psi_\C'$ instead of $\psi_\C \circ \mathrm{Nm}_{F'/F}$ and $\psi_\C'\circ \mathrm{Nm}_{F'/F}$. We again denote by $\pi$ the base change of $\pi$ to $F'$.

Let $D$ be the quaternion algebra over $F'$ ramified at all infinite places and split at all finite places. Since $\pi_v$ is either principal series or Steinberg at each $v|p$, there is some $a\ge 1$ such that $\pi_v^{\mathrm{Iw}_1(v^a)} \ne 0$ for any $v|p$ in $F'$. Let $U \subseteq \GL_2(\mathcal{O}_{F'}\otimes_\Z \widehat{\Z})$ be a compact open subgroup such that $U_v = \mathrm{Iw}_1(v^a)$ for all $v|p$, and small enough such that $\pi^U \ne 0$ and 
	\begin{equation}\label{adneatagain}
	(U(\A_{F'}^\infty)^\times \cap t^{-1} D^\times t)/(F')^\times = 1
	\end{equation}
for every $t \in (D\otimes_{F'} \A_{F'}^\infty)^\times$.

Using the Jacquet-Langlands-Shimizu correspondence, cf. \ref{JacquetLanglands}, we transfer $\pi$ to a Hecke eigenform in $S_{\kappa,\psi}^\mathrm{no}(U,\mathcal{O})$ (enlarging $\mathcal{O}$ if necessary). Since $U$ satisfies \eqref{adneatagain}, we have
	\[ S_{\kappa,\psi}^\mathrm{no}(U,\mathcal{O}) \otimes_\mathcal{O} \F
		\cong S_{\kappa,\psi}^\mathrm{no}(U,\F), \]
and similarly for $S_{\kappa',\psi'}^\mathrm{no}(U,\mathcal{O})$. Since $U = U(p^{1,1})$ and $\psi \cong \psi'$ modulo the maximal ideal of $\mathcal{O}$, \ref{WeightChangeFin} gives Hecke equivariant isomorphisms
	\[ S_{\kappa,\psi}^\mathrm{no}(U,\F) \cong S_{2,\psi}^\mathrm{no}(U,\F)
		\cong S_{\kappa',\psi'}^\mathrm{no}(U,\F) .
	\]
This implies the existence of a Hecke eigenform $g$ in $S_{\kappa',\psi'}^\mathrm{no}(U,\mathcal{O})$ that is a $\overline{\chi}$-good lift of $\rhobar$. The fact that the lift is $\overline{\chi}$-good follows form the fact that the above isomorphism is equivariant for $T_{\varpi_v}$ for each $v|p$, and for $\langle y \rangle^\mathrm{no}$ for each $y \in (\mathcal{O}_{F'} \otimes_\Z \Z_p)^\times$.

Now let $V \subseteq \GL_2(\mathcal{O}_{F'}\otimes_\Z \widehat{\Z})$ be another compact open subgroup such that $V_v = \mathrm{Iw}_1(v)$ for all $v|p$, such that $V_v \subseteq U_v$ for all $v\nmid p$, and small enough such that $V$ also satisfies \eqref{adneatagain} (with $V$ in place of $U$). Note that $V\cap U = V(p^{a,a})$. Viewing $g$ as a Hecke eigenform in $S_{\kappa',\psi'}^\mathrm{no}(V(p^{a,a}),\mathcal{O})$, we let $\mathfrak{m}$ denote the corresponding maximal ideal of the universal nearly ordinary Hecke algebra $\mathbf{T}_{\psi'}(V)$. In particular, $S_{\kappa',\psi'}^\mathrm{no}(V(p^{a,a}),\mathcal{O})_\mathfrak{m} \ne 0$. Since $V$ satisfies \eqref{adneatagain} (with $V$ in place of $U$) and $V(p^{a,a})/V$ is a $p$-group, we get that $S_{\kappa',\psi'}^\mathrm{no}(V(p^{a,a}),\mathcal{O})_\mathfrak{m}$ is free over $\mathcal{O}[V(p^{a,a})/V]$; hence, $S_{\kappa',\psi'}(V,\mathcal{O})_\mathfrak{m} \ne 0$. Letting $f'$ be any Hecke eigenform in $S_{\kappa',\psi'}(V,\mathcal{O})_\mathfrak{m}$, its transfer $\pi'$ to $\GL_2(\A_{F'})$ satisfies the requirements of the lemma.
	\end{proof} 

\begin{lem}\label{LevelLower}

Let $\pi$ be a regular algebraic cuspidal automorphic representation of $\GL_2(\A_F)$ with central character $\psi_\C$. Assume that $\pi$ is a $p$-nearly ordinary $\overline{\chi}$-good lift of $\overline{\rho}$. Let $\Sigma$ be some (possibly empty) set of finite places of $F$, disjoint from $\{v|p\}$, such that for each $v\in \Sigma$, we have $\pi_v \cong (\gamma_v \circ \det) \otimes\mathrm{St}$ with $\gamma_v$ a character of $F_v^\times$.

There is a $p$-split allowable base change $F'/F$ and a $p$-nearly ordinary regular algebraic cuspidal automorphic representation $\pi'$ of $\GL_2(\A_{F'})$, such that
	\begin{itemize}
	\item[-] the central character of $\pi'$ is $\psi_\C \circ \mathrm{Nm}_{F'/F}$;
	\item[-] for any archimdean place $v$ of $F$ and $w|v$ in $F'$, $\pi_v \cong \pi'_w$ as $(\mathfrak{gl}_2,\mathrm{O}(2))$-modules;
	\item[-] for any finite place $w$ not above $p$ or any of the places of $\Sigma$, $\pi'_v$ is unramified;
	\item[-] for $v\in \Sigma$ and $w|v$, $\pi'_w \cong (\gamma_v\circ\mathrm{Nm}_{F'_w/F_v}\circ
	\det)\otimes\mathrm{St}$;
	\item[-] for any $v|p$ in $F$, if $a\ge 1$ is such that $\pi_v^{\mathrm{Iw}_1(v^a)} \ne 0$, then for any $w|v$ in $F'$, we have $(\pi_w')^{\mathrm{Iw}_1(w^a)} \ne 0$;
	\item[-] $\pi'$ is a $\overline{\chi}$-good lift of $\overline{\rho}$.
	\end{itemize}
\end{lem}

\begin{proof}

For ease of notation, we will throughout denote the base change of a character by the same letter, in particular again write $\psi$ and $\gamma_v$ for $\psi \circ \mathrm{Nm}_{F'/F}$ and $\gamma_v \circ \mathrm{Nm}_{F'_w/F_v}$, respectively. For each $v|p$ in $F$, let $a_v \ge 1$ be minimal such that $\pi_v^{\mathrm{Iw}_1(v^{a_v})} \ne 0$.

Let $F_1/F$ be a $p$-split allowable base change such that, letting $\pi_1$ denote the base change of $\pi$ to $F_1$, if $v$ is a finite place of $F_1$ with $v \nmid p$ at which $(\pi_1)_v$ is ramified, then $(\pi_1)_v$ is an unramified twist of the Steinberg representation, and such that the number of places above $\Sigma$ in $F_1$ is even. Let $\Sigma_1$ denote the set of places of $F_1$ above $\Sigma$. Let $Q_1$ denote the set of places $v$ in $F_1$ such that $(\pi_1)_v$ is ramified and $v \notin \Sigma_1 \cup\{v|p\}$. 

Choose a finite place $w_0 \notin \Sigma_1\cup Q \cup \{v|p\}$, and let $N_{w_0}$ be the order of $\GL_2(k_{w_0})$. Choose a $p$-split allowable base change $F_2/F_1$, split at $w_0$, such that for any finite place $v$ above $Q_1$, the order of the $p$-subgroup of $k_v^\times$ is divisible by the $p$-part of $2p(4N_{w_0})$. Let $\Sigma_2$ denote the set of places of $F_2$ above $\Sigma_1$ and similarly for $Q_2$. Let $\pi_2$ denote the base change of $\pi_1$ to $F_2$. Because $F_2/F$ is split completely at places above $p$ in $F$, for any $v|p$ in $F$ and $w|v$ in $F_w$, we have $(\pi_2)_w \cong \pi_v$ as $\GL_2(F_w) \cong \GL_2(F_v)$ representations; in particular, $(\pi_2)_w^{\mathrm{Iw}_1(w^{a_v})} \ne 0$.

Let $D$ denote the quaternion algebra over $F_2$ ramified at all archimedean places and all places in $\Sigma_2$, and split elsewhere. Fix a maximal order $\mathcal{O}_D$ of $D$ and isomorphisms $(\mathcal{O}_D)_v \cong \mathrm{M}_{2\times 2}(\mathcal{O}_{F_v})$ for every split $v$. Let $U$ denote the open subgroup of $(D\otimes_{F_2} \A_{F_2}^\infty)^\times$ such that
	\begin{itemize}
	\item[-] $U_v = \GL_2(\mathcal{O}_{F_{2,v}})$ for $v \notin \Sigma_2 \cup Q_2 \cup \{w|p\}$;
	\item[-] $U_v = D_v^\times$ for $v \in \Sigma_2$;
	\item[-] $U_v = \mathrm{Iw}(v)$ for $v \in Q_2$;
	\item[-] $U_w = \mathrm{Iw}_1(w^{a_v})$ for $w|p$.
	\end{itemize}
By our choice of $F_2$, if $V$ denotes the open compact subgroup of $\GL_2(\A_{F_2}^\infty)$ given by $V_v = U_v$ for $v\notin \Sigma_2$ and $V_v = \mathrm{Iw}(v)$ for $v\in \Sigma_2$, then $\pi_2^{V} \ne 0$. Let $\kappa$ denote the weight of $\pi_2$. Via the Jaquet-Langlands-Shimizu correspondence, $\pi_2$ is generated by a Hecke eigenform in $S_{\kappa,\psi}^\mathrm{no}(U,\mathcal{O})$ (enlarging $\mathcal{O}$ if necessary).

Define an open subgroup $U'$ of $(D\otimes_{F_2}\A_{F_2}^\infty)^\times$ be letting $U'_v = U_v$ for $v \notin Q_2$, and for $v \in Q_2$ we set
	\[ U'_v = \left\{ \left( \begin{array}{cc} a & b \\ c & d \end{array} \right) \in \mathrm{Iw}(v) :
		a = d \;\mathrm{mod}\; \mathfrak{m}_v \right\}.
	\]
By \ref{CharKillIsotropy}, we can take a non-trivial character $\chi : \prod_{v\in Q_2} k_v^\times \rightarrow \mathcal{O}^\times$ of $p$-power order, and of order divisible by $4$ if $p=2$, such that when viewed as a character of $U(\A_{F_2}^\infty)^\times$, with kernel containing $U'(\A_{F_2}^\infty)^\times$, it is trivial on $(U(\A_{F_2}^\infty)^\times \cap t^{-1} D^\times t)/F^\times$ for any $t \in (D \otimes_{F_2} \A_{F_2}^\infty)^\times$. 

Viewing $W_\kappa(\mathcal{O}) \otimes_\mathcal{O} \mathcal{O}(\chi)$ as a representation of $U(\A_F^\infty)^\times$, we let $S_{\kappa\otimes\chi,\psi}(U,\mathcal{O})$ denote the space of functions
	\[ f : D^\times \backslash (D \otimes_F \A_{F_2}^\infty)^\times \longrightarrow
		W_\kappa(\mathcal{O}) \otimes_\mathcal{O} \mathcal{O}(\chi)
	\]
such that $f(xu) = u^{-1} f(x)$ for all $x \in U$ and $f(zx) = \psi(z) f(z)$ for all $z\in (\A_{F_2}^\infty)^\times$. Fixing a set of representatives $\{t\}$ for the double cosets $D^\times\backslash (D \otimes_{F_2} \A_{F_2}^\infty)^\times/U(\A_{F_2}^\infty)^\times$, we have an isomorphism of $\mathcal{O}$-modules
	\[
	S_{\kappa\otimes\chi,\psi}(U,\mathcal{O}) \cong \bigoplus_{\{t\}} 
	W_\kappa(\mathcal{O})(\chi)^{(U(\A_{F_2}^\infty)^\times \cap t^{-1}D^\times t)/F_2^\times}.
	\]
Since $\chi$ is trivial on each of the $(U(\A_{F_2}^\infty)^\times \cap t^{-1}D^\times t)/F_2^\times$, and $\chi$ is of $p$-power order, $S_{\kappa\otimes\chi,\psi}(U,\mathcal{O})$ and $S_{\kappa,\psi}(U,\mathcal{O})$ have the same $\mathcal{O}$-rank and the same image in $S_{\kappa,\psi}(U,\F)$. The isomorphism
	\begin{equation}\label{actertainiso}
	S_{\kappa,\psi}(U,\mathcal{O})\otimes_\mathcal{O}\F
	\cong S_{\kappa\otimes\chi,\psi}(U,\mathcal{O})\otimes_\mathcal{O}\F
	\end{equation}
then implies the existence of a nearly ordinary maximal ideal $\mathfrak{m}$ in the Hecke subalgebra of $\mathrm{End}(S_{\kappa\otimes\chi,\psi}(U,\mathcal{O}))$, with $\overline{\rho}_\mathfrak{m} \cong \overline{\rho}$. Note that the nearly ordinary condition is preserved since $Q$ does not containing any places above $p$, so the isomorphism \eqref{actertainiso} is equivariant for $T_{\varpi_v}$ for each $v|p$, and for $\langle y \rangle^\mathrm{no}$ for each $y \in (\mathcal{O}_{F'} \otimes_\Z \Z_p)^\times$.

Choose a Hecke eigenform in $S_{\kappa\otimes\chi,\psi}(U,\mathcal{O})_\mathfrak{m}$, and $\pi_2'$ denotes the cuspidal automorphic representation of $\GL_2(\A_{F_2})$ obtained from it via the Jacquet-Langlands-Shimizu correspdonance. Then we have
	\begin{itemize}
	\item[-] the central character of $\pi_2'$ is $\psi_\C \circ \mathrm{Nm}_{F_2/F}$;
	\item[-] for any archimdean place $v$ of $F$ and $w|v$ in $F_2$, $(\pi'_2)_w \cong  \pi_v$ as $(\mathfrak{gl}_2,\mathrm{O}(2))$-modules;
	\item[-] for any finite place $w \notin \Sigma_2 \cup Q_2$ and not above $p$, $(\pi_2')_w$ is unramified;
	\item[-] for any $v \in \Sigma_2$, $(\pi_2')_v \cong (\gamma_v \circ \det)\otimes\mathrm{St}$;
	\item[-] for any $v \in Q_2$, $(\pi_2')_v$ is a ramified principal series;
	\item[-] $\pi_2'$ is $p$-nearly ordinary;
	\item[-] for each $v|p$, $(\pi_2')_v^{\mathrm{Iw}_1(v)} \ne 0$;
	\item[-] $\pi_2'$ is a $\overline{\chi}$-good lift of $\overline{\rho}|_{G_{F_2}}$.
	\end{itemize}
We can then find another $p$-split allowable base change $F'/F_2$, such that the characters defining the principal series representations $(\pi'_2)_v$, for $v\in Q_2$, become unramified. Letting $\pi'$ denote the base change of $\pi_2'$ to $F'$ gives the result.
\end{proof}

\begin{lem}\label{SwitchSt}
Assume $p=2$. Let $D$ be a quaternion algebra with centre $F$, ramified at all archimedean places and at a nonempty set $\Sigma$ of finite places not containing any places above $p$. Fix a maximal order $\mathcal{O}_D$ of $D$. Fix an algebraic weight $\kappa$, a continuous character $\psi : F^\times \backslash (\A_F^\infty)^\times \rightarrow \mathcal{O}^\times$, and an open subgroup $U$ of $(D \otimes_F \A_F^\infty)^\times$ such that
	\begin{itemize}
	\item[-] $U_v\subseteq \GL_2(\mathcal{O}_{F_v})$ for $v$ at which $D$ is split;
	\item[-] $U_v = D_v^\times$ for each $v \in \Sigma$, 
	\item[-] $U_v \supseteq \mathrm{Iw}(v)$ for each $v|p$,
	\item[-] the action of $U\cap (\A_F^\infty)^\times$ on $W_\kappa(\mathcal{O})$ is given by $\psi^{-1}$,
	\item[-] $(U(\A_F^\infty)^\times \cap t^{-1} D^\times t)/F^\times = 1$ for each $t \in (D \otimes_F \A_F^\infty)^\times$. 
	\end{itemize}
Let $U'$ denote the maximal compact subgroup of $U$. Let $\mathfrak{m}$ be a maximal ideal of $\mathbf{T}_{\kappa,\psi}^\mathrm{no}(U',\mathcal{O})$ and let $f \in S_{\kappa,\psi}^\mathrm{no}(U',\mathcal{O})_\mathfrak{m}$ be an eigenform such that for each $v \in \Sigma$, the action of $D_v^\times$ on $f$ is via $\gamma_v\circ \nu_D$, with $\gamma_v : F_v^\times \rightarrow \calO^\times$ an unramified character and $\nu_D$ the reduced norm of $D$.

For each $v\in \Sigma$ let $\gamma_v' : F_v^\times \rightarrow \mathcal{O}^\times$ be either $\gamma_v$ or $-\gamma_v$. There is a Hecke eigenform $f' \in S_{\kappa,\psi}^\mathrm{no}(U',\mathcal{O})_\mathfrak{m}$ such that $D_v^\times$ acts on $f'$ via $\gamma_v'$ for each $v \in \Sigma$.

\end{lem}

\begin{proof}

Let $\Delta = \prod_{v \in \Sigma} F_v^\times / (F_v^\times)^2\calO_{F_v}^\times$. The reduced norm defines an isomorphism
	\[ U(\A_F^\infty)^\times / U'(\A_F^\infty)^\times 
	\stackrel{\sim}{\longrightarrow} \Delta. 
	\]
Fix coset representative $\{t\}$ for $D^\times \backslash (D \otimes_F \A_F^\infty)^\times/U(\A_F^\infty)^\times$. By \ref{freeovergroupring}, our assumptions on $U$ implie that we have an isomorphism of $\calO[\Delta]$-modules
	\[ S_{\kappa,\psi}(U'(p^{a,a}),\calO) \cong \bigoplus_{\{t\}} W_{\kappa}(\calO) \otimes_\calO
		\calO[\Delta], \]
and $S_{\kappa,\psi}(U',\calO)$ is free over $\calO[\Delta]$. Since $\Delta$ is a $2$-group, $\calO[\Delta]$ is a local ring and $S_{\kappa,\psi}^\mathrm{no}(U'(p^{a,a}),\calO)_\frakm$ is also free over $\calO[\Delta]$. Since
	\[ E[\Delta] \cong \prod_{\chi : \Delta \rightarrow \{\pm 1\}} E(\chi) \]
we see that the $\calO$-rank of the submodule of $S_{\kappa,\psi}^\mathrm{no}(U',\calO)_\frakm$ consisting of elements on which $D_v^\times$ acts via $\gamma_v$ for all $v\in\Sigma$ is equal to the $\calO$-rank of the submodule on which $D_v^\times$ acts via $\gamma_v'$ for all $v\in\Sigma$. \end{proof}

\begin{lem}\label{LevelRaiseSt}

Let $\pi$ be a $p$-nearly ordinary regular algebraic cuspidal automorphic representation of $\GL_2(\A_F)$ that is a $\overline{\chi}$-good lift of $\overline{\rho}$, and let with $\psi_\C$ denote its central character. Let $\Sigma$ be a finite set of finite places not containing any above $p$ such that for every $v \in \Sigma$, $\pi_v$ is unramified and $\overline{\rho}|_{G_v}$ is an extension of $\overline{\gamma}_v$ by $\overline{\gamma}_v\overline{\epsilon_p}$, for some unramified character $\gamma_v : G_v \rightarrow \mathcal{O}^\times$ with $\gamma_v^2 = \psi|_{G_v}$.

Fix some finite place $w_0$ not in $\Sigma$ and not above $p$. There is a $p$-split allowable base change $F'/F$ and a $p$-nearly ordinary regular algebraic cuspidal representation $\pi'$ of $\GL_2(\A_F')$ such that
	\begin{itemize}
	\item[-] the central character of $\pi'$ is $\psi_\C \circ \mathrm{Nm}_{F'/F}$;
	\item[-] for any archimdean place $v$ of $F$ and $w|v$ in $F'$, $\pi'_w \cong \pi_v$ as $(\mathfrak{gl}_2,\mathrm{O}(2))$-modules;
	\item[-] $\pi$ is unramified outside the places above $\Sigma$, the places above $p$, and the places above $w_0$;
	\item[-] for any $v \in \Sigma$ and $w|v$ in $F'$, $\pi'_w \cong (\gamma_v\circ\mathrm{Nm}_{F'_w/F_v} \circ \det) \otimes \mathrm{St}$;
	\item[-] for any $v|p$ in $F$ and $w|v$ in $F'$, if $a\ge 1$ is such that $\pi_v^{\mathrm{Iw}_1(v^a)} \ne 0$, then $(\pi_w')^{\mathrm{Iw}_1(w^a)} \ne 0$;
	\item[-] $\pi'$ is a $\overline{\chi}$-good lift of $\overline{\rho}$.
	\end{itemize}
\end{lem}

\begin{proof}

We prove this as in \cite{KisinFinFlat}*{Lemma 3.5.3}. For ease of notation, we will throughout denote the base change of a character by the same letter, in particular again write $\psi$ and $\gamma_v$ for $\psi \circ \mathrm{Nm}_{F'/F}$ and $\gamma_v$ for $\gamma_v \circ \mathrm{Nm}_{F'_w/F_v}$, respectively. Also, given an algebraic weight $\kappa =(\mathbf{k},\mathbf{w})$ for $F$ we will again denote by $\kappa = (\mathbf{k},\mathbf{w})$ the algebraic weight for $F'$ given by letting $(k_{\tau'},w_{\tau'}) = (k_\tau,w_\tau)$ if $\tau' : F' \hookrightarrow \R$ extends $\tau : F \hookrightarrow \R$. For each $v|p$, let $a_v \ge 1$ be minimal such that $\pi_v^{\mathrm{Iw}(v^{a_v})} \ne 0$.

 By first performing a quadratic $p$-split allowable base change if necessary, we may assume that $[F:\Q]$ is even. Let $\kappa = (\mathbf{k},\mathbf{w})$ denote the weight of $\pi$, and $\psi$ denote the $\mathcal{O}^\times$-valued character of $(\A_F^\infty)^\times$ corresponding to the central character of $\pi$. Let $D_0$ denote the quaternion algebra ramified at all archimedean places of $F$ and split at all finite places of $F$. Fix some finite place $w_0 \notin \Sigma\cup \{v|p\}$. Let $U$ be an open compact subgroup of $\GL_2(\mathcal{O}_F \otimes_\Z \widehat{\Z})$ such that 
\begin{itemize}
	\item[-] $U_v = \GL_2(\mathcal{O}_{F_v})$ for all finite places $v\nmid p w_0$ such that $\pi_v$ is unramified;
	\item[-] $U_v = \mathrm{Iw}_1(v^{a_v})$ for each $v|p$;
	\item[-] $\pi^U \ne 0$;
	\item[-] $(U(\A_F^\infty)^\times \cap t^{-1} D^\times t)/F^\times = 1$ for every $t \in (D \otimes_F \A_F^\infty)^\times$.
\end{itemize}
The last condition can be ensured by taking $U_{w_0}$ small enough. By the Jacquet-Langlands-Shimizu correspondence, $\pi$ is generated by some $\mathbf{T}_{\kappa,\psi}^\mathrm{no}(U,\mathcal{O})$-eigenform in $S_{\kappa,\psi}^\mathrm{no}(U,\mathcal{O})$ (enlarging $\mathcal{O}$ if necessary).

Take $2 \le k \le p+1$ and $w\in \Z$ such that $k+2w = k_\tau + 2w_\tau$ for any $\tau \in I$, and let $\kappa_0 = ((k,\ldots,k),(w,\ldots,w))$. We have a Hecke equivariant isomorphism, cf. \ref{WeightChangeFin},
	\[ 
	S_{\kappa,\psi}^\mathrm{no}(U,\mathcal{O}) \otimes_\mathcal{O} \F \cong
	S_{\kappa_0,\psi}^\mathrm{no}(U,\mathcal{O}) \otimes_\mathcal{O} \F.
	\]
So, there is an eigenfunction $f_0 \in S_{\kappa_0,\psi}^\mathrm{no}(U,\mathcal{O})$ that is a $\overline{\chi}$-good lift of $\overline{\rho}$.

Write $\Sigma = \{v_1,\ldots,v_r\}$. Choose a tower of quadratic extensions $F = F_0 \subset F_1 \subset \cdots \subset F_r$, with each $F_i/F_{i-1}$ a $p$-split allowable base change such that for any $1 \le j \le r$, any $w \in F_{i-1}$ above $v_j$ splits in $F_i$ if $i=j$ and is inert in $F_i$ otherwise. Let $D_i$ be the quaternion algebra with centre $F_i$, ramified at all archimedean places and all places above $v_1,\ldots,v_i$, and split elsewhere. Let $\mathcal{O}_{D_i}$ be a fixed maximal order of $D_i$. We will show, by induction, that there is an open subgroup $U_i$ of $(D_i \otimes_{F_i} \A_{F_i}^\infty)^\times$ and an eigenfunction $f_i \in S_{\kappa_0,\psi}^\mathrm{no}(U_i,\mathcal{O})$ such that
	\begin{enumerate}
		\item[(a)] for any $1\le j\le i$, if $w|v_j$ then $(U_i)_w = (D_i)_w^\times$ acting on $W_{\kappa_0}(\mathcal{O})$ by $\gamma_{v_j}^{-1} \circ \nu_{D_i}$, with $\nu_{D_i}$ the reduced norm of $D_i$;
		\item[(b)] $(U_i)_w = \mathrm{Iw}_1(w^{a_v})$ for any $v|p$ in $F$ and $w|v$ in $F_i$;
		\item[(c)] $(U_i)_w = \GL_2(\mathcal{O}_{F_{i,w}})$ for any finite $w$ not above $\{v_1,\ldots,v_i,w_0\}$ such that $U_v = \GL_2(\mathcal{O}_{F_v})$;
		\item[(d)] $(U_i(\A_{F_i}^\infty)^\times \cap t^{-1} D_i^\times t)/F_i^\times = 1$ for any $t \in (D_i \otimes_{F_i} \A_{F_i}^\infty)^\times$;
		\item[(e)] $f_i$ is a $\overline{\chi}$-good lift of $\overline{\rho}$.
	\end{enumerate}

Take $0\le i <r$, and denote by $\mathfrak{m}$ the maximal ideal of $\mathbf{T}^\mathrm{no}_{\kappa_0,\psi}(U_i,\mathcal{O})$ corresponding to $f_i$. Since $2 \le k \le p+1$, there is a perfect pairing $\langle \;,\; \rangle$ on $W_k(\mathcal{O})$, cf. \cite{TaylorMeroDeg2}*{\S 1}, and thus a perfect pairing on $W_{\kappa_0}(\mathcal{O})$, which we also denote by $\langle \;,\; \rangle$. Fix a set of coset representatives $\{t\}$ for $D_i^\times\backslash (D_i \otimes_{F_i} \A_{F_i}^\infty)^\times/U(\A_{F_i}^\infty)^\times$. Assumption (c) allows us to define  a perfect pairing on $S_{\kappa_0,\psi}(U_i,\mathcal{O})$ by
	\[
	\langle h_1 , h_2 \rangle_i = \sum_{\{t\}} \langle h_1(t), h_2(t) \rangle \psi\circ\nu_{D_i}(t)^{-1}.
	\]
Define an open subgroup $U_i'$ of $(D_i \otimes_{F_i} \A_{F_i}^\infty)^\times$ by letting $(U_i')_w = (U_i)_w$ if $w$ is not above $v_{i+1}$ and $(U_i')_w = \mathrm{Iw}(w)$ if $w$ is the unique place in $F_i$ above $v_{i+1}$. By our assumption on $\overline{\rho}|_{G_{v_{i+1}}}$, we see that
	\[
	(T_{v_{i+1}}^2 - (\mathrm{Nm}(v_{i+1})+1)^2\psi(\varpi_{v_{i+1}})) S_{\kappa_0,\psi}
	(U_i,\mathcal{O}) \subset \mathfrak{m} S_{\kappa_0,\psi}(U_i,
	\mathcal{O}).
	\]
Pulling back $\mathfrak{m}$ to a maximal ideal of $\mathbf{T}_{\kappa_0,\psi}^\mathrm{no}(U_i',\mathcal{O})$, \cite{KisinFinFlat}*{Corollary 3.1.11} shows there is an eigenfunction $h \in S_{\kappa_0,\psi}^\mathrm{no}(U_i',\mathcal{O})$ that is in the support of $\mathfrak{m}$ and is $w$-new, for $w$ the unique place of $F_i$ above $v_{i+1}$. If $(D_i)_{v_{i+1}}^\times$ does not act on $h$ by $\gamma_{v_{i+1}}\circ \nu_{D_i}$, then we must have $p=2$ and $(D_i)_{v_{i+1}}^\times$ acts on $h$ via $-\gamma_{v_{i+1}}\circ\nu_{D_i}$. Applying \ref{SwitchSt} allows us to assume $(D_i)_{v_{i+1}}^\times$ acts on $h$ by $\gamma_{v_{i+1}}\circ \nu_{D_i}$. Considering the base change of $h$ to $F_{i+1}$ together with the Jacquet-Langlands-Shimizu correspondence then yields the desired $f_{i+1}$, and with $U_{i+1}$ an open subgroup of $(D_{i+i}\otimes_{F_{i+1}}\A_{F_{i+1}}^\infty)^\times$ such that
	\begin{itemize}
		\item[-] $(U_{i+1})_w = (D_{i+1})_w^\times$ for any $1\le j \le i+1$ and $w|v_j$;
		\item[-] $(U_{i+1})_w = \mathrm{Iw}_1(w^{a_v})$ for any $v|p$ in $F$ and $w|v$ in $F_{i+1}$;
		\item[-] $(U_{i+1})_w = \GL_2(\mathcal{O}_{F_{i+1,w}})$ for any finite $w$ not above $\{v_1,\ldots,v_{i+1},w_0\}$ such that $U_v = \GL_2(\mathcal{O}_{F_v})$;
		\item[-] for $w|w_0$ in $F_{i+1}$, we take $(U_{i+1})_w$ small enough so that $(U_{i+1}(\A_{F_{i+1}}^\infty)^\times \cap t^{-1} D_{i+1}^\times t)/F_{i+1}^\times = 1$ for any $t \in (D_{i+1} \otimes_{F_{i+1}} \A_{F_{i+1}}^\infty)^\times$
	\end{itemize}

Having obtained $f_r$, we again use \ref{WeightChangeFin} to obtain a $\overline{\chi}$-good lift $f_r'$ of $\overline{\rho}$ of level $U_r$ and weight $\kappa'$, where $(k_{\tau'},w_{\tau'}) = (k_\tau,w_\tau)$ for any $\tau' \in I_{F_r}$ extending $\tau \in I$. Then, applying Jacquet-Langlands-Shimizu to $f_r'$ yields the result.  \end{proof}

We now group most of the above lemmas into one proposition.

\begin{prop}\label{SpecifiedLift}

Let $\pi$ be a $p$-nearly ordinary regular algebraic cuspidal automorphic representation of $\GL_2(\A_F)$ that is a $\overline{\chi}$-good lift of $\overline{\rho}$. Let $\kappa = (\bold{k},\bold{w})$ be an algebraic weight and $\psi_\C : F^\times \backslash \A_F^\times \rightarrow \C^\times$ a character such that $\psi_\C(z) = z_\infty^{2-\bold{k}-2\bold{w}}$ on some open subgroup of $\A_F^\times$. Let $\Sigma$ be a finite set of finite places not containing any places above $p$. For each $v \in \Sigma$, fix an unramified character $\gamma_v : F_v \rightarrow \mathcal{O}^\times$ such that $\overline{\rho}|_{G_v}$ is an extension of $\overline{\gamma}_v$ by $\overline{\gamma}_v\overline{\epsilon_p}$ and such that $\gamma_v(\varpi_v)^2 = \psi(\varpi_v)$. Fix a quadratic extension $L/F$.

There is an $L$-allowable base change $F'/F$ and a $p$-nearly ordinary cuspidal automorphic representation $\pi'$ of $\GL_2(\A_{F'})$ such that
	\begin{itemize}
		\item[-] $\pi$ has central character $\psi_\C \circ \mathrm{Nm}_{F'/F}$;
		\item[-] if $w$ is an archimedean place of $F'$ that extends the embedding $\tau : F \hookrightarrow \R$, then $\pi'_w$ is discrete series of lowest weight $k_\tau-1$ and central character $z_w \mapsto \mathrm{sgn}(z_w)^{k_\tau} \lvert z_w \rvert^{2-k_\tau-2w_\tau}$;
		\item[-] if $w$ is a finite place of $F'$ with $w\nmid p$ and $w$ not above any place in $\Sigma$, then $\pi_w$ is unramified;
		\item[-] if $w$ is a finite place of $F'$ lying above some place in $v \in \Sigma$, then $\pi_w \cong (\gamma_{v,\C} \circ \mathrm{Nm}_{F'_w/F_v} \circ \det) \otimes \mathrm{St}$;
		\item[-] for any $w|p$, $(\pi'_w)^{\mathrm{Iw}_1(w)}\ne 0$;
		\item[-] $\pi'$ is a $\overline{\chi}$-good lift of $\overline{\rho}$.
	\end{itemize}
\end{prop}

\begin{proof}

Note that if $p$ is odd, the unramified characters $\gamma_v$ are determined uniquely by $\psi$ and $\overline{\rho}$, but if $p=2$, then they are only determined up to sign. We first apply \ref{WeightChange} to obtain $\pi_1$ with the right weight and central character, and such that $(\pi_1)_w^{\mathrm{Iw}_1(w)}\ne 0$ for any $w|p$. We then apply \ref{LevelLower} to obtain $\pi_2$ that is unramified for every $v \nmid p$. We then apply \ref{LevelRaiseSt} to obtain $\pi_3$ such that $\pi_3$ is our desired twist of Steinberg at each $w$ above $v\in \Sigma$. However, $\pi_3$ may no longer be unramified at all places not above $\Sigma$ and $p$. So, we apply \ref{LevelLower} one more time to get our desired $\pi'$. Note that when applying \ref{LevelLower} this last time we can ensure that the places above those in $\Sigma$ remain the desired twist of Steinberg. \end{proof}

%% file: MainThms/Thms.tex
We can now prove our main theorem.

\begin{thm}\label{finalthm}

Let $F$ be a totally real subfield of $\overline{\Q}$.  Let $J_F$ denote set of embeddings $F \hookrightarrow \overline{\Q}$. Fix embeddings $\overline{\Q}\hookrightarrow \Qbar_2$ and $\overline{\Q}\hookrightarrow \C$. Via these embeddings we view $J_F$ as the set of embeddings $\{F \hookrightarrow \R\}$ as well as the set of embeddings $\{F \hookrightarrow \Qbar_2\}$. For any $v|2$ in $F$, let $J_{F,v}\subseteq J_F$ denote the subset of $\tau$ that give rise to $v$. We identify $J_{F,v}$ with the set of embedding $F_v \hookrightarrow \Qbar_2$.

Let
	\[ \rho: G_F \longrightarrow \GL_2(\Qbar_2)
	\]
be a continuous representation unramified outside finitely many primes. Assume there is some $(\mathbf{k},\mathbf{w}) \in J_F^2$, with $k_\tau \ge 2$ for each $\tau\in J_F$ and $ w= k_\tau + 2w_\tau$ independent of $\tau$, and such that
\begin{enumerate}
	\item $\det\rho = \phi\epsilon_p^{w-1}$, with $\phi$ a finite order character;
	\item for each $v|2$, $\rho|_{G_v} \cong \left( \begin{array}{cc} \ast & \ast \\ 
			& \chi_v \end{array} \right)$ such that viewing $\chi_v$ as a character of $F_v^\times$ via class field theory, $\chi_v(y) = \prod_{\tau \in J_{F,v}} y^{-w_\tau}$ on some open subgroup of $\mathcal{O}_{F_v}^\times$;
	\item for each choice of complex conjugation $c$, $\det\rho(c) = -1$.
\end{enumerate}

\noindent Let $\overline{\rho}: G_F \rightarrow \GL_2(\overline{\F})$ denote the residual representation associated to $\rho$. We assume
\begin{enumerate}
	\item[(4)] $\overline{\rho}$ is absolutely irreducible;
	\item[(5)] If $L/F$ is a CM extension such that $\overline{\rho}|_{G_L}$ is abelian, then there is some $v|2$ in $F$ that does not split in $L$;
	\item[(6)] there is a $2$-nearly ordinary regular algebraic cuspidal automorphic representation $\pi_0$ of $\GL_2(\A_F)$ with $\overline{\rho_{\pi_0}} \cong \overline{\rho}$.
\end{enumerate}
	
Under these assumptions $\rho$ is modular, i.e. there is a $2$-nearly ordinary regular algebraic cuspidal automorphic representation $\pi$ of $\GL_2(\A_F)$ such that $\rho \cong \rho_\pi$.

\end{thm}

\begin{proof}

First note that by solvable base change, it suffices to prove the theorem after replacing $F$ with any finite solvable totally real extension of $F$.

In the case that $\rho|_{G_v}$ is split for $v|2$, we fix $\chi_v$ as in (2). Recall that if $\rhobar$ is solvable, then there is a unique quadratic extension $L/F$ such that $\rhobar|_{G_L}$ is abelian. In the case that this $L/F$ is CM, we fix a $v_0|p$ satisfying (5). We can find a totally real solvable extension $F'/F$ of even degree such that
\begin{itemize}
	\item[-] $F'/F$ is disjoint from the fixed field of $\ker\overline{\rho}$
	\item[-] for any $v|2$ in $F'$, we have $[F'_v:\Q_2] \ge 4$;
	\item[-] for any $v|2$ in $F'$, the image of $\rhobar|_{G_v}$ has order two if $v$ is above our fixed place $v_0$, and is trivial otherwise;
	\end{itemize}

Now assume $\overline{\rho}$ is dihedral and  that $L/F$ is not CM. Let $r$ and $s$ denote the number of real and complex embeddings of $L$ into $\C$, respectively. By assumption, $r\ge 1$. Let $L_S^\mathrm{ab}$ denote the maximal pro-$p$ abelian extension of $L$ unramified outside $S$. Let $L_S^-$ denote the maximal sub-extension of $L_S^\mathrm{ab}/L$ such that the nontrivial element of $\Gal(L/F)$ act on $\Gal(L_S^-/L)$ by $-1$. We distinguish two sub-cases depending on whether or not $L$ is contained in the $2$-adic cyclotomic extension of $F$.

First assume that $L/F$ is not contained in the $2$-adic cyclotomic extension. Let $F_n$ denote the totally real subfield of the cyclotomic extension $F(\mu_{2^n})$. Set $L_n = F_nL$, and let $r_n$ and $s_n$ denote the number of real and complex embeddings of $L_n$ into $\C$, repsectively. Note that $r_n \ge [F_n:F]$. The subgroup of $\mathcal{O}_{L_n}^\times$ on which $\Gal(L_n/F_n)$ acts via the nontrivial character has rank
	\[ r_n + s_n - [F_n:\Q] = \frac{r_n}{2} \ge \frac{[F_n:F]}{2}.
	\]
The weak Leopoldt conjecture is known to hold for the tower $\{L_n/L\}_{n\ge 1}$, c.f. \cite{NSWCohomNumFields}*{Theorem 10.3.25}. Hence, letting $\overline{\mathcal{O}_{L_n}^\times}$ denote the closure of $\mathcal{O}_{L_n}^\times$ in $(\mathcal{O}_{L_n}\otimes_\Z \Z_2)^\times$ and letting $(\overline{\mathcal{O}_{L_n}^\times})^1$ denote its maximal pro-$2$ subgroup, there is a constant $c$ such that
	\[ \mathrm{rank}_\Z \mathcal{O}_{L_n}^\times
	 - \rank_{\Z_2}(\overline{\mathcal{O}_{L_n}^\times})^1 < c
	\]
for all $n$. Then, we have
	\[ \rank_{\Z_2} \Gal((L_n)_S^-/L_n) \le [F_n:\Q] -\frac{[F_n:F]}{2} + c
	\]
for all $n$. Be replacing $F$ with $F_n$ for $n$ sufficiently large, we may assume that $\rank_{\Z_2}\Gal(L_S^-/L) < [F:\Q] - 3$.

In the case that $L/F$ is contained in the $2$-adic cyclotomic extension, we use an idea of \cite{SWcorrection}. Notice that, if $L_0$ denotes the maximal abelian subextension of $L/\Q$, then $L = L_0F$. Set $F_0 \cap L_0$. Choose an odd prime $\ell$ such that $\Q(\mu_{\ell^\infty})$ is disjoint from the fixed field of $\ker\overline{\rho}$. Let $M_n$ denote the sub-extension of $\Q(\mu_{\ell^\infty})/\Q$ of degree $\ell^n$; in particular, $M_n$ is totally real. Since $L_0/F_0$ is not CM, there is a subgroup of $\mathcal{O}_{M_nL_0}^\times$ of rank at least $\ell^n-1$ on which $\Gal(M_nL_0/M_nF_0)$ acts via the nontrivial character. Since Leopoldt's conjecture is known for abelian extensions of $\Q$, the closure of this subgroup has a maximal pro-$2$ subgroup of $\Z_2$-rank at least $\ell^n-1$. The same is then true for $M_nL/M_nF$, and
	\[ \rank_{\Z_2} \Gal( (M_nL)_S^-/(M_nL)) \le [M_nF:\Q] - \ell^n+1.
	\]
Be replacing $F$ with $M_nF$ for $n$ sufficiently large we may assume that $\rank_{\Z_2}\Gal(L_S^-/L) < [F:\Q] - 3$.

Let $\Sigma$ denote the set of finite places not above $p$ at which $\rho$ is ramified. Recall that an $L$-allowable base change is an extension $F'/F$ such that
	\begin{itemize}
		\item[-] $F'/F$ is finite of even degree, solvable and totally real
		\item[-] $F'/F$ is disjoint from $\overline{F}^{\ker \overline{\rho}}$;
		\item[-] if $\rhobar|_{G_L}$ is abelian, and $v|p$ in $F$ does not split in $L$, then no $w|v$ in $F'$ splits in $F'L$.
	\end{itemize}
By replacing $F$ with an $L$-allowable base change we may assume that $[F:\Q]$ and $\Sigma$ are both even.

By replacing $F$ with an $L$-allowable base change we may assume that for any $v\in \Sigma$, the local representation $\rho|_{G_v}$ is an extension of $\gamma_v$ by $\gamma_v\epsilon_p$, for some unramified character $\gamma_v$. By further replacing $F$ with an allowable base change, \ref{SpecifiedLift} implies we may assume the same properties for $\rho_{\pi_0}$, with the same characters $\gamma_v$ for $v\in \Sigma$, as well as that $\pi_0$ has weight $(\mathbf{k},\mathbf{w})$, and that $\det\rho = \det \rho_{\pi_0}$. Moreover, because \ref{SpecifiedLift} allows us to assume that $(\pi_0)_v^{\mathrm{Iw}_1(v)}\ne 0$ for each $v|2$, writing
	\[ \rho_{\pi_0}|_{G_v}  \cong \left( \begin{array}{cc} \ast & \ast \\ 
			& \chi_v^0 \end{array} \right) \]
with $\chi_v^0$ lifting $\overline{\chi_v}$, we have $\chi_v^0(y) = \prod_{\tau\in J_{F,v}} y^{-w_\tau}$ for every $y \in \mathcal{U}_v^1$, the maximal pro-$2$ subgroup of $\mathcal{O}_{F,v}^\times$.

For each $v|2$, let $\phi_v$ denote the finite order character of $\mathcal{U}_v^1$ given by $\phi_v(y) = \chi_v(y)\prod_{\tau\in J_{F,v}}y^{w_\tau}$. We want a finite totally ramified extension $F_v'/F_v$ of order a power of $2$ such that $\phi_v\circ\mathrm{Nm}_{F_v'/F_v}$ is trivial. However, if $v$ does not split in $L$, we further require this extension is disjoint from $L_v/F_v$, where $L_v$ denotes the completion of $L$ at the unique prime above $v$. We require this in order to ensure the resulting base change still satisfies condition (5) of the theorem. If $L_v/F_v$ is unramified, this is automatic. In the case that $L_v/F_v$ is ramified, there is a choice of uniformizer $\varpi_v$ in $F_v$ such that $\varpi_v$ is not a norm from $L_v$. By class field theory, the subgroup of $F_v^\times$ generated by the kernel of $\phi_v$, the group of prime-to-$2$ roots of unity, and the element $\varpi_v$ then gives the desired extension. We fix such an $F_v'/F_v$ for each $v|p$. Let $F'/F$ be a finite solvable, totally real extension of even degree, disjoint from $\overline{F}^{\ker \overline{\rho}}$, such that for any $v|2$ in $F$ and $w|v$ in $F'$, the completion of $F'$ at $w$ is our fixed extension $F_v'/F_v$. Then, $F'/F$ is $L$-allowable, and replacing $F$ with $F'$, we can now assume that for every $v|2$ in $F$ and every $y \in \mathcal{U}_v^1$, we have $\chi_v(y) =  \prod_{\tau\in J_{F,v}} y^{-w_\tau}$.

Let $E$ be a finite sub-extension of $\Qbar_2/\Q_2$ that contains the image of $\rho$ and $\rho_{\pi_0}$, and let $\mathcal{O}$ denote its ring of integers. For each $v|2$, recall that $G_v^\mathrm{ab}$ denotes the abelianization of $G_v$, $I_v^\mathrm{ab}$ denotes it inertia subgroup, and $G_v^\mathrm{ab}(2)$ and $I_v^\mathrm{ab}(2)$ denote their respective maximal pro-$2$ quotients. Recall also, the notation $\Lambda(G_v) = \mathcal{O}[[G_v^\mathrm{ab}(2)]]$ for $v|2$, and $\Lambda(G_2) = \hat{\otimes}_{v|2} \Lambda(G_v)$. By our choice of base changes, for every $v|2$, the restrictions of $\chi_v$ and $\chi_v^0$ to the maximal pro-$2$ subgroup of $I_v^\mathrm{ab}$ are equal. In particular, the restrictions of $\chi_v$ and $\chi_v^0$ to the $2$-torsion subgroup of $G_v^\mathrm{ab}$ are equal. Viewing $\chi_v$ and $\chi_v^0$ as lifts of $\overline{\chi_v}$ for each $v|2$, this implies that the $\mathcal{O}$-valued points of $\Lambda(G_2)$ determined by $(\chi_v)_{v|2}$ and $(\chi_v^0)_{v|2}$ lie on the same irreducible component.

Let $S = \Sigma\cup \{v|p\} \cup \{v|\infty\}$. Using the Jacquet-Langlands-Shimizu correspondence, the properties of $\pi_0$ allow us to transfer $\pi_0$ to a nearly ordinary eigenform $f$ as in \ref{HeckeAssumptions}. The field $F$ and residual representations satisfy the assumptions of \ref{ResRepAssumptions} as well as \ref{ProMod}. Then letting $\overline{R}_{F,S}^\psi$ be the quotient of $R_{F,S}\hat{\otimes}\Lambda(G_p)$ as in \ref{DefTheoryAssumptions}, we see that $(\rho,(\chi_v)_{v|2})$ defines a point of $\Spf \overline{R}_{F,S}^\psi$. By \ref{ProMod}, the representation $\rho$ arrises as a specialization of the big modular Galois representation in \ref{BigGalRep} via an $\mathcal{O}$-algebra homomorphicm $\mathbf{T}_\psi(U) \rightarrow \Qbar_2$ (for appropriate $U$). By our assumptions on the determinant and local behaviour of $\rho$, the kernel of this specialization is an arithmetic prime, and it factors through $\mathbf{T}_{\kappa,\psi}^\mathrm{no}(U(p^{a,a}),\mathcal{O})$ for some $a\ge 1$ by \ref{HorControl}. The theorem now follows from the Jacquet-Langlands-Shimizu correspondence \ref{JacquetLanglands}.
\end{proof}

Our theorem from the introduction is a result of \ref{finalthm} above together with \ref{OrdinaryLift}.

%% file: MainThms/Examples.tex
We show how the main theorem implies the corollary from the introduction and conclude by giving some examples of elliptic curves satisfying the assumptions of the corollary from the introduction. For any finite place $v$ of $F$, we normalize the valuation $\val_v$ at $v$ so that a uniformizer has valuation $1$.

\subsubsection{Proof of the main corollary}\label{MainCorProof}

Let $E$ be an elliptic curve over $F$ given by the the Weierstrass equation
	\[ y^2 = x^3 + ax + b.
	\]
The curve $E$ has $2$-torsion defined over $F$ if and only if $x^3 + ax + b$ is reducible. Further, the $G_F$-representation $E[2](\Qbar)$ is absolutely irreducible if and only if the splitting field of $x^3 + ax +b$ is an $S_3$-extension, which happens if and only if $x^3+ax+b$ is irreducible and the discriminant 
	\[\Delta = -16(4a^3 + 27b^2)
	\]
of the Weierstrass equation is not a square in $F$. If the action of $G_F$ on $E[2](\Qbar)$ is absolutely irreducible, then the unique quadratic extension $L/F$ for which the action of $G_L$ on $E[2](\Qbar)$ is abelian is $F(\sqrt{\Delta})$. Lastly, $E$ has multiplicative reduction at $v|2$ if and only if $j_E$ is non-integral at $v$, and has potentially ordinary reduction if and only if $j_E$ is a unit at $v$, c.f. \cite{SilvermanEllCurves}*{Chapter V, Exercise 5.7}. The corollary now follows from the main theorem. \hfill \qed

Note that any elliptic curve satisfying our assumptions is not CM. If it were, then there would be a quadratic CM extension $L/F$ such that the representation $\rho_E|_{G_L}$ is abelian.  By the absolute irreducibility of $E[2](\Qbar)$, we have $L = F(\sqrt{\Delta})$, and $\Delta$ is totally negative. The reducibility of $\rho_E|_{G_v}$ for each $v|2$ then implies that every $v|2$ splits in $F(\sqrt{\Delta})$, a contradiction.

\subsubsection{An example}\label{TheExamples}

We give some examples of elliptic curves satisfying the assumptions of the corollary. Let $j_0 \in F$ be any real number satisfying 
\begin{enumerate}
	\item[(a)] $\mathrm{val}_v(j_0) \le 0$ for all $v|2$;
	\item[(b)] there is some $v\nmid 6$ such that $\val_v(j_0) =1$;
	\item[(c)] $j_0-1728$ is not a square in $F$;
	\item[(d)] either
		\begin{itemize}
		\item[-] there is some $\sigma: F\hookrightarrow \R$ such that $\sigma(j_0) > 1728$, or
		\item[-] there is some $v|2$ such that $\mathrm{val}_{v}(j_0)$ is odd.
	\end{itemize}
\end{enumerate}
Note that condition (c) can be ensured by assuming either that there is some real embedding $\sigma$ such that $\sigma(j_0) < 1728$, or by assuming there is some finite place $v$ such that $\val_v(j_0-1728)$ is odd.

We will show that any elliptic curve $E$ over $F$ with $j(E) = j_0$ satisfies the conditions of the Corollary. Note that if $E$ and $E'$ are non-CM elliptic curves over $F$ with the same $j$-invariant, then they are quadratic twists of each other. Since a quadratic twist has no effect on $2$-torsion, it suffices to show one elliptic curve with $j$-invariant $j_0$ satisfies the hypotheses of the corollary.

Let $E_0$ be the elliptic curve over $F$ given by the Weierstrass equation
	\[ y^2 = x^3 - \frac{27j_0}{4(j_0-1728)}x - \frac{27j_0}{4(j_0-1728)}.
	\]
One computes that the $j$-invariant of this elliptic curve is equal to $j_0$, and that the discriminant $\Delta$ of the cubic in the above Weierstrass equation is
	\[
	\Delta = -16\left(\left(- \frac{27j_0}{4(j_0-1728)}\right)^3 + 
	27\left(- \frac{27j_0}{4(j_0-1728)}\right)^2\right)
	= \frac{2^6\cdot 3^{12} \cdot j_0^2}{(j_0-1728)^3},
	\]
which is not a square in $F$ since $j_0-1728$ is not a square in $F$. Since there is some finite $v\nmid 6$ such that $\val_v(j_0) =1$, we have $\val_v(j_0-1728) = 0$ and $\val_v(-\frac{27j}{4(j_0-1728)})=1$. Then,
	\[  x^3 - \frac{27j_0}{4(j_0-1728)}x - \frac{27j_0}{4(j_0-1728)}
	\]
is irreducible by the Eisenstein criterion at $v$. Lastly, if there is a real embedding $\sigma$ of $F$ such that $\sigma(j_0) > 1728$, we have
	\[ \sigma(\Delta) = \sigma\left(\frac{2^6\cdot 3^{12} \cdot j_0^2}{(j_0-1728)^3}\right) > 0,
	\]
and $\Delta$ is not totally negative. If there is some $v|2$ such that $\val_v(j_0)$ is odd, then since it is also less than than or equal to zero, $\val_v(j_0-1728) = \val_v(j_0)$. Then, $j_0-1728$ is not a square in $F_v$, and neither is $\Delta$.